\documentclass[10pt,letterpaper,reqno]{amsart}
\usepackage{geometry}
\usepackage{graphicx} 
\usepackage{tikz-cd}
\usepackage{amssymb,amsmath,amsthm}
\usepackage{euscript}
\usepackage{hyperref}
\usepackage{aliascnt}
\newtheorem{theorem}{Theorem}[section]

\theoremstyle{proposition}
\newtheorem{proposition}[theorem]{Proposition}
\theoremstyle{corollary}
\newtheorem{corollary}[theorem]{Corollary}

\theoremstyle{definition}
\newtheorem{definition}[theorem]{Definition}
\newtheorem{example}[theorem]{Example}

\theoremstyle{remark}
\newtheorem{remark}[theorem]{Remark}
\theoremstyle{question}
\newtheorem{question}[theorem]{Question}
\numberwithin{equation}{section}

\newcommand{\bfP}{\mathbf{P}^1}
\newcommand{\bfA}{\mathbb{A}_{\mathbb{C}}^1}
\newcommand{\an}{\mathrm{an}}
\newcommand{\Map}{\underline{\mathrm{Map}}}
\newcommand{\Perf}{\underline{\mathrm{Perf}}}
\newcommand{\RPerf}{\mathbb{R}\underline{\mathrm{Perf}}}
\newcommand{\RCoh}{\mathbb{R}\underline{\mathrm{Coh}}}
\newcommand{\Sect}{\mathbb{R}\underline{\mathrm{Sec}}}
\newcommand{\ASect}{\mathbb{R}\underline{\mathrm{AnSec}}}
\newcommand{\AnMap}{\underline{\mathrm{AnMap}}}
\newcommand{\AnPerf}{\underline{\mathrm{AnPerf}}}
\newcommand{\RAnPerf}{\mathbb{R}\underline{\mathrm{AnPerf}}}
\newcommand{\RAnCoh}{\mathbb{R}\underline{\mathrm{AnCoh}}}
\newcommand{\Dol}{\mathrm{Dol}}
\newcommand{\DR}{\mathrm{DR}}
\newcommand{\Hod}{\mathrm{Hod}}
\newcommand{\Del}{\mathrm{Del}}
\newcommand{\B}{\mathrm{B}}
\newcommand{\AG}{[\mathbb{A}^1/\mathbb{G}_m]}

\newcommand{\Q}{\mathsf{QCoh}}
\newcommand{\Per}{\mathsf{Perf}}

\numberwithin{equation}{section}

\title[Lagrangian structures, NAH-theory and Twistor Geometry]{Lagrangian correspondences of nonabelian Hodge type and shifted twistor structures}


\author{Jacob Kryczka}
\address{School of Mathematics, Harbin Institute of Technology (HIT), Harbin 150001, China \&
Beijing Institute of Mathematical Sciences and Applications (BIMSA), Huairou District, Beijing 101408, China.
}
\email{jkryczka@hit.edu.cn}

\author{Yuuji Tanaka}
\address{Beijing Institute of Mathematical Sciences and Applications (BIMSA), No. 544, Hefangkou
Village, Huaibei Town, Huairou District, Beijing 101408, China
}
\email{ytanaka@bimsa.cn}

\author{Shing--Tung Yau}
\address{Yau Mathematical Sciences Center, Tsinghua University, Haidian District, Beijing,
China\\
}
\email{styau@tsinghua.edu.cn}

\date{29 June 2026}

\begin{document}

\maketitle

\begin{abstract} 
Classical nonabelian Hodge theory identifies Dolbeault and de Rham moduli spaces by providing a real-analytic isomorphism. In this paper, motivated by the Kapustin--Witten theory, we study 
this correspondence in the more general framework of perfect complexes on proper
varieties, paying special attention to the surface case.  We establish a Lagrangian correspondence which relates the shifted symplectic geometries by Pantev--To\"en--Vaqui\'e--Vezzosi (PTVV) between the derived stacks of flat and Higgs perfect complexes.

Furthermore, we investigate the existence of the derived twistor structure of hyperkähler type on the moduli stack of perfect complexes endowed with $\lambda$-connections by Deligne--Hitchin--Simpson.  
We establish a version of the AKSZ/PTVV transgression, Lagrangian intersection, and (hyperk\"ahler) symplectic reduction theorems in this context.  
Moreover, we prove that the derived Riemann--Hilbert correspondence of Porta and Holstein--Porta, which states an equivalence of derived analytic stacks of perfect complexes on $X_{\mathrm{Betti}}$ and $X_{\mathrm{DR}}$,  is compatible with the natural shifted--symplectic structures.

We then study the relation between the shifted twistor structure and the shifted symplectic forms on the fibers, and prove that the analytic Deligne--Hitchin--Simpson moduli stack on a smooth projective variety $X$ has a canonical $2(1-\dim X)$ shifted twistor structure over $\mathbb{P}^1_{\mathbb{C}}$, a result which has been anticipated for some time. 
 In particular, the moduli stack of solutions to the Kapustin--Witten equations modulo gauge equivalence on a smooth proper complex algebraic surface exibits a $(-2)$-shifted (pre-)twistor structure as a family over
$\mathbb{P}^1_{\mathbb{C}}$. 
\end{abstract}
\tableofcontents

\newpage 

\section{Introduction}
In this paper, we initiate the study of hyperk\"ahler structures in derived algebraic and analytic geometry.
\subsection{Overview}
\label{OverviewIntro}
Our main results can be summarized as follows: 
\begin{enumerate}
\item[\hypertarget{(1.)}{1.}] 
We define \textit{derived analytic (pre)-twistor spaces} and study their formal properties. Roughly speaking, they are locally geometric relative derived $\mathbb{C}$-analytic stacks endowed with the action of the Galois group $\mathrm{Gal}(\mathbb{C}/\mathbb{R})$ which encodes a real structure, together with a relative shifted symplectic structure invariant for the action. This requires a detailed study of shifted symplectic structures on derived analytic stacks, modelled as in \cite{HolsteinPorta2025} (see references therein).  
Using this formalism, we answer questions of Katzarkov--Pandit--Spaide \cite[Problems~6.5,6.6]{KPS2021} on their interaction with fundamental constructions in shifted symplectic geometry \cite{PTVV13}:
\begin{itemize}
    \item We prove three new existence results for shifted twistor families which are of use in constructing new examples of such hyperk\"ahler manifolds. We prove twistor structures are compatible with: the derived Alexandrov--Kontsevich--Schwarz--Zaboronsky (AKSZ) construction (Thm.~\ref{MainTheorem2}), forming derived (hyper-)Lagrangian intersections (Thm.~\ref{MainTheorem3}) and derived symplectic reduction\footnote{Informally, suppressing the adjective “derived”, the latter may be summarized as Hyperk\"ahler reduction$:=$ Symplectic reduction of a twistor family of hyperk\"ahler type.} along connected complex analytic reductive groups $G$ (Thm.~\ref{MainTheorem4}).   

    \item  We study these structures in relation to various moduli problems coming from nonabelian Hodge theory and its relation with Kapustin--Witten gauge theory explored by the second author with Rayan and Liu \cite{LRT2022}.
\end{itemize}
   \item[\hypertarget{(2.)}{2.}] 
    We apply the shifted twistor formalism to study the Hodge stack of perfect complexes with $\lambda$-connections on a smooth projective complex variety $X$ with parameter $\lambda \in\mathbb{A}^1$, introduced by Simpson \cite{Sim09}. 
   Recall the derived stack of perfect complexes is the mapping stack
\begin{equation}
    \label{eqn: RPerfIntro}
\RPerf(X):=\underline{\mathrm{Map}}_{}(X,\underline{\mathrm{Perf}}),
\end{equation}
where $\underline{\mathrm{Perf}}$ is the (locally geometric\footnote{A derived stack is locally geometric if it the union of open substacks
that are geometric.}) classifying stack of perfect complexes \cite{toenvaquie2007} and $\underline{\mathrm{Map}}$ is taken in the category of derived \emph{algebraic} stacks. More generally, one may consider the relative derived mapping stack 
$$\RPerf(X/S):=\underline{\mathrm{Map}}_{/S}(X,\underline{\mathrm{Perf}}\times S),$$ of perfect complexes on the fibers of a smooth morphism $\pi:X\to S$ of projective varieties. 

In this paper, we are mainly concerned with the analytic version $\AnPerf$ of $\Perf.$ Its functor of points sends a derived analytic space $W$, in the sense of \cite{LurieDAGV,PortaYuDerivedGAGA2019}, to the stable $\infty$-category $\mathsf{Perf}(W)$ of perfect complexes on $W.$ It is a locally geometric derived analytic stack \cite[Thm.~1.2]{HolsteinPorta2025}, related to \eqref{eqn: RPerfIntro} via the functor of analytification $(-)^{\an}$, whose definition and properties are given in \cite[\S.~3]{HolsteinPorta2025}. 

\begin{remark}[Notations]
Given a (higher) moduli stack $\mathcal{M}$ e.g. of coherent sheaves, we write $M$ for its underlying coarse scheme, and $\mathbb{R}\underline{\mathcal{M}}$ for a derived enhancement \cite{STV}, and $t_0(\mathbb{R}\underline{\mathcal{M}})$ for its classical truncation. If it exists, we write $\mathbb{R}\underline{\mathrm{An}\mathcal{M}}$ for the derived analytic analog. Sometimes it arises via analytification $\mathbb{R}\underline{\mathcal{M}}^{\an}$ of $\mathbb{R}\underline{\mathcal{M}}.$
In this paper, we mainly work with perfect complexes $\Perf_{\mathbb{C}}$ over $\mathbb{C}$. By \cite[Thm.~ 1.2 (1)]{HolsteinPorta2025} we have
$\Perf^{\an}_{\mathbb{C}}\simeq \AnPerf_{\mathbb{C}}.$
We will sometimes use these notations interchangeably.
\end{remark}
Let $X$ be a smooth $\mathbb{C}$-analytic
space. Following ideas originating with Simpson \cite{Simpson1994ModuliI,Simpson1994ModuliII}, elaborated in detail and in greater generality by Porta--Sala \cite{PortaSala}, various aspects of the geometry of $X$ (e.g. topological or holomorphic) can be encoded by so-called \emph{Simpson shape operations} e.g. Betti $X_{\mathrm{B}}$, Dolbeault $X_{\mathrm{Dol}}$ and de Rham $X_{\mathrm{DR}}$ shapes.

The \emph{Hodge shape} $X_{\mathrm{Hod}},$ which lives over $\Theta:=[\mathbb{A}^1/\mathbb{G}_m]$ where $\mathbb{G}_m$ acts on $\mathbb{A}^1$ with weight $1$, is defined as the deformation to the normal bundle construction along the canonical map $X\to X_{\DR}$ (see e.g. \cite[§5.~1.3]{GR17b}). It interpolates between $X_{\mathrm{Dol}}$ and $X_{\mathrm{DR}}$ and geometrically, it encodes the Hodge filtration in de Rham cohomology.
According to \cite{Sim09}, the locally geometric derived stack of $\lambda$-complexes is the family
\begin{equation}
    \label{RPerfHodIntro}
\RPerf_{\mathrm{Hod}}(X):=\Map_{/\Theta}(X_{\Hod},\Perf\times\Theta)\to\Theta.\end{equation}
By mapping out of Simpson's shapes as in \eqref{eqn: RPerfIntro}, one obtains the derived moduli stack of
flat perfect complexes 
\begin{equation}
    \label{RPerfDRIntro}
\RPerf_{\DR}(X):=\RPerf(X_{\mathrm{DR}}),
\end{equation}
and perfect complexes with Higgs fields
\begin{equation}
    \label{RPerfDolIntro}
\RPerf_{\Dol}(X):=\RPerf(X_{\Dol}).
\end{equation}
The latter possess interesting open substacks e.g. of $\mathcal{O}_X$-perfect complexes of Higgs sheaves whose cohomology sheaves are locally free, semistable and with vanishing Chern classes, of importance in nonabelian Hodge theory \cite{SiNonabelian,SimpsonHodgeFiltrationNonabelian}. 

Both derived moduli stacks \eqref{RPerfDRIntro} and \eqref{RPerfDolIntro} possess shifted symplectic structures \cite{PTVV13}, and in this paper we study shifted-symplectic aspects of the nonabelian Hodge correspondence in the analytic category via shifted twistor geometry;
    \begin{itemize}
    \item We prove via pullback along $\mathbb{A}^1\to \Theta$, that the analytic Hodge stack \eqref{RPerfHodIntro} is part of a relative shifted Lagrangian correspondence between the \emph{analytifications} \cite{HolsteinPorta2025} of the derived stacks \eqref{RPerfDRIntro} and \eqref{RPerfDolIntro}(Thm.~\ref{MainTheorem}).
\end{itemize}

\item[\hypertarget{(3.)}{3.}] The final main result of this paper is Thm.~\ref{MainTheorem5}. Using derived analytic techniques, it proves the existence of a locally geometric relative derived analytic stack, realizing the original method sketched by
 Simpson \cite{Sim09} to construct Hitchin's twistor space of perfect complexes. The original formulation predated the crucial ingredient of the derived Riemann--Hilbert correspondence \cite{porta2017derived}). Even before the crucial developments of derived analytic geometry, it was expected that a derived version of Hitchin's space has a derived twistor structure \cite{KPS2021}. Our proof presents a first step towards a global derived NAH correspondence (see Rem.~\ref{DNCRemark} for further comments).

\begin{itemize}
    \item We first prove the derived Riemann--Hilbert (RH) correspondence \cite{porta2017derived} with values in a suitable geometric derived stack (made precise in \cite[Cor.~7.6]{HolsteinPorta2025}), which is formulated as an equivalence of derived analytic mapping stacks (see \eqref{DerRHIntro} below), is compatible with the natural shifted symplectic structures (Thm.~\ref{AnalyticNAHIntro}). This result has important implications for shifted-symplectic aspects of nonabelian Hodge theory for perfect complexes,  

    \item We construct the derived analytic Deligne--Hitchin--Simpson stack of perfect complexes on a $\mathbb{C}$-analytic variety $X$ over $\bfP$, and prove via the twistor-formalism it has a shifted pretwistor structure (Thm.~\ref{MainTheorem5}).
\end{itemize}
      \end{enumerate}  
Our construction extends one given in \cite{FrancoHanson2024}, and we are mainly interested in its relative shifted-symplectic geometry over $\mathbb{P}^1.$
Finally, let us mention a conjectural derived \emph{Twistor Geometric Langlands Conjecture} (GLC) (see, \cite[eqn.~1.5]{Hanson2025}), interpolating Dolbeault \cite{DonagiPantev2012} and Betti Geometric Langlands \cite{BZN2018}.
Such a twistor GLC incorporates the Hodge filtrations of the de Rham
theory, similarly to what is done in the setting of classical limit categories in Dolbeault GLC \cite{PT2025}, but restricted to a rigid hyperk\"ahler class
of objects with physical significance in four-dimensional super Yang--Mills e.g. stacky BAA-branes and BBB-branes \cite{Gaiotto2018} (compatible with Hodge deformations). 
\begin{remark}[Related work]
    While finalizing this draft, the interesting preprint \cite{DDX2026}  sharing a similar title appeared. Their Lagrangian correspondence is non-derived, and over compact connected Riemann surfaces. Interestingly, the authors announce 
    their expectations that the Dolbeault GLC
is generically realized by the Fourier transforms induced by their Lagrangian correspondences for
Hitchin moduli spaces \cite[Conj.~1.5]{DDX2026}, and the de Rham GLC is generically realized by the corresponding quantized
Fourier transforms \cite[Conj.~1.6]{DDX2026}. A comparison will appear elsewhere.
\end{remark}

\subsection{Background}
Every hyperk\"ahler manifold is holomorphic symplectic; conversely, by the Calabi--Yau theorem \cite{Yau1978} any compact holomorphic symplectic manifold admits a hyperk\"ahler metric in each K\"ahler class. This deep correspondence connects differential-geometric extremal problems to algebro-geometric and categorical stability properties \cite{Dona83, Do, Dona87, UY}. It plays a central role in Mirror Symmetry \cite{Kontsevich1994HMS, StromingerYauZaslow1996} and for Bridgeland stability \cite{Br},\cite{BLMNPS2021}.

Compared with holomorphic symplectic manifolds, examples of hyperk\"ahler manifolds are generally harder to come by; examples are Hilbert--Douady schemes $S^{[n]}$ of $n$-points on a $K3$ surface $S$ and generalized Kummer varieties associated to abelian surfaces \cite{Beauville1983}, the Kirwan blowup of moduli of semistable torsion-free sheaves of rank two $\mathcal{M}_{(2,0,4)}(S)$ with $c_1=0,c_2=4$ on a $K3$ surface $S$ \cite{OGrady} (or abelian surface \cite{OGrady2003}).

Therefore it is important to find methods to discover new examples or deduce when a smooth projective variety is hyperk\"ahler or not. As an example of the latter, if $X$ and $Y$ are derived equivalent smooth
projective varieties then $X$ is hyperk\"ahler if and only if $Y$ is hyperk\"ahler \cite{HuybrechtsNieperWisskirchen2011}.

While holomorphic symplectic geometry has seen substantial advancements via the theory of shifted symplectic and Poisson structures on derived Artin stacks \cite{PTVV13}, later extended to the relative \cite{CalaqueHaugsengScheimbauer2025} and non-geometric settings \cite{calaque2024shifted}, there is at present no satisfactory derived analogue of a hyperk\"ahler structure, despite its central role in moduli theory.  

In the present paper, we take up the fundamental question:

\begin{question}
\label{Question1}
What is the appropriate notion of a hyperk\"ahler structure on a derived stack?
\end{question}

Motivated in part by the relationship between Kapustin--Witten gauge theory \cite{KapustinWitten2007} and nonabelian Hodge theory explored in \cite{LRT2022}, we propose an approach of both an algebraic and analytic nature to construct hyperk\"ahler structures on derived stacks in terms of their twistor-theoretic description, summarized in \S~\ref{HKViaTwistorIntro}. Employing well-known results in derived symplectic geometry, such as the AKSZ/PTVV transgression and derived Lagrangian intersection theorems \cite{PTVV13,CPTVV17}, we can construct new hyperk\"ahler structures with relative ease compared to classical approaches. 


We now briefly recall the classical nonabelian Hodge (NAH) and Riemann--Hilbert (RH) correspondences, together with the hyperk\"ahler geometry of the associated Deligne--Hitchin--Simpson twistor family; the primary objective of this paper is to study these structures from the derived analytic viewpoint.

\subsection{Recollections}
Let $G$ be a reductive algebraic group and let $\Sigma$ be a compact Riemann surface. The moduli space $\mathcal{M}_{\mathrm{Dol}}^{\mathrm{ss}}(\Sigma,G)$ of semistable $(G-)$Higgs bundles with vanishing rational Chern classes carries a natural hyperk\"ahler metric defined by the Hitchin equations \cite{Fujiki1991}, whose symplectic geometry is governed by the Hitchin fibration \cite{Hitchin1987}. Its regular locus is an algebraic completely integrable system admitting a semiflat metric \cite{Freed1999}. However, classically, the hyperk\"ahler metrics are generally defined only over smooth loci, with the full moduli space (highly) singular and non-compact,  particularly in higher dimensions. Extending the metric near singularities e.g. in the large Higgs field limit, leads to substantial analytic difficulties \cite{GMN2013}, such as the failure of metric perturbation methods to achieve transversality.

Let $X$ be a smooth projective variety of dimension $d.$ The nonabelian Hodge correspondence \cite{SiNonabelian, SimpsonHodgeFiltrationNonabelian} offers a uniform study relating the Betti $\mathcal{M}_{\mathrm{B}}(X),$ de Rham $\mathcal{M}_{\mathrm{DR}}(X)$ and Dolbeault stacks $\mathcal{M}_{\mathrm{Dol}}^{\mathrm{ss}}(X)$. The Riemann--Hilbert (RH) correspondence further yields a complex analytic isomorphism
\begin{equation}
    \label{RHIntro}
\mathrm{RH}:\mathcal{M}_{\mathrm{B}}(X)\xrightarrow{\simeq}\mathcal{M}_{\mathrm{DR}}(X),
\end{equation}
while the existence of pluriharmonic metrics yields a real analytic isomorphism of good moduli spaces $M_{\mathrm{DR}}(X)$ and $M_{\mathrm{Dol}}(X)$ \cite{alper2013good}. Although nonalgebraic, \eqref{RHIntro} is compatible with the Atiyah--Bott \cite{Jeffrey1997} and Goldman symplectic form on character varieties \cite{Goldman1984}.

These moduli spaces arise as fibers of a single \emph{twistor} family \cite[Thm.~4.1]{SimpsonHodgeFiltrationNonabelian}; a complex analytic $\mathbb{P}_{\mathbb{C}}^1$-family of K\"ahler manifolds, in fact a Poisson variety with holomorphic symplectic fibers,
\begin{equation}
    \label{DeligneSimpsonIntro}
\eta:\mathcal{M}_{\mathrm{Del}}^{\mathrm{ss}}(X,G)=\mathrm{Tw}\big(\mathcal{M}_{\mathrm{Dol}}^{\mathrm{ss}}(X,G)\big)\rightarrow\mathbb{P}_{\mathbb{C}}^1
\end{equation}
interpolating between the $\mathcal{M}_{\mathrm{Dol}}(X,G)$ (fiber over $\lambda=0$) and $\mathcal{M}_{\mathrm{DR}}(X)$ (fiber over $\lambda=1$). For $\lambda \neq 0,$ one has Deligne's moduli space parametrizing semistable $\lambda$-connections. For a review see \cite{Garcia2015introduction}, or \cite{Huang2020} for details of this classical correspondence.
Informally, it may be summarized by the diffeomorphism
\begin{equation}
    \label{NAH}
    \mathrm{NAH}:\mathcal{M}^{\mathrm{ss}}_{\mathrm{Dol}}(X,n)\to \mathcal{M}^{\mathrm{}}_{\mathrm{DR}}(X,n),\hspace{5mm}
    (E,\overline{\partial}_E,\theta)\mapsto (E,\overline{\partial}_E+\theta^{\dagger_h}+\partial_{E,h}+\theta),
\end{equation}
where $h$ is the unique pluri-harmonic metric of $(E,\overline{\partial}_E,\theta).$

For compact K\"ahler surfaces this provides a correspondence between the moduli spaces of solutions to the equations in Kapustin--Witten gauge theory \cite{KapustinWitten2007}, at $t = 0$ and $t \in \mathbb{R} \setminus \{0\}$, as described in \cite{LRT2022}. We recall this in \S~\ref{sec: KWReview}.

The moduli spaces related by \eqref{NAH}, thus also the Kapustin--Witten moduli space $\mathcal{M}_{\text{KW}}$ (with parameter $t$), are fibers of the hyperk\"ahler geometry fibered over the $\lambda$-line in the twistor space \eqref{DeligneSimpsonIntro}. 
As pointed out in \cite{LRT2022}, these moduli spaces would be, in suitable form, diffeomorphic to moduli spaces of $\lambda$-connections, providing an explicit gauge-theoretic realization of the nonabelian Hodge correspondence by the identification 
$\lambda\longleftrightarrow t$.

From the viewpoint of derived geometry, the moduli stack of perfect complexes of local systems on compact oriented four‑manifolds carries a natural $(-2)$-shifted symplectic structure. This is viewed as corresponding to the $t\neq 0$ regime of Kapustin--Witten gauge theory and suggests the full moduli space of solutions should admit a uniform shifted symplectic derived enhancement. Indeed, when the underlying four‑manifold is a complex projective surface, such a structure is known to exist, and the orientation problem for this moduli space at $t=0$ has been resolved \cite[Thm.~4.9]{JoyceTanakaUpmeier2020}, with the same conclusion extending to $t\neq 0$.
\subsubsection{This paper}

In this paper, we propose replacing classical analytic methods used e.g. in proving existence of harmonic metrics, with cohomological tools using ideas of derived algebraic \cite{TV2} and analytic geometry \cite{Lurie2011}. In this way, we recast nonlinear elliptic PDE problems in formal and categorical terms. This turns out to be fruitful and necessary for a number of other PDE-related problems (see e.g. \cite{Pardon23},\cite{Ste24},\cite{KSY23},\cite{KSY2},\cite{PoTe24},\cite{HPT2024}).

In fact, such a perspective is already implicit in observations of Deligne, Hitchin \cite{HKLR87} and Simpson. Namely, the twistor construction encodes the nonabelian Hodge correspondence in a purely geometric manner, bypassing much of the analytic input\footnote{Note that the homeomorphism $M_{\mathrm{DR}}(X,G)\simeq M_{\mathrm{Dol}}(X,G)$ comes from the real‑analytic harmonic metric construction, and the corresponding twistor space trivialization $\mathrm{Tw}(M^{\mathrm{smth}}(X,G))\simeq M^{\mathrm{smth}}(X,G)\times\mathbf{P}^1$, which is a $C^{\infty}$-isomorphism over smooth points $M_{\mathrm{Dol}}^{\mathrm{smth}}$, also depends on harmonic metrics \cite{SimpsonHodgeFiltrationNonabelian}.}. 

As discussed in \S~\ref{OverviewIntro} above, our approach is a formal one based on Simpson’s shape operations (see \cite{PortaSala}), which have favorable functorial properties extending to various $\infty$-toposes of derived stacks \cite{Lur09}. Thanks to the series of works due to Holstein, Porta, Sala, Yu, and others e.g. \cite{PortaYuHigherGaga2016,PortaYuDerivedGAGA2019,PortaYu2020,HolsteinPorta2025}, many properties of these operations are well understood, and their behaviour in relation with analytification is known for large classes of geometric objects useful for applications e.g. locally geometric derived algebraic stacks. This simplifies our situation greatly, since by working with $\Perf$, many properties can be transported between algebraic and analytic settings, subject to some other natural properties. This abstract approach permits the extension of hyperk\"ahler metrics beyond smooth loci to the full singular moduli stack, when treated as an appropriately derived object in the language of stable $\infty$-categories \cite{Lur17}.

\begin{remark}
    Even if we work over underived i.e. classical schemes, language of higher category theory is unavoidable. Indeed, if $X$ is a smooth proper variety over $\mathbb{C}$, the Betti shape is naturally a $\mathsf{Spc}$-valued sheaf on the \'etale-site of $X,$ since it encodes the full topological homotopy-type of $X$.  
\end{remark}
Throughout this paper, when $X$ is a classical scheme, we always assume it is proper over $\mathbb{C}.$ A choice of fundamental class $[X]$ in algebraic de Rham cohomology (resp. Dolbeault cohomology) then determines a $2(1-d)$-shifted symplectic structure on $\RPerf_{\DR}(X)$ (resp.  $\RPerf_{\Dol}(X)$) (see \cite[Cor.~2.13]{PTVV13}). These symplectic structures are the derived extensions of the Atiyah--Bott and Goldman symplectic forms.

\begin{remark}[Non-proper case]
In recent work \cite{PantevToen2021,PantevToen2022} (see also \cite{Tom2026}), the derived stack of flat complexes on a general smooth algebraic variety was studied. In this case, the derived stacks are not representable, although formally representable at field‑valued point. Their tangent complexes are infinite-dimensional which poses challenges for non‑degeneracy conditions \cite[Sect.~1.1]{KSY2}. This setting necessitates a more involved ind-pro (Tate) framework \cite{Hennion2017}. Throughout this paper we mainly work in the proper setting.
\end{remark}

We now introduce the main twistor-theoretic structures studied in this paper, giving further motivations from nonabelian Hodge and gauge theory. We state the corresponding results in Subsect.~\ref{ssec: MainResultsPartAIntro} and Subsect.~\ref{StatementofResults}, respectively.

\subsection{Hyperk\"ahler structures via shifted twistor spaces} 
\label{HKViaTwistorIntro}
The starting point is the classical fact that a hyperk\"ahler structure can be encoded in purely holomorphic data via Penrose’s twistor
geometry \cite[Thm.~3.3]{HKLR87}. In the classical picture, a hyperk\"ahler manifold $M$ gives rise to a holomorphic fibration
$ \pi : Z \to \mathbb{P}^1$, 
whose fibers parametrize the different complex structures on $M$. Moreover, $Z$ carries a holomorphic symplectic structure varying quadratically along $\mathbb{P}^1$, and a real structure covering the antipodal map. 

Our goal is to extend this paradigm to derived geometry. Instead of working with smooth manifolds and classical moduli spaces (such as the Deligne--Hitchin space), we consider derived stacks and encode hyperk\"ahler structures entirely in terms of their \emph{twistor geometry}, following and extending ideas\footnote{This notion is due to \cite{KPS2021}, but predated the necessary mathematical preliminaries on relative shifted symplectic derived algebraic geometry \cite{CalaqueHaugsengScheimbauer2025, calaque2024shifted} and derived analytic geometry \cite{PortaYu2020, HolsteinPorta2025}.} of Katzarkov--Pandit--Spaide \cite{KPS2021}. 

In this approach, one replaces metric data by holomorphic and homotopical structures, as we now explain in further detail. 

\subsubsection{Shifted pretwistor structures}
Let $Z$ be a derived analytic stack equipped with a morphism 
$\eta : Z \to \mathbb{P}^1$.
For foundations on derived analytic geometry over a complete non-archimedean field $K$ with a non-trivial valuation we refer to \cite{Lurie2011,PortaYu2020,PortaYu2018,PortaYuHigherGaga2016,PortaYuDerivedGAGA2019}. For most of our applications we will take $K=\mathbb{C}.$
\begin{remark}[Notation]
\label{GalActionRemark}
 Let $\mathbb{A}_{\mathbb{C}}^n$ be the $n$-dimensional affine space and $\mathbf{A}_{\mathbb{C}}^n:=(\mathbb{A}_{\mathbb{C}}^n)^{\mathrm{an}},$ its analytification. Similarly, $\bfP:=(\mathbb{P}_{\mathbb{C}}^1)^{\mathrm{an}}$ is the complex manifold with sheaf of holomorphic functions $\mathcal{O}_{\mathbf{P}^1(\mathbb{C})}$, whose underlying topological space has a $\mathrm{Gal}(\mathbb{C}/\mathbb{R})$-action, for which the non-trivial element $\sigma$ acts via complex conjugation.
\end{remark}
We think of $Z$ as a derived analogue of a twistor space whose fibers at $\lambda\in\bfP$, in the setting of NAH theory, play the role of moduli spaces of $\lambda$-connections. 
The key additional structure is a {\it shifted symplectic form}  $\omega$,  
which should be viewed as a derived version of the classical holomorphic symplectic form varying with weight $2$ along $\mathbb{P}^1$.
More precisely, fix $n\in\mathbb{Z}$ and assume that $\eta$ is equipped with a relative $n$-shifted closed $2$-form \cite{CalaqueHaugsengScheimbauer2025} twisted by the rank one locally free sheaf $\EuScript{O}_{\bfP}(2)$ on $\bfP$ \cite{calaque2024shifted}. In other words, it defines a degree $n$ section of the weighted (relative) negative cyclic complex i.e. is a shifted (derived) global section
 \begin{equation}
     \label{RelFormIntro}
\omega \in \Gamma\big(Z,\wedge^2\mathbb{L}_{Z/\mathbf{P}^1}^{\an}\otimes \eta^*\EuScript{O}_{\mathbf{P}^1}(2)\big)[n],
 \end{equation}
 where $\mathbb{L}_{Z/\mathbf{P}^1}^{\an}$ is the analytic relative cotangent complex. Recall in the derived setting, a closed $2$-form is a $2$-form equipped with additional closure \emph{data}.
Moreover, the form \eqref{RelFormIntro} satisfies a non-degeneracy condition given originally by \cite[Def.~1.18 (1)]{PTVV13}, and in the generalized form we require by \cite[Thm.~2.7, Def.~2.11]{calaque2024shifted}. It requires that $\omega$ induces a quasi-isomorphism
\begin{equation}
    \label{TwistedNDIntro}
\Theta_{\omega}:(\mathbb{L}_{Z/\bfP}^{\an})^{\vee}\xrightarrow{\sim}\mathbb{L}_{Z/\bfP}^{\an}[n]\otimes \eta^*\EuScript{O}_{\bfP}(2),
\end{equation}
where $\mathbb{T}_{Z/\bfP}^{\an}:=(\mathbb{L}_{Z/\bfP}^{\an})^{\vee}$ is the (derived) dual,  
which ensures that each fiber carries a shifted symplectic structure.

\begin{remark}
Following \cite{Simpson1997MixedTwistor}, we require this structure to be compatible with a real involution covering the antipodal map\footnote{This defines a real structure on $\mathbb{P}^1$, although its set of real points is empty.} on $\mathbb{P}^1$ given by the antiholomorphic involution $\sigma_{\mathbb{P}^1}:\mathbb{P}^1\to \mathbb{P}^1,[x:y]\mapsto [\overline{y}:-\overline{x}],x,y\in\mathbb{P}^1$. Equivalently, $\sigma_{\mathbb{P}^1}(z)=-1/\bar{z},$ for $z:=x/y\in\mathbb{P}^1$.
Via the canonical $\mathbb{G}_m$-action on $\mathbb{P}_{\mathbb{C}}^1$ given by translation we then impose the requirement that the underlying $n$-shifted $2$-form \eqref{RelFormIntro} is compatible with a given \emph{real structure on $Z$} (over $\mathbb{C})$ covering  $\sigma_{\mathbb{P}^1}.$
\end{remark}
In the classical theory, the twistor space carries an anti-holomorphic involution covering the antipodal map on $\mathbb{P}^1$. In the derived setting, we encode this using a homotopy coherent action of the group $C_2 = \mathrm{Gal}(\mathbb{C}/\mathbb{R})$.
More concretely, this consists of an automorphism $\sigma_Z : Z \to Z$ together with higher coherence data expressing that $\sigma_Z^2$ is homotopic to the identity. This replaces strict real structures by derived ones, where compatibility holds up to higher homotopies.

\begin{remark}
\label{ConformalRemark}
    We are interested in encoding real structures via $\mathrm{Gal}(\mathbb{C}/\mathbb{R}),$ but mention also the larger group $\mathsf{H}$ of conformal automorphisms of $\mathbf{P}^1$ preserving $\{0,\infty\},$ which act trivially when orientation is preserved and non-trivially when changed. The connected component of $\mathsf{H}$ is the group of holomorphic automorphisms fixing $0,\infty$ i.e. $\mathbb{G}_m$. The group $\mathsf{H}$ may be expressed as the group of automorphisms generated by $\mathbb{G}_m$ and by $\sigma_{\mathbf{P}^1}$ i.e.
$\mathbb{G}_m\cdot \sigma_{\mathbf{P}^1}.$ 
\end{remark}
We formulate this in Subsect.~\ref{sssec: DerivedGalois}--\ref{C2WaldCons}, by characterizing 
real structures on an object in derived geometry over $\mathbb{C}$ via a homotopy coherent action of the group $C_2$.

In this formulation, the main technical result (Proposition \ref{GaloisLemma}) has the following intuitive and explicit description: 
Let $C_2$ be presented by $\{\sigma,e: \sigma^2=e\},$ with classifying space $\mathsf{B}C_2$ given by the simplicial set with unique $0$-simplex $*$, nondegenerate $1$-simplex $\sigma$ and nondegenerate $2$-simplex witnessing the homotopy $\sigma \circ \sigma \simeq \mathrm{id}_*$ (together with all degeneracies).

Consider the $\infty$-category $\mathsf{Funct}(\mathsf{B}C_2,\mathsf{dSt}_{\mathbb{C}}^{\mathrm{laft}}),$ whose objects are $\mathsf{B}C_2$-shaped diagrams of derived stacks locally of finite presentation over $\mathbb{C},$
\begin{equation}
    \label{F}
    \mathcal{T}:\mathsf{B}C_2\to \mathsf{dSt}_{\mathbb{C}}^{\mathrm{laft}}.
\end{equation}

Unwinding \eqref{F}, it consists of: a derived stack locally of finite type $Z := \mathcal{T}(*)$, a $1$-morphism $\sigma_Z := \mathcal{T}(\sigma) : Z \to Z$, a homotopy ($2$-morphism) $\beta : \sigma_Z \circ \sigma_Z \Rightarrow \mathrm{id}_Z$, the image of the $2$-simplex, together with higher coherences.

Giving a diagram \eqref{F} is therefore equivalent to prescribing a pair $(Z,\sigma)$ where $\sigma$ is an automorphism whose square is homotopic to the identity.


The key notion introduced in this paper can be paraphrased
as follows; see Definitions \ref{DerivedTwistor} and \ref{DerivedTwistor2} for the precise formulation.
\begin{definition}
An \emph{$n$-shifted pre-twistor family} for $n\in\mathbb{Z}$ is a triple
$(\eta:Z\to \bfP,\omega,\sigma_Z)$ where the section $\omega\in\Gamma\big(Z,\wedge^2\mathbb{L}_{Z/\bfP}^{\an}\otimes\eta^*\EuScript{O}_{\mathbb{P}^1}(2)\big)$ has cohomological degree $n$, satisfies the non-degeneracy condition \eqref{TwistedNDIntro} and is compatible with the real involution $\sigma_Z$ on $Z$ covering the antipodal map on $\mathbf{P}^1.$  
\end{definition}
Via the $\mathbb{G}_m$-action on $\bfP$ we prove an important Proposition \ref{WeightVersion} characterizing pretwistor data in terms of an additional $\mathbb{Z}$-grading on differential forms \cite{GinRoz}, called the \emph{weight} grading. In this picture, the symplectic form \eqref{RelFormIntro} has has weight $2$, namely,  $t\cdot \omega = t^2 \omega$. 

The definition of a \emph{twistor family of hyperk\"ahler type} is obtained by imposing a geometric condition on the analytic stack of sections. It is modeled on the classical twistor space reconstruction theorem \cite[Thm.~3.3]{HKLR87} by noticing the real sections of $\eta$ are computed via the homotopy fixed point stack (for the $C_2$-action), in the setting of derived analytic geometry \cite{PortaYu2020, HolsteinPorta2025}.
\subsubsection{Shifted twistor structures}
Let $\mathsf{dAnSt}_{\mathbb{C}}$ denote the $\infty$-category of derived analytic stacks over $\mathbb{C}$ in the sense of \cite{PortaYuHigherGaga2016,PortaYuDerivedGAGA2019}, recalled in Subsect.~\ref{AnalyticConventions}. The analytification functor
\begin{equation}
\label{AnIntro}
(-)^{\mathrm{an}} : \mathsf{dSt}_{\mathbb{C}}^{\mathrm{laft}} \longrightarrow \mathsf{dAnSt}_{\mathbb{C}}
\end{equation}
preserves $C_2$-actions, commuting with colimits and extends\footnote{In more precise, yet abstract terms that we aim to avoid wherever possible, \eqref{AnIntro} provides a morphism of \emph{(pre) geometries} \cite[Thm.~2.1.1]{LurieDAGV}.} to other classes of derived stacks \cite{HolsteinPorta2025}. Thus, we may extend $\mathsf{B}C_2$-diagrams \eqref{F} relative to $\mathbb{P}_{\mathbb{C}}^1$ by analytification \eqref{AnIntro} in the \'etale topology
\begin{equation}
    \label{RelAnFunct}
(-)^{\an}:\mathsf{dSt}_{/[\mathbb{P}_{\mathbb{C}}^1/C_2]}\to \mathsf{dAnSt}_{/[\mathbf{P}^1/C_2]}.\end{equation}

Let $\eta:Z\to \bfP$ be such a derived analytic stack given by $\mathcal{T}^{\mathrm{an}}.$ Its stack of analytic sections is defined by the relative analytic mapping stack whose functor of points on $T \in \mathsf{dAn}_{\mathbb{C}}$ is given by
\begin{equation}
\label{AnSectIntro}
\ASect(Z/\bfP)(T)
:= \AnMap_{/\bfP}(T\times \bfP,Z),
\end{equation}
 where mapping spaces are taken in $\mathsf{dAnSt}_{\mathbb{C}/\mathbb{P}_{\mathbb{C}}^1}$, computed by (analytic) Weil restriction \cite{Lurie2018}. See also \cite[Def.~1.3.3]{KMMP}.

There is an induced automorphism of \eqref{AnSectIntro}, explained in Proposition  \ref{ActAnMaps}, and we study its homotopy fixed points stack. This stack is constructed in \S.\ref{sssec: FixedPoints}, and results in an analytic stack 
\begin{equation}
\label{AnSectFPIntro}
\ASect^{\sigma}(Z/\bfP):=\ASect(Z/\bfP)^{hC_2},
\end{equation}
parameterizing real-analytic maps i.e. $T$-families of $C_2$-invariant holomorphic sections
$
\phi_T : T \times \mathbb{P}_{\mathbb{C}}^1 \to Z.$
Importantly, compatibility with the real structure is not a strict condition on the family of holomorphic sections $\phi_T$, but rather part of derived structure encoded by additional homotopical data. 

With standard universal (relative) GAGA assumptions \cite[Def.~5.2 (2)]{HolsteinPorta2025}, \cite{PortaYuDerivedGAGA2019}, the twisted AKSZ construction \cite[Thm.~2.35]{calaque2024shifted} endows \eqref{AnSectIntro} with a twisted shifted symplectic structure, which descends to \eqref{AnSectFPIntro}. 
\begin{remark} Proposition \ref{WeightVersion} states this form has the same $\mathbb{G}_m$-weight. See also \cite[Thm.~3.3]{GinRoz}.
\end{remark}
Let $(Z\to \bfP,\omega,\sigma)$ be an $n$-shifted pre-twistor family. Following \cite{KPS2021}, we say that it is an \emph{$n$-shifted twistor family of hyperk\"aler type} if there exists a connected component
\begin{equation}
    \label{TwistorCondition1Intro}
\mathcal{X} \subset \ASect^{\sigma}(Z/\bfP)
\end{equation}
such that the universal evaluation morphism
$
\mathrm{ev}^{\sigma}:\ASect^{\sigma}(Z/\bfP)\times \bfP \longrightarrow Z,$
restricts to a $C^{\infty}$-equivalence
\begin{equation}
\label{TwistorCondition2Intro}
\Phi:\mathcal{X}\times \bfP \xrightarrow{\sim} Z.
\end{equation}
Some remarks should be made, since the nonabelian Hodge correspondence can not be applied for stacks, thus twistor lines and $C^{\infty}$-trivializations \eqref{TwistorCondition2Intro} do not exist\footnote{Indeed, proving such a trivialization would then give, in the case of $Z=\Perf(X)$ or $\Map(X,\mathsf{B}G)$ for a reductive group $G$, a global derived nonabelian Hodge correspondence, which is for now out of reach.}.
In this work, we impose a milder version of the above twistor condition given in \cite{KPS2021} by classical truncation and descending to a good moduli space (Def. \ref{DerivedTwistor2}). This is an initial ansatz to be shaped by future insights in derived and categorical (hyper)K\"ahler geometry.

\begin{remark}
As a first approximation, to be refined later by analyzing the deformation theory of (preferred) sections (see Rem.~ \ref{DNCRemark}), we impose that the classical truncation $t_{0}(Z)$ contains an open substack
    $U \subseteq t_{0}(Z)$ that admits a good moduli space to the twistor space $\mathrm{Tw}(M)$ of a hyperk\"ahler manifold $M$ (called the \emph{underlying classical hk-manifold}).
 This definition is sufficient for our intended applications, where $Z$ will be $\RAnPerf_{\Del}(X)$ and $U$ is the open substack
of Gieseker semistable objects (cf. Sect.~\ref{JHmaps} and Theorem \ref{MainTheoremBody} (iii)). The twistor conditions state that 
the classical truncation  of the derived Deligne stack \cite{alper2013good, FrancoHanson2024},
once restricted to the semistable locus  \eqref{DeligneSimpsonIntro}, is equivalent
(as a real-analytic stack over $\bfP$) to the product
$\mathrm{Tw}(M_{\Dol}(X))\times\bfP$ via the horizontal
twistor lines by NAH theory \eqref{NAH}.

\end{remark}
In other words, the twistor condition we ask for is interpreted as a trivialization of the underlying good moduli space $
\Phi^{\mathrm{cl}} :
\mathcal{X} \times \mathbb{P}^1
\xrightarrow{\sim}
t_0(Z)^{\mathrm{ss}}.$ 
Together, \eqref{TwistorCondition1Intro} and \eqref{TwistorCondition2Intro} abstractify the existence of twistor lines; here they are the above family of real analytic sections extending the classical fact one may reconstruct
a hyperk\"ahler manifold from its twistor space via a distinguished family of real sections \cite[Thm.~3.3]{HKLR87}. In this way, the derived twistor families we study in this paper are to be interpreted as derived enhancements of classical twistor spaces. 
\begin{remark}
    In this work, we mainly consider pre-twistor structures. We will give a more systematic treatment elsewhere of the twistor conditions and the existence of derived enhancements of horizontal twistor lines in connection with hyperholomorphic complexes \cite[Sect.~4.4]{Hanson2025}.
\end{remark}
From the viewpoint of nonabelian Hodge theory, the fibers of $Z \to \bfP$ play the role of moduli spaces of $\lambda$-connections, interpolating between Higgs bundles and flat connections. The shifted symplectic structure encodes the universal holomorphic symplectic form, while the real structure captures the reality conditions appearing in gauge theory. A key difference from the classical setting is that we now consider the derived moduli space of sections itself as a geometric object. This framework promotes the classical twistor description of hyperk\"ahler geometry to a derived and functorial setting, suitable for applications to gauge theory and the Kapustin--Witten equations.

We now state our main results which lead to new constructions of (derived) hyper-k\"ahler structures.

\subsection{Main results on hyperk\"ahler structures on derived stacks}
\label{ssec: MainResultsPartAIntro}
A key advantage of
formulating derived hyperk\"ahler geometry via the above notion of shifted twistor space is that it relies on the derived algebraic and analytic geometry of derived mapping stacks of sections, where known existence results for shifted symplectic structures are available \cite{PTVV13,CPTVV17,CalaqueHaugsengScheimbauer2025,Calaque15,calaque2024shifted}.

The main contents on the hyperk\"ahler structures on derived stacks appear in \S.\ref{sec:HKinDAG}. We now state the main theorems.
 \subsubsection{Existence results on twistor structures}
We first treat the twistor analogue of the well-known model-independent AKSZ/PTVV transgression theorem for construction of symplectic structures on derived mapping stacks \cite[Thm.~2.5]{PTVV13}. See \cite{CalaqueHaugsengScheimbauer2025}, for the relative setting and \cite{calaque2024shifted} for non-geometric derived stacks.

In the analytic context, there are two cases which occur in practice: \hypertarget{(1)}{(1)}
when the mapping stack has target $Z$ which is already a derived analytic stack for which the analytic mapping stack is locally geometric, and \hypertarget{(2)}{(2)} when $Z$ is an algebraic derived stack locally of finite presentation over $\mathbb{C}$ which moreover satisfies \emph{Tannakian} conditions of \cite[Def.~6.1]{HolsteinPorta2025}. Roughly, a derived algebraic stack $Z$ is Tannakian if for every $X\in\mathsf{dSt}_{\mathbb{C}},$
there is a fully-faithful functor
$$\mathrm{Maps}_{\mathsf{dSt}_{\mathbb{C}}}(X,Z)\to \mathsf{Funct}_{L,\mathbb{C}}^{\otimes}\big(\mathsf{Perf}(Z),\mathsf{Perf}(X)\big),$$
whose essential image consists of exact $\mathbb{C}$-linear symmetric monoidal functors commuting with colimits (i.e. cocontinuous $\mathbb{C}$-linear functors) which preserve both connective\footnote{Recall \emph{connectivity} refers to objects which are cohomologically non-positively graded.}
and flat objects.

Tannakian duality in the algebraic setting for geometric stacks is studied in \cite{Lurie2004,LurieDAG8, Lurie2018}. Its requirement here ensures certain analytic mapping stacks may be obtained via analytification of an algebraic mapping stack. 

Our first main result on pre-twistor transgression, states that in both scenarios \hyperlink{(1)}{(1)} and \hyperlink{(2)}{(2)}, the mapping stacks inherit a canonical shifted pretwistor structure.

\begin{theorem}[Theorem \ref{MainTheorem2Body}]
\label{MainTheorem2}
Fix $n\in\mathbb{Z}$ and an integer $d\geq 0.$ Let $X$ be a proper geometric derived algebraic stack with a relative $d$-orientation $[X],$ compatible with the $C_2$-action. 
\begin{enumerate}
   \item[(i)] Let $\eta_Z:Z\to \mathbb{P}^1$ be an $n$-shifted derived algebraic twistor family, with $Z$ geometric and satisfying $\mathsf{QCoh}(Z)\simeq\mathsf{Ind}\big(\mathsf{Perf}(Z)\big)$. Assume it is Tannakian and that $\Map_{/\mathbb{P}_{\mathbb{C}}^1}(X\times\mathbb{P}_{\mathbb{C}}^1,Z)$ is geometric. Then, $\AnMap_{/\bfP}(X^{\an}\times \bfP,Z^{\an})\to \bfP$ is canonically $(n-d)$-shifted pretwistor.

 \item[(ii)]
Let $\eta_Z: Z \to \bfP$ be an $n$-shifted derived analytic pretwistor family. Assume for the analytic mapping stack 
$\AnMap_{/\bfP}(X^{\an}\times \bfP,Z)$ is locally geometric. 
Then, it is canonically an $(n-d)$-shifted pretwistor family.
\end{enumerate}
\end{theorem}
The assumptions in Theorem \ref{MainTheorem2} are satisfied for the objects of interest in NAH-theory i.e. $Z=\Perf$ or the classifying stack $\mathsf{B}G$ of a connected reductive group $G$ over $\mathbb{C}$, and $X$ of the form $Y_{\sharp}$ for any of Simpson's shapes $\sharp\in\{\DR,\Dol,\B\}$ with $Y$ a smooth and proper scheme over $\mathbb{C}$. Of course, the local geometricity of the mapping stacks is always satisfied if $X$ is a proper underived scheme over $\mathbb{C}.$ 
\begin{remark}
 The assumptions on the source $X$ can be significantly weakened by requiring $X$ to satisfy the universal GAGA property \cite[Def.~5.1]{HolsteinPorta2025}. For example, if $X$ is a proper and geometric derived stack over $\mathbb{C}$, then $X_{\Dol}$ satisfies universal GAGA \cite[Prop.~5.32]{HolsteinPorta2025}.
 \end{remark}
The algebraic transgression for Tannakian derived algebraic stacks actually implies analytic transgression i.e. Theorem \ref{MainTheorem2} $(i)\Rightarrow(ii)$, due to the existence of a functorial comparison morphism, which we recall in \S.\ref{sssec: Analytic mapping stacks}, (see e.g. \eqref{eqn: Map to AnMap}), given by
$$\Map_{/\mathbb{P}_{\mathbb{C}}^1}(X\times\mathbb{P}_{\mathbb{C}}^1,Z)^{\an}\rightarrow\AnMap_{/\bfP}(X^{\an}\times \bfP,Z^{\an}).$$
By the assumptions of the theorem, it is an equivalence of derived analytic stacks by \cite{HolsteinPorta2025}.

The notion of \emph{hyper-Lagrangian} arises naturally in hyperk\"ahler geometry \cite{leung2007hyper}, which is a condition on \emph{complex} Lagrangian sub-manifolds, as opposed to real ones. Studying the notion of a Lagrangian structure \cite{Calaque15} on a morphism $f:L\to Z$ of derived analytic stacks relative to $\bfP$, we prove
a Proposition \ref{LagSectProp} using Theorem.~\ref{MainTheorem2} which extends to the shifted twistor setting an important result in derived symplectic geometry due to Calaque \cite[Thm.~2.10]{Calaque15}.

Since homotopy fiber products of derived analytic spaces exist in $\mathsf{dAnSt}_{\mathbb{C}}$ and the analytic cotangent complex has the expected base-change properties \cite[Thm.~1.5]{PortaYu2020} 
(recalled in Proposition \ref{AnCotangentProperties} below), we prove the following derived Lagrangian intersection theorem for twistor families.
\begin{theorem}[Theorem \ref{MainTheorem3Body}]
\label{MainTheorem3}
Let $Z \to \bfP$ be an $n$-shifted twistor family of hyperk\"ahler type. Let $f_1:L_1\to Z,f_2:L_2\to Z$ be morphism of derived analytic stacks relative to $\bfP$, equipped with relative Lagrangian structures. Consider the homotopy pullback square in $\mathsf{dAnSt}_{\bfP},$
\[
\begin{tikzcd}
    L\arrow[d]\arrow[r] & L_2\arrow[d,"f_2"]
    \\
    L_1\arrow[r,"f_1"] & Z,
\end{tikzcd}
\]

Then $\eta_{L}:L\to \bfP$ has the structure of an $(n-1)$-shifted twistor structure of hyperk\"ahler type.
\end{theorem}

Equipped with Theorem \ref{MainTheorem3}, we closely follow the ideas of Safronov expressing derived symplectic reduction as a derived Lagrangian intersection \cite{Safronov2016}. See also \cite[Thm.~2.1]{AnelCalaque2022}.

We prove a version of derived symplectic reduction as a derived Lagrangian intersection for twistor families of hyperk\"ahler-type denoted by $Z_{\mathrm{HK}-\mathrm{red}}$\footnote{If $Z$ is a classical hyper\"kahler manifold with complex structures $I,J,K$ and $\mu:Z\to \mathfrak{g}$ is a moment map, we note that it is sometimes customary in the literature to denote the hyperk\"ahler reduction by $\mu^{-1}(0)/\!/\!/\mathfrak{g},$ to indicate the three symplectic structures defined by $I,J,K.$}.

\begin{theorem}[Theorem \ref{MainTheorem4Body}]
  \label{MainTheorem4}
Let $Z \to \mathbb{P}^1$ be an $n$-shifted derived (pre-)twistor family, with $G$ a connected complex reductive group with Lie algebra $\mathfrak{g}$. Let $\mu: Z \to \mathfrak{g}^*[n]_{\mathbb{P}^1}$ be a relative $\EuScript{O}(2)$-twisted $n$-shifted moment map. Then the derived stack $Z_{\mathrm{HK-red}} \to \mathbb{P}^1$ defined by the cartesian square 
\[
\begin{tikzcd}[column sep=large,row sep=large] Z_{\mathrm{HK-red}} \arrow[r] \arrow[d] & \mathsf{B}G\times\mathbb{P}^1 \arrow[d,"0"] \\ {[\![Z/G]\!]}\arrow[r,"{[\mu]}"'] & {[\![\mathfrak{g}^*[n]/G]\!]}\times\mathbb{P}^1, \end{tikzcd}
\]
in $\mathsf{dSt}_{[\mathbb{P}^1/\mathrm{Gal}(\mathbb{C}/\mathbb{R})]}$ is canonically endowed with the structure of an $n$-shifted (pre-)twistor family, compatibly with the induced functor \eqref{RelAnFunct} of analytification.
\end{theorem}
In the second part of the paper, we apply the shifted twistor space formalism to study the analytic moduli stacks of flat and Higgs perfect complexes. The main results are stated in \S.\ref{StatementofResults}, but first we provide some further background.

\subsection{Motivation from Gauge theory}
In prior work \cite{LRT2022}, the second author studied nonlinear PDE systems arising from the Kapustin–Witten theory \cite{KapustinWitten2007}, obtained via a topological twist of $\mathcal{N}=4$ super Yang–Mills theory on a compact oriented four-manifold. A key observation is that the Kapustin–Witten equations naturally fit into the framework of nonabelian Hodge theory \cite{SiNonabelian,SiHiggs,SimpsonHodgeFiltrationNonabelian}. This ultimately relies on Simpson’s extension of the Uhlenbeck–Yau continuity method \cite{UY}, showing that a Higgs bundle admits a Hermitian–Yang–Mills metric if and only if it is polystable \cite{Simpson1988}.
 
Over a curve $C$, this relationship interacts with the
hyperk\"ahler structure $\mathcal{M}_{\mathrm{Dol}}^{\mathrm{ss}}(C)$. The $S$-duality transformation selects distinguished branes that are either symplectic or complex with regards to each complex structure in the moduli
space that are transformed under mirror symmetry, which itself is a form of global
Langlands duality, discussed in detail by Donagi--Pantev \cite{DonagiPantev2012}.

In higher dimensions these moduli spaces become highly singular, making them natural objects for shifted symplectic derived geometry \cite{PTVV13,CPTVV17,CalaqueHaugsengScheimbauer2025}; their shifted analogues exist, recalled in Subsection.~\ref{ShiftedHiggs}. A significant goal and formadible open problem is to understand the derived enhancements of the moduli spaces underlying the Corlette--Donaldson--Hitchin--Uhlenbeck--Yau--Simpson correspondence \cite{Corlette88,Corlette92} and the natural structures they possess in gauge theory and complex geometry. 

Note that there is another topological twist by Vafa and Witten \cite{VafaWitten1994}. For this twist, one expects that the generating series of invariants defined through the moduli spaces of solutions to the gauge-theoretic equations arising from the twist may exhibit modular properties reflecting $S$-duality, which $\mathcal{N}=4$ super Yang--Mills theory in four dimensions would have. For projective surfaces, Thomas and the second author of this paper \cite{TaTh1, TaTh2} gave a mathematical definition of these invariants using moduli spaces of semistable Higgs bundles, via a Donaldson--Uhlenbeck--Yau type theorem for this system \cite{AlGa, Tana2}. Vafa and Witten \cite{VafaWitten1994} also discussed that the blowup formula in this twist should possess modular properties, which ought to be explained in terms of conformal field theory. On the mathematical side, Kuhn, Leigh, and the second author of this paper \cite{KuLeTa} obtained exactly the same modular forms predicted by Vafa and Witten for the instanton part of the Vafa–Witten invariants, by using more general blowup formulas for enumerative invariants on projective surfaces by Kuhn and the second author \cite{KuTa}. 
We do not pursue this direction here; instead, our focus is on the comparison and relationship between the moduli spaces arising from the Kapustin--Witten theory and their Derived enhancements.

\subsection{Gauge-theoretic moduli spaces in the Kapustin--Witten theory}
\label{ssec: Gaugetheoretic KW-theory}
Derived algebraic geometry offers a powerful alternative approach that circumvents certain analytic difficulties in gauge theory, such as the failure of standard metric perturbation methods to obtain the transversality of the moduli spaces in analytic settings. On one hand, the moduli stack of perfect complexes of local systems on a compact oriented topological four-manifold carries a canonical $(-2)$-shifted symplectic structure \cite{PTVV13}, which corresponds to the $t \neq 0$ side of Kapustin–-Witten theory. Therefore, it is natural to expect that the moduli spaces of solutions to the Kapustin--Witten equations for $t \neq 0$ would carry a $(-2)$-shifted symplectic structure, as they do when the underlying four-manifold is a complex projective surface. On the other hand, the orientation problem for the moduli space of solutions to the Kapustin--Witten equations at $t = 0$ was solved in \cite[Thm.~4.9]{JoyceTanakaUpmeier2020}, and the same conclusion extends to the $t \neq 0$ case.

Derived algebraic geometry has been used to describe moduli spaces arising from topological twists of $\mathcal{N}=4$ super Yang–Mills theory. Elliott–Yoo \cite{ElliotYoo2018} describe the geometric quantization of the derived stack assigned to $C\times\mathbb{C}$ for each twist, where $C$ is a compact Riemann surface. In their sequel \cite{ElliottYoo2025}, they further explain the role of these moduli spaces in the (quantum) categorical geometric Langlands programme (see Subsection \ref{ssec: Relation to other works}).

Two natural problems arise: (a) understand the deformation spaces of solutions to the Kapustin--Witten equations for both $t = 0$ and $t \neq 0$ on a closed four-manifold in an appropriate sense and to compare these moduli spaces to obtain a certain form of correspondence predicted by nonabelian Hodge theory, (b) \emph{define}, study and ultimately clarify the hyperk\"ahler geometry arising in this context.

Using derived (analytic) geometry, we identify a Lagrangian correspondence, which we say is of \textit{nonabelian Hodge-type}. Our approach is compatible with previous descriptions via topological twists, but we focus on developing the basic features of hyperk\"ahler structures in derived geometry, which are of independent interest (see \cite[Rem.~2.5(ii)]{GinRoz}).

This leads to our second family of results in this paper, which prove the existence of a shifted twistor family structure on the Deligne--Simpson moduli space for perfect complexes, generalizing the classical twistor family structure \eqref{DeligneSimpsonIntro}. 
Via the nonabelian Hodge correspondence, our results may be fruitfully interpreted as a deformation of the classical hyperk\"ahler structure in derived directions, controlling the obstruction theory of the moduli space of perfect complexes with flat $\lambda$-connections.

\subsection{Main results on Lagrangian correspondences of nonabelian Hodge-type}
\label{StatementofResults}
Let $X$ be a smooth algebraic variety over $\mathbb{C}$ of dimension $d$, with Hodge stack $X_{\Hod},$ over $\Theta=[\mathbb{A}^1/\mathbb{G}_m].$ Its fiber over $[1]$ is $X_{\mathrm{DR}}$, and its special fiber over $[0]$ defines a \emph{formal moduli problem} under $X\times \Theta$, in the sense of Gaitsgory--Rozenblyum \cite{GR17b}. Roughly, a formal moduli problem under a derived stack locally of finite-type is equivalent to a Lie algebroid object in ind-coherent sheaves. See \cite[\S.2.2.3]{ElliotYoo2018} for a nice explanation in terms of algebraic gauge theories.

The special fiber is {\it Dolbeault stack} $X_{\Dol}$ over $\mathsf{B}\mathbb{G}_m,$ with natural $\mathbb{G}_m$-action, so $X_{\Hod}$ is the Dolbeault degeneration of $X_{\DR}$ to the total space of the $1$-shifted tangent bundle $\mathsf{T}[1]X.$

Functoriality of mapping stacks \eqref{eqn: RPerfIntro} induce morphisms to the generic fiber, given by the stack of flat perfect complexes \eqref{RPerfDRIntro}, and special fiber, given by Higgs perfect complexes \eqref{RPerfDolIntro}.
Both fibers are $2(1-d)$-shifted symplectic derived stacks, and we construct a series of Lagrangian correspondences for these symplectic structures, that we refer to as being of \emph{nonabelian Hodge-type.} 

\begin{remark}[Notation]
In this paper, we adopt the convention of writing a correspondence between two objects $X,Y$ implemented by a third object $Z$ as $Z:X\dashrightarrow Y,$ as opposed to the usual 'hat' notation,
\[
\begin{tikzcd}[row sep=small, column sep=small]
    & \arrow[dl,"\ell"] Z \arrow[dr,"r"] & \\
    X & & Y.
\end{tikzcd}
\]
When the maps need specification, we use the notation $Z\overset{\ell\times r}{\rightarrow}X\times Y$.
\end{remark}
Let $F$ be a derived Artin stack locally of finite presentation over $\mathbb{C}.$ Let $\mathsf{Filt}(F),\mathsf{Grad}(F)$ be the stacks of filtered and graded objects \cite{Hal14,Moul2021}. There is a known correspondence 
\begin{equation}
\label{AttractorCorrIntro} 
\mathsf{Filt}(F): \mathsf{Grad}(F)\dashrightarrow F,
    \end{equation}
    via maps $\mathsf{Filt}(F)\xrightarrow{(\mathsf{Gr}\times\mathsf{Un})}\mathsf{Grad}(F)\times F,$ taking the `associated graded' and 'underlying object'. 
If $F$ is further endowed with an $n$-shifted symplectic structure $\omega_F$, the correspondence \eqref{AttractorCorrIntro} is an $n$-shifted Lagrangian correspondence 
\cite[Cor.~5.19]{KinjoParkSafranov2024}).
Specifically, the map $\mathsf{Filt}(F)\to \mathsf{Grad}(F)\times \overline{F}$ has a Lagrangian structure, where $\overline{F}$ denotes the stack $F$ with symplectic form $\omega_{\overline{F}}:=-\omega_F$ (e.g., \cite{Calaque15}).

In Subsect.~\ref{ssec: Analytification}, we prove an analogue of  \eqref{AttractorCorrIntro} in the $\mathbb{C}$-analytic setting \cite{Lurie2011,HolsteinPorta2025}, when $F$ is the analytification of the $2(1-d)$-shifted derived stack of flat perfect complexes. This makes use of nice formal properties of analytification; for any shape operation $(-)_{\sharp},\sharp\in\{\mathrm{DR},\mathrm{Dol},\mathrm{B}\},$ there is a natural transformation 
\begin{equation}
\label{NatTransform}
(-)^{\mathrm{an}}\circ (-)_{\sharp}\rightarrow (-)_{\sharp}\circ (-)^{\mathrm{an}}.
\end{equation}
When evaluated on a smooth proper complex scheme $X$, \eqref{NatTransform} is an equivalence \cite[Lem.~5.24, Lem.~5.31]{HolsteinPorta2025}. The notation 
$X_{\mathrm{DR}}^{\mathrm{an}}$ and $X_{\mathrm{Betti}}^{\mathrm{an}}$ is therefore unambiguous.

\begin{proposition}[Proposition \ref{MainTheoremBody}]
Let $X$ be a smooth proper algebraic variety over $\mathbb{C}$. Let $X^{\an}$ be its analytification. Then 
$$\mathsf{AnFilt}(\RAnPerf(X_{\DR}^{\an})\big): \mathsf{AnGrad}\big(\RAnPerf(X_{\DR}^{\an})\big)\dashrightarrow \RPerf_{\DR}(X)^{\an},$$
is a $2(1-d)$-shifted Lagrangian correspondence of derived analytic stacks with $\mathbf{G}_m$-actions, where $\RAnPerf(X_{\DR}^{\an})$ has the trivial action. Thus, $\mathsf{AnFilt}(\RAnPerf_{\DR}(X^{\an})\big)$ is naturally endowed with $(1-2d)$-shifted Poisson structure.
\end{proposition}
Since $X_{\Hod}$ is a deformation to the normal cone, we obtain the following NAH-type Lagrangian correspondence, where $\mathcal{O}_{\Theta}(p)$ is the trivial line bundle on $\mathbb{A}^1$ for the weight $p$ representation of $\mathbb{G}_m$ with $\mathcal{O}_{\Theta}(1)$ the universal bundle on $\Theta$ \cite[Prop.~5.1]{calaque2024shifted}.
\begin{theorem}[Theorem \ref{MainTheoremBodyBody}]
\label{MainTheorem}
Let $X$ be a smooth proper variety of dimension $d$ over $\mathbb{C}.$ 
Then $X_{\Hod}\to \Theta$ has a natural $\mathcal{O}_{\Theta}(1)$-twisted $2d$-orientation such that by pullback along $\mathbb{A}^1\to \Theta,$ the morphisms of evaluation at special and generic fiber $(\mathrm{ev}_{\lambda =0},\mathrm{ev}_{\lambda =1})$ yield a relative shifted Lagrangian correspondence
$$\big(\Map_{/\Theta}(X_{\Hod},\Perf\times\Theta)\times\mathbb{A}_{\mathbb{C}}^1\big)^{\an}: \RAnPerf(X_{\Dol}^{\an})\dashrightarrow\RAnPerf(X_{\DR}^{\an}),$$
of derived analytic stacks over $\mathbf{A}^1=(\bfA)^{\an}.$
\end{theorem}
It is important to emphasize
Theorem \ref{MainTheorem} is a statement about shifted symplectic structures and their compatibilities, not about homeomorphism or diffeomorphism of coarse moduli spaces. In particular, it does not assert the existence of an analytic isomorphism, nor does 
it impose stability conditions cf. \eqref{NAH}, although we briefly discuss polynomial stability conditions in Subsect.~\ref{ssec: Stability considerations}. We prove a Corollary \ref{SSCohCor} stating a restricted Lagrangian correspondence for semi-stable coherent complexes with fixed numerical behaviour.

\begin{remark}
We mention that there is also a $(1-2d)$-shifted \emph{coistoropic} correspondence of derived algebraic stacks, for a related derived moduli functor of complexes up to chain isomorphism $\mathbb{R}\underline{\mathrm{Cplx}}(X)$ introduced by Hua--Polishchuk \cite{HuaPolishchuk2018}. Working with the latter $\Perf$, allows additional flexibility for endowing complexes with additional structures term-wise e.g. hermitian forms.
\end{remark}

\subsubsection{Derived Riemann--Hilbert correspondence}
Recall that given a smooth complex variety $X$ the algebraic de Rham theorem identifies the cohomology of the constant sheaf (derived endomorphisms of the constant local system),
$H^*(X,\underline{\mathbb{C}}_X)\simeq H^*\mathbb{R}\mathrm{End}_{\mathsf{Shv}(X)}(\underline{\mathbb{C}}_X),$
with de Rham cohomology
$H^*(X,\mathrm{DR}_X^{\bullet}),$ or equivalently, derived endomorphisms $H^*\mathbb{R}\mathrm{End}_{\mathcal{D}_X}(\mathcal{O}_X,d_{DR}),$ of its structure sheaf with natural $\mathcal{D}$-module structure (flat connection). 

Natural transformations from $X_{\B}$ to the analytic stack in categories\footnote{See \cite{porta2017derived} for what this means.} $\underline{\mathrm{AnVect}}_n$ of rank $n$ holomorphic vector bundles of corresponds to representations of the fundamental groupoid
$\pi(X_{\mathrm{top}})$ of $X$ in $\mathbb{C}^n,$ where $X_{\mathrm{top}}$ is the underlying topological space of $X.$  Similarly, natural transformations
from $X_{\mathrm{DR}}$ to $\underline{\mathrm{AnVect}}_n$ give a category of holomorphic bundles on $X$ with flat connection.

The Riemann--Hilbert correspondence \cite[Thm.~2.23]{Deligne1970} (see also \cite{Kashiwara1984}), may be formulated as an equivalence
\begin{equation}
    \label{RHVectIntro}\mathrm{Map}_{\mathrm{AnSt}^{\mathrm{Cat}}}(X_{\mathrm{Betti}},\underline{\mathrm{AnVect}}_n)\simeq \mathrm{Map}_{\mathrm{AnSt}^{\mathrm{Cat}}}(X_{\mathrm{DR}},\underline{\mathrm{AnVect}}_n)\end{equation}
where $\mathrm{AnSt}_{\mathbb{C}}^{\mathrm{Cat}}$ is the $\infty$-category of sheaves on the category of Stein spaces $\mathrm{Stein}_{\mathbb{C}}$ with Grothendieck topology given by open immersions, valued in $\infty$-categories (see \cite[eqn. (1.1)]{porta2017derived}).

There is a \emph{Riemann--Hilbert transformation} 
$$\eta_{\mathrm{RH}}:X_{\mathrm{DR}}^{\mathrm{an}}\to X_{\mathrm{Betti}}^{\mathrm{an}},$$ and this leads to the Porta's \emph{derived RH theorem} \cite[Thm.~6.11]{porta2017derived}, generalizing the classical one \eqref{RHIntro}, and promoting the equivalence \eqref{RHVectIntro} from vector bundles to all perfect complexes. 
More precisely, the derived RH-correspondence is an equivalence of derived analytic stacks \cite[Cor.~7.5]{HolsteinPorta2025}, 
\begin{equation}
\label{DerRHIntro}\eta_{\mathrm{RH}}^*:\RAnPerf(X_{\mathrm{Betti}}^{\mathrm{an}})\simeq \RAnPerf(X_{\mathrm{DR}}^{\mathrm{an}}).
\end{equation}
A general derived RH-correspondence with coefficients in any derived algebraic stack $Z$ of finite presentation over $\mathbb{C}$ satisfying the
Tannakian property is given in \cite[Cor.~7.6]{HolsteinPorta2025}.

Our contribution provides a shifted symplectic enhancement of the derived RH theorem.
First, we observe the AKSZ construction is compatible with analytification and with relative GAGA properties for $X_{\DR},X_{\B}$. This fact allows us to prove the relative orientations (in the sense of \cite[Sect.~2]{PTVV13}) on $X_{\B}^{\an}$ and on $X_{\DR}^{\an}$ are preserved under $\eta_{\mathrm{RH}}^*.$

In other words, \eqref{DerRHIntro} is compatible with the natural shifted symplectic structures on each side and we obtain the following result.
\begin{theorem}[Theorem.~\ref{AnalyticNAH} (i)]
\label{AnalyticNAHIntro}
Let $X$ be a smooth proper complex variety of $\dim_{\mathbb{C}}X=d.$ Then, the derived RH--correspondence \eqref{DerRHIntro} is an equivalence of $(2-2d)$-shifted symplectic derived analytic stacks.
\end{theorem}
We actually prove a more general statement allowing for Tannakian coefficients other than $Z=\Perf$, which should have future applications (Theorem.~\ref{AnalyticNAH} (ii)).

In Section \ref{sec: Derived NAH and KW}, we apply 
the theory of derived twistor families to the Deligne--Hitchin--Simpson moduli space interpreted as the moduli of solutions to the Kapustin--Witten equations in the surface case \cite{LRT2022}.
\subsubsection{Deligne--Simpson perfect complexes}
\label{sec:intro_hkr}
We now come to the main result of this paper, which generalizes the Deligne–Simpson moduli space \eqref{DeligneSimpsonIntro} of $\lambda$-connections to the derived analytic moduli stack of  perfect complexes on a smooth projective variety $X$ of dimension $d.$ For a related construction for $G$-bundles see \cite{FrancoHanson2024}.

In the surface case $d=2,$ we relate it to the Kapustin--Witten moduli space which, in light of Theorems \ref{MainTheorem2}, \ref{MainTheorem3}, and \ref{MainTheorem4}, is obtained as a derived hyperk\"ahler reduction. 

Let $X^{\mathrm{conj}}$ be the complex (Galois) conjugate variety of $X$, as in \S.~\ref{PreTwistorSection}, defined by taking complex conjugates of the coefficients of equations defining $X.$
 There is a real-analytic homeomorphism,
 $\gamma:X_{\B}\xrightarrow{\simeq}X^{\mathrm{conj}}_{\B},$
 as $X_{\B}$ is the topological space underlying the complex analytic manifold $X^{\an}.$
 Via Theorem \ref{AnalyticNAHIntro}, 
 we construct a closed immersion of derived analytic stacks to the analytification $\RAnPerf_{\Hod}(X)$ and to $\RAnPerf_{\Hod}(X^{\mathrm{conj}}).$ 
 We compute the relative homotopy-pushout in $\mathsf{dAnSt}_{/\bfP}$by gluing the along the shifted symplectic equivalence \ref{DerRHIntro}, and prove that it naturally inherits a relative shifted symplectic form.

Altogether, by the results on shifted twistor formalism, shifted symplectic aspects of the derived nonabelian Hodge Lagrangian correspondence and the shifted symplectic RH-correspondence, we prove the following main result completing the original construction by Simpson in his seminal paper \cite{Sim09}, and answering \cite[Questions 6.7 (1), (2)]{KPS2021}.

Recall $(\mathbb{A}^1\backslash \{0\})^{\mathrm{an}}\simeq\mathbb{C}^*\simeq \mathbf{G}_m,$ and our conventions regarding $\mathbf{P}^1$ are that it is always endowed with its $\mathrm{Gal}(\mathbb{C}/\mathbb{R})$-action (cf. Rem.~\ref{GalActionRemark}).

\begin{theorem}[Theorem \ref{DHSTwistor1}, Proposition \ref{DHSTwistor2}]
    \label{MainTheorem5}
    Let $X$ be a smooth $\mathbb{C}$-analytic space of dimension $d$. Then there exists closed-immersion of derived analytic stacks $\EuScript{J},\EuScript{J}_{\mathrm{conj}}$ and a homotopy-pushout diagram in derived analytic stacks
    \begin{equation}
        \label{DeligneHoPushout}
    \begin{tikzcd}
       \Map(X_{\B},\Perf)^{\an}\times \mathbf{G}_m\arrow[d,"\EuScript{J}"] \arrow[r,"\EuScript{J}_{\mathrm{conj}}"] & 
       \RAnPerf_{\Hod}(X^{\mathrm{conj}})\arrow[d,"p"]
\\
\RAnPerf_{\Hod}(X)\arrow[r,"q"] & \RAnPerf_{\Del}(X)
    \end{tikzcd}
\end{equation}
which endows the relative analytic Deligne--Hitchin--Simpson derived stack of perfect complexes,
   $$\eta_{\Del}:\RAnPerf_{\Del}(X)\to\bfP,$$
   with the structure of a $2(1-d)$-shifted pretwistor family of hyperk\"ahler type.
 
\end{theorem}
The push-out \eqref{DeligneHoPushout} must be computed in derived analytic stacks, since it is not the analytification of a derived algebraic stack, simply because the equivalence \eqref{DerRHIntro} only holds in the analytic setting (c.f, Remark \ref{AnalyticRequirementRemark}).

Applying Theorem \ref{MainTheorem5} to the $2$-shifted symplectic stack $\mathsf{B}G$, in place of $\Perf$, we have the following natural corollary.

\begin{corollary}
\label{KWCor}
The derived stack $\AnMap_{\Del}(X,G)$ is a derived hyperkahler reduction. Moreover, the restriction of its classical truncation to the locus of Gieseker semistable sheaves is a good moduli space for the classical Deligne--Hitchin--Simpson twistor family $\mathcal{M}_{\mathrm{Del}}^{\mathrm{ss}}(X)$ in \eqref{DeligneSimpsonIntro}.
\end{corollary}
Via pullback along $\mathbb{A}^1\to \Theta,$ by identifying $\mathbb{P}^1$ with two copies $\mathbb{A}_0^1,\mathbb{A}_{\infty}^1$ of $\mathbb{A}^1$, given a $2(1-d)$-shifted $\EuScript{O}(2)$-twisted relative moment map $\mu$ restricted to one copy:
$$\mu|_{\mathbb{A}_0}:\RPerf_{\mathrm{Hod}}(X)|_{\mathbb{A}_0^1}\to \mathfrak{g}^*\otimes \mathcal{O}_{\mathbb{P}^1}(2)|_{\mathbb{A}_0^1},$$
the reduction as a relative derived Lagrangian intersection given by
\begin{equation}
\label{RKW(X)1}
[\![\RPerf_{\mathrm{Hod}}(X)|_{\mathbb{A}_0^1}/G]\!]\times_{\mathfrak{g}^*\otimes\EuScript{O}_{\mathbb{P}^1}(2)|_{\mathbb{A}_0^1}/G}^h(\mathsf{B}G\times\mathbb{A}_{0}^1).
\end{equation}
in the surface case, yields a shifted-symplectic fiber of a relative twisted $(-2)$-shifted derived (pre)twistor family of hyperk\"ahler type by Theorem \ref{MainTheorem4}. As emphasized in Subsect.~\ref{ssec: Gaugetheoretic KW-theory}, motivated by results of \cite{LRT2022} it is interpreted as an analog of the Kapustin--Witten moduli space of perfect complexes; for $\mathsf{B}G$, the result is compatible with the conclusions of \cite{ElliotYoo2018}, however our construction differs substantially in nature and interpretation.
\begin{remark}
  \label{DNCRemark}
    In \cite{NABMHS2000}, it was conjectured the local deformation theory along preferred sections of a derived twistor stack should be governed
by some kind of mixed Hodge complex. More specifically,
the prefered sections of Hitchin-Deligne’s twistor space for perfect complexes in Theorem \ref{MainTheorem5} correspond to $\mathcal{O}_X$-perfect mixed Hodge complexes over $X$, in the sense of T. Mochizuki \cite{Mochizuki2015}. See also, works of Sabbah \cite{SabbahTwistor} and Saito \cite{Saito1988,Saito1990}. 
This perspective is summarized by stating that the formal neighborhood of a prefered section should be a nonabelian mixed Hodge
structure \cite{NABMHS2000}. 
Recall the conformal automorphism group $\mathsf{H}$ of Remark \ref{ConformalRemark}. An $\mathsf{H}$-equivariant mixed twistor structure i.e. a mixed twistor structures which is $\mathbb{G}_m$-equivariant
and equipped with antipodal involution, is 
equivalent to a $\mathbb{R}$-mixed Hodge structure \cite[Prop.~2.1]{Simpson1997MixedTwistor}. Moreover, Hodge structures are $\mathbb{C}^*$-equivariant twistor structures. With this picture in mind, one may carry out the derived deformation to the normal cone construction in the $\mathsf{H}$-equivariant topos. We will revisit this computation later.
\end{remark}

\subsubsection{Relation to other works}
\label{ssec: Relation to other works}
There is a profound connection between Kapustin--Witten theory and (categorical) Geometric Langlands. Motivated from physics perspective, specifically via topological field theory, this was studied mathematically within the context of derived algebraic geometry by C. Elliot and P. Yoo. \cite[Thm 1.3,1.4]{ElliotYoo2018}.
The authors compute the derived moduli space of the equation of motion of holomorphic, $B$ and $A$-twists of $\mathcal{N}=4$ super Yang--Mills theory as an assignment of derived stacks. Their starting point is the twistor space point of view, and fits into a general theory of computing topological twists of $\mathcal{N}=4$ supersymmetric gauge theories with gauge group $G$. 
Note that their parameter $\Psi$ is defined as

$$ \Psi := \frac{\tau + \bar{\tau}}{2}  + \frac{\tau - \bar{\tau}}{2} \left( \frac{t - t^{-1}}{t + t^{-1}} \right), $$ 
where $\tau = \frac{\theta}{2 \pi} + \frac{4 \pi i}{e^2}$, as in \cite[Sect.~3.5]{KapustinWitten2007}. Thus,  their B-model at $\Psi =\tau$ corresponds to $t=0$ in the twisting parameter of supersymmetry, or $\lambda =0$ in the Deligne--Simpson  moduli space, and the A-model at $\Psi = \infty$ corresponds to $t=1$ in the supersymmetry parameter, or $\lambda =1$ in the Deligne--Simpson  moduli space but the structure group or the {\it gauge group} replaced by the Langlands dual ${}^{L}G$.  

The moduli spaces of solutions to the equations of motion in the $A$ and $B$-twists of $\mathcal{N}=4$ super Yang--Mills on a smooth proper complex algebraic surface $S$ are equivalent to
$\mathbf{Higgs}_G(S)_{\mathrm{DR}},$ and $\mathsf{T}^*[-1]\mathbf{Flat}_G(S)_{\mathbf{Flat}_G(S)}^{\wedge},$ respectively as $(-1)$-shifted symplectic stacks.

Note the $B$-module twist 
is a deformation of $\mathbf{Higgs}(S)$; it sees the algebraic structure of
$\mathbf{Flat}_G(S)$ as opposed to the stack of locally constant
sheaves. The $A$-twist gives rise to a de Rham stack, explaining the appearance of
$D$-modules \cite{GR14}, and avoiding Fukaya-type categories in the formulation of the categorical geometric
Langlands correspondence.

Summarizing the moduli spaces involved in \cite{ElliottYoo2025}, where the authors consider complex algebraic surfaces of the form $\mathbb{C}\times C$, where $C$ is a closed
algebraic curve, whose higher dimensional generalization (which is possible via results of \cite{calaque2024shifted}) as discussed in  Subsect.~\ref{ShiftedHiggs}, are given by the diagram:
\begin{equation}
\label{TwistedBunGDiagramIntro}
\begin{tikzcd}
& \arrow[dl,"\lambda \to 1"] T_{\EuScript{L}_{\det}^{\lambda}}^*\mathrm{Bun}_G(C)\arrow[dr,"\lambda\to 0"]\arrow[d,"\simeq"]
\\
T_{\EuScript{L}_{\det}}^*\mathrm{Bun}_G(C) \arrow[d,"\simeq"] & \arrow[dl,"\lambda \to 1"] Flat_G^{\lambda}(C)\arrow[dr,"\lambda \to 0"] & \arrow[d,"\simeq"] T^*\mathrm{Bun}_G(C)
\\
Flat_G(C) & & Higgs_G(C).
\end{tikzcd}
\end{equation}

\subsection{Notations and definitions}
Here and throughout, we work over a field $K$ of characteristic zero, or more general ring over $K$.
We follow closely \cite{Lur09,Lur17}, and by a \emph{category}, we mean a $K$-linear stable presentable $(\infty,1)$-category, denoted by $\mathsf{C}.$ We let $\mathrm{Maps}_{\mathsf{C}}$ denote the mapping space. The category of spaces is denoted by $\mathsf{Spc}.$
We assume familiarity with derived algebraic geometry \cite{LurieDAGV}. See \cite{T2014} for an overview.

\begin{itemize}
\item We let $\mathrm{dg}_K$ denote the category of cochain complexes of
$K$-vector spaces considered with its projective model structure and standard tensor product, so it is a symmetric monoidal model category. 
Given a symmetric monoidal model category $M$ satisfying some assumptions \cite[Subsect~1.1]{CPTVV17}, the associated 
symmetric monoidal $\infty$-category, obtained via homotopy coherent
nerve of the Dwyer--Kan localization $L(M)$ along weak equivalences is denoted by $\mathsf{M}$. For example, $\mathsf{dg}_K=L(\mathrm{dg}_K).$
\item Let $\mathsf{cdga}_K^{\leq 0}$ denote connective commutative dg-algebras in $\mathsf{dg}_K.$ Let $\mathsf{PSt}_K$ denote the $\infty$-category of $K$-prestacks i.e. accessible $\infty$-functors $\mathsf{cdga}_K^{\leq 0}\to \mathsf{Spc}.$ We denote by $\mathsf{dSt}_{K}$ the $\infty$-category of derived stacks over $K$ for the standard \'etale topology \cite{TV2}. For $n\geq -1,$ we use the notion of $n$-geometric derived stack and $n$-representable morphism of derived stacks given in \cite[Def.~1.3.3.1]{TV2}.
\item Working over $K=\mathbb{C}$, denote the $\infty$-category of connective commutative differential graded algebras $A$ of finite presentation over $\mathbb{C}$ as $\mathsf{cdga}_{\mathbb{C}}^{\leq 0}$ and by $\mathsf{dAff}_{\mathbb{C}}$ the opposite category of derived affine $\mathbb{C}$-schemes. Let $\mathsf{Mod}(A)$ denote the $\infty$-category of connective $A$-modules.

\item Given a prestack $S$, denote by $\mathsf{dSt}_{S/K}$ the category of $S$-prestacks. Given a derived prestack $S$ and derived $S$-prestacks $E,F$, the relative mapping prestack is $\Map_{/S}(E,F)\in\mathsf{PStk}_S.$

\item 
Let $\underline{\mathsf{Q}}\to \mathsf{dSt}^{op}$ the cocartesian fibration for the functor $\underline{\mathsf{QCoh}}^*(-)$, where the $*$ refers to the fact we consider this assignment with respect to $*$-pull-back functors. Equivalently, it is the cartesian fibration for the corresponding right-adjoint pushforward functors, which for a morphism of prestacks $f:B\to S$ is denoted $f_*(-):=\Gamma_{S}(B,-).$
Then, given a derived stack $Z$, we have
$$\mathsf{QCoh}(Z)\simeq \underset{\mathsf{Spec}(A)\to Z}{\lim} \mathsf{Mod}(A).$$
We write $\mathsf{Mod}^{\mathrm{gr}}(A)$ the $\mathbb{Z}$-graded $A$-modules, and $\mathsf{Mod}^{\mathrm{gr},\epsilon}(A)$ the graded mixed $A$-modules. Similarly, $\mathsf{QCoh}^{\mathrm{gr}}(Z),\mathsf{QCoh}^{\mathrm{gr},\epsilon}(Z).$
\item 
Let $\underline{\mathcal{O}}:\mathsf{dSt}^{op}\to \underline{\mathsf{Q}},$ be the cocartesian section sending a derived stack $S$ to its structure sheaf $\mathcal{O}_S.$ 

\item Let $\mathsf{Perf}(A)\subseteq \mathsf{Mod}(A)$ the stable idempotent complete $\infty$-category spanned by perfect $A$-modules. An $A$-module $M$ is \emph{almost perfect} if $\pi_i(M)=0,i\ll0$ and for every $n\in\mathbb{Z}$, the truncation $\tau^{\leq n}M$ is compact as an object of $\mathsf{Mod}(A)^{\leq n}.$ Denote this category by $\mathsf{APerf}(A).$
$$\mathsf{Perf}(Z)\simeq \underset{\mathsf{Spec}(A)\to Z}{\lim}\mathsf{Perf}(A),\hspace{2mm}\mathsf{APerf}(Z)\simeq\underset{\mathsf{Spec}(A)\to Z}{\lim}\mathsf{APerf}(A).$$

\item Let $\mathsf{Coh}(Z)$ be the full subcategory of $\mathsf{Mod}(\mathcal{O}_Z)$ consisting of $\mathcal{F}$ for which there exists an atlas $\{f_i:U_i\to Z\}_{i\in I}$ such that for each $i\in I,n\in\mathbb{Z}$, the $\mathcal{O}_{U_i}$-modules $\pi_n(f_i^*\mathcal{F})$ are coherent sheaves.

Let $f:B\to S$ be a morphism of derived stacks, with $E\in\mathsf{QCoh}(B)$. We say that it has \emph{tor-amplitude $[a,b]$ relative to $S$} (resp. $\leq n$ relative to $S$) if for every $F\in \mathsf{QCoh}(S)^{\heartsuit},$ 
$\pi_i(E\otimes f^*F)=0,$ for $i\notin[a,b]$ (resp. $i\notin [0,n]$). We write the corresponding full subcategories of $\mathsf{QCoh}(B)$ spanned by them as
$\mathsf{QCoh}_S^{\leq n}(B),\mathsf{APerf}_S^{\leq n }(B).$ 
Set
$$\mathsf{Coh}_S(B):=\mathsf{APerf}_{S}^{\leq 0}(B).$$

\item In this work, we also consider the derived stacks
$\RCoh(Z),\mathbb{R}\underline{\mathrm{APerf}}(Z),\RPerf(Z):\mathsf{dAff}^{op}\to \mathsf{Spc},$
where $\mathbb{R}\underline{\mathrm{APerf}}(Z\times S)$ is gives $\mathsf{APerf}(Z\times S)^{\simeq},$ and 
$\RCoh(Z)$ sends a derived affine scheme $S$ to the full sub $\infty$-groupoid,
$$\mathsf{Coh}_S(Z\times S)^{\simeq}\hookrightarrow \mathsf{APerf}(Z\times S)^{\simeq}.$$
They are locally geometric and finitely presented derived stacks by \cite[Prop.~3.13, Cor.~3.29]{toenvaquie2007} i.e. a union of open substacks (of fixed tor-amplitudes) that are geometric. 

If $Z$ is a proper, underived complex scheme $X$, the derived stack $\RCoh(X)$ admits a global cotangent complex. Moreover, $^{\mathrm{cl}}\RCoh(X)$ is the usual stack $M_{coh}(X)$ of coherent sheaves on $X.$ Set $\mathbb{R}\underline{\mathrm{Vect}}(Z):=\coprod_{n\geq 0}\Map(Z,\mathsf{\mathsf{B}GL}_n)$ to be the open derived substack of $\RCoh(Z)$, to be called the \emph{derived stack of vector bundles on $Z$.}
The derived analytic stack $\RAnCoh$ of coherent sheaves is constructed in \cite[Sec.~7.1]{PortaSala2023}.
\end{itemize}

\subsection{Acknowledgments}
The authors thank S. Heller, M. Porta and P. Pandit for helpful discussions. 
In the early stages of this project, J.K. was supported by the Simons Foundation (grant SFI‑MPS‑T‑Institutes‑00007697) and the Ministry of Education and Science of the Republic of Bulgaria (grant DO1‑239/10.12.2024). 
Y.T. was supported by JSPS KAKENHI Grant Numbers JP21H00973, JP21K03246, Startup Grant at BIMSA, and the Beijing NSF Beijing International Scientist Project IS25031. 
Part of this work was conducted while J.K. and Y.T. were visiting Kyoto University and was completed while J.K. held a visiting position at BIMSA; they thank these institutions for their hospitality.

\section{Kapustin--Witten and nonabelian Hodge theory}
\label{sec: KWReview}
In this section, we provide background, recalling the main results of \cite{LRT2022}.

\subsection{Kapustin--Witten equations}
\label{ssec: KW}
Let $M$ be a closed, oriented, smooth, Riemannian 4-manifold with Riemannian metric $g$, and let $P \to M$ be a principal $G$-bundle over $M$ for a compact Lie group $G$. 
Following \cite{KapustinWitten2007}, via a topological twist
of $\mathcal{N}=4$ super Yang–-Mills theory in four dimensions, one has the following family of equations parameterized by $t \in \mathbb{P}^1$ for pairs $(A,\mathfrak{a})$  where $A$ is a connection on $P$ with its curvature $F_A$, and $\mathfrak{a} \in \Omega^1(M, \operatorname{ad} P):=\Gamma^{\infty}(M, \operatorname{ad} P  \otimes T^*M)$: 
\begin{align}
    &\big(F_A-[\mathfrak{a}\wedge\mathfrak{a}]+t d_A \mathfrak{a} \big)^+=0,
   \label{eqn:KW1}
    \\
    &\big(F_A-[\mathfrak{a}\wedge\mathfrak{a}]-t^{-1}d_A\mathfrak{a} \big)^-=0,
    \label{eqn:KW2}
    \\
    &d_A^{*} \mathfrak{a}=0,
    \label{eqn:KW3}
\end{align}
where $\pm$ in the first two conditions are the projections to the self-dual and anti-selfdual parts of $\mathrm{Ad}(P)$-valued two-forms, respectively, and $d_A^*$ is the formal adjoint of the covariant derivative $d_A.$ 
A pair $(A, \mathfrak{a})$ is said to be a solution to the $(G-)$ Kapustin--Witten equations on $M$ if it satisfies the above equations.

For the parameter $t =0$, the Kapustin--Witten equations \eqref{eqn:KW1}, \eqref{eqn:KW2}, \eqref{eqn:KW3} for the pair $(A,\mathfrak{a})$ become: 
\begin{align}
   &F_A^{+}-[\mathfrak{a}\wedge\mathfrak{a}]^+=0,
    \label{eqn:KWt01}
    \\    
   &d_A^{*} \mathfrak{a} =0,\hspace{2mm} \big(d_A\mathfrak{a} \big)^-=0.
    \label{eqn:KWt02}
\end{align}
Note that after gauge fixing, the equation $d_A^{*} \mathfrak{a} =0$ renders this system elliptic.

When the underlying 4-manifold is a K\"ahler surface, which we denote by $S$ from here, the equations \eqref{eqn:KWt01} and \eqref{eqn:KWt02} reduce to the Simposon equations (see \cite{Naka}, \cite{Tana}). Namely, the moduli space of solutions to the equations \eqref{eqn:KWt01} and  \eqref{eqn:KWt02} on a compact K\"ahler surface can be identified with the moduli space of stable Higgs bundles, via Corrette--Donaldson--Hitchin--Simpson--Uhlenbeck--Yau correspondence. 

On the other hand, for the parameter $t \in \mathbb{R} \setminus \{ 0 \}$ the Kapustin--Witten equations \eqref{eqn:KW1}, \eqref{eqn:KW2} read
\begin{gather}
F_{A} - [\mathfrak{a} \wedge \mathfrak{a} ] =0, \, 
d_{A} \mathfrak{a} = d_{A}^{*} \mathfrak{a} =0 .  
\label{eqn:KWt1}
\end{gather}
When the underlying four-manifold is a K\"ahler surface $S$, the system of the equations \eqref{eqn:KWt1} is equivalent to those defining a pluriharmonic metric on a Higgs bundle over $S$ in nonabelian Hodge theory \cite{LRT2022}.
Namely, the moduli space of solutions to the equation \eqref{eqn:KWt1} on a compact K\"aler surface $S$ coincide with the moduli space of semisimple local system on $S$. 

By classical NAH,  one has an interpretation of the moduli spaces of solutions to the Kapustin--Witten equations at $t=0$ and that at $t \in \mathbb{R} \setminus \{ 0 \}$, which are diffeomorphic if the underlying smooth four-manifold is a compact K\"ahler surface. This raises the question of how the two moduli spaces defined by equations \eqref{eqn:KWt01}--\eqref{eqn:KWt02} and \eqref{eqn:KWt1}
are related when the underlying manifold is more generally a compact, oriented, smooth four-manifold.

In fact, \cite[Thm.~5.1, Thm.~5.2, Cor.~5.3]{LRT2022} shows that the two moduli spaces have the same virtual dimension;
the expected dimension is the dimension of the moduli space when it is smooth
and unobstructed, computed e.g. via the index of the Atiyah--Hitchin--Singer
deformation complex. 

In this paper, we pursue a different direction by relaxing the assumption on the structures of underlying manifold. We focus on the case where the underlying manifold is in fact a algebraic surface and investigate the geometry of the moduli spaces at the level of their derived enhancements. More specifically, each of these moduli spaces carries a $(-2)-$shifted symplectic structure, and our goal is to clarify their relationship with gauge-theoretic structures within the framework of derived algebraic and analytic geometry.

\begin{remark}
    The parameter $t$ indexes the family \eqref{eqn:KW1}, \eqref{eqn:KW2}, \eqref{eqn:KW3},  
    but note that for 
    $t\in \bfP \setminus (\mathbb{R}\cup \{\infty\}),$ 
    this system is overdetermined, while for $t=\infty$ are orientation-reversed equations for $t=0.$ Foundations of derived geometry of overdetermined systems of non-linear PDEs appear in \cite{KSY23,KSY2} and \cite{Kry2026}. A more detailed treatment is unfortunately outside the scope of the current article, and by keeping $t$ real, we may follow instead the approaches of \cite{BJ}. 
\end{remark}

In the next section, we recall the basic facts concerning derived symplectic geometry which will be used to develop the notion of a hyperk\"ahler structure in the context of derived algebraic and analytic geometry.

\section{Shifted symplectic geometry}
\label{sec: Shifted symplectic geometry}
We recall basic notions from shifted symplectic derived algebraic geometry, taking the opportunity to further fix some notations. We closely follow \cite{PTVV13,CPTVV17} and \cite[App.~B]{CalaqueHaugsengScheimbauer2025} in the relative setting.

First, in \S\ref{ssec: RelSSS} we recall shifted symplectic in the relative setting. In \S\ref{AnalyticConventions} we recall aspects of derived analytic geometry and in \S\ref{ssec: Lags2} we recall the definition of relative shifted Lagrangian correspondences, formulated in the analytic setting relative to $\bfP$ which are suitable for our needs.

Given a derived scheme $S$ with an action of an algebraic group $G$, denote by $S/G$ the simplicial derived scheme, obtained via the nerve of the action groupoid. Denote by $[S/G]$ the associated derived stack. The basic example is the classifying stack $\mathsf{B}G\simeq [\mathrm{pt}/G].$ 

\subsection{Relative shifted symplectic structures}
\label{ssec: RelSSS}
Let $f:B\to S$ be a morphism of derived prestacks with relative cotangent complex $\mathbb{L}_{B/S}$. Consider its relative
de Rham algebra $\mathrm{DR}(B/S)$, a graded mixed commutative dg-algebra in $\mathsf{QCoh}(S),$ whose graded pieces give the relative $p$-forms and whose mixed structure is provided by the de Rham differential $d_{\mathrm{DR}}$ \cite{CalaqueHaugsengScheimbauer2025,calaque2024shifted}.

More explicitly, there is a canonical morphism 
$$\Gamma_S(B,\mathrm{Sym}(\mathbb{L}_{B/S}[-1]))\to \mathrm{DR}(B/S)^{\epsilon},$$
with $p$-th graded piece $\Gamma_S(B,\mathrm{Sym}^p(\mathbb{L}_{B/S}[-1])).$

For $n\in\mathbb{Z}$ and $p\geq 0$ a non-negative integer let $\EuScript{A}^{p}_S(B,n)$ and $\mathcal{A}_{S}^{p,cl}(B,n)$ be the spaces of $n$-shifted relative $p$-forms and closed $p$-forms respectively.
In the absolute case, when $B$ is derived Artin stack over $K$, by \cite[Prop.~1.14]{PTVV13}, there is an equivalence
\begin{equation}
\label{MapsShiftedForms}
\mathrm{Maps}_{\mathsf{QCoh}(B)}(\mathcal{O}_B,\wedge^q_{\mathcal{O}_B}\mathbb{L}_{B/K}[n])\simeq \mathcal{A}_K^q(B,n).
\end{equation}

As a representable functor on $\mathsf{dSt}_{S}$, we denote the representing object $A_S^p(n)$. We follow \cite{CalaqueHaugsengScheimbauer2025} and denote by 
$\mathrm{ClDF}^{2,n}\to \mathsf{dSt}^{\Delta^1},$
the right fibration corresponding to the functor 
$$\underline{\mathcal{A}}_{(-)}^{2,cl}(-,n):\mathsf{dSt}^{\Delta^1,op}\to \mathsf{Spc},\hspace{2mm}(X\to S)\mapsto
\mathcal{A}_S^{2,cl}(X,n).$$ Thus,
$\mathrm{ClDF}^{2,n}\to \mathsf{dSt}^{\Delta^1,op}
\xrightarrow{ev}\mathsf{dSt}$ is a cartesian fibration.
\begin{remark}
With these conventions, whenever we have (shifted) closed $(p=2)$-forms, we call them \emph{pre-symplectic forms}.
\end{remark}
In the relative setting, an $n$-presymplectic derived Artin stack $Z$ over $S$ is equivalently a morphism $Z\to A_S^{2,cl}(n)$ in $\mathsf{dSt}_S$. 
The $\infty$-category of $n$-shifted presymplectic stacks over $S$ (resp. $n$-shifted presymplectic stacks) is
\begin{equation}
    \label{nPreSymp}
\mathsf{PreSympSt}_{S,n}:=\mathsf{dSt}_S^{\mathrm{Art}}\times_{\mathsf{dSt}_S}\mathsf{dSt}_{S/A_S^{2,cl}(n)},
\end{equation}
(resp. $\mathsf{PreSympSt}_{n}:=\mathsf{dSt}^{\Delta^1,\mathrm{Art}}\times_{\mathsf{dSt}^{\Delta^1}}\mathrm{ClDF}^{2,n}$).
\begin{remark}
For each $p,n\in\mathbb{Z}$ there is an equivalence $\mathrm{ClDF}_{S}^{p,n}\simeq \mathsf{dSt}_{S/A_S^{p,\mathrm{cl}}(n)}.$ 
    \end{remark} 
Fix $\EuScript{L}\in \mathrm{Pic}^{gr}(S),$ where $\mathrm{Pic}^{gr}$ is the classifying stack of graded line bundles.
The space of $\EuScript{L}$-twisted presymplectic structures on $Z\to S$ is given by 
\begin{equation}
    \label{eqn:Ltwistedpresymp}
\EuScript{A}_{S}^{2,\mathrm{cl}}(Z;\EuScript{L}):=\mathrm{Maps}_{\mathsf{QCoh}(S)^{\mathrm{gr},\epsilon}}(\mathcal{O}_S(p)[-p],\mathrm{DR}_S(Z)\otimes\EuScript{L}).
\end{equation}
Note that $\EuScript{L}=\mathcal{O}_S[n]$-twisted closed relative $p$-forms are just $n$-shifted closed relative $p$-forms.

An element of \eqref{eqn:Ltwistedpresymp} is equivalent to a morphism $\omega_Z:Z\to A_S^{2,cl}(\EuScript{L}),$ in $\mathsf{PSt}_S.$ If $f:Y\to Z$ is a morphism of $S$-prestacks and $Z$ has an $\EuScript{L}$-twsited presymplectic structure $\omega_Z$, we put $f^*\omega:=\omega_Z
\circ f:Y\to A_S^{2,cl}(\EuScript{L}),$ for induced presymplectic form on $Y.$

\subsubsection{Lagrangian structures}
\label{Lags1}
Let $f:B\to S$ be a morphism of derived Artin stacks, with $S$ equipped with an $n$-shifted symplectic form $\omega.$

An isotropic structure on $f:B\to S$ is a nullhomotopy $h:0\sim f^*\omega$ in $\EuScript{A}^{2,\mathrm{cl}}(B,n)$ i.e. an element $h\in \EuScript{A}^{2,cl}(B,n-1)$ such that $(\delta+d_{\mathrm{DR}})h=f^*(\omega).$ Explicitly, it is given by a collection of forms $\{h_i\}$ such that $dh_0=f^*(\omega_0),$ and $\delta h_{i+1}+d_{\mathrm{DR}}h_{i}=f^*\omega_{i+1}.$ The underlying form $h_0$ defines a morphism $\mathbb{T}_{\B}\to \mathbb{L}_{\B}[n-1].$

From the fiber sequence $\mathbb{L}_{B/S}[n-1]\to f^*\mathbb{L}_{S}[n]\to \mathbb{L}_{\B}[n],$ there is a canonical morphism 
\begin{equation}
    \label{eqn: Thetah}
\Theta_h:\mathbb{T}_{\B}\to \mathbb{L}_{B/S}[n-1].
\end{equation}
A \emph{Lagrangian structure on $f$} is an isotropic structure which is non-degenerate in the sense that \eqref{eqn: Thetah} is a quasi-isomorphism. 

We will need twisted versions, stated for not necessarily geometric derived stacks. For a general treatment of coisotropic and Lagrangian correspondences, we refer to \cite{HMS2022},\cite{Haugseng2018}. Isotropic structures are special cases of isotropic correspondences.
\begin{definition}
    Let $f:L\to Z$ be a morphism of derived $S$-prestacks. An $\EuScript{L}$-twisted isotropic structure on $f$ is a commuting square
    \[
    \begin{tikzcd}
    L\arrow[d]\arrow[r,"f"]& Z\arrow[d,"\omega_Z"]
    \\
    S\arrow[r] & A_S^{2,cl}(\EuScript{L}).
    \end{tikzcd}
    \]
An $\EuScript{L}$-twisted isotropic correspondence on $Z\to S$ to $Z'\to S$ is a commutative diagram of $S$-prestacks,
\[
\begin{tikzcd}
L\arrow[d]\arrow[r] & Z'\arrow[d,"\omega_{Z'}"]
\\
Z\arrow[r,"\omega_Z"]& A_S^{2,cl}(\EuScript{L}).
\end{tikzcd}
\]
\end{definition}
We postpone the definition of twisted Lagrangian correspondence to \S\ref{ssec: Lags2}, since we want to formulate it in the analytic setting. 
Thus, we now briefly recal conventions regarding derived analytic geometry, referring to \cite[Sect.~2]{HolsteinPorta2025} for comprehensive overview.
\subsection{Conventions from analytic geometry}
\label{AnalyticConventions}
We let $\EuScript{T}_{\mathrm{an}}$ denote the (ordinary) category whose objects are Stein open subsets of $\mathbb{C}^n$ and whose morphisms are holomorphic functions. We endow $\EuScript{T}_{\mathrm{an}}$ with the analytic topology, denoted $\tau_{\mathrm{an}}$. 
The definition of derived analytic stacks begins with a more abstract notion of pre-geometry $\EuScript{T}$ and of a $\EuScript{T}$-structure on an $\infty$-topos $\EuScript{X}$ \cite[3.1.1, 3.1.4]{LurieDAGV}. We do not go into details, referring instead to \cite[Defs 2.1--2.4]{PortaYu2020} 
for a comprehensive treatment of the relevant pre-geometries for derived $\mathbb{C}$-analytic geometry.
\begin{definition}\label{AnalyticTopos}
Let $\EuScript{X}$ be an $\infty$-topos. A \emph{$\EuScript{T}_{\mathrm{an}}$-structure} is a functor 
$$
\mathcal{O} : \EuScript{T}_{\mathrm{an}} \to \EuScript{X},$$
such that:
\begin{enumerate}
\item $\mathcal{O}$ commutes with products and pullbacks along open immersions in $\EuScript{T}_{\mathrm{an}}$;
\item $\mathcal{O}$ takes $\tau_{\mathrm{an}}$-covers to morphisms which are $\pi_0$-epimorphic in $\EuScript{X}$.
\end{enumerate}
The $\infty$-category of $\EuScript{T}_{\mathrm{an}}$-structures in $\EuScript{X}$ is denoted $\mathsf{Str}_{\EuScript{T}_{\mathrm{an}}}(\EuScript{X})$. A morphism $f : \mathcal{O} \to \mathcal{O}'$ of $\EuScript{T}_{\mathrm{an}}$-structures is \emph{local} if for every open immersion $\phi : U \to V$ in $\EuScript{T}_{\mathrm{an}}$, the square
\[
\begin{tikzcd}
\mathcal{O}(U) \arrow[d]\arrow[r] & \mathcal{O}(V) \arrow[d] \\
\mathcal{O}'(U) \arrow[r] &  \mathcal{O}'(V),
\end{tikzcd}
\]
is a pullback square. The $\infty$-category of $\EuScript{T}_{\mathrm{an}}$-structures and local morphisms between them is denoted $\mathsf{AnRing}(\EuScript{X})$, and we refer to it as the $\infty$-category of analytic rings in $\EuScript{X}$.
\end{definition}

\begin{example}
\label{AnExample}
\normalfont
Let $X$ be a $\mathbb{C}$-analytic space and let $X_{\mathrm{top}}$ denote the underlying topological space of $X$ and consider the $\infty$-topos $\mathsf{Shv}(X_{\mathrm{top}})$ of sheaves on $X_{\mathrm{top}}$. We can define an analytic ring $\mathcal{O}$ on $X$ as the functor
$$\mathcal{O}:\EuScript{T}_{\mathrm{an}}\to \mathsf{Shv}(X_{\mathrm{top}}),$$
sending $U \in \EuScript{T}_{\mathrm{an}}$ to the sheaf $\mathcal{O}(U)$ on $X_{\mathrm{top}}$ defined by
$$
\mathsf{Opens}(X_{\mathrm{top}}) \ni V \mapsto \mathcal{O}(U)(V) := \mathrm{Hom}_{\mathsf{An}}(V,U) \in \mathsf{Set},$$
where $\mathsf{An}_{\mathbb{C}}$ denotes the category of $\mathbb{C}$-analytic spaces. Here $\mathsf{Open}(X_{\mathrm{top}})$ is the lattice of open subsets of $X_{\mathrm{top}}.$ 
In particular, $\mathcal{O}(\mathbf{A}_{\mathbb{C}}^1)$
corresponds with the sheaf of
holomorphic functions on $X.$
\end{example}
Therefore, the $\infty$-topos for derived analytic geometry is $\mathsf{Str}_{\EuScript{T}_{\mathrm{an}}}\big(\mathsf{Shv}(X_{\mathrm{top}})\simeq \mathsf{AnRing}\big(\mathsf{Shv}(X_{\mathrm{top}})\big).$
\begin{definition}
\label{AnalyticSpace}
A \emph{derived $\mathbb{C}$-analytic space} is a pair \((\EuScript{X}, \mathcal{O}_X)\), where \(\EuScript{X}\) is an \(\infty\)-topos and \(\mathcal{O}_X\) is an analytic ring in \(\EuScript{X}\), such that:
\begin{enumerate}
\item locally on $\EuScript{X}$, \((\EuScript{X}, \pi_0\mathcal{O}_X)\) is equivalent to a structured topos arising from Example \ref{AnExample}, where locally means there exists a collection of objects \(\{U_i\}\) of \(\EuScript{X}\) such that \(U_i \to \mathbf{1}_\EuScript{X}\) is an effective epimorphism and such that the structured topos \((\EuScript{X}/U_i, \mathcal{O}_X|_{U_i})\) satisfies the given condition;
\item For every $i\in\mathbb{Z}$, the sheaves \(\pi_i(\mathcal{O}^{\mathrm{alg}}_X)\) are coherent as sheaves of \(\pi_0(\mathcal{O}^{\mathrm{alg}}_X)\)-modules.
\end{enumerate}
\end{definition}
Finally, we need a definition of closed immersion \cite{Lurie2011}.
\begin{definition}
\label{ClImmersions}
A morphism $f : \EuScript{X} = (\EuScript{X},\mathcal{O}_X) \to \EuScript{Y} = (\EuScript{Y},\mathcal{O}_Y)$ of derived $\mathbb{C}$-analytic spaces is a \emph{closed immersion} if:
\begin{enumerate}
\item the underlying morphism of $\infty$-topoi $f^{-1}: \EuScript{Y} \rightleftarrows \EuScript{X} : f_*$ is a closed immersion of $\infty$-topoi;\footnote{More details can be found in \cite[Subsect.7.3.2]{Lur09}.}
\item the induced morphism $f^{-1}\mathcal{O}_Y \to \mathcal{O}_X$ is an effective epimorphism ($\pi_0$-surjection) in the sense that for any $U \in \EuScript{T}_{\mathrm{an}}$ the morphism $f^{-1}\mathcal{O}_Y(U) \to \mathcal{O}_X(U)$ is an effective epimorphism (surjection on $\pi_0$) in $\EuScript{X}$.
\end{enumerate}
\end{definition}
The following can be found in \cite[Section.11]{Lurie2011}.
\begin{proposition}
\label{PullbackClosedImmersionLemma}
Let
\[
\begin{tikzcd}
X' \ar[r,"g"] \ar[d,"i"] & Y' \ar[d,"j"] \\
X \ar[r,"f"] & Y
\end{tikzcd}
\]
be a pullback diagram in $\mathsf{dAn}_{\mathbb{C}}$, where $j$ is a closed immersion. Then:
\begin{enumerate}
\item the underlying algebra of the structure sheaf of $X'$ can be computed as the tensor product
\[
\mathcal{O}_{Y'} \otimes_{f^{-1}\mathcal{O}_Y} f^{-1}j_*\mathcal{O}_X.
\]
\item For any $F \in \mathcal{O}_X\mathsf{-Mod}$, the natural transformation $
f_*j_*(F) \to i_*g_*(F)$
is an equivalence.
\end{enumerate}
\end{proposition}
Analytic stacks can be defined to take values in $\infty$-categories. We denote by
\[
\mathsf{dAnSt}_{\mathbb{C}}^{\mathrm{Cat}} := \mathsf{Shv}_{\mathsf{Cat}_\infty}(\mathsf{dStn}_{\mathbb{C}}, \tau_{\mathrm{an}})
\]
the $\infty$-category of sheaves on $(\mathsf{dStn}_{\mathbb{C}}, \tau_{\mathrm{an}})$ with values in $\mathsf{Cat}_\infty$. Since $\mathsf{Spc}$ canonically embeds in $\mathsf{Cat}_{\infty},$ we have a canonical natural embedding 
\begin{equation}
    \label{CatStks}
J:\mathsf{dAnSt}_{\mathbb{C}}\hookrightarrow \mathsf{dAnSt}_{\mathbb{C}}^{\mathrm{Cat}}.\end{equation}
In other words, $\mathsf{dAnSt}_{\mathbb{C}}$ is the $\infty$-category of hypercomplete sheaves $\mathsf{dAnSt}_{\mathbb{C}}:=\mathsf{Shv}\big(\mathsf{dAfd}_{\mathbb{C}},\tau_{\text{\'et}}\big)^{\wedge},$ on the Grothendieck site $(\mathsf{dAfd}_{\mathbb{C}},\tau_{\text{\'et}})$ (see \cite[Subsect.~3.1]{PortaYuDerivedGAGA2019} and \cite[Subsect.~7.1]{PortaYu2018}).
\subsubsection{Analytic mapping stacks}
\label{sssec: Analytic mapping stacks}
All details are found in \cite{HolsteinPorta2025}, we recall only the basic definition. 
Let $Z,Y$ be objects of $\mathsf{dAnSt}_{\mathbb{C}}$. The derived analytic mapping stack is denoted by
$$\AnMap(Z,Y):\mathsf{dAn}_{\mathbb{C}}^{\mathrm{op}}\to \mathsf{Spc},\hspace{2mm}U\mapsto \mathrm{Maps}_{\mathsf{dAnSt}_{\mathbb{C}}}(U\times Z,Y).$$
Relative to $\bfP$, we write $\AnMap_{/\bfP}(Z,Y).$
\begin{remark}
    The inclusion \eqref{CatStks} is strong monoidal and commutes with internal hom spaces \cite{PortaYu2020}, so we will use the same notation for both.
\end{remark}
Let $Z$ be a derived algebraic stack and $X$ a proper geometric derived stack locally almost of finite presentation over $\mathbb{C}$. There is a natural morphism
\begin{equation}
\label{eqn: Map to AnMap}
C_{X,Z}:\Map(X,Z)^{\mathrm{an}}\to \AnMap(X^{\mathrm{an}},Z^{\mathrm{an}}),
\end{equation}
of derived analytic stacks.

With certain assumptions, found in the main theorem \cite[Thm.~6.14]{HolsteinPorta2025}, this is an equivalence. We will frequently use their theorem, since the derived algebraic stacks we study in this paper e.g. $Z=\Perf,\mathsf{B}G$ satisfy such assumptions, and so the comparison map \eqref{eqn: Map to AnMap} will be an equivalence for our applications.
\subsubsection{Analytic cotangent complexes}
Throughout, we always assume our derived analytic stacks $Z$ admit perfect cotangent complexes $\mathbb{L}_{Z}^{\mathrm{an}}.$ Properties and constructions are found in \cite[Thm. 1.5]{PortaYu2020}.

The analytic cotangent complex enjoys nice finiteness properties, differing from the algebraic one. For instance, for any morphism of derived analytic spaces $f:Z\to W$, its analytic cotangent complex
$\mathbb{L}_{Z/W}^{\mathrm{an}}:=\mathbb{L}_{\mathcal{O}_Z/f^{-1}\mathcal{O}_W}^{\mathrm{an}},$
is connective with coherent cohomology \cite[Cor.~5.40]{PortaYu2020}. For example, $\mathbb{L}_{\mathbf{A}^n}^{\mathrm{an}}$ is perfect of tor-amplitude $0$ (cf., \cite[Cor.~5.26]{PortaYu2020}).

We will require two standard properties of $\mathbb{L}_{(-)}^{\mathrm{an}}.$
\begin{proposition}{\cite[Thm. 1.5 (2),(3)]{PortaYu2020}}
\label{AnCotangentProperties}
Let $W\xrightarrow{f}V\xrightarrow{g}Z$ be composeable morphisms of derived analytic stacks. Then, there is a fiber sequence
$$f^*\mathbb{L}_{V/Z}^{\mathrm{an}}\to \mathbb{L}_{W/Z}^{\mathrm{an}}\to \mathbb{L}_{W/V}^{\mathrm{an}}.$$
Furthermore, for any pull-back square of derived analytic spaces,
\[
\begin{tikzcd}
    W'\arrow[d,"g"] \arrow[r] & V'\arrow[d,"f"]
    \\
    W\arrow[r]& V,
\end{tikzcd}
\]
there is a canonical equivalence $\mathbb{L}_{W'/V}^{\mathrm{an}}\simeq g^*\mathbb{L}_{W/V}^{\mathrm{an}}.$
\end{proposition}
Recall the $+$-pushforward $
\eta_+$ along a proper flat map $\eta:Z\to \bfP$ \cite{PortaYu2020}. Via base-change, we have the following.
\begin{proposition}
\label{AnMapCotComplex}
    Let $X$ be a derived geometric stack locally of finite presentation such that there is a map $X^{\an}\to \bfP$ which is proper and flat. Let $Z\to\bfP$ be finitely presented with perfect relative analytic cotangent complex $\mathbb{L}_{Z/\bfP}^{\an}.$ Then the map $\AnMap_{/\bfP}(X^{\an},Z)\to\bfP$   admits a relative analytic cotangent complex, given by 
    $\pi_+\big(\mathrm{ev}^*\mathbb{L}_{Z/\bfP}^{\an})\in \mathsf{Perf}\big(\AnMap_{/\bfP}(X^{\an},Z)\big),$ where $\pi$ is the projection in the diagram \eqref{CartesianDiaghAnMap} below.
\end{proposition}
\begin{proof}
    Let $U$ be derived affinoid with map $u:U\to \AnMap_{/\bfP}(X^{\an},Z)$ relative to $\bfP$ with $f_u:X^{\an}\times_{\bfP}U\to Z$ the morphism classified by it.
    Consider the cartesian diagran
    \begin{equation}
    \label{CartesianDiaghAnMap}
    \begin{tikzcd} X^{\an}\times_{\bfP} U \ar[r] \ar[d,"\pi_u"] & X\times_{\bfP}\AnMap_{/\bfP}(X^{\an},Z) \ar[d,"\pi"] \\ U \ar[r,"u"] & \AnMap_{/\bfP}(X^{\an},Z). \end{tikzcd}
    \end{equation}
We need to confirm the universal property of the analytic cotangent complex relative to $\bfP$. To this end, let $\mathcal{E}\in \mathsf{Coh}^{\geq 0}(U)$, and letting $\underline{\mathrm{Der}}^{\an}_{(-)/\bfP}(-,\mathcal{E})$ denote the space of analytic derivations, it suffices to note there are equivalences
$$\underline{\mathrm{Der}}^{\an}_{\AnMap_{\bfP}(X^{\an},Z)/\bfP}(U,\mathcal{E})\simeq \mathrm{Maps}_{\mathsf{Coh}^+(X^{\an}\times_{\bfP}U)}\big(f_u^*\mathbb{L}_{Z/\bfP}^{\an},\pi_u^*\mathcal{E}\big)\simeq \mathrm{Maps}_{\mathsf{Coh}^+(U)}\big(\pi_{u+}(f_u^+\mathbb{L}_{Z/\bfP}^{\an}),\mathcal{E}\big).$$
By proper base-change for $+$-pushforwards, since \eqref{CartesianDiaghAnMap} is cartesian,  
$\pi_{u+}(f_u^*\mathbb{L}_{Z/\bfP}^{\an})\simeq u^*\big(\pi_+\mathrm{ev}^*\mathbb{L}_{Z/\bfP}^{\an}\big)=u^*\mathcal{F},$ and an equivalence
$$\underline{\mathrm{Der}}^{\an}_{\AnMap_{\bfP}(X^{\an},Z)/\bfP}(U,\mathcal{E})\simeq \mathrm{Maps}_{\mathsf{Coh}^+(U)}\big(u^*\mathcal{F},\mathcal{E}\big).$$
\end{proof}
More generally, let $\underline{\mathrm{AnRes}}_{X^{\an}/\bfP}(Z)$ be the analytic Weil restriction of $Z\to X^{\an}$ along $X^{\an}\to\bfP,$ with universal evaluation diagram
\begin{equation}
    \label{UniverAnResEval}
\begin{tikzcd}
X^{\an}\times_{\bfP}\underline{\mathrm{AnRes}}_{X^{\an}/\bfP}(Z) \ar[r,"\mathrm{ev}"] \ar[d,"\pi"] & Y \ar[d] \\ \underline{\mathrm{AnRes}}_{X^{\an}/\bfP}(Z) \ar[r] & \bfP.\end{tikzcd}
\end{equation}
With the same hypothesis as in Proposition \ref{AnMapCotComplex}, an analogous argument the universal property for $f_u:X^{\an}\times_{\bfP}U\to Z,$ over $X^{\an}$ is given by
$$\underline{\mathrm{Der}}^{\an}_{\underline{\mathrm{AnRes}}_{X^{\an}/\bfP}(Z)/\bfP}\big(U,\mathcal{E}\big)\simeq \mathrm{Maps}_{\mathsf{Coh}^+(X^{\an}\times_{\bfP}U))}\big(f_U^*\mathbb{L}_{Z/X^{\an}}^{\an},\pi_u^*\mathcal{E}\big),$$
which shows that the induced map $q:\underline{\mathrm{AnRes}}_{X^{\an}/\bfP}(Z) \to \bfP$ admits a relative analytic cotangent complex which is perfect.

In order to define de Rham algebra and thus make sense of shifted symplectic structures on derived analytic stacks, the following result is central.
\begin{proposition}
\label{AnDeRhamAlgs}
Let $Z$ be a derived geometric stack locally of finite presentation over $\mathbb{C}$ with cotangent complex $\mathbb{L}_{Z/\mathbb{C}}.$
Let $Z^{\an}$ denote its analytification. 
Then, there exists a canonical equivalence of graded mixed commutative dg-algebras
$$\mathrm{DR}(Z/\mathbb{C})^{\an}\xrightarrow{\simeq} \mathrm{DR}^{\an}(Z^{\an}/\mathbb{C}),$$
where $\mathrm{DR}^{\an}(Z^{\an}/\mathbb{C})=\mathrm{Sym}(\mathbb{L}_{Z^{\an}/\mathbb{C}}^{\an}[-1]).$
\end{proposition}
\begin{proof}
Let $\mathrm{DR}(Z/\mathbb{C})=\mathrm{Sym}_{\mathcal{O}_Z}(\mathbb{L}_{Z/\mathbb{C}}[-1])$ be the algebraic de Rham algebra with universal differential $d_{Z/\mathbb{C}}:\mathcal{O}_Z\to\mathbb{L}_{Z/\mathbb{C}}.$ We must show that since analytification is a left-adjoint it commutes with colimits and yields an equivalence
\begin{equation}
    \label{DeRhamAn}
\mathrm{DR}(Z/\mathbb{C})^{\an}\simeq \big(\mathrm{Sym}_{\mathcal{O}_Z}(\mathbb{L}_{Z/\mathbb{C}}[-1])\big)^{\an}\simeq\mathrm{Sym}\big((\mathbb{L}_{Z/\mathbb{C}}[-1])^{\an}\big)\simeq \mathrm{Sym}_{\mathcal{O}_{Z^{\an}}}\big(\mathbb{L}_{Z^{\an}/\mathbb{C}}^{\an}[-1]\big).\end{equation}
By the universal property of analytification and 
$\mathbb{L}_{(-)/\mathbb{C}},$ both $\mathbb{L}_{Z^{\an}/\mathbb{C}}^{\an}$ and the analytification $(\mathbb{L}_{Z/\mathbb{C}})^{\an}$ satisfy descent and by reduction to the case of finitely presented derived affine schemes, there is a canonical equivalence $\phi:(\mathbb{L}_{Z/\mathbb{C}})^{\an}\xrightarrow{\simeq}\mathbb{L}_{Z^{\an}/\mathbb{C}}^{\an}$ \cite[Lem.~2.2]{Hanson2025}, which induces \eqref{DeRhamAn}. Indeed, by analytification of the universal algebraic differential, there is a map $d_{Z/\mathbb{C}}^{\an}:\mathcal{O}_{Z^{\an}}\to (\mathbb{L}_{Z/\mathbb{C}})^{\an}.$ Composing with the equivalence $\phi,$ gives an analytic derivation
$$\delta:\mathcal{O}_{Z^{\an}}\xrightarrow{d_{Z/\mathbb{C}}^{\an}}(\mathbb{L}_{Z/\mathbb{C}})^{\an}\xrightarrow{\phi}\mathbb{L}_{Z^{\an}/\mathbb{C}}^{\an},$$
which one may verify is the universal analytic derivation for $Z^{\an}$, such that
\[
\begin{tikzcd}
    \mathcal{O}_{Z^{\an}}\arrow[d,"\mathrm{id}"] \arrow[r,"d_{Z/\mathbb{C}}^{\an}"] & (\mathbb{L}_{Z/\mathbb{C}})^{\an}\arrow[d,"\phi"]
    \\
    \mathcal{O}_{Z^{\an}}\arrow[r,"d_{Z^{\an}/\mathbb{C}}"] & \mathbb{L}_{Z^{\an}/\mathbb{C}}^{\an},
\end{tikzcd}
\]
commutes. Thus, it is compatible with the graded-mixed structures and induces an equivalence of graded mixed algebras is given by $\mathrm{Sym}(\phi[-1]).$
\end{proof}
Using Proposition \ref{AnDeRhamAlgs}, we follow \eqref{MapsShiftedForms} and may define the space of $n$-shifted $q$-forms on the analytification $Z^{\mathrm{an}}$ is 
$\mathrm{AnMaps}(\mathcal{O}_{Z^{\mathrm{an}}},\wedge^q\mathbb{L}_{Z}^{\mathrm{an}}[n]\big).$ Denote it by $\EuScript{A}^q(Z^{\mathrm{an}},n).$
Analogous to \eqref{nPreSymp}, specializing over $\mathbf{A}^k=(\mathbb{A}_{\mathbb{C}}^k)^{\mathrm{an}},$ (for future applications, with $k=1$), denote the category of $n$-shifted relative presymplectic analytic stacks over $\mathbf{A}^1_{\mathbb{C}}$ by
\begin{equation}
    \label{nPreSympAnalytic}
\mathsf{PreSympAnSt}_{\mathbf{A}^1_{\mathbb{C}},n}:=\mathsf{dAnSt}_{\mathbf{A}^1_{\mathbb{C}}}^{\mathrm{lft}}\times_{\mathsf{dAnSt}_{\mathbf{A}^1_{\mathbb{C}}}}\mathsf{dAnSt}_{\mathbf{A}^1_{\mathbb{C}}/A_{\mathbf{A}^1_{\mathbb{C}}}^{2,\mathrm{cl}}(n)}.
\end{equation}
It is directly related with \eqref{nPreSymp}, via $(-)^{\mathrm{an}}$ since it preserves local geometricity of derived stacks.

\begin{remark}[Notation]
Later, with added non-degeneracy conditions, we study objects of 
\begin{equation}
    \label{DAnStk1}
\mathsf{dAnSt}_{\mathbb{P}_{\mathbb{C}}^1/A_{\mathbb{P}^1}^{2,cl}(n)^{\mathrm{nd}}},
\end{equation}
to be denoted $(\eta_Z:Z\to \bfP,\omega_Z).$
\end{remark}

\subsection{Analytic Lagrangian structures}
\label{ssec: Lags2}

We now state the definition of relative shifted Lagrangian correspondences we will use in this paper. It is stated for analytic derived stacks and we note that while such a definition has not appeared in the literature before, it is essentially the same as the algebraic setting, using instead analytic cotangent complexes.

With twistor-space geometry in mind, we state the definition for derived analytic stacks relative to $\bfP,$ via the $\infty$-topos \eqref{DAnStk1}, but it holds for more general bases.

\begin{definition}
\label{LagStructure}
 An \emph{$n$-shifted analytic relative Lagrangian correspondence over $\mathbf{P}^1$}, denoted 
 $L:Z_1\dashrightarrow  Z_2,$ is a morphism in 
 $$\mathrm{Corr}\big(\mathsf{dAnSt}_{\mathbf{P}^1/\EuScript{A}^{2,cl}[n]}\big),$$
 such that $L\to \mathbf{P}^1$ is locally geometric and locally of finite presentation, for which there is a cartesian square of perfect complexes,
\[
\begin{tikzcd}
\mathbb{T}_{L/\mathbf{P}^1}^{\mathrm{an}}\arrow[d]\arrow[r] & \pi_{2}^*\mathbb{T}_{Z_2/\mathbf{P}^1}^{\mathrm{an}}\simeq \mathbb{L}_{Z_2/\mathbf{P}^1}^{\mathrm{an}}[n]\arrow[d]
\\
\pi_{1}^*\mathbb{T}_{Z_1/\mathbf{P}^1}^{\mathrm{an}}\simeq \pi_1^*\mathbb{L}_{Z_1/\mathbf{P}^1}^{\mathrm{an}}[n] \arrow[r]& \mathbb{L}_{L/\mathbf{P}^1}^{\mathrm{an}}[n],
\end{tikzcd}
\]
where $\pi_1,\pi_2,$ are the projections.
\end{definition}

Fixing a perfect complex $\EuScript{F}\in \mathsf{Perf}(\bfP)$, there is an analogue of definition \ref{LagStructure}. The more appropriate notion for our purposes is an $\EuScript{F}$-twisted notion of relative Lagrangian correspondence (Def. \ref{defn:HyperLagStructure}).


\subsubsection{$\mathcal{O}$-compactness, $\mathcal{O}$-orientations} 
For notations and definitions of $\mathcal{O}$-compact derived stacks and
$\mathcal{O}$-orientations, we refer the reader to \cite[\S~2]{PTVV13}. We need a twisted version which uses the notion of \emph{universally cocontinuous} (resp, \emph{universally perfect}) morphism $\pi:B\to S$ of prestacks \cite[Def.~2.17]{calaque2024shifted}. This means quasi-coherent pushforward of the morphism induced from base-change along a derived affine preserves colimits (resp. perfect complexes).
\begin{definition}
Let $f:B\to S$ be universally cocontinuous, let $\EuScript{L}\in \mathrm{Pic}^{gr}(B)$ and let $d\in\mathbb{Z}$. A \emph{$d$-preorientation ($\EuScript{L}$-twisted) on $B$ over $S$} is a morphism $[B]_{\EuScript{L}}:\Gamma_S(B,\EuScript{L})\to \mathcal{O}_S[-d].$

\end{definition}
An $\EuScript{L}$-twisted $d$-preorientation is an element of $\mathrm{Maps}_{\mathsf{QCoh}(S)}(\Gamma_S(B,\EuScript{L})[d],\mathcal{O}_S).$ By the Yoneda embedding $\mathsf{dPSt}_S\to \mathsf{Funct}(\mathsf{dPSt}_S,\mathsf{Spc})^{op},$ we view an $S$-prestack as a functor $\mathsf{dPSt}_S\to \mathsf{Spc},$ where $B\mapsto\mathrm{Maps}_{\mathsf{dPSt}_S}(B,-).$ Set
$$\mathsf{PrOr}_{S}^{(-)}(d):\mathsf{dPSt}_{S\times \mathsf{Pic}^{gr}}\to \mathsf{Spc},\hspace{2mm}(B\to S,\EuScript{L})\to \mathrm{Maps}_{\mathsf{QCoh}(S)}(\Gamma_S(B,\EuScript{L}),\mathcal{O}_S[-d]).$$
Therefore, an $\EuScript{L}$-twisted $d$-preorientation on $B\to S$ is just a morphism of $S$-prestacks
$\mathsf{PrOr}_S^{(-)}(d)\to B.$ 

\subsubsection{Shifted Poisson structures from Lagrangians}
We now recall an important result attributed first to Costello--Rozenblyum (see also \cite[Lem.~2.3]{Spaide2016}) as stated in \cite[Thm.~4.22]{MelaniSafronovII2018}, which constructs shifted Poisson structures from shifted Lagrangian thickenings\footnote{Without going into details, an $n$-shifted Poisson structure on a morphism of derived prestacks $\mathfrak{X}\to \mathfrak{S}$ is equivalent to the existence of a Lagrangian structure on a morphism $f:\mathfrak{X}\to \mathfrak{Y}$, to an $(n+1)$-shifted symplectic formal derived stack $\mathfrak{Y}.$} \cite{Cal}. A more general version is given by \cite[Thm.~4.3]{Tom2025}, generalized further in \cite[Thm.~4.1]{Tom2026}.  
\begin{proposition}
\label{LagThicken}
Let $f:B\to S$ be a morphism of locally geometric derived stacks locally of finite presentation. Suppose that $S$ is equipped with an $n$-shifted symplectic form, for some $n\in\mathbb{Z}$ and that $f$ has a Lagrangian structure. Then $B$ is canonically equipped with an $(n-1)$-shifted Poisson structure.
\end{proposition}


\subsection{Analytification, shifted cotangents and completions} 
\label{ssec: ShapesCotCompl}

We use the general convention that given a derived stack $Z$ admitting a perfect cotangent complex $\mathbb{L}_{Z}$ with dual $\mathbb{T}_{Z}:=\mathbb{L}_{Z}^{\vee},$ for any $n\in\mathbb{Z}$, the $n$-shifted total cotangent derived stack is denoted by $\mathsf{T}^*[n]Z:=\mathsf{Spec}_{Z}\big(\mathrm{Sym}_{\mathcal{O}_{Z}}^{\bullet}(\mathbb{T}_{Z}[-n])\big)$. It is shifted symplectic \cite{Cal}. 

The de Rham stack $Z_{\mathrm{DR}}$ is defined by its functor of points for $A\in\mathsf{cdga}^{\leq 0}$ by $Z_{\mathrm{DR}}(A)\simeq Z(A^{\mathsf{red}}).$ There is a canonical map $q_{\mathrm{DR}}^Z:Z\to Z_{\mathrm{DR}},$ and for any morphism $f:Z\to Y$, let $f_{\mathrm{DR}}:Z_{\mathrm{DR}}\to Y_{\mathrm{DR}}$ be the induced map.

Following \cite[Defn.~2.1.3 (3)]{CPTVV17}, the \emph{formal completion} $Y_f^{\wedge}$ of $Y$ along $f$ (also to be denoted by $Y_Z^{\wedge}$) is the homotopy pull-back,
\begin{equation}
    \label{FormalCompletionSq}
\begin{tikzcd}
Y_f^{\wedge}\arrow[d]\arrow[r] & Z_{\mathrm{DR}}\arrow[d,"f_{\mathrm{DR}}"]
    \\
    Y\arrow[r,"q_{\mathrm{DR}}^Y"] & Y_{\mathrm{DR}}.
\end{tikzcd}
\end{equation}
If $Z,Y$ are geometric derived stacks lafp over $\mathbb{C}$, then analytification preserves formal completions. Indeed, there is a natural morphism 
$(Y_Z^{\wedge})^{\an}\to(Y^{\an}_{Z^{\an}})^{\wedge},$ which is an equivalence since analytification commutes with colimits.

An important example of a formal completion is taken along the zero-section of a shifted cotangent stack. Namely, let $Z$ be a geometric derived stack locally of finite presentation and the formal completion along the zero-section $Z\to \mathsf{T}^*[n]Z$ for some $n\in\mathbb{Z}$, denoted 
\begin{equation}
    \label{eqn: FormalCompletion}
\mathsf{T}_f^*[n]Z:=\big(\mathsf{T}^*[n]Z\big)_{Z}^{\wedge}:=Z_{\mathrm{DR}}\times_{(\mathsf{T}^*[n]Z)_{\mathrm{DR}}}\mathsf{T}^*[n]Z.
\end{equation}

\begin{definition}
The \emph{Dolbeault} stack $Z_{\mathrm{Dol}}$ of $Z$ is the relative classifying stack $Z_{\mathrm{Dol}}:=\mathsf{B}_Z\widehat{\mathsf{T}Z},$ along the formal completion of the tangent stack along the zero section $\widehat{\mathsf{T}Z}:=Z_{\mathrm{DR}}\times_{(\mathsf{T}Z)_{\mathrm{DR}}}\mathsf{T}Z.$  
The \emph{nilpotent Dolbeault stack} is $Z_{\mathrm{Dol}}^{\mathrm{nil}}\simeq \mathsf{B}_Z\mathsf{T}Z.$ 
\end{definition}
There are natural maps $\xi_Z:Z\to Z_{\mathrm{Dol}},\xi_Z^{\mathrm{nil}}:Z\to Z_{\mathrm{Dol}}^{\mathrm{nil}}$ and there is a canonical morphism
\begin{equation}
\label{NilpDolbShape}
\nu_Z:Z_{\mathrm{Dol}}\to Z_{\mathrm{Dol}}^{\mathrm{nil}},
\end{equation}
with $\xi_Z^{\mathrm{nil}}=\nu_Z\circ \xi_Z.$
\begin{example}
\label{UniversalDolbeault}
    Let $\mathcal{E}$ be a given sheaf on $X_{\mathrm{Dol}}.$ There is a map $\xi_X:X\to X_{\mathrm{Dol}}.$ The Higgs sheaf associated to $\mathcal{E}$ is $(E,\phi):=\xi_Z^*(\mathcal{E}).$ Thus $H^{\bullet}(X_{\mathrm{Dol}};\mathcal{E})\simeq H_{\mathrm{Dol}}^{\bullet}(E,\varphi).$ For the trivial Higgs sheaf $(\mathcal{O}_X,0)$, we have
$$\mathbb{R}\Gamma(\mathcal{O}_{X_{\mathrm{Dol}}})\simeq H^{\bullet}(X_{\mathrm{Dol}},\mathcal{O}_{X_{\mathrm{Dol}}})\simeq H^{\bullet}_{\mathrm{Dol}}(\mathcal{O}_X,0).$$

Sheaves on $X_{\mathrm{DR}}$ encode sheaf-valued de
Rham cohomology on $X$;
$$\mathbb{R}\Gamma(X_{\mathrm{DR}},\mathcal{O}_{X_{\mathrm{DR}}})\simeq H_{\mathrm{DR}}^{\bullet}(X)\simeq H_{\mathrm{DR}}^{\bullet}(\mathcal{O}_X,d_{DR}).$$
\end{example}

We conclude by collecting some supplementary results on the interaction between the formation of shifted cotangent bundles, formal completions and mapping stacks under analytification. They are likely well-known, and certainly follow from known properties (see e.g. \cite[Cor. 5.20, Prop. 5.23]{HolsteinPorta2025}), but we find it convenient to state them explicitly since they extend many algebraic results of \cite{ElliotYoo2018} to the analytic setting.

\begin{proposition}
\label{PropsAn}
    Let $Z$ be a geometric derived stack locally of finite presentation over $\mathbb{C}$. 
    \begin{itemize}

\item[1.] Assume $Z$ is endowed with an $n$-shifted symplectic structure. Then, there is a canonical equivalence
$\mathsf{T}_f[1]Z\simeq \mathsf{T}_f^*[n+1]Z,$ compatible with analytification.
\item[2.] Let $X$ be a reduced classical scheme of dimension $d.$ Suppose that $Z$ is Tannakian satisfying $\mathsf{QCoh}(Z)\simeq \mathsf{Ind}(\mathsf{Perf}(Z))$. Then there is a commutative diagram in derived analytic stacks over $\mathbb{C}$
\[
\begin{tikzcd}
\Map(X,Z_{\DR})^{\an}\arrow[d]\arrow[r] & \Map(X,Z)_{\DR}^{\an}\arrow[d]
\\
\AnMap(X^{\an},Z_{\DR}^{\an})\arrow[r] & \AnMap(X^{\an},Z^{\an})_{\DR},
\end{tikzcd}
\]
whose morphisms are all equivalences.
\item[3.] Let $Z$ be as in the above 2. Suppose that $X$ is a smooth proper classical scheme of dimension $d,$ with trivial canonical bundle $\mathcal{O}_X\simeq K_X.$ Then, there is an equivalence of $(n-d)$-shifted derived algebraic stacks 
\begin{equation}
    \label{FormCompMaps}
F_{X,Z}:\mathsf{T}_f^*[n-d]\Map(X,Z)\xrightarrow{\simeq}\Map\big(T[1]X,Z\big)_{\Map(X,Z)}^{\wedge},
\end{equation}
compatible with analytification.
    \end{itemize}
\end{proposition}

\begin{proof}
1. The equivalence follows immediately by definition of shifted cotangent bundles, by using the symplectic form on $Z$. Compatibility with analytification follows directly from the definition \eqref{eqn: FormalCompletion}, by analytifying the equivalence of algebraic derived stacks. Indeed, analytification commutes with formal completions and de Rham spaces, so we have a commuting diagram of equivalences
\[
\begin{tikzcd}
    (\mathsf{T}_{f}[1]Z)^{\an}\arrow[r,"\sim"] \arrow[d,"\sim"] & \mathsf{T}_f^{\an}[n+1]Z\arrow[d,"\sim"]
    \\
\big(\mathsf{T}_{f}^{*,\an}[n+1]\big)Z^{\an}& \arrow[l,"\sim"] (\mathsf{T}_f^*[n+1]Z)^{\an},
\end{tikzcd}
\]
were $\mathsf{T}_{f}^{*,\an}$ denotes is the analytic version of \eqref{eqn: FormalCompletion} arising via Proposition \ref{AnDeRhamAlgs}. 

2. Since $X$ is reduced, the algebraic equivalence $\Map(X,Z_{\DR})\simeq \Map(X,Z)_{\DR}$ is clear by definition of $(-)_{\DR}.$ The vertical arrows, and hence the bottom horizontal arrow are equivalences by the assumptions on $Z$ allowing us to apply the main theorem of \cite{HolsteinPorta2025}, and due to the more general fact that for any derived geometric stack locally of finite presentation $Y$, thus in particular the reduced scheme $X$, the canonical morphism 
$(Y_{\mathrm{red}})^{\an}\to (Y^{\an})_{\mathrm{red}}$ is an equivalence of derived analytic stacks.

3. The existence of \eqref{FormCompMaps} is proven in \cite[Thm. 4.20 (2.)]{ElliotYoo2018}.
We must establish an analytic morphism,
$$F_{X^{\an},Z^{\an}}:\mathsf{T}_f^{*,\an}[n-d]\AnMap(X^{\an},Z^{\an})\to \AnMap(T^{\an}[1]X^{\an},Z^{\an})_{\AnMap(X^{\an},Z^{\an})}^{\wedge},$$
and prove it is an equivalence of derived analytic stacks with $(n-d)$-shifted symplectic structures. However, by part (1.) and (2.) of the proposition, using again the assumptions on $Z$, we have the comparison map \eqref{eqn: Map to AnMap}, is an equivalence. The existence of $F_{X^{\an},Z^{\an}}$ is determined by the diagram,
\[
\begin{tikzcd}
    \big(\mathsf{T}_f^*[n-d]\Map(X,Z)\big)^{\an}\arrow[r,"\sim"] \arrow[d,"(F_{X,Z})^{\an}"] & \mathsf{T}_f^{*,\an}[n-d]\AnMap(X^{\an},Z^{\an})\arrow[d,"F_{X^{\an},Z^{\an}}"]
    \\
\big(\Map(T[1]X,Z)_{\Map(X,Z)}^{\wedge}\big)^{\an}\arrow[r,"\sim"]& \AnMap(T^{\an}[1]X^{\an},Z^{\an})_{\AnMap(X^{\an},Z^{\an})}^{\wedge},
\end{tikzcd}
\]
Since $F_{X,Z}$ is an equivalence, and the horizontal arrows are induced by analytification and \eqref{eqn: Map to AnMap}, we have $F_{X^{\an},Z^{\an}}$ is an equivalence.
\end{proof}
\begin{proposition}
\label{DolDRIncreaseShift}
Let $Z$ be a geometric derived stack locally of finite presentation over $\mathbb{C}$ with an $n$-shifted symplectic structure. Then $Z_{\Dol}$ (and trivially $Z_{\DR}$) are $(n+1)$-shifted symplectic. 
Moreover, $(\mathsf{T}_f^*[-n]Z)_{\mathrm{DR}}$ is a formal moduli problem under $\mathsf{T}^*[-n]Z.$
\end{proposition}
\begin{proof}
Let $Z \in \mathsf{dSt}_{\mathbb{C}}^{\mathrm{lafp}}$ with an $n$-shifted symplectic structure. Since $Z_{\mathrm{Dol}}\simeq \mathsf{T}_f[1]Z,$ the first claim follows from Proposition \ref{PropsAn} (1). The second claim follows via the equivalence,
    $(\mathsf{T}_f^*[-n]Z)_{\mathrm{DR}}\simeq \mathsf{T}_f^*[1-n](Z_{\mathrm{DR}}),$
    and using the composition,
    $\mathsf{T}_f^*[-n]Z\to \big(\mathsf{T}_f^*[-n]Z\big)_{\mathrm{DR}}\to Z_{\mathrm{DR}}\simeq \mathsf{T}_f^*[1-n](Z_{\mathrm{DR}}).$
\end{proof}

\subsubsection{The Weil restriction} 
Let $f:B\to S$ be a morphism of prestacks and let $E$ be a $B$-prestack. The \emph{Weil restriction} of $E$ along $f:B\rightarrow S$ is the $S$-prestack $\underline{\mathrm{Res}}_{B/S}(E)$ characterized by the universal property:
\begin{equation}
    \label{eqn: WeilRes}
\mathrm{Map}_{\mathsf{PStk}_S}(T,\underline{\mathrm{Res}}_{B/S}(E))\simeq\mathrm{Map}_{\mathsf{PStk}_B}(T\times_S B,E).
\end{equation}
One example is given by relative mapping prestacks; let $E,B\in \mathsf{PStk}_S$, then $E\times B\in \mathsf{PStk}_B$ and $\underline{\mathrm{Res}}_{B/S}(E\times B)\simeq \Map_S(B,E).$ Another describes mapping stacks of flat sections.
Let $\pi:E\to X$ be a morphism of derived prestacks and $f:X\to S$ a morphism with $\EuScript{M}\in \mathrm{Pic}^{gr}(S)$ and equipped with a $\EuScript{L}$-orientation. Consider the twisted total cotangent bundles and natural projections
$p:\mathsf{T}_{\EuScript{L}\otimes f^*\EuScript{M}}^*(E/X)\to E,$ and $\widetilde{p}:\mathsf{T}_{\EuScript{M}}^*\big(\underline{\mathrm{Res}}_{X/S}(E)/S\big)\to \underline{\mathrm{Res}}_{X/S}(E).$
Consider tautological $1$-forms,
$$\lambda_E\in \Gamma_X\big(\mathsf{T}_{\EuScript{L}\otimes f^*\EuScript{M}}^*(E/X),\mathbb{L}_{\mathsf{T}_{\EuScript{L}\otimes f^*\EuScript{M}}^*(E/X)/X}\otimes \pi^*(\EuScript{L}\otimes f^*\EuScript{M})\big),$$
as well as 
$$\lambda_{\underline{\mathrm{Res}}_{X/S}(E)}\in \Gamma_S\big(\mathsf{T}_{\EuScript{M}}^*(\underline{\mathrm{Res}}_{X/S}(E)/S),\mathbb{L}_{\mathsf{T}_{\EuScript{M}}^*(\underline{\mathrm{Res}}_{X/S}(E)/S)/S}\otimes\widetilde{\pi}^*\EuScript{M}.\big)$$
In \cite[Prop.~3.9]{calaque2024shifted}, the following is proven.
\begin{proposition}
\label{CS24Prop3.9}
    Let $\pi:E\to X$ be a morphism of derived prestacks, $f:X\to S$ a morphism of prestacks, with $\EuScript{M}$ a graded line bundle on $S$, equipped with an $\EuScript{L}$-orientation. Then, 
there is an equivalence of derived prestacks,
$$\underline{\mathrm{Res}}_{X/S}\big(\mathsf{T}_{\EuScript{L}\otimes f^*\EuScript{M}}^*(E/X)\big)\simeq \mathsf{T}_{\EuScript{M}}^*\big(\underline{\mathrm{Res}}_{X/S}(E)/S\big).$$
\end{proposition}
This equivalence is uniquely characterized by the commutative diagram
\[
\begin{tikzcd}
\underline{\mathrm{Res}}_{X/S}\big(\mathsf{T}_{\EuScript{L}\otimes f^*\EuScript{M}}(E/X)\big)\arrow[dr,"\underline{\mathrm{Res}}_{X/S}(p)"] \arrow[rr] & & \mathsf{T}_{\EuScript{M}}^*\big(\underline{\mathrm{Res}}_{X/S}(E)/X\big)\arrow[dl,"\widetilde{p}"]
    \\
    & \underline{\mathrm{Res}}_{X/S}(E) &
\end{tikzcd}
\]
and the equivalence of the corresponding canonical $1$-forms,
$\int_{X/S}\widetilde{\lambda}_E\simeq \widetilde{\lambda}_{\underline{\mathrm{Res}}_{X/S}(E)}.$

Looking to future applications in Section.~\ref{sec: LagCorrNAHType}, we recall \cite[Example~3.11]{calaque2024shifted}.
\subsubsection{Example: Shifted Higgs bundles}
\label{ShiftedHiggs}
Let $X$ be a smooth (proper) variety. Let $\mathcal{K}_X:=K_X[\dim_{\mathbb{C}}X]$ be the graded canonical bundle. Let $G$ be a reductive algebraic group and consider
$\underline{\mathrm{Bun}}_G(X):=\Map(X,\mathsf{B}G).$
By Proposition \ref{CS24Prop3.9}, there is an equivalence of $n$-shifted symplectic derived stacks,
\begin{equation}
\label{DerHiggsBunG}
\mathsf{T}^*[n]\underline{\mathrm{Bun}}_G(X)\simeq\underline{\mathrm{Res}}_{X/pt}\big(\mathsf{T}_{K_X[n+\dim X]}^*\big((\mathsf{B}G\times X)/X\big)\big),
\end{equation}
since 
\begin{equation}
    \label{TwistedCotangent}
\mathsf{T}_{K_X[n+d]}^*(\mathsf{B}G\times X/X)\simeq \mathsf{T}^*[n+d]\mathsf{B}G\times^{\mathbb{G}_m} K_X^{\times}.\end{equation}
Via its functor of points description e.g. over $\mathbb{C}$, the stack $\underline{\mathrm{Res}}_{X/pt}\big(\mathsf{T}_{K_X[n+d]}^*(\mathsf{B}G\times X/X)\big)$ parameterizes  $G$-bundles $P\to X$, together with 
$\phi\in \Gamma(X,K_X\otimes \mathrm{coAd}(P)\big)[n+d-1],$
where $\mathrm{coAd}(P)$ is the coadjoint bundle.
\begin{remark}
The section $\phi$ is shifted with degree $n+d-1$. Thus, \eqref{DerHiggsBunG}  describes a higher-dimensional version of the moduli stack of Higgs bundles (classically arising from $n=0,d=1$ case).
\end{remark}
The equivalence \eqref{DerHiggsBunG} thus provides a derived generalization to higher dimensional smooth proper varieties $X$, of the $\lambda\to 0,$ component of \eqref{TwistedBunGDiagramIntro}. 
Since mapping stacks are particular cases of Weil restriction, we describe the derived generalization of the $\lambda \to 1$ component. First, let us remind the case for curves.
\begin{remark}
\label{Remark:TwBunDR}
Let $C$ be a smooth proper curve and $(Z,\omega)$ an $n$-shifted symplectic derived stack. Consider the graded mixed morphism (cf, \cite[Sect.~2.1]{PTVV13}),
$\mathrm{DR}(Z\times C)\to \mathrm{DR}(Z)\otimes \mathrm{DR}(C),$ and the natural map (of graded mixed complexes) given by Serre duality,
$$\int_C:\mathrm{DR}(C)\to \mathbb{C}(1)[-1].$$
Letting $\int_C\mathrm{ev}^*\omega$ be the image of $\omega$ under
\begin{equation}
    \label{eqn: DR-composites}
\mathrm{DR}(Z)\to \mathrm{DR}\big(\Map(C,Z)\times C)\to \mathrm{DR}(\Map(C,Z))\otimes\mathrm{DR}(C)\to \mathrm{DR}(\Map(C,Z))\otimes\mathbb{C}(1)[-1].\end{equation}
By \cite[Prop.~1.24]{Safronov2023} there is an equivalence of derived mapping stacks,
\begin{equation}
    \label{LocSysTwistedCot1}
\Map(C_{\mathrm{DR}},Z)\simeq\mathsf{T}_{\int_C\mathrm{ev}^*\omega}^*[n-2]\Map(C,Z).\end{equation}
\end{remark}
Following the method of \cite{calaque2024shifted}, one may generalize \eqref{LocSysTwistedCot1} of Remark \ref{Remark:TwBunDR} to higher dimensional smooth proper varieties\footnote{See related   \cite[Conj.~1.25]{Safronov2023}.}. 
What we will do in Sect.~\ref{sec: LagCorrNAHType}, is analytify these constructions, and work with $\Perf$ in place of $\mathsf{B}G,$ to study the $\lambda \neq 0,1$ component of \eqref{TwistedBunGDiagramIntro}, in terms of an analytic moduli stack of $\lambda$-connections.

\section{Hyperk\"ahler structures in derived geometry}
\label{sec:HKinDAG}
In this section,  we aim to provide an interpretation of the results of \cite{LRT2022} in the setting of derived geometry. First, we make precise the candidate notion of a \emph{derived twistor family} proposed in \cite{KPS2021}. We then construct the Deligne--Simpson moduli stack as a derived  hyperk\"ahler reduction, in Sect.~\ref{sec: Derived NAH and KW}. 

To formulate twistor geometry for derived stacks over 
$\mathbb{C}$ we begin by establishing several technical preliminaries. The reader may wish to proceed directly to Subsect.~\ref{PreTwistorSection}, returning to these results as needed. For convenience, we summarize their main content:
\begin{itemize}
    \item \textit{Subsect.~\ref{ssec: C2Diags}}: proposes a category for derived stacks over $\mathbb{C}$, with real structure. Since $\mathsf{dSt}_{\mathbb{C}}$ is an $\infty$-topos it has all limits and colimts, and thus the category of $I$-diagrams $\mathsf{Diag}(I,\mathsf{dSt}_{\mathbb{C}})$ is an $\infty$-topos for small $\infty$-categories $I.$ In particular, this holds for $\mathsf{B}C_2$, since $C_2$ is finite, hence $\mathsf{dSt}_{\mathbb{C}}[C_2]$ is an $\infty$-topos;
    \begin{enumerate}
    \item \textit{Subsect.~\ref{sssec: DerivedGalois}}: relates the $\infty$-topoi of presheaves on derived affines over $K,$ and over $L$, with $L/K$ a finite field extension. The main result is Proposition \ref{GaloisLemma}, relating $\infty$-presheaves over $K$ satisfying \'etale hyperdescent and \'etale sheaves over $L$ which are $\mathrm{Gal}(L/K)$-equivariant\footnote{Important results on the base change of stability structures for field extensions $L/K$ are given in \cite[Thm.~12.17]{BLMNPS2021}.};

    \item \textit{Subsect.~\ref{sssec:C2Groth}}: states explicitly a standard generalization of the $\infty$-Grothendieck construction \cite[Thm.~3.2.0.1]{Lur09} to the equivariant setting (see Proposition \ref{C2Grothendieck}). This is used to make sense of $C_2$-equivariant presheaves of $\mathcal{O}_S$-modules on a derived prestack $S$ with $C_2$-action, given by Proposition \ref{PresheavesC2Stacks};

    \item \textit{Subsect.~\ref{C2WaldCons}}: we explicitly describe a version of the Waldhausen construction for $2$-Segal spaces \cite{dyckerhoff2019higher}, for $\mathsf{B}C_2$-shaped diagrams in $\mathsf{dSt}_{\mathbb{C}}$. The main result is Proposition \ref{C2Waldhausen}.
    \end{enumerate}
\end{itemize}
These results are applied to the case with $K=\mathbb{R},L=\mathbb{C}$ with 
$\mathrm{Gal}(\mathbb{C}/\mathbb{R})=C_2.$ 
We understand $\mathbb{P}^1(\mathbb{C})$ as a complex manifold with sheaf of holomorphic functions $\mathcal{O}_{\mathbb{P}^1(\mathbb{C})}$, whose underlying topological space has a $\mathrm{Gal}(\mathbb{C}/\mathbb{R})$-action, for which the non-trivial element $\sigma$ acts via complex conjugation.
\begin{remark}
This $\mathrm{Gal}(\mathbb{C}/\mathbb{R})$-action induces a sheaf isomorphism,
$\sigma^{\sharp}:\sigma^{-1}\mathcal{O}_{\mathbb{P}^1(\mathbb{C})}\xrightarrow{\simeq}\mathcal{O}_{\mathbb{P}^1(\mathbb{C})},$ defined by $\sigma^{\sharp}(f)(z):=\overline{f(\overline{z})},$ satisfying  $\sigma^{\sharp}\circ \sigma^{-1}(\sigma^{\sharp})\simeq \mathrm{id}_{\mathcal{O}_{\mathbb{P}^1(\mathbb{C})}}:\mathcal{O}_{\mathbb{P}^1(\mathbb{C})}\simeq \mathcal{O}_{\mathbb{P}^1(\mathbb{C})}.$
\end{remark}
Recall from \cite{Simpson1997MixedTwistor}, an \emph{antipodal real twistor structure} is a vector bundle $E$ on $\mathbb{P}^1$ equipped with an antilinear involution
$
\sigma: E \longrightarrow \sigma^*E$
covering $\sigma_{\mathbb{P}^1}$, and satisfying $\sigma^2 = \mathrm{id}$. 
Sheaf-theoretically, the pullback $\sigma^*E$ is defined by
$
\sigma^*(E)(U) := E(\sigma_{\mathbb{P}^1}(U)),$
with the $\mathcal{O}_{\mathbb{P}^1}$-module structure twisted by $\sigma$. 
\begin{remark}
Similarly, one defines \emph{circular real twistor structures} using the antiholomorphic involution $\tau_{\mathbb{P}^1}$ preserving the unit circle, with analogous notions of involution and compatible filtrations.
\end{remark}
We will now generalize this to consider objects of the $\infty$-topos $\mathsf{dAnSt}_{/\bfP}[C_2],$ of relative derived analytic stacks for the analytic \'etale topology, admitting a relative analytic cotangent complex $\mathbb{L}_{Z/\bfP}^{\an},$ endowed with compatible real structures (covering $\sigma_{\bfP}$), induced from the $\mathrm{Gal}(\mathbb{C}/\mathbb{R})$-action.

\subsection{Diagrams and homotopy-fixed points}
\label{ssec: C2Diags}
Complex conjugation is real-analytic, so in particular does not define a morphism in the $\mathbb{C}$-analytic category. This motivates introducing an auxiliary category of derived stacks over $\mathbb{C}$, with additional datum of a homotopy-coherent real involution; the real structure is encoded via descent data.
More precisely, since $\mathbb{C}$ is an $\mathbb{R}$-algebra, we may study derived stacks over $\mathbb{C}$ with \emph{effective Galois descent} along $\mathrm{Spec}(\mathbb{C})\to \mathrm{Spec}(\mathbb{R}).$ This is equivalent to the $\infty$-category
\begin{equation}
\label{eqn:RealDAnSt}
\mathsf{dSt}_{\mathbb{C}}[C_2]:=\mathsf{Diag}(\mathsf{B}C_2,\mathsf{dSt}_{\mathbb{C}}),
\end{equation}
of $\mathsf{B}C_2$-shaped diagrams in derived stacks over $\mathbb{C}$, with $C_2:=\mathbb{Z}/2\mathbb{Z}$ the cyclic group.
Let $C_2 = \mathbb{Z}/2\mathbb{Z} = \{e,\sigma\}$ denote the cyclic group of order two. Its classifying space $\mathsf{B}C_2$ is the $\infty$-category with a single object $\ast$ and endomorphism space given by $C_2$, i.e.\ 
\[
\mathrm{End}_{\mathsf{B}C_2}(\ast) \simeq C_2,
\]
with composition induced by the group law. Concretely, $\mathsf{B}C_2$ may be viewed as the simplicial nerve of the group $C_2$, whose nondegenerate simplices encode the relation $\sigma^2 = e$.

We may also speak of  Galois conjugate algebraic varieties or schemes, which fail to have the same topology, in fact have isomorphic cohomologies with $\mathbb{C}$-coefficients.
\begin{example}
Given a number field $K$, and $X$ an algebraic variety over $K$, for any embedding $\sigma:K\to \mathbb{C},$ we can consider the algebraic variety $X_{\sigma}:=X\times_{\mathrm{Spec}(K)}\mathrm{Spec}(\mathbb{C}).$ 
If $\sigma':K\to \mathbb{C}$ is another embedding, we call $X_{\sigma}$ and $X_{\sigma'},$ \emph{Galois conjugate} as varieties.
\end{example}
We will now extend this conjugation (by Kan extension) to derived stacks. 

\subsubsection{Derived stacks with Galois descent}
\label{sssec: DerivedGalois}
Let $L/K$ be a finite field extension with Galois group $\mathrm{Gal}(L/K).$ 
Assume that there is a good notion of connective derived geometry over $L$ and over $K$.
\begin{remark}
  More precisely, we tacitly assume the existence of homotopical geometric contexts \cite{TV2}, satisfying conditions for geometry relative to simplicial objects in an exact category, in the sense of \cite{BenBassatKellyKremnizer2024}. See also \cite[Sect.1.1]{CPTVV17}. It is of course satisfied by our case of interest i.e. for $\mathbb{C}/\mathbb{R}.$
\end{remark}
Let $\mathsf{dAff}_{L},\mathsf{dAff}_{K}$ denote the corresponding $\infty$-categories of affine derived schemes.
Let $A\in\mathsf{cdga}_{L}^{\leq 0}$. Then, every element of $\mathrm{Gal}(L/K)$ defines a Galois conjugate scheme via $\mathsf{Spec}(g^*A),$ where $g^*A:=A\otimes_{L,g}A,g\in\mathrm{Gal}(L/K).$
For $S$ not necessarily affine, we have $g^*S$ defined via its functor of points as 
$g^*S(\mathsf{Spec}B):=S\big(\mathsf{Spec}(g^*B)\big).$ This yields an endofunctor $g^*(-):\mathsf{dAff}_{/L}\to \mathsf{dAff}_{/L}.$ Base-change gives a functor
\begin{equation}
    \label{L/Kbc}
    (-)_L:\mathsf{dAff}_{/K}\to \mathsf{dAff}_{/L},\hspace{2mm}\mathsf{Spec}_KA\to \mathsf{Spec}_K\big(A\otimes_K L).
    \end{equation}
Note $\mathrm{Gal}(L/K)$ acts naturally on $\mathsf{dAff}_{/L}$ via conjugation and the essential image of \eqref{L/Kbc} is characterized by $\mathrm{Gal}(L/K)$-equivariant derived affine schemes i.e. which are endowed with equivalences 
$$(\mathsf{Spec}A)_L\xrightarrow{\simeq}g^*\big((\mathsf{Spec}A)_L\big),$$
for each group element $g.$ Consequently, we have the following.
\begin{proposition}
\label{GaloisLemma}
    There is an induced functor
$\mathsf{dSt}_K\to \mathsf{dSt}_L^{\mathrm{Gal}(L/K)},Z\mapsto Z_L,$
which is an equivalence of $\infty$-categories, whose inverse equivalence is the functor of homotopy fixed points i.e. $Z\simeq(Z_L)^{h\mathrm{Gal}(L/K)}.$
In particular, for an affine $K$-scheme $\mathsf{Spec}A$ the following statements hold:
\begin{enumerate}
    \item For every $K$-stack $Z$, we have
    $$Z\big(\mathsf{Spec}_KA)\simeq\big(Z_L(\mathsf{Spec}_L(A\otimes_KL))\big)^{h\mathrm{Gal}(L/K)},$$

    \item There is an induced full-faithful functor $i:\mathsf{dAff}_K\to \mathsf{dAff}_{L}^{\mathrm{Gal}(L/K)},$ whose right-Kan extension to $\mathsf{PShv}\big(\mathsf{dAff}_{L}^{\mathrm{Gal}(L/K)}\big)$, denoted by $i_*:\mathsf{PShv}\big(\mathsf{dAff}_L^{\mathrm{Gal}(L/K)}\big)\to \mathsf{PShv}\big(\mathsf{dAff}_K\big),$ is computed object-wise by
    $$(i_*\EuScript{F})(\mathsf{Spec}_KA)\simeq \underset{\mathsf{Spec}B\in(\mathsf{dAff}_L^{\mathrm{Gal}(L/K)})_{/i(\mathsf{Spec}A)}}{\mathrm{lim}}\hspace{1mm}\EuScript{F}(\mathsf{Spec}B).$$
\end{enumerate}
\end{proposition}
We will actually prove a stronger statement, namely that base-change defines a geometric morphism of $\infty$-topoi compatible with \'etale (hyper)descent for objects of $\mathsf{PShv}$. In other words, it extends to sheaves.
\begin{proof}
    Let $b:\mathsf{dAff}_K\to \mathsf{dAff}_{L}^{\mathrm{Gal}(L/K)},$ be the base-change, whose target is $\mathsf{Funct}\big(\mathsf{B}\mathrm{Gal}(L/K),\mathsf{dAff}_{/L}\big)$. It is essentially surjective with inverse functor
    $$(-)^{h\mathrm{Gal}(L/K)}:\mathsf{dAff}_{/L}^{\mathrm{Gal}(L/K)}\to \mathsf{dAff}_K.$$
    Via the usual \'etale topologies, $b$ is a morphism of sites i.e. preserves \'etale covers, and defines an equivalence of $\infty$-topoi of \'etale sheaves, 
    $$b_*:\mathsf{Shv}(\mathsf{dAff}_K)\xrightarrow{\simeq}\mathsf{Shv}\big(\mathsf{dAff}_{L}^{\mathrm{Gal}(L/K)}\big),$$
    given by right-Kan extension along $b$, with inverse $b^*$ given via precomposition with $b$. Explicitly, for an \'etale sheaf $\EuScript{F}$ on the $\infty$-site over $K$, we will show that
    $$\big(b_*\EuScript{F}\big)(Y)\simeq \EuScript{F}\big(Y^{h\mathrm{Gal}}\big),$$
    for each $Y\in\mathsf{dAff}_{/L}^{\mathrm{Gal}}.$
First, note the forgetful functor $u:\mathsf{dAff}_{/L}^{\mathrm{Gal}}\to \mathsf{dAff}$ sends an equivariant affine $L$-scheme to its underlying $L$-scheme. It is both continuous and cocontinuous for the \'etale topologies, and hence induces a geometric morphism of topoi
$$u_*:\mathsf{Shv}(\mathsf{dAff}_{/L}^{\mathrm{Gal}(L/K)})\to \mathsf{Shv}(\mathsf{dAff}),$$
with a left-adjoint $u^*.$ Base-change $\EuScript{F}_L$ of a $K$-stack $\EuScript{F}$ is given by left-Kan extension along $\mathsf{dAff}_K\hookrightarrow \mathsf{dAff}_{L},$ thus $\EuScript{F}_L\simeq u_*(b_*\EuScript{F}).$ That is, for every $L$-algebra $B$, 
$$\EuScript{F}_L(\mathsf{Spec}B)\simeq \underset{A\to B^{hG}}{\mathrm{colim}}\hspace{1mm}\EuScript{F}(\mathsf{Spec}A)\simeq (b_*\EuScript{F})(\mathsf{Spec}B)\simeq u_*(b_*\EuScript{F})(\mathsf{Spec}_LB),$$
where we consider $(b_*\EuScript{F})(\mathsf{Spec}B)$ with trivial $\mathrm{Gal}(L/K)$-action.
To prove descent, let $\EuScript{F}\in \mathsf{dSt}_{K}$ and let $A$ be a $K$-algebra. Consider the object $Y_0:=\mathsf{Spec}A\otimes_KL$ with $\mathrm{Gal}(L/K)$-action and let $Y$ denote the same underlying $L$-scheme with trivial action. There is a canonical map $Y_0\to Y$ acting as the identity on the underlying scheme, which is moreover $\mathrm{Gal}(L/K)$-equivariant. Moreover, the following diagram commutes: 
\[
\begin{tikzcd}
    \EuScript{F}(\mathsf{Spec}A)\arrow[d,"\mathrm{id}"]\arrow[r,"\sim"] & b_*\EuScript{F}(Y_0)\arrow[d]
    \\
    \EuScript{F}_L\big(\mathsf{Spec}(A\otimes_KL)\big)\arrow[r,"\sim"] & b_*\EuScript{F}(Y).
\end{tikzcd}
\]
Indeed, the top equivalence is given by $\EuScript{F}(\mathsf{Spec}A)\simeq b_*\EuScript{F}(Y_0)$ as $Y_0^{h\mathrm{Gal}}\simeq\mathsf{Spec}A,$ and the bottom follows by definition since $\EuScript{F}_L=u_*b_*\EuScript{F}.$
The vertical right map is induced from $Y_0\to Y$ by applying $b_*\EuScript{F}(-)$, viewed as a morphism in $\mathsf{dAff}_{/L}^{\mathrm{Gal}}.$ Since $b_*\EuScript{F}$ is $\mathrm{Gal}$-equivariant, $b_*\EuScript{F}(Y_0)$ has a canonical $\mathrm{Gal}$-action induced from $Y_0$ and the morphism $b_*\EuScript{F}(Y_0)\to b_*\EuScript{F}(Y),$ is $\mathrm{Gal}$-equivariant, where we take the target with its trivial action. Via descent for \'etale sheaves it follows that 
$b_*\EuScript{F}(Y_0)\simeq\big(b_*\EuScript{F}(Y)\big)^{h\mathrm{Gal}},$ which is a formal consequence of the equivalence of $\infty$-topoi, 
$$\mathsf{Shv}\big(\mathsf{dAff}_K\big)\xrightarrow{\simeq}\mathsf{Shv}(\mathsf{dAff}_{L}^{\mathrm{Gal}(L/K)}).$$
In other words, for a $\mathrm{Gal}(L/K)$-equivariant sheaf $\EuScript{G}$ and an object $Z\in\mathsf{dAff}_L^{\mathrm{Gal}(L/K)},$ we have $\EuScript{G}(Z)\simeq \EuScript{G}\big(uZ\big)^{h\mathrm{Gal}(L/K)}.$
Then, applying this formula to $\EuScript{G}=b_*\EuScript{F}$ and $Z=Y_0,$
it follows that
$$b_*\EuScript{F}(Y_0)\simeq\big(b_*\EuScript{F}(uY_0)\big)^{h\mathrm{Gal}}\simeq\big(b_*\EuScript{F}(Y)\big)^{h\mathrm{Gal}}.$$
In particular, we have
$$\EuScript{F}(\mathsf{Spec}A)\simeq b_*\EuScript{F}(Y_0)\simeq(b_*\EuScript{F}(Y))^{h\mathrm{Gal}}\simeq\big(\EuScript{F}_K(\mathsf{Spec}A\otimes_KL)\big)^{h\mathrm{Gal}},$$
as required.
\end{proof}
We now apply Proposition \ref{GaloisLemma} to our case of interest, when $K=\mathbb{R},L=\mathbb{C}$ with 
$\mathrm{Gal}(\mathbb{C}/\mathbb{R})=C_2.$ 

\subsubsection{Derived analytic stacks with $\mathrm{Gal}(\mathbb{C}/\mathbb{R})$-action}

Let $G := \mathrm{Gal}(\mathbb{C}/\mathbb{R}) \simeq \mathbb{Z}/2\mathbb{Z}$, viewed as a discrete group stack in derived analytic stacks \cite[Def.~7.2.2.1]{Lur09}. To emphasize this point of view, we denote it by $\mathcal{G}(\mathbb{C}/\mathbb{R})_{\bullet}:\Delta^{\mathrm{op}}\to \mathsf{dAnSt}.$
The functor of classifying stacks is understood to be $\mathsf{B}\mathcal{G}_{\bullet}\simeq\underset{n\in\Delta^{\mathrm{op}}}{\mathrm{colim}}\hspace{1mm}\mathcal{G}_n.$
\begin{remark}[Notation]
Throughout, where context will be clear, we will understand notation $\mathsf{B}\mathrm{Gal}(\mathbb{C}/\mathbb{R})$, to mean either the (co)simplicial diagram
$$\bigg[\cdots \;\substack{\longrightarrow\\[-0.3em]
               \longrightarrow\\[-0.3em]
               \longrightarrow\\[-0.3em]\longrightarrow}\; \mathrm{Gal}(\mathbb{C}/\mathbb{R})^{\times 3}   \;\substack{\longrightarrow\\[-0.3em]
               \longrightarrow\\[-0.3em]
               \longrightarrow}\; \mathrm{Gal}(\mathbb{C}/\mathbb{R})\times \mathrm{Gal}(\mathbb{C}/\mathbb{R}) \;\substack{\longrightarrow\\[-0.3em]
               \longrightarrow}\; \mathrm{Gal}(\mathbb{C}/\mathbb{R})\bigg],$$
or its geometric realization. 
\end{remark}

Consider a group object in the relative setting of $\mathsf{dAnSt}_{/\mathbf{P}^1}.$ They form a full $\infty$-subcategory of $\mathsf{Funct}(\Delta^{\mathrm{op}},\mathsf{dAnSt}_{/\mathbf{P}^1})$ spanned by functors $\mathcal{G}_{\bullet}:\Delta^{\mathrm{op}}\to \mathsf{dAnSt}_{/\bfP},$ such that: $\mathcal{G}_{\bullet}(0)\simeq \bfP,$ and for every partition $[n]=I_1\cup I_2$ of $[n]\in\Delta,$ with cardinality of intersection $1$, then $\mathcal{G}_n\to \mathcal{G}_{\bullet}(I_1)\times_{\mathbf{P}^1}\mathcal{G}_{\bullet}(I_2)$ is an equivalence.

 Following \cite[Def.~4.2.2.2]{Lur17}, the $\infty$-category of derived analytic stacks with $\mathrm{Gal}(\mathbb{C}/\mathbb{R})$-action is the full subcategory
\begin{equation}
    \label{C2dAnStIntro}
\mathsf{dAnSt}_{/\bfP}^{\mathrm{Gal}(\mathbb{C}/\mathbb{R})}
\subseteq
\mathsf{Funct}(\Delta^{\mathrm{op}}, \mathsf{dAnSt}_{/\bfP})_{/\mathsf{B}\mathrm{Gal}(\mathbb{C}/\mathbb{R})},
\end{equation}
consisting of morphisms $Z_{\bullet}\to \mathcal{G}(\mathbb{C}/\mathbb{R})_{\bullet}$ of functors $\Delta^{\mathrm{op}}\to \mathsf{dAnSt}_{/\bfP}$ i.e. simplicial objects, such that for every $[n]\in \Delta$, the canonical map
\begin{equation}
\label{nSimplexRelativeIntro}
Z_n\to Z_0\times_{\bfP}\mathcal{G}(\mathbb{C}/\mathbb{R})_n,
\end{equation}
induced from $[0]\to [n],0\mapsto n$ in $\Delta$, is an equivalence. 

The following result states internal hom-stacks inherit a natural $\mathrm{Gal}$-action when the target is endowed with one.
\begin{proposition}
    \label{ActAnMaps}
 Suppose that $Z$ is has an action of the group stack associated with $\mathrm{Gal}(\mathbb{C}/\mathbb{R})$ and $W$ is an arbitrary derived analytic stack. Consider
$$\mathsf{dAnSt}_{/\bfP}^{\mathrm{act}}\subseteq \mathsf{Funct}\big(\Delta^1\times \Delta^{\mathrm{op}},\mathsf{dAnSt}_{/\bfP}\big),$$
the full subcategory of group stack actions on objects of $\mathsf{dAnSt}_{/\bfP}.$ Then there is an induced group stack action on $\underline{\mathrm{AnMap}}_{/\bfP}(W,Z)$.
\end{proposition}
\begin{proof}
    Recall from \eqref{eqn: WeilRes} the internal mapping derived stack (inner hom)
$$\underline{\mathrm{AnMap}}_{/\bfP}:\mathsf{dAnSt}_{/\bfP}^{\mathrm{op}}\times \mathsf{dAnSt}_{/\bfP}\to \mathsf{dAnSt}_{/\bfP}.$$
It is characterized by the universal property:
$$\mathrm{Maps}_{\mathsf{dAnSt}_{/\bfP}}\big(T,\underline{\mathrm{AnMap}}_{/\bfP}(Z,W)\big)\simeq \mathrm{Maps}_{\mathsf{dAnSt}_{/\bfP}}(T\times_{\bfP}Z,W).$$
Consider \eqref{C2dAnStIntro} and the full subcategory
$$\mathsf{dAnSt}_{/\bfP}^{\mathrm{act}}\subseteq \mathsf{Funct}\big(\Delta^1\times \Delta^{\mathrm{op}},\mathsf{dAnSt}_{/\bfP}\big),$$
of group stack actions on objects of $\mathsf{dAnSt}_{/\bfP}.$ By definition an object is given by a morphism $Z_{\bullet}\to \mathcal{G}_{\bullet}$ of functors $\Delta^{\mathrm{op}}\to \mathsf{dAnSt}_{/\bfP},$ with $\mathcal{G}_{\bullet}$ a group stack and $Z_{\bullet}$ an object of \eqref{C2dAnStIntro}.
It is clear there is a lift of $\underline{\mathrm{AnMap}}_{/\bfP},$ to an inner-hom functor
\begin{equation}
\label{eqn: ActAnMaps}
\underline{\mathrm{AnMap}}_{\bfP}^{\mathrm{act}}:\mathsf{dAnSt}_{/\bfP}^{\mathrm{op}}\times \mathsf{dAnSt}_{/\bfP}^{\mathrm{act}}\to \mathsf{dAnSt}_{/\bfP}^{\mathrm{act}}.
\end{equation}
Since $Z$ has an action of the group stack associated with $\mathrm{Gal}(\mathbb{C}/\mathbb{R})$ and $W$ is an arbitrary derived analytic stack, $Z$ is an object of $\mathsf{dAnSt}_{/\bfP}^{\mathrm{act}},$ and the claim follows by identifying $\underline{\mathrm{AnMap}}_{/\bfP}(W,Z)$ together with its group-stack action with the functor \eqref{eqn: ActAnMaps}.
\end{proof}

\subsubsection{$C_2$-unstraightening}
\label{sssec:C2Groth}

Let $\mathsf{B}C_2$ denote the classifying stack of $C_2.$ Let $\mathsf{C}$ be an $\infty$-category, for which we assume that all $\mathsf{B}C_2$-shaped diagrams exist. A $C_2$-action on $\mathsf{C}$ is an $\infty$-functor $\mathsf{B}C_2\to \mathsf{C}.$ 

The $\infty$-category of $C_2$-categories is denoted by $\mathsf{Cat}_{\mathrm{st},\mathrm{pres}}^{\infty}[C_2],$ whose objects are called \emph{$\infty$-categories with $C_2$-action}. A morphism is a $C_2$-equivariant functor $F:\mathsf{C}\to\mathsf{D},$ the totality of which is denoted $\mathsf{Funct}^{C_2}(\mathsf{C},\mathsf{D}).$
Here is a fact that we will use without proof, whose justification may be found, for instance in \cite{CalmesDottoHarpazHebestreitLandNardinMoiNikolausSteimle2023} and references cited therein.
\begin{remark}
\label{UsefulRmk}
Consider a $C_2$-equivariant functor $F:\EuScript{C}\to \EuScript{D}$ between objects of $\mathsf{Cat}_{\mathrm{st},\mathrm{pres}}^{\infty}[C_2]$. Assume that the underlying functor (after forgetting the $C_2$-actions) admits a right adjoint. Then this right adjoint refines to a $C_2$-equivariant functor, such that the unit and counit
are $C_2$-equivariant natural transformations. \end{remark}

We need a $C_2$-equivariant $\infty$-Grothendieck construction \cite[Thm.~3.2.0.1]{Lur09}
Classify $F:I\to \mathsf{Cat}^{\infty}$ by cocartesian fibrations $X\to I$ resulting in equivalence:
$$\mathsf{Funct}(I,\mathsf{Cat}^{\infty})\simeq (\mathsf{Cat}_{\infty}^{cocart})_{/I},$$
the rhs is the subcategory of cocartesian fibrations to $I$ and
functors over $I$ preserving cocartesian lifts. Straightening relative to $\mathsf{B}C_2$, there is an equivalence
$$\mathsf{Cat}_{\mathrm{st},\mathrm{pr}}^{\infty}[C_2]\simeq \mathsf{Cat}_{\mathrm{st},\mathrm{pr}/\mathsf{B}C_2}^{\infty},\hspace{2mm} \EuScript{C}\mapsto \EuScript{C}_{hC_2}.$$

Letting $\mathrm{I}$ be an $\infty$-category with $C_2$-action, a $C_2$-equivariant functor $F:\mathrm{I}\to \mathsf{Cat}_{\infty}$, with $\mathsf{Cat}_{\infty}$ carrying trivial action, is classified by: $C_2$-equivariant cocartesian fibrations $\mathsf{X}\to\mathsf{I}$.

The refined Grothendieck constrution, states the $C_2$-equivariant functor $F$ corresponds to a functor 
$I_{hC_2}\to \mathsf{Cat}^{\infty}$
classified by cocartesian fibration $Y\to I_{hC2}$ over $\mathsf{B}C_2$.

Straightening over $\mathsf{B}C_2$ is precisely classified by $C_2$-equivariant $X\to I$ whose underlying functor after forgetting the $C_2$-action (cf., Remark \ref{UsefulRmk}), is a cocartesian fibration by \cite[Prop.~2.4.2.11]{Lur09} which classifies the functor
$$F:I\to I_{hC_2}\to \mathsf{Cat}^{\infty}.$$
\begin{proposition}
\label{C2Grothendieck}
Let $I$ be a small $\infty$-category with $C_2$-action i.e. an object of $\mathsf{Cat}_{\mathrm{st},\mathrm{pres}}^{\infty}[C_2].$ Then, there is an equivalence of $\infty$-categories,
\begin{equation}
\label{C2GrothCons}
\mathsf{Funct}_{\mathsf{B}C_2}\big( I,\, \mathsf{Cat}_\infty \big) \;\simeq\; \big( \mathsf{Cat}_{\infty,/I}^{\mathrm{cocart}}\big)[C_2].
\end{equation}
\end{proposition}
The results of the preceeding section and this one are now applied to study equivariant quasi-coherent sheaves on $C_2$-diagrams of stacks.
That is, we return to our case of interest and consider $\mathsf{cdga}^{\leq 0}[C_2].$ By Proposition \ref{GaloisLemma}, and Proposition \ref{C2Grothendieck}, we have the following equivariant analogue of \cite[Prop.~1.17]{calaque2024shifted}, by applying \cite[Prop.~7.1]{GHN2017}

\begin{proposition}
\label{PresheavesC2Stacks}
    Let $S$ be a derived Artin stack with $C_2$-action i.e. a functor $S:\mathsf{B}C_2\to \mathsf{dSt}^{\mathrm{Art}}.$ Then, the $\infty$-category of $C_2$-equivariant presheaves of $\mathcal{O}_S$-modules, $\mathsf{Mod}_{\mathcal{O}_S}^{\mathrm{pre},C_2},$ is equivalent to 
    $$\mathsf{Mod}_{\mathcal{O}_S}^{\mathrm{pre},C_2}\simeq \int_{T\in\mathsf{dAff}^{C_2,\mathrm{op}}_{/S}}\mathsf{Funct}^{C_2}\big(\mathsf{dAff}_{T// S}^{C_2,\mathrm{op}},\mathsf{QCoh}^{C_2}(T)\big).$$
\end{proposition}
\begin{proof}
Let $I:=(\mathsf{dAff}^{C_2})_S$. Then, following \cite[Section~5.2.1]{Lur17}, let $\mathsf{Tw}^{C_2}(I)$ be $\infty$-category of equivariant twisted arrows in $I$. In particular, an object is a morphism $f:i\to j$ together with an isomorphism $\varphi:\sigma f\xrightarrow{\sim} f,$ where $\sigma\in C_2,$ with higher coherences. By construction, there is a canonical projection
$p^{C_2}:\mathsf{Tw}^{C_2}(I)\to I^{op}\times I$, whose fiber of $(i,j)$ is the mapping space $\mathrm{Map}_{I}^{C_2}(i,j)$ of $C_2$-equivariant morphism in $I$ from $i$ to $j.$ Then, for any $C_2$-biequivariant bifunctor $G:I^{op}\times I\to \mathsf{Cat}_{\infty},$ one computes the end,
$$\int_{i\in I}^{C_2}G(i,i)=\underset{\mathsf{Tw}^{C_2}(I)^{op}}{\mathrm{lim}}\hspace{1mm}\big(G\circ p^{C_2}\big).$$
Then, applying this to $\underline{\mathsf{QCoh}}^{C_2}:I^{op}\to \mathsf{Cat}_{\infty},$ from Proposition \ref{C2Grothendieck} we have that 
$$\mathsf{Mod}_{\mathcal{O}_S}^{\mathrm{pre,C_2}}\simeq \int_{T\in I}\mathsf{QCoh}^{C_2}(T)\simeq \underset{\mathsf{Tw}^{C_2}(I)^{op}}{\mathrm{lim}}(\mathsf{QCoh}^{C_2}(-)\circ p^{C_2}),$$
and the result follows by the proof of \cite[Prop.~1.17]{calaque2024shifted}.
\end{proof}
\subsubsection{$C_2$-Waldhausen construction}
\label{C2WaldCons}
In this subsection we elaborate a certain combinatorial structure defining a simplicial object in $\mathsf{dSt}[C_2]$. Specifically, we adapt the Waldhausen construction for $2$-Segal spaces \cite{dyckerhoff2019higher}, to $\mathsf{B}C_2$-shaped diagrams in $\mathsf{dSt}_{\mathbb{C}}$.

In Section \ref{ssec: Analytification}, we will also need it in the $\mathbb{C}$-analytic category, namely, $\mathsf{dAnSt}$. It will suffice to first perform the construction in the algebraic setting, extending to analytic stacks via the derived analytification functor $(-)^{\mathrm{an}}:\mathsf{dSt}^{afp}\to \mathsf{dAnSt},$ as it commutes with finite limits and arbitrary colimits \cite{HolsteinPorta2025}. 
Its relevance is that it leads to a canonical $\infty$-functor
$$2-\mathrm{Seg}(\mathsf{dSt}[C_2])\to \mathrm{Alg}_{\mathbb{E}_1}\big(\mathrm{Corr}(\mathsf{dSt}[C_2])\big),$$
so that we assign to every $2$-Segal object, an associative monoid object in the $\infty$-category of $C_2$-equivariant correspondences in derived stacks over $\mathbb{C}.$ This makes use of \cite[Thm.~11.1.6]{dyckerhoff2019higher}.

\begin{proposition}
\label{C2Waldhausen}
    Let $G$ be a reductive algebraic group with complexification $G^{\mathbb{C}}$ acting on $Z\in \mathsf{dSt}_{\mathbb{C}}[C_2].$ Then, there is a $C_2$-equivariant simplicial object in derived stacks over $\mathsf{B}G^{\mathbb{C}},$ 
    $$\EuScript{S}_{\bullet}\RPerf_{G^{\mathbb{C}}}(Z):\Delta^{op}\to \mathsf{dSt}_{\mathbb{C}/\mathsf{B}G^{\mathbb{C}}}[C_2],$$
    with $\Delta$ given the canonical $C_2$-structure, mapping $[n]$ to $[n]$ and $\alpha:[n]\to [m]$ to $i\mapsto m-\alpha(n-i).$ Furthermore, there is a canonical equivalence,
    \begin{equation}
\label{eqn: G equiv versus stack}
\EuScript{S}_{\bullet}\RPerf_{G^{\mathbb{C}}}(Z)\simeq\big[\EuScript{S}_{\bullet}\RPerf(Z)/G^{\mathbb{C}}\big],
\end{equation}
    inducing an equivalence 
    $$\mathsf{Coh}^b(\RPerf_{G^{\mathbb{C}}}(Z)\big)\simeq \mathsf{Coh}_{G^{\mathbb{C}}}^b(\RPerf(Z)\big),$$
    in $\mathsf{Cat}_{\mathbb{C},\mathrm{st}}^{\infty}[C_2],$ where the right hand side is the $G^{\mathbb{C}}$-equivariant stable $\infty$-category of bounded coherent complexes on the derived stack $\RPerf(Z).$
\end{proposition}
\begin{proof}
The claim for $\Delta$ is clear, and denoted by $(\Delta,(-)^{o})\in\mathsf{Cat}_{\infty}[C_2],$ the $C_2$-structure. Since $\mathsf{PShv}(\Delta)$ has an obvious induced $C_2$-action, then by localizing 
$\mathsf{PShv}(\Delta)\leftrightarrows \mathsf{Cat}_{\infty}$, whose local objects are complete Segal spaces \cite[Thm.~14.6.3]{BarwickSchommerPries2021}, since the $C_2$-action on simplicial presheaves preserves complete Segal spaces, it restricts to a non-trivial $C_2$-action on $\mathsf{Cat}_{\infty},$ denoted again by $(-)^o.$

    For any grouplike $\mathbb{E}_1$-monoid in $\mathsf{dSt}_{\mathbb{C}}^{\times},$  $G$ with complexification $G^{\mathbb{C}},$ the $G^{\mathbb{C}}$-action may be encoded by

\[
\begin{tikzcd}[column sep=huge]
\cdots 
  \arrow[r, shift left=3]
  \arrow[r, shift left=1]
  \arrow[r, shift right=1]
  \arrow[r, shift right=3]
& G^{\mathbb{C}}\times G^{\mathbb{C}}\times X 
  \arrow[d]
  \arrow[r, shift left=2]
  \arrow[r]
  \arrow[r, shift right=2]
& G^{\mathbb{C}}\times Z 
  \arrow[d]
  \arrow[r, shift left]
  \arrow[r, shift right]
& Z
  \arrow[d]
\\
\cdots
  \arrow[r, shift left=3]
  \arrow[r, shift left=1]
  \arrow[r, shift right=1]
  \arrow[r, shift right=3]
& G^{\mathbb{C}}\times G^{\mathbb{C}}
  \arrow[r, shift left=2]
  \arrow[r]
  \arrow[r, shift right=2]
& G^{\mathbb{C}}
  \arrow[r, shift left]
  \arrow[r, shift right]
& \mathrm{Spec}(\mathbb{C})
\end{tikzcd}\simeq \hspace{5mm} \begin{tikzcd}\big[Z/G^{\mathbb{C}}\big]
\arrow[d]
\\
\mathsf{B}G^{\mathbb{C}}
\end{tikzcd},
\]
    and thus a single diagram since the monoidal structure on $\mathsf{dSt}$ is cartesian by \cite[Prop.~4.2.2.9]{Lur17}
    $$Act_{G^{\mathbb{C}}}(Z):\Delta^{op}\times \Delta^1\to \mathsf{dSt}_{\mathbb{C}}[C_2],$$
    which is $C_2$-equivariant satisfying the relative $1$-Segal condition, and whose geometric realizations give a canonical map $[Z/G^{\mathbb{C}}]\to \mathsf{B}G^{\mathbb{C}},$ as indicated. Then,
    $$\EuScript{S}_{\bullet}\RPerf_{G^{\mathbb{C}}}(Z):\Delta^{op}\to \mathsf{dSt}[C_2]_{/\mathsf{B}G^{\mathbb{C}}},$$
    is defined 
$\EuScript{S}_{\bullet}\RPerf_{G^{\mathbb{C}}}(Z):=\Map_{/\mathsf{B}G^{\mathbb{C}}}([Z/G^{\mathbb{C}}],\EuScript{S}_{\bullet}\Perf\times \mathsf{B}G^{\mathbb{C}}).$
The canonical decomposition of $\mathbb{G}_m$-equivariant perfect complexes is compatible with the real
structure \cite[Prop.~7]{NABMHS2000}.
Note that 
$$\mathrm{Spec}(\mathbb{C})\times_{\mathsf{B}G^{\mathbb{C}}}\EuScript{S}_{\bullet}\RPerf_{G^{\mathbb{C}}}(Z)\simeq \Map(Z,\EuScript{S}_{\bullet}\Perf).$$
By descent of $\infty$-topoi \cite[Thm.~6.1.3.9, Proposition.~6.1.3.10]{Lur09}, we obtain the first equivalence, which extends to a $2$-Segal object precisely since the functor 
$$\Map_{/\mathsf{B}G^{\mathbb{C}}}([Z/G^{\mathbb{C}}],-\times \mathsf{B}G^{\mathbb{C}}):\mathsf{dSt}[C_2]\to \mathsf{dSt}_{/\mathsf{B}G^{\mathbb{C}}}[C_2],$$
commutes with limits. The final equivalence of $\infty$-categories is clear.
\end{proof}

\subsubsection{The fixed-points cotangent complex}
\label{sssec: FixedPoints}
For an object $Z\in \mathsf{dSt}_{\mathbb{P}^1}[C_2],$ we have 
$(Z/\mathbb{P}^1)^{hC_2}:\mathsf{dSt}_{\mathbb{P}^1}^{op}\to \mathsf{Spc},$ such that for each $T$-point over $\mathbb{P}^1$,
$$\mathrm{Maps}_{\mathsf{dSt}_{\mathbb{P}^1}}(T,(Z/\mathbb{P}^1)^{hC_2})\simeq \mathrm{Maps}_{\mathsf{dSt}_{\mathbb{P}^1}}\big(T,\Map_{/\mathbb{P}^1}^{C_2}(\mathbb{P}^1,Z)\big)\simeq \mathrm{Maps}_{\mathsf{dSt}_{\mathbb{P}^1}}^{C_2}(T,Z).$$
The identity map $\mathrm{Id}_{(Z/\mathbb{P}^1)^{hC_2}}:(Z/\mathbb{P}^1)^{hC_2}\to (Z/\mathbb{P}^1)^{hC_2},$ induces a canonical $C_2$-equivariant $\mathbb{P}^1$-morphism of derived stacks,
$\mu:(Z/\mathbb{P}^1)^{hC_2}\to Z.$ The relative homotopy fixed points stack can be equivalently desribed as Weil restriction of $[Z/C_2]\to \mathsf{B}C_2$ along the canonical map $\mathsf{B}C_2 \to\mathbb{P}^1,$ i.e. as the homotopy fiber product,
\begin{equation}
\label{eqn: HoFP as Res}
\begin{tikzcd}
(Z/\mathbb{P}^1)^{hC_2}\arrow[d]\arrow[r] & \Map_{/\mathbb{P}^1}(\mathsf{B}C_2,[Z/C_2])
\arrow[d]
\\
\mathbb{P}^1\arrow[r,"\mathrm{Id}_{\mathsf{B}C_2}"] & \Map_{/\mathbb{P}^1}(\mathsf{B}C_2,\mathsf{B}C_2).
\end{tikzcd}
\end{equation}
From the presentation \eqref{eqn: HoFP as Res}, it is clear how to compute the relative cotangent complex $\mathbb{L}_{(Z/\mathbb{P}^1)^{hC_2}/\mathbb{P}^1}.$ 
\begin{proposition}
\label{FixedPointCotangents}
Let $\mu:(Z/\mathbb{P}^1)^{hC_2}\to Z$ denote the canonical map, $q:(Z/\mathbb{P}^1)^{hC_2}\times \mathsf{B}C_2\simeq[Z^{hC_2}/C_2]\to [Z/C_2],$ the quotient map of $\mu$ and $p:(Z/\mathbb{P}^1)^{hC_2}\times_{\mathbb{P}^1}\mathsf{B}C_2\to (Z/\mathbb{P}^1)^{hC_2},$ the canonical projection map.
Then, 
$$\mathbb{L}_{(Z/\mathbb{P}^1)^{hC_2}/\mathbb{P}^1}\simeq p_+q^*\big(\mathbb{L}_{([Z/C_2]/\mathsf{B}C_2)/\mathbb{P}^1}\big).$$
\end{proposition}
\begin{proof}
See e.g. \cite[Prop.~19.1.4.3]{Lurie2018}.
\end{proof}
Proposition \ref{FixedPointCotangents} states the relative cotangent complex $\mathbb{L}_{(Z/\mathbb{P}^1)^{hC_2}/\mathbb{P}^1}$ is given by taking (derived) $C_2$-coinvariants $p_+$, of the pullback to $Z^{hC_2}$ of the equivariant relative cotangent complex of $Z\to \mathbb{P}^1$ (i.e. the relative cotangent complex with canonical $C_2$-action).

\subsection{Derived (pre-)twistor families}
\label{PreTwistorSection}
In this subsection we turn to defining twistor structures. Consider a relative derived analytic stack locally geometric and finitely presented $Z\in \mathsf{dAnSt}_{\mathbb{C}}.$ Let $\mathbb{L}_{Z}^{\mathrm{an}}$ denote its analytic cotangent complex. Viewing $\mathbb{P}_{\mathbb{C}}^1$ as a derived analytic stack, we consider the $\infty$-category $\mathsf{dAnSt}_{\mathbf{P}_{\mathbb{C}}^1}$ of derived analytic stacks relative to $\mathbf{P}^1$ over $\mathbb{C},$  where we also view $\mathbf{P}^1$ as gluing two copies $\mathbf{A}^1_{\mathbb{C}}$. They have conjugate coordinate variables, and if $\beta\in\mathbb{G}_m$, the conjugate action is $\beta^{\mathrm{conj}}:=1/\beta\in\mathbb{G}_m.$
We would like to construct an $\infty$-functor,
$$(-)^{\mathrm{conj}}:\mathsf{dAnSt}_{\mathbb{C}}\to \mathsf{dAnSt}_{\mathbb{C}},$$
defined via the base-change along the `conjugation' morphism $\overline{(-)}:\mathrm{Spec}(\mathbb{C})\to \mathrm{Spec}(\mathbb{C}),$ defined for analytic algebras and then by Kan extension to stacks. This is defined in terms of involutive structures from the action of $\mathrm{Gal}(\mathbb{C}/\mathbb{R}).$ Specifically, from the functor of points description, extend first $\overline{(-)}$ to connective derived analytic $\mathbb{C}$-algebras $(-)^{\mathrm{conj}}:\mathsf{dAnRing}_{\mathbb{C}}\to \mathsf{dAnRing}_{\mathbb{C}},$ via base-change\footnote{As mentioned in \cite[Sect.~1.3]{Pridham2020dgDerivedAnalytic}, complex conjugation extends to any derived analytic setting modelled on non-Archimedean entire functional calculus algebras as a semilinear operation.}. Then, 
$$Z^{\mathrm{conj}}:\mathsf{dAnRing}_{\mathbb{C}}^{op}\to \mathsf{Spc},\hspace{2mm} A\mapsto Z^{\mathrm{conj}}(A):=Z(A^{\mathrm{conj}}).$$
Similarly we obtain a base-change $\infty$-functor $A\mathrm{mods}\to A^{\mathrm{conj}}\mathrm{mods}.$
A \emph{real} structure on a relative derived analytic stack $Z\in \mathsf{dAnSt}_{\mathbb{P}_{\mathbb{C}}^1},$ is a morphism $\sigma_Z:Z\to Z^{\mathrm{conj}},$ such that $(\sigma_Z)^{\mathrm{conj}}\circ \sigma_Z\simeq \mathrm{id}_Z.$

\begin{definition}
    \label{RealStructure}
Let $Z$ be an object of $\mathsf{dAnSt}_{\mathbb{P}_{\mathbb{C}}^1}$ equipped with a real structure $\sigma_Z:Z\to Z^{\mathrm{conj}}\simeq \sigma^*Z$.
The real structure on $\eta_Z:Z\to \mathbb{P}^1$ is said to \emph{cover the antipodal map} (on $\mathbb{P}_{\mathbb{C}}^1$) if the following diagram commutes in $\mathsf{dAnSt}_{\mathbb{P}^1}:$
\begin{equation}
    \label{eqn: CoveringAntipodal}
\begin{tikzcd}
Z\arrow[d,"\eta_Z"]\arrow[r,"\sigma_Z"]& Z^{\mathrm{conj}}\arrow[d,"\eta_{Z^{\mathrm{conj}}}"] 
\\
\mathbb{P}_{\mathbb{C}}^1\arrow[r,"a"] & \mathbb{P}_{\mathbb{C}}^1.
\end{tikzcd}
\end{equation}
\end{definition}
In other words, $\overline{\eta}:Z^{\mathrm{conj}}\to \overline{\mathbb{P}}^1$ is the conjugate of $\eta_Z$, and we identify $\mathbb{P}_{\mathbb{C}}^1$ with its conjugate as analytic stacks. Via the conjugate decomposition $\mathbf{P}_{\mathbb{C}}^1$ into $\mathbf{A}^1\cup \mathbf{A}^1,$ any $Z\in\mathsf{dAnSt}_{/\mathbf{P}^1}$ has a $\mathbb{G}_m$-equivariant structure morphism to $\mathbf{A}^1$, compatible with conjugations.

Given such a structure $(\eta_Z:Z\to \mathbb{P}^1,\sigma_Z),$ note the following properties of conjugation as an $\infty$-functor.

\begin{proposition}
\label{AnalyticDRConj}
Let $Z$ be a derived analytic stack with perfect analytic cotangent complex $\mathbb{L}_{Z}^{\mathrm{an}}.$ Then, there is an equivalence $(Z^{\mathrm{conj}})^{\mathrm{conj}}\simeq Z,$ and an equivalence
$$\mathbb{L}_{Z^{\mathrm{conj}}}^{\mathrm{an}}\simeq (\mathbb{L}_Z^{\mathrm{an}})^{\mathrm{conj}},$$ of $\mathcal{O}_{Z^{\mathrm{conj}}}$-modules.
\end{proposition}
\begin{proof}
This follows from a more general fact: let $\rho: K\hookrightarrow \mathbb{C}$ be any sub-field. Consider the structured analytic $\infty$-topose $\EuScript{X}$ associated. Then there is a base-change topoi $\EuScript{X}_{\rho}:=(\EuScript{X},\mathcal{O}_{\EuScript{X}_{\rho}})$ with underlying topological space the same, but with $\mathcal{O}_{\EuScript{X}_{\rho}}:=\mathcal{O}_{\EuScript{X}}\otimes_{K,\rho}\mathbb{C}$, where $\mathbb{C}$ is a $K$-module via $\rho.$ Using base-change via Proposition \ref{AnCotangentProperties}, the result follows since there is an equivalence 
$\mathbb{L}_{\EuScript{X}_{\rho}}^{\mathrm{an}}\simeq \mathbb{L}_{\EuScript{X}}^{\mathrm{an}}\otimes_{K,\rho}\mathbb{C}.$
Similarly,
$$\tau^*(\mathrm{DR}(Z/\mathbb{P}^1)\otimes \eta^*\EuScript{O}(2))\simeq \mathrm{DR}(\tau^*Z/\mathbb{P}^1)\otimes \eta_{\tau^*Z}^*\EuScript{O}(2)\simeq \mathrm{DR}(\tau^*Z/\mathbb{P}^1)\otimes \eta_Z^*\EuScript{O}(2),$$
since
$\tau^*(\eta^*\EuScript{O}(2))\simeq \eta_Z^*\EuScript{O}(2).$
\end{proof}
Suppose $Z\to \mathbb{P}^1$ has an $\EuScript{O}_{\mathbb{P}^1}(2)$-twisted $n$-shifted symplectic form $\omega.$ 
We can always obtain a conjugate $2$-form $\overline{\omega}$ on $Z^{\mathrm{conj}}.$ Fixing a choice of a real structure $\sigma_Z:Z\to Z^{\mathrm{conj}},$ we can compare $\omega$ with $\sigma_Z^*(\overline{\omega}).$ 

Indeed, fixing a choice of real structure $\sigma_Z$ on $Z,$ then there is a canonical map 
$$\sigma_Z^*:\Gamma(Z^{\mathrm{conj}},\wedge^2\mathbb{L}_{Z^{\mathrm{conj}}/\mathbb{P}^1}^{\mathrm{an}}[n]\otimes \eta_{Z^{\mathrm{conj}}}^*\EuScript{O}_{\mathbb{P}^1}(2)\big)\to \Gamma\big(Z,\wedge^2\mathbb{L}_{Z/\mathbb{P}^1}^{\mathrm{an}}[n]\otimes \eta_Z^*\mathcal{O}_{\mathbb{P}}^1(2)\big).$$
Indeed, via Proposition \ref{AnalyticDRConj}, there are equivalences
$$\sigma_Z^*\Theta_{\overline{\omega}}\simeq \Theta_{\sigma_Z^*(\overline{\omega})}:\mathbb{T}_{Z/\mathbb{P}^1}^{\mathrm{an}}\to \mathbb{L}_{Z/\mathbb{P}^1}^{\mathrm{an}}[n]\otimes \eta_Z^* a^*\EuScript{O}_{\mathbb{P}^1}(2),$$
from which we may use the diagram \eqref{eqn: CoveringAntipodal} to see $\sigma^*(\eta_{Z^{\mathrm{conj}}})^*\simeq (a\circ \eta_Z)^*.$ Thus there are induced maps on spaces of non-degenerate closed $2$-forms.

The condition of real-compatibility of a relative shifted two form is formulated in terms of spaces of $\mathrm{Gal}$-equivariant and $\mathrm{Gal}$-invariant forms.
\begin{definition}
Let $Z \to \mathbf{P}^1$ be a derived analytic stack with perfect relative cotangent complex $\mathbb{L}_{Z/\bfP}.$ Suppose that $Z$ is equipped with a $\mathrm{Gal}(\mathbb{C}/\mathbb{R})$-action covering the antipodal map on $\bfP$. For a sheaf $\EuScript{F}$ on $\bfP$, and $n\in\mathbb{Z}$, we have:
\begin{enumerate}
    \item The space of \emph{$\mathrm{Gal}(\mathbb{C}/\mathbb{R})$-invariant $\EuScript{F}$-twisted closed $p$-forms of degree $n$}, 
    $$\EuScript{A}^{p,\mathrm{cl}}_{\mathsf{B}\mathrm{Gal}(\mathbb{C}/\mathbb{R})\times\bfP}\big([\![Z/\mathrm{Gal}(\mathbb{C}/\mathbb{R})]\!],n\,;\,\EuScript{F}\big):=\mathrm{Maps}\big(\mathcal{O}_{\mathsf{B}\mathrm{Gal}(\mathbb{C}/\mathbb{R})},\mathrm{Fil}^p\mathrm{DR}([\![Z/\mathrm{Gal}(\mathbb{C}/\mathbb{R})]\!]\otimes \EuScript{F}[n+p]\big),$$
    where $\mathrm{Maps}$ is the mapping space in the $\infty$-category of coherent sheaves on the relative group stack $\mathsf{B}\mathrm{Gal}(\mathbb{C}/\mathbb{R})_{,\mathbf{P}^1}.$

    \item The space of \emph{$\mathrm{Gal}(\mathbb{C}/\mathbb{R})$-equivariant $\EuScript{F}$-twisted closed $p$-forms of degree $n$,}
    $$\EuScript{A}_{\bfP}^{p,\mathrm{cl}}([\![Z/\mathrm{Gal}(\mathbb{C}/\mathbb{R})]\!],n\,;\,\EuScript{F})=\mathrm{Maps}\big(\mathcal{O}_{\bfP},\mathrm{Fil}^p\mathrm{DR}([\![Z/\mathrm{Gal}(\mathbb{C}/\mathbb{R})]\!]\otimes \EuScript{F}[n+p]\big).$$
\end{enumerate}
\end{definition}
All results in this section work for general prestacks $Z\to S$ where $Z$ has an action by a group stack $G$, and sheaves $\EuScript{F}\in\mathsf{Perf}(S).$ For our intended applications we always take $\EuScript{F}$ to be $\EuScript{O}(2):=\EuScript{O}_{\bfP}(2)$ and state all results for $\EuScript{O}(2)$ and $C_2=\mathrm{Gal}(\mathbb{C}/\mathbb{R}).$

For every $p\geq 0,$ there is a commutative diagram of spaces,
\begin{equation}
    \label{EqInvForms}
    \begin{tikzcd}
\EuScript{A}_{\bfP}^{p,\mathrm{cl}}\big([\![Z/C_2]\!], n\,;\,\EuScript{O}(2)\big)\arrow[d]\arrow[dr,"\mathrm{oblv}_{\mathrm{eq}}"]\arrow[rr] & & \arrow[dl, swap,"\mathrm{oblv}_{\mathrm{inv}}"]\EuScript{A}^{p,\mathrm{cl}}_{\mathsf{B}C_{2,\bfP}}\big([\![Z/C_2]\!],n\,;\,\EuScript{O}(2)\big)\arrow[d]
    \\
   \EuScript{A}_{\bfP}^{p}\big([\![Z/C_2]\!],n\,;\,\EuScript{O}(2)\big)\arrow[dr] & \EuScript{A}_{\bfP}^{p,\mathrm{cl}}\big(Z,n\,;\,\EuScript{O}(2)\big)\arrow[d] & \arrow[dl] \EuScript{A}^{p}_{\mathsf{B}C_{2,\bfP}}\big([\![Z/C_2]\!],n\,;\,\EuScript{O}(2)\big)
   \\
   & \EuScript{A}_{\bfP}^{p}\big(Z,n\,;\,\EuScript{O}(2)\big) & 
    \end{tikzcd}
\end{equation}
where the vertical morphisms are the usual maps sending a closed $p$-form to its underlying $p$-form \cite{PTVV13}, and $\mathrm{oblv}_{C_2-\mathrm{eq}}$ and $\mathrm{oblv}_{C_2-\mathrm{inv}}$ are the maps forgetting the $C_2$-equivariance and $C_2$-invariance, respectively.
There is also an induced map $\EuScript{A}_{\bfP}^{p}\big([\![Z/C_2]\!],n\,;\,\EuScript{O}(2)\big)\rightarrow \EuScript{A}^{p}_{\mathsf{B}C_{2,\bfP}}\big([\![Z/C_2]\!],n\,;\,\EuScript{O}(2)\big).$

\subsubsection{Symplectic $\mathrm{Gal}(\mathbb{C}/\mathbb{R})$-actions.}
A $\mathrm{Gal}(\mathbb{C}/\mathbb{R})$-action on a relative $n$-shifted symplectic derived analytic stack $(Z \to \mathbf{P}^1, \omega)$ is \emph{symplectic} if there exists a $C_2$-invariant closed $2$-form $\omega^{inv}\in \EuScript{A}_{C_2-\mathrm{inv}}^{2,\mathrm{cl}}(Z/\mathbf{P}^1,n\,;\,\EuScript{O}(2)),$ and an equivalence 
$$\text{oblv}_{\mathrm{inv}}(\omega^{inv})\simeq \omega,\hspace{5mm}\text{ in }\EuScript{A}^{2,\mathrm{cl}}_{\mathbf{P}^1}(Z,n\,;\,\EuScript{O}(2)\big),$$
via the diagram \eqref{EqInvForms}.
Recall that $\mathrm{Symp}(Z,n)= \EuScript{A}^{2,\mathrm{nd}}(Z,n)\times_{\EuScript{A}^2(Z,n)}\EuScript{A}^{2,\mathrm{cl}}(Z,n)\in\EuScript{S}$, is the space of $n$-shifted symplectic forms on $Z$ \cite[Def.~1.18 (2)]{PTVV13}. The \emph{space of symplectic $\mathrm{Gal}$-actions}, is the homotopy fiber product,
\[
\begin{tikzcd} \EuScript{S}\mathrm{ymp}\big(\mathrm{Gal} \circlearrowright Z/\mathbf{P}^1\big)\arrow[d]\arrow[r] & \EuScript{A}_{C_2-\mathrm{inv}}^{2,\mathrm{cl}}\big(Z/\mathbf{P}^1,n\,;\,\EuScript{O}(2)\big)\arrow[d,"\mathrm{oblv}_{\mathrm{inv}}"]
    \\
    *\arrow[r,"\omega"] & \EuScript{A}_{\mathbf{P}^1}^{2,\mathrm{cl}}\big(Z,n\,;\,\EuScript{O}(2)\big).
\end{tikzcd}\]
Following \cite[Prop.~1.11]{PTVV13}, we let $\underline{\EuScript{A}}_{\mathbf{P}^1,C_2}^{2,\mathrm{cl}}(n\,;\,\EuScript{O}(2))$ be the object of $\mathsf{dAnSt}_{/\mathbf{P}^1}^{C_2},$ satisfying the universal property that
$$\mathrm{Maps}_{\mathsf{dAnSt}_{/\mathbf{P}^1}^{C_2}}\big(Z,\underline{\EuScript{A}}_{\mathbf{P}^1,C_2}^{2,\mathrm{cl}}(n\,;\,\EuScript{O}(2))\big)\simeq \EuScript{A}_{C_2-\text{inv}}^{2,\mathrm{cl}}(Z/\mathbf{P}^1,n\,;\,\EuScript{O}(2)).$$ 
Using the general formalism in \cite{HHLN23}, we study higher categories of correspondences (see e.g. \cite{Haugseng2017HigherMorita,Haugseng2018}); we consider
\begin{equation}
    \label{CorrGalCat}
\mathrm{Corr}_{/\mathbf{P}^1}^{\mathrm{Gal}(\mathbb{C}/\mathbb{R})}:=\mathrm{Corr}\big(\mathsf{dAnSt}_{/\mathbf{P}^1}^{\mathrm{Gal}(\mathbb{C}/\mathbb{R})}\big).\end{equation}
More explicitly, an object of \eqref{CorrGalCat} is a derived analytic stack $Z$ over $\bfP$ with $\mathrm{Gal}$-action, while a morphism $Z\to Z'$ in \eqref{CorrGalCat} is a $\mathrm{Gal}$-invariant diagram $Z\leftarrow Y\rightarrow Z'$ of derived analytic stacks over $\bfP$. 
We may consider the generalization of \eqref{nPreSymp}, or more precisely its analytic version \eqref{nPreSympAnalytic} to include group-stack actions,
$$\mathsf{PrSymp}_{\mathbf{P}^1,n}^{C_2}:=\mathrm{Corr}_{/\underline{\EuScript{A}}_{\mathbf{P}^1,C_2}^{2,\mathrm{cl}}(n\,;\,\EuScript{O}(2))}^{C_2}\equiv \mathrm{Corr}\big(\mathsf{dAnSt}_{/\underline{\EuScript{A}}_{\mathbf{P}^1,C_2}^{2,\mathrm{cl}}(n\,;\,\EuScript{O}(2))}^{\mathrm{Gal}(\mathbb{C}/\mathbb{R})}\big).$$
Similarly for $\mathrm{Symp}_{\mathbf{P}^1,n}^{C_2}$. 
\begin{definition}
    The category of \emph{$n$-shifted $\EuScript{O}(2)$-twisted-symplectic derived analytic stacks over $\mathbf{P}^1$ with symplectic $\mathrm{Gal}$-action}, is
$\mathsf{SympAct}_{\mathbf{P}^1,n}^{\mathrm{Gal}}:=\mathsf{Symp}_{/\mathbf{P}^1\times\mathsf{B}\mathrm{Gal},n}.$
\end{definition}

This is well-defined, in particular since the stack 
$\underline{\EuScript{A}}_{\mathbf{P}^1,C_2}^{2,\mathrm{cl}}(n\,;\,\EuScript{O}(2))$ may be interpreted via an action of a right-adjoint.

\begin{proposition}
\label{HomotopyFixedPoints1}
Let $X_*$ be a smooth Segal groupoid object in derived Artin stacks over $\mathbb{P}^1$ with $C_2$-action, with quotient $Z=|X_*|.$ Then for each $n\in\mathbb{Z}$,  
$$\mathsf{Symp}(Z/\mathbb{P}^1;n)^{hC_2}\simeq \underset{\mathbb{R}\mathsf{Spec}(A)\to Z\in (\mathsf{dAff}_{/Z/\mathbb{P}^1}[C_2])^{op}}{\mathrm{lim}}\mathsf{Symp}(A,n)^{hC_2},$$
 is an equivalence in $\mathsf{Spc}.$
\end{proposition}
\begin{proof}
First, note there is a canonical morphism 
\begin{eqnarray*}
\underset{\mathsf{B}C_2}{\mathrm{lim}}\hspace{1mm}\mathrm{Maps}_{\mathsf{QCoh}(Z)}(\mathcal{O}_Z,\mathrm{DR}(Z/\mathbb{P}^1)\otimes \eta^*\EuScript{O}(2)\big)&\simeq &\mathrm{Maps}_{\mathsf{Spc}[C_2]}(\star,\mathrm{Maps}_{\mathsf{QCoh}(Z)}(\mathcal{O}_Z,\mathrm{DR}(Z/\mathbb{P}^1)\otimes \eta^*\EuScript{O}(2)\big)
\\
&\to& \mathrm{Maps}_{\mathsf{QCoh}(Z)}(\mathcal{O}_Z,\mathrm{DR}(Z/\mathbb{P}^1)\otimes \eta^*\EuScript{O}(2)\big),
\end{eqnarray*}
in $\mathsf{Spc}.$
It restricts to simplicial sets $\mathsf{Symp}(Z/\mathbb{P}^1,n).$
Indeed, for any simplicial set $K_{\bullet}\in\mathsf{Spc}[C_2]$ consider the unit 
$$K_{\bullet}\to \mathrm{Triv}(K_{\bullet}^{hC_2}),$$
where the trivial action functor
$\mathrm{triv}$ has a right-adjoint $(-)^{hC_2}.$ As a right-adjoint, it commutes with limits, and we can consider the composition of functors,
$$\big(\mathsf{dAff}_{/Z/\mathbb{P}^1}[C_2]\big)^{op}\xrightarrow{\mathsf{Symp}(-/\mathbb{P}^1,k)}\mathsf{Spc}[C_2]\to \mathsf{Spc},$$
to obtain that, for each $C_2$-equivariant affine cover $\{U_i\to Z\}$, form the nerve $U_{\bullet}$ as a simplicial object in $\mathsf{dAff}_{/Z}[C_2]$, and thus 
$\mathsf{Symp}(Z/\mathbb{P}^1,k)\simeq \mathrm{lim}_{\Delta}\mathrm{Sypl}(U_{\bullet}/\mathbb{P}^1,k)$ in $C_2$-spaces. Then, we obtain
$$\mathsf{Symp}(Z/\mathbb{P}^1;n)^{hC_2}\simeq \underset{\mathbb{R}\mathsf{Spec}(A)\to Z\in (\mathsf{dAff}_{/Z/\mathbb{P}^1}[C_2])^{op}}{\mathrm{lim}}\mathsf{Symp}(A,n)^{hC_2},$$
which is an equivalence in $\mathsf{Spc}.$
Moreover, from Proposition \ref{C2Waldhausen}, by possibly assuming cofibrancy in the local model structure on the category of stacks, from \eqref{eqn: G equiv versus stack} we may compute that
$$\mathrm{Maps}\big(\EuScript{S}_{\bullet}\RPerf_{G^{\mathbb{C}}}(Z),\EuScript{A}^{2}_d)\simeq \mathrm{holim}_k\mathrm{Maps}(\EuScript{S}_{k}\RPerf_{G^{\mathbb{C}}}(Z),\EuScript{A}^2_d)
    \simeq \mathrm{holim}_k\EuScript{A}^2([\EuScript{S}_k\RPerf(Z)/G^{\mathbb{C}}];d),$$
    which coincides with $\mathrm{holim}_k\EuScript{A}^{2}(G^{\mathbb{C}\times k}\times Z;d).$
\end{proof}
The relative $2$-form $\omega \in \EuScript{A}^{2,cl}(Z/\mathbb{P}^1,n)$ is said to be \emph{real} if $\sigma_Z^*(\overline{\omega})\simeq\omega,$ for some real structure $\sigma_Z$ on $Z.$
More precisely, by Proposition \ref{HomotopyFixedPoints1},
there is a canonical map 
\begin{equation}
\label{eqn: RealSymplForget}
\mathsf{Symp}^{\sigma}(Z/\mathbb{P}^1,n):=\mathsf{Symp}(Z/\mathbb{P}^1,n)^{hC_2}\to \mathsf{Symp}(Z/\mathbb{P}^1,n),
\end{equation}
in $\mathsf{Spc}.$ The essential image of \eqref{eqn: RealSymplForget} consists of symplectic forms that are equivalent to their conjugates.
Thus, by a form with real structure, we mean lies in this essential image.

\subsubsection{Pre-twistor data}
We can now make precise the idea sketched in \cite{KPS2021}, which formulate a twistor family of hyperk\"ahler type.
\begin{definition}
\label{DerivedTwistor}
Let $n\in \mathbb{Z}.$
    An \emph{$n$-shifted derived pre-twistor family of hyperk\"ahler type} is a pair $(Z,\omega_{rel})$, consisting of an object $Z\in \mathsf{dSt}_{\mathbb{P}^1}[C_2]\simeq \mathsf{Funct}(\mathsf{B}C_2,\mathsf{dSt}_{/\mathbb{P}_{\mathbb{C}}^1})$, namely,  a relative derived stack with a real structure covering the antipode, meaning a $C_2$-equivariant map
    $\eta:(Z,\tau)\to (\mathbb{P}^1,\sigma),$
    together with an element $\omega_{rel}\in \mathsf{Symp}^{\EuScript{O}(2)}(Z/\mathbb{P}^1;n),$ compatible with the $C_2$-action. 
    
    \end{definition}
Recalling an $n$-shifted presymplectic stack over a base $S$ is an object \ref{nPreSymp},
we may also characterize the $\infty$-category of shifted pre-twistor families.
\begin{corollary}
\label{nShiftedPreTwistorGroth}
    The $\infty$-category of $n$-shifted derived pre-twistor spaces is equivalent to
    $$\mathsf{PreTwdSt}(n)\simeq \int_{Z\in \mathsf{dAnSt}_{/\mathbb{P}^1}[C_2]}\mathsf{Symp}^{\EuScript{O}(2)}(Z/\mathbb{P}^1,n)^{hC_2}.$$
\end{corollary}
\begin{proof}
Both  $\mathsf{dSt}_{\mathbb{P}^1/\EuScript{A}_{\mathbb{P}^1}^{2,\mathrm{cl}}(n)^{\EuScript{O}(2)}}$ and $\mathsf{Funct}(\mathsf{B}C_2,\mathsf{dSt}_{\mathbb{P}^1/\EuScript{A}_{\mathbb{P}^1}^{2,\mathrm{cl}}(n)^{\EuScript{O}(2)}})$ have all limits, and the forgetful functor $\mathsf{dSt}_{\mathbb{P}^1/\EuScript{A}_{\mathbb{P}^1}^{2,\mathrm{cl}}(n)^{\EuScript{O}(2)}}[C_2]\to \mathsf{dSt}_{\mathbb{P}^1}[C_2]$ preserves weakly contractible limits. Since $(-)^{hC_2}$ commutes with limits, the composite functor $\mathsf{dSt}_{\mathbb{P}^1/\EuScript{A}_{\mathbb{P}^1}^{2,\mathrm{cl}}(n)^{\EuScript{O}(2)}}[C_2]\to \mathsf{dSt}_{\mathbb{P}^1}[C_2]\to \mathsf{dSt}_{\mathbb{P}^1},$ also preserves such limits. Therefore, by 
Proposition \ref{C2Grothendieck} and Proposition \ref{HomotopyFixedPoints1}, the result follows by the same argument given in \cite[Rmk.~B.12.6]{CalaqueHaugsengScheimbauer2025}.
\end{proof}
Consider the $\mathbb{G}_m$-equivariant line bundle $\EuScript{O}_{\mathbb{P}^1}(2)$ on $\mathbb{P}^1$ associated to the character $t\mapsto t^2.$ We provide an alternative characterization of pretwistor data via \emph{weighted} symplectic structures \cite{GinRoz}. 

Given a $\mathbb{G}_m$-stack $Z'\to X$, with $X$ any smooth projective variety with $L\in\mathrm{Pic}(X),$ letting $Z_{\mathbb{L}}:=Z\times_{\mathsf{B}\mathbb{G}_m}X$ be the associated bundle via the classifying map $X\to\mathsf{B}\mathbb{G}_m$ for $L$. 

For $m\geq 1,$ let $\EuScript{A}^{2,\mathrm{cl}}(Z)(m),$ denote the weight $m$-component. Then, there is a canonical map
\begin{equation}
    \label{TwistingMap}
    \mathrm{twist}_{L}:\EuScript{A}^{2,\mathrm{cl}}(Z)(m)\to \mathbb{R}\Gamma\big(X,\widetilde{\mathcal{A}}_X^{2,\mathrm{cl}}(Z_{\mathbb{L}}(m)\otimes L^{\otimes m}\big),
\end{equation}
sending an $n$-shifted symplectic form of weight $m$ on the $\mathbb{G}_m$-stack $Z$ to a \emph{relative} $n$-shifted closed $2$-form on $Z_{\mathbb{L}}\to X$, twisted by the line bundle $L^{\otimes m}$ on $X$.

Applying this to $X=\mathbb{P}_{\mathbb{C}}^1$ with $L=\mathcal{O}_{\mathbb{P}^1}(1),$ then an $n$-shifted form of weight $2$ on a $\mathbb{G}_m$-stack $Z$ produces a relative $n$-shifted form on $Z_{\mathcal{O}(1)}\to \mathbb{P}_{\mathbb{C}}^1$, twisted by $\EuScript{O}_{\mathbb{P}^1}(2)=\mathcal{O}(1)^{\otimes 2}.$

\begin{proposition}
    \label{WeightVersion}
Let $\eta:Z\to \mathbb{P}^1$ be a relative derived Artin stack. Then, the following are equivalent:
\begin{itemize}
\item[(i)] $(\eta:Z\to \mathbb{P}_{\mathbb{C}}^1,\omega,\sigma_Z)$ has the structure of an $n$-shifted pretwistor family;
\item[(ii)] there exists a $\mathbb{G}_m$-stack $Z'$ with $C_2$-action and invariant $n$-shifted symplectic form $\omega_{Z'}$ of weight $2$, with an equivalence
$$Z\simeq Z_{\mathcal{O}_{\mathbb{P}^1}(1)}:=Z'\times_{\mathsf{B}\mathbb{G}_m}\mathbb{P}^1,$$
where $\mathbb{P}^1\to\mathsf{B}\mathbb{G}_m$ classifies $\mathcal{O}_{\mathbb{P}^1}(1),$ and where
the relative symplectic form on $Z/\mathbb{P}^1$ is the image of $\omega_{Z'}$ by $\mathrm{twist}_{\mathcal{O}(1)}.$
\end{itemize} 
\end{proposition}

\begin{proof}
    Let $(\pi: Z \to \bfP)$ be a derived $\mathbb{G}_m$-stack, and let $\mathcal{O}_{\mathbb{P}^1}(1)^\times \to \mathbb{P}^1$ denote the associated $\mathbb{G}_m$-torsor. Define the associated bundle $
Z_{\mathcal{O}(1)} := Z \times_{\mathbb{G}_m} \mathcal{O}_{\mathbb{P}^1}(1)^\times.$
Let
$
\Sect_{\mathbb{P}^1}(Z_{\mathcal{O}(1)})
:=
{\mathrm{id}_{\mathbb{P}^1}}
\times_{\Map(\mathbb{P}^1,\mathbb{P}^1)}
\Map(\mathbb{P}^1, Z_{\mathcal{O}(1)})$
denote the derived stack of sections of the projection $Z_{\mathcal{O}(1)} \to \mathbb{P}^1$.

The $\mathbb{G}_m$-action on $Z$ induces a fiberwise $\mathbb{G}_m$-action on the associated bundle $Z_{\mathcal{O}(1)} \to \mathbb{P}^1$, and hence a natural $\mathbb{G}_m$-action on $\Sect_{\mathbb{P}^1}(Z_{\mathcal{O}(1)})$ by Proposition \ref{ActAnMaps}.

Moreover, if $\pi: Z \to \mathbb{P}^1$ is equipped with a relative $\EuScript{O}_{\mathbb{P}^1}(2)$-twisted $n$-shifted symplectic structure of $\mathbb{G}_m$-weight $2$, then the induced structure on $\Sect_{\mathbb{P}^1}(Z_{\mathcal{O}(1)})$ is naturally compatible with this $\mathbb{G}_m$-action, with weight $2$ arising from the identification $
\EuScript{O}_{\mathbb{P}^1}(2) \simeq \mathcal{O}_{\mathbb{P}^1}(1)^{\otimes 2}.$
Letting $q:\mathbb{P}_{\mathbb{C}}^1\times\mathsf{B}\mathbb{G}_m\to \mathsf{B}\mathbb{G}_m$ be the natural map classifying the canonical line bundle $\mathcal{O}_{\mathbb{P}^1}(1)\boxtimes \mathcal{O}(-1),$ note there is a cartesian diagram
\[
\begin{tikzcd}
    Z_{\mathcal{O}(1)}\arrow[r]\arrow[d] & Z_{\mathcal{O}(1)}/\mathsf{B}\mathbb{G}_m\arrow[d,"\alpha"] \arrow[r] & \mathbb{P}_{\mathbb{C}}^1\times Z'/\mathbb{G}_m\arrow[d,"\beta"]
    \\
    \mathbb{P}_{\mathbb{C}}^1\arrow[r] & \mathbb{P}_{\mathbb{C}}^1\times\mathsf{B}\mathbb{G}_n\arrow[r,"p_X\times q"] & \mathbb{P}_{\mathbb{C}}^1\times\mathsf{B}\mathbb{G}_m.
\end{tikzcd}
\]
It induces an equivalence,
\begin{eqnarray*}
\alpha_*\big(\wedge^2\mathbb{L}_{(Z_{\mathcal{O}(1)}/\mathbb{G}_m)/\mathbb{P}_{\mathbb{C}}^1\times\mathsf{B}\mathbb{G}_m}\big)&=&\mathbb{R}\Gamma_{\mathbb{P}_{\mathbb{C}}^1\times\mathsf{B}\mathbb{G}_m}\big(Z_{\mathcal{O}(1)}/\mathbb{G}_m,\wedge^2\mathbb{L}_{(Z_{\mathcal{O}(1)}/\mathbb{G}_m)/\mathbb{P}_{\mathbb{C}}^1\times\mathsf{B}\mathbb{G}_m}\big)
\\
&
\simeq& (p_{\mathbb{P}^1}\times q)^*\mathbb{R}\Gamma_{\mathbb{P}_{\mathbb{C}}^1\times\mathsf{B}\mathbb{G}_m}\big(\mathbb{P}^1\times Z'/\mathbb{G}_m,\wedge^2\mathbb{L}_{\mathbb{P}^1\times Z'/\mathbb{G}_m/\mathbb{P}_{\mathbb{C}}^1\times\mathsf{B}\mathbb{G}_m}\big).
\end{eqnarray*}
 We have a commutative diagram,
\[
\begin{tikzcd}
\widetilde{\mathcal{A}}_{\mathbb{P}_{\mathbb{C}}^1\times\mathsf{B}\mathbb{G}_m}(Z_{\mathcal{O}(1)}/\mathsf{B}\mathbb{G}_m)\arrow[d] \arrow[r,"\sim"] & (p_{\mathbb{P}^1}\times q)^*\big(\widetilde{\mathcal{A}}_{\mathbb{P}_{\mathbb{C}}^1\times\mathsf{B}\mathbb{G}_m}^{2,\mathrm{cl}}(\mathbb{P}^1\times Z'/\mathbb{G}_m)\big)\arrow[d]
 \\
\mathbb{R}\Gamma_{\mathbb{P}_{\mathbb{C}}^1\times\mathsf{B}\mathbb{G}_m}\big(Z_{\mathcal{O}(1)}/\mathbb{G}_m,\wedge^2\mathbb{L}_{(Z_{\mathcal{O}(1)}/\mathbb{G}_m)/\mathbb{P}_{\mathbb{C}}^1\times\mathsf{B}\mathbb{G}_m}\big)\arrow[r] & (p_{\mathbb{P}^1}\times q)^*\mathbb{R}\Gamma_{\mathbb{P}_{\mathbb{C}}^1\times\mathsf{B}\mathbb{G}_m}\big(\mathbb{P}^1\times Z'/\mathbb{G}_m,\wedge^2\mathbb{L}_{\mathbb{P}^1\times Z'/\mathbb{G}_m/\mathbb{P}_{\mathbb{C}}^1\times\mathsf{B}\mathbb{G}_m}\big).
\end{tikzcd}
\]
The vertical maps are simply the canonical maps sending a closed $2$-form to its underlying $2$-form.
Since the diagram is cartesian, via base change (see e.g, \cite[Thm.~2.7]{calaque2024shifted}), comparing the weight $2$-component of
$\widetilde{\mathcal{A}}_{\mathbb{P}_{\mathbb{C}}^1}^{2,\mathrm{cl}}(Z_{\mathcal{O}(1)})$, we see it is given by a twist of $\widetilde{\mathcal{A}}^{2,\mathrm{cl}}(Z')(2)$ by powers of the tatuological line bundle and we conclude via \eqref{TwistingMap}, by noticing there is a homotopy pullback diagram
\[
\begin{tikzcd} Z \ar[r] \ar[d] & Z' \ar[d] \\ \mathbb{P}^1 \ar[r] &\mathsf{B}\mathbb{G}_m, \end{tikzcd}
\]
where the right vertical map is the projection from the $\mathbb{G}_m$-stack $Z'$ to the classifying stack and the bottom map classifies $\mathcal{O}(1),$ and the $n$-shifted symplectic form on $Z'$ has weight $2.$

\end{proof}

Motivated by classical theorems of 
Hitchin--Karlhede--Lindstr\"om--Roček \cite{HKLR87}, we now introduce a geometric condition which characterizes when a pre-twistor family is a twistor family of hyperk\"ahler type.

To this end, recall Weil restriction is defined by the universal property,
$$\mathrm{Maps}_{\mathsf{dSt}_{\mathbb{P}_{\mathbb{C}}^1}}(T,\underline{\mathrm{Res}}(Z/\mathbb{P}^1)\big)\simeq \mathrm{Maps}(T\times \mathbb{P}^1,Z),$$
and $\Phi$ is defined by the universal evaluation map (also stated via \eqref{evaldiagram} below):
\[
\begin{tikzcd} \mathbb{P}^1\times \underline{\mathrm{Res}}(Z/\mathbb{P}^1)\arrow[d,"\mathrm{pr}"]\arrow[r,"\mathrm{ev}"] & Z\arrow[d,"\eta"] \\ \underline{\mathrm{Res}}(Z/\mathbb{P}^1)\arrow[r] & \mathbb{P}^1. \end{tikzcd}
\]

Derived twistor families are cocartesian sections, of the $C_2$-Grothendeick construction (as in Subsection \ref{ssec: C2Diags}), by applying \eqref{C2GrothCons}, classifying
$$\big( \mathsf{dSt}_{/\bfP}\big)_{/ \mathsf{Symp}^{\eta^*\EuScript{O}(2)}(-/\mathbb{P}^1, n)}[C_2],$$
by imposing geometric constraints on the homotopy fixed points of the analytic Weil restriction. Informally, there is an equivalence $\mathcal{X}\times \mathbb{P}^1\to Z$ of $C^{\infty}$-differential graded manifolds (cf. \cite{BKSY2025}), stating that $Z$ is constructed from real sections, smoothly over $\mathbb{P}^1$ and that the holomorphic structures are allowed to vary nontrivially along 
$\mathbb{P}^1.$\footnote{This is an analogue of the classical twistor space reconstruction of a hyperk\"ahler manifold.}

\begin{definition}
    \label{DerivedTwistor2}
Let $\eta:(Z,\tau)\to (\mathbb{P}^1,\sigma_{\mathbb{P}}^1)$ be an $n$-shifted pre-twistor family with relative symplectic form $\omega.$ It is a  \emph{twistor family of hk-type} if there exists a (union of) connected components $\mathcal{X}\subset t_0\underline{\mathrm{Res}}(Z/\mathbb{P}^1)^{hC_2},$ of the fixed points sections such that:  Zariski locally there exists an open substack $U\subset t_0(Z)$ admitting a good moduli space $p:U\to \mathrm{Tw}(M)$ for some \emph{underlying hyperkahler manifold} $M$, such that the canonical evaluation map  
$\Phi:\mathcal{X}\times\mathbb{P}^1\xrightarrow{\sim} U,$
is a $C^{\infty}$-equivalence compatible with $\mathrm{Tw}(M)\simeq M\times \mathbb{P}^1.$
\end{definition}
This definition is a first approximation to the existence of derived substacks of real sections\footnote{What could tentatively be called \emph{derived twistor lines.}}.

Thus, 
$\mathcal{X}$ can be viewed as spanned by those sections $s\in \underline{\mathrm{Res}}(Z/\mathbb{P}^1)^{hC_2}$, for which the following holds: consider the associated truncation $t_0(s):\mathbb{P}^1\to t_0(Z)$, then there exists $\sigma':\mathbb{P}^1\to U$ and a composite section $\tau_{\sigma'}:=p\circ \sigma \in \mathrm{Sec}(\mathbb{P}^1,\mathrm{Tw}(M)),$ which is a horizontal twistor line in the standard sense \cite{HKLR87}, making the diagram commute:
\begin{equation}
    \label{eqn: Twistor Diag}
\begin{tikzcd}
& U\arrow[d]\arrow[r,"p"] & \mathrm{Tw}(M)
\\
\mathbb{P}^1\arrow[ur,dotted,"\sigma'"]\arrow[r,"t_0(\sigma)"] & t_0(Z) &
\end{tikzcd}
\end{equation}
Note from Proposition \ref{WeightVersion}, that $\underline{\mathrm{Res}}(Z/\mathbb{P}^1)^{hC_2}$ is naturally endowed with a real torus action.

\begin{remark}
In our intended application, given in Sect. \ref{sec: Derived NAH and KW}, we will have that $Z$ is the (analytic) derived Deligne-stack $\RPerf_{\Del}(X)^{\an}$ and $U=t_0(\RPerf_{\Del}(X)^{\an})^{ss}$ is the semistable locus of its classical truncation. 
\end{remark}
The first interesting example of a pretwistor space in this level of generality is associated with the relative derived analytic stack $\AnPerf(\mathrm{Tw}(M)/\bfP)\to \bfP$
of perfect complexes along fibers of the twistor space $\mathrm{Tw}(M)\to \bfP$ of a classical hyperkahler manifold $M$ e.g. a smooth projective $K3$ surface for which the fiberwise derived moduli stack is the analytification of the $0$-shifted symplectic structure given in \cite{STV}. Details and further examples will be given elsewhere. 
\section{Existence of shifted twistor structures}
In this section we prove existence results for $n$-shifted twistor familes of hyperk\"ahler type, as introduced in Definition \ref{DerivedTwistor}. 
It is mainly achieved by expoiting the familiar existence theorems from (relative) derived symplectic geometry \cite{PTVV13,calaque2024shifted,CalaqueHaugsengScheimbauer2025}, adapted to the $\mathsf{B}C_2$-equivariant setting of derived (analytic) stacks relative to $\mathbf{P}^1.$

\subsection{Transgression with derived hyperk\"ahler structures}
\label{ssec: TransgressionHKStructures}
In this subsection, we prove an analogue of \cite[Thm.~2.5]{PTVV13}, for derived twistor families of hyperk\"ahler type.
\begin{remark}[Notation]
Let $Z\to S$ be an $n$-shifted $S$-prestack and let $X$ be an $S$-prestack with relative $d$-orientation $[X].$ We follow \cite[\S.3.2]{CalaqueHaugsengScheimbauer2025} and denote the transgression map to the mapping stack $\Map_{/S}(X,Z)$ by 
\begin{equation}
    \label{AKSZ}
\mathrm{aksz}_{(X,Z)}(-):=\int_{X/S}\mathrm{ev}^*(-):\EuScript{A}_{S}^{2,\mathrm{cl}}(Z,n)\to \EuScript{A}_{S}^{2,\mathrm{cl}}\big(\Map_{/S}(X,Z),n-d).
\end{equation}
\end{remark}
We now state the AKSZ-twistor theorem for constructing twistor structures on mapping stacks in either of the algebraic and analytic settings.
    \begin{theorem}
 \label{MainTheorem2Body}
Fix $n\in\mathbb{Z}$ and an integer $d\geq 0.$ Let $X$ be a proper geometric derived algebraic stack with a relative $d$-orientation $[X].$ 
\begin{enumerate}
\item Let $Z\to \mathbb{P}_{\mathbb{C}}^1$ be an $n$-shifted algebraic twistor family. Assume that
 \begin{itemize}
     \item Z is Tannakian,
     \item The $\infty$-category of quasi-coherent sheaves on $Z$ is saturated, $\mathsf{QCoh}(Z)\simeq\mathsf{Ind}\big(\mathsf{Perf}(Z)\big),$
     \item The mapping stack $\eta_M:\Map_{/\mathbb{P}_{\mathbb{C}}^1}(X\times\mathbb{P}_{\mathbb{C}}^1,Z) \to\mathbb{P}_{\mathbb{C}}^1$ is geometric.
 \end{itemize}
 Then, its analytification
 $$\eta_{\mathrm{Map}}^{\an}:\Map_{/\mathbb{P}_{\mathbb{C}}^1}(X\times \mathbb{P}_{\mathbb{C}}^1,Z)^{\an}\to \bfP,$$
 is canonically endowed with the structure of an $(n-d)$-shifted analytic pretwistor structure of hyperkahler-type.
 
\item Let $\eta_Z: Z \to \bfP$ be an $n$-shifted derived analytic pretwistor family such that 
$\eta_{\mathrm{AnMap}}:\AnMap_{/\bfP}(X^{\an}\times \bfP,Z)\to \bfP$ is locally geometric. 
Then it is canonically an $(n-d)$-shifted analytic pretwistor family.
\end{enumerate}
Moreover, if $Z\to\mathbb{P}_{\mathbb{C}}^1$ is a derived algebraic stack satisfying (1). Then the canonical map of derived analytic stacks
$$\Map_{/\mathbb{P}_{\mathbb{C}}^1}(X\times\mathbb{P}_{\mathbb{C}}^1,Z)^{\an}\rightarrow\AnMap_{/\bfP}(X^{\an}\times \bfP,Z^{\an}),$$
 is an equivalence of $(n-d)$-shifted twistor structures.
\end{theorem}
\begin{proof}
The relative $\EuScript{O}_{\mathbb{P}^1}(2)$-twisted $2$-form of degree $n$ along $\eta_Z$, denoted $\omega_Z^{rel}$ has underlying $2$-form a section 
$\omega_Z\in \Gamma\big(Z,\wedge^2\mathbb{T}_{Z/\mathbb{P}^1}^{\vee}\otimes \eta_Z^*\EuScript{O}_{\mathbb{P}^1}(2)\big)[n].$ It is non-degenerate, thus induces a quasi-isomorphism,
$$\mathbb{T}_{Z/\mathbb{P}^1}\xrightarrow{\simeq}\mathbb{T}_{Z/\mathbb{P}^1}^{\vee}[n]\otimes \eta_Z^*\EuScript{O}_{\mathbb{P}^1}(2).$$
The compatibility with AKSZ-construction in relative algebraic setting follows from \cite[Prop.~3.2.1]{CalaqueHaugsengScheimbauer2025}. We will prove the transgressed form, fiber-wise nondegenerate, is compatible with the induced real structures and is preserved under analyitifcation. In other words, we prove that symplectic $\mathrm{Gal}$-action is compatible with algebraic AKSZ-procedure  \eqref{AKSZ}, which analytifies to our desired analytic result using the canonical morphism $(\mathbb{L}_{Z/\mathbb{P}_{\mathbb{C}}^1})^{\an}\to \mathbb{L}_{Z^{\an}/\bfP}^{\an}$ and Proposition \ref{AnDeRhamAlgs}.
To this end, note the diagram
\begin{equation}
    \label{evaldiagram}
\begin{tikzcd}
X\times \Map_{/\mathbb{P}^1}(X,Z)\arrow[d,"\mathrm{pr}_{M}"]\arrow[r,"\mathrm{ev}"] & Z\arrow[d,"\eta_Z"]
\\
\Map_{/\mathbb{P}^1}(X,Z)\arrow[r,"\eta_{M}"] & \mathbb{P}^1_{\mathbb{C}}
\end{tikzcd}
\end{equation}
is $C_2$-equivariant by the construction of the real structure on the relative mapping stack, in Proposition \ref{ActAnMaps},  which holds in both the algebraic and analytic category. We have that 
$$\mathrm{ev}^*\eta_Z^*\EuScript{O}_{\mathbb{P}^1}(2)\simeq (\eta_M\circ \mathrm{pr}_M)^*\EuScript{O}_{\mathbb{P}^1}(2)\simeq \mathrm{pr}_M^*\eta_M^*\EuScript{O}_{\mathbb{P}^1}(2),$$
thus a twisted closed $2$-form of degree $n$, given by the section
\begin{equation}
    \label{EvOmega}
\underline{\omega}_{\mathrm{ev}}:=\underline{\mathrm{ev}^*\omega_Z}\in \Gamma\big(X\times \Map_{/\mathbb{P}^1}(X,Z)\wedge^2\mathbb{T}_{X\times \Map_{/\mathbb{P}^1}(X,Z)/\mathbb{P}^1}^{\vee}\otimes \mathrm{pr}_M^*\eta_M^*\EuScript{O}_{\mathbb{P}^1}(2)\big)[n].
\end{equation}
Thus, 
$$\mathrm{ev}^*\Theta_{\omega_Z}:\mathrm{ev}^*\mathbb{T}_{Z/\bfP}^{\an}\xrightarrow{\simeq}\mathrm{ev}^*\mathbb{L}_{Z/\bfP}^{\an}[n]\otimes (\mathrm{ev}^*\circ \eta^*)\EuScript{O}(2)\simeq \mathrm{ev}^*\mathbb{L}_{Z/\bfP}^{\an}[n]\otimes (\mathrm{pr}_M^*\circ \eta_M^*)\EuScript{O}(2).$$
Applying $\mathrm{pr}_{M,+}$ we have
$$\mathrm{pr}_{M,+}\mathrm{ev}^*\mathbb{T}_{Z/\bfP}^{\an}\xrightarrow{\simeq}\mathrm{pr}_{M,+}\big(\mathrm{ev}^*\mathbb{L}_{Z/\bfP}^{\an}[n]\otimes \mathrm{pr}_M^*\eta_M^*\EuScript{O}(2)\big)\simeq (\mathrm{pr}_{M,+}\mathrm{ev}^*\mathbb{L}_{Z/\bfP}^{\an}[n]\big)\otimes \eta_M^*\EuScript{O}(2),$$
and using the orientation, we obtain the map $\mathrm{pr}_+\mathrm{ev}^*\mathbb{T}_{Z/\bfP}^{\an}\to (\mathrm{pr}_+\mathrm{ev}^*\mathbb{T}_{Z/\bfP}^{\an})^{\vee}[n-d]\otimes \eta_M^*\EuScript{O}(2).$
Thus, by base-change induced from evaluation $\mathrm{ev}^*\mathbb{L}_{Z/\bfP}^{\an}\to \mathbb{L}_{X\times_{\bfP}\Map(X,Z)/\bfP}^{\an}\simeq \mathrm{pr}_M^*\big(\mathrm{pr}_{M,+}\mathrm{ev}^*\mathbb{T}_{Z/\bfP}^{\an}\big)^{\vee},$ corresponding to the counit $\mathrm{pr}^*\mathrm{pr}_+\mathrm{ev}^*\mathbb{T}_{Z/\bfP}^{\an}\to \mathrm{ev}^*\mathbb{T}_{Z/\bfP}^{\an}.$ We thus have a map
$$\wedge^2\mathrm{pr}_+\mathrm{ev}^*\mathbb{T}_{Z/\bfP}^{\an}\to \eta_M^*\EuScript{O}(2)[n-d],$$
which by the computation of the analytic cotangent complex in Proposition 
\ref{AnMapCotComplex}, gives the desired equivalence,
$$\Theta_{\omega_M}:\mathbb{T}_{\Map_{/\bfP}(X,Z)/\bfP}^{\an}\xrightarrow{\simeq}\mathbb{L}_{\Map_{\bfP}(X,Z)/\bfP}^{\an}[n-d]\otimes\eta_M^*\EuScript{O}(2).$$
To observe the compatibility with real structures, since $Z\to\bfP$ has a symplectic $\mathrm{Gal}$-action, there is an invariant form $\omega_Z^{\mathrm{inv}}$ and a homotopy $\mathrm{oblv}_Z^{\mathrm{inv}}(\omega_Z^{\mathrm{inv}})\simeq\omega_Z.$
We need to show that there exists an invariant form $\omega_M^{\mathrm{inv}}$ on the mapping stack and a homotopy $\mathrm{oblv}_{M}^{\mathrm{inv}}(\omega_M^{\mathrm{inv}})\simeq \omega_M.$
For this, consider the diagram \eqref{EqInvForms} and the corresponding maps \eqref{AKSZ}, giving
\[
\begin{tikzcd}
\EuScript{A}_{C_2-inv}^{2,\mathrm{cl}}\big(Z/\bfP,n;\EuScript{O}(2)\big)\arrow[d,"\mathrm{obvl}_Z^{\mathrm{inv}}"] \arrow[r,"\mathrm{aksz}^{inv}"]& \EuScript{A}_{C_2-inv}^{2,\mathrm{cl}}\big(\Map(X,Z)/\bfP; n-d,\EuScript{O}(2)\big)\arrow[d,"\mathrm{oblv}_M^{inv}"]
\\
\EuScript{A}_{\bfP}^{2,\mathrm{cl}}(Z,n;\EuScript{O}(2))\arrow[r,"\mathrm{aks}_{Z}"]& \EuScript{A}_{\bfP}^{2,cl}\big(\Map_{/\bfP}(X,Z)/\bfP,n-d;\EuScript{O}(2)\big).
\end{tikzcd}
\]
Setting $\omega_M^{\mathrm{inv}}:=\mathrm{aksz}_{(X,Z)}^{\mathrm{inv}}(\omega_Z^{\mathrm{inv}})$, we need to show the diagram commutes. One has 
$\mathrm{oblv}_M^{\mathrm{inv}}\omega_M^{\mathrm{inv}}\simeq \int_{X/\bfP}\mathrm{ev}^*\omega_Z,$ and by applying \eqref{AKSZ} to $\omega,$ we see
$$\mathrm{aksz}_{(X,Z)}(\mathrm{oblv}_Z^{\mathrm{inv}}\omega_Z^{\mathrm{inv}})\simeq \int_{X/\bfP}\mathrm{ev}^*\mathrm{oblv}_Z^{\mathrm{inv}}(\omega_Z^{\mathrm{inv}})\simeq \int_{X/\bfP}\mathrm{ev}^*\omega_Z=\omega_M,$$
as required. 
We now turn to the twistor assumption. First, observe if $\sigma_Z:Z\to Z$ is the real structure covering $\sigma_{\mathbb{P}^1}$ i.e.
$\eta_Z\circ \sigma_{\mathbb{P}^1}\simeq \sigma_{\mathbb{P}^1}\circ \eta_Z,$ where $\sigma_Z$ is anti-holomorphic write $\sigma_Z^*(\omega_Z)\simeq \overline{\omega_Z},$ for the conjugate $\mathbb{C}$-linear form.
Considering the diagram
\[
\begin{tikzcd}
    X\times \Map_{/\bfP}(X,Z)\arrow[d,"\mathrm{id}_X\times \sigma_M"]\arrow[r,"\mathrm{ev}"] & Z\arrow[d,"\sigma_Z"]
    \\
    X\times \Map_{/\bfP}(X,Z)\arrow[r,"\mathrm{ev}"]& Z.
\end{tikzcd}
\]
Then, we have that $(\mathrm{id}_X\times \sigma_M)^*\circ \mathrm{ev}^*\simeq \mathrm{ev}^*\circ \sigma_Z^*,$ and since $\sigma_Z^*(\omega_Z)\simeq \overline{\omega_Z},$ we obtain that 
$(\mathrm{id}_X\times \sigma_M)^*(\mathrm{ev}^*\omega_Z)\simeq\overline{\mathrm{ev}^*\omega_Z},$
in the complex of $\mathrm{pr}_M^*\eta_M^*\EuScript{O}_{\mathbb{P}^1}(2)$-twisted relative closed $2$-forms. Since $\int_{[X]}$ is $\mathbb{C}$-linear
$\sigma_M^*(\omega_M)\simeq \sigma_M^*\big(\int_{[X]}\mathrm{ev}^*\omega_Z\big)\simeq \int_{[X]}(\mathrm{id}_X\times \sigma_M)^*\mathrm{ev}^*\omega_Z\simeq\overline{\omega_M}.$
\end{proof}

 \subsection{Lagrangian intersections}
We now establish the twistor-family version of the lagrangian intersection theorem \cite[Thm.~2.9]{PTVV13}.
Consider a morphism of relative analytic prestacks with a Lagrangian structure $f:L\to Z$. It is said to have a \emph{real structure} if the following commutes
\[\begin{tikzcd} L \arrow[r,"f"] \arrow[d,"\rho_L"'] & Z \arrow[d,"\rho_Z"] \\ \sigma^*L \arrow[r,"\sigma^*f"] & \sigma^*Z . \end{tikzcd}\]
Following the standard Definition \ref{LagStructure} of a Lagrangian structure on a morphism of prestacks, it is convenient to enunciate the current setting with the following similar notion.
\begin{definition}
\label{defn:HyperLagStructure}
Let $\eta_Z:Z\to \mathbb{P}_{\mathbb{C}}^1$ be an $n$-shifted twistor family of hyperkahler-type. Consider a morphism $f:L\to Z$ of relative analytic prestacks. A \emph{real Lagrangian structure on $f$} is the datum of a compatible real structure $(L,\sigma_L)$ on $L$, and a $\sigma$-equivariant twisted isotropic structure
$\gamma: 0\sim f^*\omega,$
in $\mathcal{A}^{2,cl}(L/\mathbb{P}^1,n)^{\sigma},$ such that 
$$\Theta_{\gamma}:\mathbb{T}_{L/\mathbb{P}^1}^{\an}\xrightarrow{\simeq} \mathbb{L}_{L/Z}^{\an}[n-1]\otimes (f^*\circ\eta_Z^*)\EuScript{O}_{\bfP}(2),$$
is a quasi-isomorphism.
Dropping the real structure, we simply call this a \emph{Lagrangian structure}.
\end{definition}
We are mainly interested in Lagrangian structures, as a condition more closely related with classical hyper-Lagrangian condition on complex submanifolds in a K\"ahler manifold \cite{leung2007hyper}.

\begin{proposition}
\label{LagSectProp}
    Let $\eta:Z\to \mathbb{P}_{\mathbb{C}}^1$ be an $n$-shifted twistor family of hyperkahler type. 
    \begin{itemize}
        \item[1.] The derived stack of sections $\Sect(Z/\mathbb{P}_{\mathbb{C}}^1)$ is canonically $(n-1)$-shifted symplectic.
        \item[2.] Let $f:L\to Z$ be equipped with a relative Lagrangian structure (in the sense of Defn.\ref{defn:HyperLagStructure}). Then the induced map 
    \begin{equation}
        \label{RelLagSects}
f_*:\Sect(L/\mathbb{P}_{\mathbb{C}}^1)\to \Sect(Z/\mathbb{P}_{\mathbb{C}}^1),
\end{equation}
has a Lagrangian structure.
\item[3.] The homotopy-fixed points $\Sect(Z/\mathbb{P}_{\mathbb{C}}^1)^{hC_2}$ inherit the symplectic structure and preserve the relative Lagrangian structure on \eqref{RelLagSects}.
\end{itemize}
\end{proposition}
\begin{proof}
(1.) The result follows from \cite[Thm.~2.35]{calaque2024shifted} using the twisted $K_{\mathbb{P}^1}$-orientation since the evaluation map $\Sect(Z/\mathbb{P}_{\mathbb{C}}^1)\times\mathbb{P}^1\to Z$, induces the required form since $\mathcal{O}_{\mathbb{P}_{\mathbb{C}}^1}(2)\otimes K_{\mathbb{P}_{\mathbb{C}}^1}$ is the trivial bundle.
(2.) Follows from \cite[Prop.~2.4 (ii)]{GinRoz}. (3.) is clear since homotopy fixed points are a right-adjoint and the $C_2$-action is symplectic.
\end{proof}
\begin{remark}
    A more general statement is true: suppose $X\to \mathbb{P}^1$ is a $d$-oriented $\mathcal{O}$-compact derived stack with relative $K_{\mathbb{P}^1}$-orientation $[X].$ Let $f:L\to Z$ be a morphism of relative prestacks over $\mathbb{P}^1$ with a Lagrangian structure. Then the induced morphism 
    $F:=f_*:\Map_{/\mathbb{P}^1}(X,L)\to \Map_{/\mathbb{P}^1}(X,Z),$
    has the structure of a Lagrangian morphism.
\end{remark}
We now prove a twistor-family analogue of the 
derived intersection theorem for shifted Lagrangian structures.

\begin{theorem}
\label{MainTheorem3Body}
Let $\eta:Z\to \bfP$ be an $n$-shifted analytic twistor family of hk-type, with real involution $\sigma_Z.$ Let $f_1:L_1\to Z$ and $f_2:L_2\to Z$ be morphisms of relative analytic stacks with perfect relative cotangent complexes. Suppose $f_1,f_2$ have a Lagrangian structure. Then their homotopy fiber product
$L:=L_1\times_{Z}^h L_2 \to \bfP,$
taken in $\mathsf{dAnSt}_{/\bfP},$ has the the structure of a $(n-1)$-shifted derived analytic twistor family of hk-type.
\end{theorem}

\begin{proof}
The statement is clear from a relative version of the proof of \cite[Thm.2.9]{PTVV13} using the properties of the analytic cotangent complex in Proposition \ref{AnCotangentProperties}. Namely, consider the cartesian diagram in $\mathsf{dAnSt}_{\bfP}$,
\[
\begin{tikzcd} L \arrow[r,"p_2"] \arrow[d,"p_1"] & L_2 \arrow[d, "f_2"] \\ L_1 \arrow[r, "f_1"] & Z. \end{tikzcd}\]
Write $\pi:L\to Z$ and $\eta_{L}:L\to \bfP.$ Since $f_1,f_2$ are invariant-isotropic structures, there are homotopies 
$h_1:f_1^*\omega_Z\sim 0,h_2:f_2^*\omega_Z\sim 0,$
whose induced morphisms
$$\mathbb{T}_{L_1/\bfP}^{\an}\to \mathbb{L}_{L_1/Z}^{\an}[n-1]\otimes (\eta_Z\circ f_1)^*\EuScript{O}_{\bfP}(2) ,\hspace{3mm}\mathbb{T}_{L_2/\bfP}^{\an}\to \mathbb{L}_{L_2/Z}^{\an}[n-1]\otimes (\eta_Z\circ f_2)^*\EuScript{O}_{\bfP}(2) ,$$
are equivalences. Via the corresponding fiber sequences for pullbacks, there is a commutative diagram with exact rows
    \begin{equation}
\begin{tikzcd}
\mathbb{T}_{L/\bfP}^{\mathrm{an}} \arrow[r,"\alpha"] \arrow[d,"\Theta_{R}"] & p_1^*\mathbb{T}_{L_1/\bfP}^{\mathrm{an}} \oplus p_2^*\mathbb{T}_{L_2/\bfP}^{\mathrm{an}} \arrow[r,"\beta"] \arrow[d,"\simeq"] & \pi^*\mathbb{T}_{Z/\bfP}^{\mathrm{an}} \arrow[d,"\simeq"] 
\\ 
\mathbb{L}_{L/\bfP}^{\mathrm{an}}[n-1] \otimes \pi_{L}^*\EuScript{O}_{\bfP}(2) \arrow[r] & \big( p_1^*\mathbb{L}_{L_1/Z}^{\mathrm{an}}[n-1] \oplus p_2^*\mathbb{L}_{L_2/Z}^{\mathrm{an}}[n-1] \big) \otimes \pi^*\EuScript{O}_{\bfP}(2) \arrow[r] & \pi^*\big( \mathbb{L}_{Z/\bfP}^{\mathrm{an}}[n] \otimes \eta_Z^*\EuScript{O}_{\bfP}(2)  \big) 
\end{tikzcd}
\end{equation}
As the second and third vertical arrows are isomorphismns, so is the left-most; there is an equivalence $\mathbb{T}_{L_{}/\bfP}^{\an}\simeq \mathbb{L}_{L/Z}^{\an}[n-1]\otimes \pi^*\EuScript{O}_{\bfP}(2).$

We now turn to the twistor property \eqref{eqn: Twistor Diag}. 
Since the $C_2$-action is symplectic, $L$ has an induced $C_2$-action with invariant $2$-form. The analytic stack of sections $\mathbb{R}\underline{\mathrm{AnSec}}(L/\bfP)$ is a Weil-restriction, thus is limit-preserving as an $\infty$-right adjoint. Hence, there is an equivalence
\begin{equation}
\label{ASectLagInt}
\ASect(L_1\times_Z^hL_2/\bfP)\simeq \ASect(L_1/\bfP)\times_{\ASect(Z/\bfP)}^h\ASect(L_2/\bfP),\end{equation}
of derived analytic stacks. Note the left-hand side has a natural $(n-2)$-shifted symplectic structure by Proposition \ref{LagSectProp} (1.), while \eqref{ASectLagInt} is itself a relative derived Lagrangian intersection by Proposition \ref{LagSectProp} (2.), and thus has an $(n-2)$-shifted symplectic structure by \cite[Thm.2.9]{PTVV13}. Moreover, from \eqref{eqn: HoFP as Res}, homotopy fixed-points $(-)^{hC_2}$ also preserves limits, and so the stack of $\sigma$-invariant analytic sections of $\pi:L\to \bfP$ is equivalent to
$$\ASect(L_1/\bfP)^{hC_2}\times_{\ASect(Z/\bfP)^{hC_2}}^h\ASect(L_2/\bfP)^{hC_2}.$$

By assumption that $Z$ is twistor, there exists $\mathcal{X}_{Z}\subset t_0(\ASect(Z/\bfP)^{hC_2})$ and an open substack $U_Z\subset t_0(Z)$ admitting a good moduli $p_Z:U_Z\to \mathrm{Tw}(M_Z)\simeq M_Z\times\bfP,$ and a $C^{\infty}$-equivalence $\Phi_{Z}:\mathcal{X}_{Z}\times \bfP\simeq U_Z.$
Consider the induced morphisms \eqref{RelLagSects} on truncations by the Lagrangian morphisms in Proposition \ref{LagSectProp}:
$$t_0(f_{i,*}):t_0(\ASect(L_i/\bfP)^{hC_2})\to t_0(\ASect(Z/\bfP)^{hC_2}),i=1,2.$$
Let $\mathcal{X}_{L_i},i=1,2$ be the union of connect components maping to $\mathcal{X}_Z$ and put $\mathcal{X}_{L}:=\mathcal{X}_{L_1}\times_{\mathcal{X}_{Z}}\mathcal{X}_{L_2}.$ Preservation of limits under truncation gives that $\mathcal{X}_{L}$ is a component of $t_0(\ASect(L/\bfP)^{hC_2}),$ and since $U_Z\subset t_0(Z)$ is open, $U_{L_i}:=t_0(f_i)^{-1}(U_Z)\subset t_0(L_i),i=1,2$ are open. There is an induced map $p_{L_i}:U_{L_i}\to \mathrm{Tw}(M_{L_i})\simeq M_{L_i}\times\bfP,$ such that $\Phi_{L_i}:\mathcal{X}_{L_i}\times \bfP\to U_{L_i},i=1,2$ are $C^{\infty}$-equivalences. Set $U_L:=U_{L_1}\times_{U_Z}U_{L_2}$. Then it is open in $t_0(L_1)\times_{t_0(Z)}t_0(L_2)\simeq t_0(L)$, since the maps are continuous. Thus $U_L\subset t_0(L)$ is open. 
By \cite[Prop.~7.9]{alper2013good}, the gluing criteria states that $U_L$ is a good moduli for $(M_{L_1}\times\bfP)\times_{(M_Z\times\bfP)}(M_{L_2}\times \bfP)\simeq (M_{L_1}\times_{M_Z}M_{L_2})\times\bfP,$ and we put $M_L:=M_{L_1}\times_{M_Z}M_{L_2}.$
Then, there are cartesian diagrams of evaluation maps and of the open substacks,
\[
\begin{tikzcd} \mathcal X_L\times \mathbb P^1 \arrow[r] \arrow[d] & \mathcal X_{L_1}\times \mathbb P^1 \arrow[d,"\Phi_{L_1}"] \\ \mathcal X_{L_2}\times \mathbb P^1 \arrow[r,"\Phi_{L_2}"] & \mathcal X_Z\times \mathbb P^1 \end{tikzcd}\hspace{2mm} \begin{tikzcd} U_L \arrow[r] \arrow[d] & U_{L_1} \arrow[d,"p_{L_1}"] \\ U_{L_2} \arrow[r,"p_{L_2}"] & U_Z. \end{tikzcd}
\]
Explicitly, due to \eqref{ASectLagInt}, for each $s\in\mathcal{X}_{L}$ we may write $s=(s_1,s_2),s_i\in\mathcal{X}_{L_i},i=1,2$ and their images agree in $\mathcal{X}_Z,$ so evaluation at $\lambda \in\bfP,$ gives 
$\mathrm{ev}_{L}(s,\lambda):=\big(\mathrm{ev}_{L_1}(s_1,\lambda),\mathrm{ev}_{L_2}(s_2,\lambda)\big)\in U_{L_1}\times_{U_Z}U_{L_2},$ which we define to be $U_{L}.$ This defines $\Phi_{L}:\mathcal{X}_{L}\times\bfP\to U_L$ and by the universal property of fiber products, is the uniquely determined $C^{\infty}$-equivalence via 
$\mathcal{X}_{L}\times \bfP\simeq (\mathcal{X}_{L_1}\times_{\mathcal{X}_{Z}}\mathcal{X}_{L_2})\times\bfP\to U_{L_1}\times_{U_Z}U_{L_2}=U_L,$ as required.
\end{proof}

An important technique for constructing new hyperk\"ahler manifolds from old ones is hyperk\"ahler
reduction. Indeed, many of the interesting examples of hyperk\"ahler manifolds arise from the moduli spaces of solutions to
the anti-self-dual instanton equations, and thus can be viewed as arising from infinite-dimensional
hyperk\"ahler reduction. In the next subsection, we discuss this in the context of derived symplectic reduction \cite{Safronov2016}, using Theorem \ref{MainTheorem3Body}.
\subsection{Derived hyperk\"ahler reduction}
\label{ssec: SymplecticReduction} 
We begin by recalling symplectic reduction in derived algebraic geometry.
\subsubsection{Symplectic reduction}
Let $G$ be a reductive algebraic group with classifying stack $\mathsf{B}G.$ For an integer $n$, the shifted cotangent stack $\mathsf{T}^*[n]\mathsf{B}G$ can be identified with the shifted coadjoint quotient stack $\mathfrak{g}^*[n-1]/G,$ where $\mathfrak{g}$ is the Lie algebra of $G.$ In other words, for any $n\in \mathbb{Z},$
there is an equivalence 
$\mathsf{T}^*[n+1]\mathsf{B}G\simeq \big[\mathfrak{g}^*[n]/G\big],$
of $(n+1)$-shifted symplectic stacks.
\begin{remark}
    The Lie algebra $\mathfrak{g}$ of an affine group $G$ admits adjoint and 
coadjoint actions, $G\times \mathfrak{g}\to \mathfrak{g}$ and $G\times\mathfrak{g}^*\to \mathfrak{g}^*.$
They are the natural actions by $G$ on the tangent and cotangent
complexes of $\mathsf{B}G.$
Under this identification, the tautological $1$-form $\theta$ of degree $n$ on the (shifted) cotangent stack is $\theta=\sum_i x_i\otimes \xi^i,$ for a basis $(x_i)$ of $\mathfrak{g}$ and dual basis $(\xi^i).$ Note that $(x_i)$ provide a choice of linear coordiantes on $\mathfrak{g}^*.$ The symplectic form is thus $\omega=\sum_i (d_{\mathrm{DR}}x_i)\otimes \xi^i.$ 
\end{remark}
The canonical morphisms,
$$\mathsf{B}G\to \mathsf{T}^*[n]\mathsf{B}G\simeq\mathfrak{g}^*[n-1]/G,\hspace{3mm} \mathfrak{g}^*[n-1]\to \mathfrak{g}^*[n-1]/G,$$
given by the zero sectio and the shifted conormal map $\mathsf{N}^*[n](\mathrm{pt}\to \mathsf{B}G),$ both have canonical Lagrangian structures.
\begin{definition}
    \label{defn: k-shiftedmomentmap}
Let $Z$ be a geometric derived algebraic stack with $n$-shifted symplectic structure, equipped with a $G$-action and a $G$-equivariant map 
$\mu:Z\to \mathfrak{g}^*[n].$ The map $\mu$ is said to have an
\emph{$n$-shifted moment map structure} if the following conditions are satisfied:
\begin{enumerate}
    \item There is a Lagrangian structure on the induced map $[\mu]:[Z/G]\to \big[\mathfrak{g}^*[n]/G\big]$,
    \item There is an equivalence 
    \begin{equation}
        \label{SympRed1}
        Z\simeq \big[Z/G\big]\times_{[\mathfrak{g}^*[n]/G]}^h\mathfrak{g}^*[n],
    \end{equation}
    of $n$-shifted symplectic derived stacks.
\end{enumerate}
\end{definition}
Given the structure of a $n$-shifted moment map $\mu:Z\to \mathfrak{g}^*[n]$ on a derived stack $Z$, the \emph{derived symplectic reduction} of $Z$ along $\mu$ is the homotopy fiber product,
\begin{equation}
\label{SymRedPb}
\begin{tikzcd}
Z_{\mathrm{Red}}\arrow[d]\arrow[r] & \mathsf{B}G\arrow[d]
    \\
\big[Z/G\big]\arrow[r,"\mu"] & \big[\mathfrak{g}^*[n]/G\big].
\end{tikzcd}
\end{equation}
This is a Lagrangian intersection
\begin{equation}
    \label{SympRed2}
Z_{\mathrm{Red}}:=[Z/G]\times_{[\mathfrak{g}^*[n]/G]}^h\mathsf{B}G,
\end{equation}
inside a $(n+1)$-shifted symplectic derived stack, therefore the derived symplectic reduction \eqref{SympRed2} is canonically an $n$-shifted symplectic derived stack.

\subsubsection{Symplectic reduction for hk-families}
We now give a \emph{algebro-geometric model} for the hyperk\"ahler reduction of moduli spaces which appear in the context of infinite-dimensional gauge theory.
\begin{remark}
It is only a model as we do not work in the setting of derived differential geometry. See \cite[\S ~5.1]{GinRoz} and \cite[\S.~5.4]{AnelCalaque2022} for details, and \cite[Obs.~8.1]{KSY2} for related remarks.
\end{remark}

Consider a groupoid object in $\mathsf{dAnSt}_{\mathbb{P}_{\mathbb{C}}^1}.$ We need to describe behaviour of face and degeneracy maps with real structures. Write the corresponding simplicial diagram with $C_2$-equivariant face maps $d_i^n:G^{\times n}\times Z\to G^{\times (n-1)}\times Z$
\[\cdots \;\substack{\longrightarrow\\[-0.3em]
               \longrightarrow\\[-0.3em]
               \longrightarrow\\[-0.3em]\longrightarrow}\; G\times G\times X  \;\substack{\longrightarrow\\[-0.3em]
               \longrightarrow\\[-0.3em]
               \longrightarrow}\; G\times X \;\substack{\longrightarrow\\[-0.3em]
               \longrightarrow}\; X.\]
 Let $[Z/G]\simeq \mathrm{hocolim}_{\Delta^{op}\ni [n]}B_{\bullet}(G,Z)$ be the corresponding colimit, taken in $\mathsf{dAnSt}_{\mathbb{P}_{\mathbb{C}}^1}[C_2].$

Let $f:Z\to W$ be a $G$-equivariant morphism of relative derived analytic stacks. Suppose $\sigma_G:G\to G^{\mathrm{conj}}$ has a real stucture e.g. $\sigma_g\circ m\simeq m^{\mathrm{conj}}\circ (\sigma_G\times \sigma_G)$ (i.e. $G$ is a group object with involution). The $G$-action on $Z$ is \emph{compatible with real structures} if
\[
\begin{tikzcd}
    G\times Z\arrow[d,"\sigma_G\times \sigma_Z"] \arrow[r,"act_Z"] & Z\arrow[d,"\sigma_Z"]
    \\
    G^{\mathrm{conj}}\times Z^{\mathrm{conj}}\arrow[r,"act_Z^{\mathrm{conj}}"] & Z^{\mathrm{conj}}
\end{tikzcd}
\]
is commutative. 
If $f:(Z,\sigma_Z)\to (W,\sigma_W)$ is a $G$-equivariant morphism of relative derived analytic stacks, it is said to be \emph{compatible with real structures} if $\sigma_W\circ f\simeq f^{\mathrm{conj}}\circ \sigma_Z.$

We now describe the quotient stack and its induced compatible real structure.
\subsubsection{$G$-equivariant real structures}
In this subsection we give a technical result supplying quotient stacks by $G$-actions with induced real structures.
\begin{proposition}
\label{RealOnQuotients}
Let $(Z\to \mathbb{P}^1)$ be a relative derived analytic stack with real structure $\sigma_Z.$ Suppose $Z$ admits a $G$-action which is compatible with real structures. Then $Z_{\bullet}:\Delta^{op}\to \mathsf{dAnSt}_{\mathbb{P}^1}[C_2]$ is a simplicial object in $C_2$-derived analytic stacks and there exists an induced real structure on geometric realizations,
$$\sigma_{[Z/G]/\mathbb{P}^1}:|Z_{\bullet}|\to \big|Z_{\bullet}^{\mathrm{conj}}\big|\simeq \big|Z_{\bullet}\big|^{\mathrm{conj}}.$$
That is, for any $Z\to \mathbb{P}^1$ with a real $G$-action, the relative derived analytic quotient stack 
$[Z/G]\to \mathbb{P}^1,$
inherits a canonical real structure.
\end{proposition}
\begin{proof}
Let $\sigma_Z,\sigma_G$ denote the real structures. By definition, $\sigma_Z\circ a_{\mathbb{P}^1}\simeq a_{\mathbb{P}^1}\circ (\sigma_G\times \sigma_Z).$
Define
$$\sigma^{(k)}:G^{\times k}\times Z\to G^{\times k}\times Z,\hspace{2mm} \sigma^{(k)}(g_1,\ldots,g_k;z):=\big(\sigma_Gg_1,\ldots,\sigma_G g_k;\sigma_Zz\big).$$
It is compatible with face maps $d_0^k:$ as
$$\sigma^{(k-1)}\circ d_0^k(g_1,\ldots,g_k,z)=\sigma^{(k-1)}(g_2,\ldots,g_k;g_1^{-1}z)=\big(\sigma_Gg_2,\ldots,s\sigma_Gg_k;\sigma_Zg_1^{-1}z).$$
Then, since
\begin{eqnarray*}
    d_0^k\circ \sigma^{(k)}(g_1,\ldots,g_k;z)&=&
    d_0^k\big(\sigma_G g_1,\ldots,\sigma_G g_k;\sigma_Z z\big)
    \\
    &=&\big(\sigma_G g_1,\ldots,\sigma_G g_k;\sigma_Gg_1^{-1}\sigma_Zz\big),
\end{eqnarray*}
since $\sigma_Z(g^{-1}\cdot z)\simeq \sigma_G(g_1)^{-1}\cdot \sigma_Z(z),$ by equivariance there are homotopies $\sigma^{(k-1)}\circ d_0^k\simeq d_0^k\simeq \sigma^{(k)}.$
Now, we have $\sigma^{(\bullet)}:=\{\sigma^{(k)}|[k]\in \Delta^{op}\},$ defines a map of simplicial objects in derived analytic stacks,
$$\sigma^{(\bullet)}:B_{\bullet}(G,Z)\to B_{\bullet}(G,Z),$$
that moreover covers the antipodal map $a_{\mathbb{P}^1}:\mathbb{P}^1\to \mathbb{P}^1.$
Taking the homotopy colimit we obtain a map
$$\sigma_{[Z/G]}:=\mathrm{hocolim}_{[k]\in \Delta^{op}}\sigma^{(k)}:[Z/G]\to [Z/G],$$
which is involutive since
$\sigma^{(k)}\circ \sigma^{(k)}(g_1,\ldots,g_k;z)=\big(\sigma_G^2g_1,\ldots,\sigma_G^2g_k;\sigma_Z^2z\big)\simeq (g_1,\ldots,g_k,z)$ as $\sigma_G\circ \sigma_G\simeq \mathrm{id}_G$ and $\sigma_Z\circ \sigma_Z\simeq \mathrm{id}_Z.$
In particular, $\mathsf{dAnSt}_{\mathbb{P}^1}[C_2]$ has colimites and $|Z_{\bullet}|\simeq \underset{\Delta^{op}}{\mathrm{colim}}Z_{\bullet},$ where $\Delta$ is given the $C_2$-structure as in Proposition \ref{C2Waldhausen}. Then, 
$\underset{\Delta^{op}}{\mathrm{colim}}Z_{\bullet}^{\mathrm{conj}}\simeq (\underset{\Delta^{op}}{\mathrm{colim}}Z_{\bullet})^{\mathrm{conj}},$ and we have the canonical map
$$\sigma_{[Z/G],\mathbb{P}^1}:=\underset{\Delta^{op}}{\mathrm{colim}}\sigma^{(\bullet)}:\underset{\Delta^{op}}{\mathrm{colim}}Z_{\bullet}\to \big(\underset{\Delta^{op}}{\mathrm{colim}}Z_{\bullet}\big)^{\mathrm{conj}},\hspace{2mm} \in \mathsf{Funct}(\Delta^{op},\mathsf{dSt}_{\mathbb{P}^1}[C_2]),$$
since limits are computed objectwise in $\mathsf{Funct}(\Delta^{op},\mathsf{dSt}_{\mathbb{P}^1}[C_2])$ and $(-)^{\mathrm{conj}}$ commutes with limits.
The quotient construction is denoted by 
$$[\![-/G]\!]:\mathsf{dAnSt}^G \xrightarrow{\sim} \mathsf{dAnSt}_{\mathsf{B}G},$$
which is an equivalence of $\infty$-categories whose action on objects is computed by
$$[\![Z/G]\!]\simeq \underset{[n]\in\Delta^{\mathrm{op}}}{\mathrm{hocolim}}\hspace{1mm}Z_n.$$
In the relative setting, for the $\infty$-category \eqref{C2dAnStIntro}, there is a quotient map over simplicial diagrams \eqref{nSimplexRelativeIntro}, denoted the same $[\![Z/G]\!].$
\end{proof}

\subsubsection{Reduction}
We now use the results of the previous two subsections to prove an analogue of symplectic reduction for
$n$-shifted derived twistor families $\eta_Z:Z\to \mathbb{P}^1.$ It can also be taken as the definition of a $n$-shifted moment map structure on a derived twistor family.

We will phrase the result for derived algebraic stacks locally almost of finite presentation over $\mathbb{C},$ obtaining the analyic version by applying $(-)^{\an}.$ In this case, it is important to emphasize that from \cite{PortaYuHigherGaga2016}, if $G$ is a reductive algebraic group over $\mathbb{C}$, then $\mathsf{B}G$ satisfies the GAGA property \cite[Prop.~5.33]{HolsteinPorta2025}, even though it is not proper as in \cite[Def.~4.8]{PortaYuHigherGaga2016}.

\begin{theorem}
\label{MainTheorem4Body}
Let $(Z\xrightarrow{\eta_Z}\mathbb{P}^1,\omega_Z)$ be an $n$-shifted derived pretwistor family of hyperk\"ahler type with fiber-wise $G$-action. Let $\mathfrak{g}^*[n]_{\mathbb{P}^1}:=\mathfrak{g}^*[n]\times \mathbb{P}^1,$ which we understood to be endowed with $\mathfrak{g}^*[n]\otimes \EuScript{O}(2).$
Assume there exists a $G$-equivariant morphism of derived stacks over $\mathbb{P}^1$
$\mu:Z\to \mathfrak{g}^*[n]_{\mathbb{P}^1}$ which satisfies the following: 
\begin{itemize}
\item[(i)] $\mu_{red}:[Z/G]\to [\mathfrak{g}^*[n]/G]\times\mathbb{P}^1$ has an  $\EuScript{O}_{\mathbb{P}^1}(2)$-twisted relative Lagrangian structure,
\item[(ii)] There is an equivalence 
$[Z/G]\times_{\mathfrak{g}^*[n]/G\times\mathbb{P}^1}^h\big(\mathfrak{g}^*[n]\times\mathbb{P}_{\mathbb{C}}^1)\simeq Z$
of $\mathbb{P}^1$-derived stacks compatible with the $\EuScript{O}(2)$-twisted relative shifted-symplectic structure and $\mathrm{Gal}(\mathbb{C}/\mathbb{R})$-action.
\end{itemize}
Then the derived $\mathbb{P}^1$-stack
\begin{equation}
    \label{eqn: DerHypRed1}
    \eta_{Z_{HK}}:Z_{\mathrm{HK}-\mathrm{red}}:=[\![Z/G]\!]\times_{\mathfrak{g}^*[n]/G\times\mathbb{P}_{\mathbb{C}}^1}^h\big(\mathsf{B}G\times\mathbb{P}_{\mathbb{C}}^1\big)\to \mathbb{P}_{\mathbb{C}}^1 
\end{equation}
is canonically equipped with the structure of an $n$-shifted derived pre-twistor family of hyperk\"ahler-type, which is moreover compatible with analytification.
\end{theorem}

\begin{remark}
Since the symplectic reduction is in particular a derived Lagrangian intersection, which by Theorem \ref{MainTheorem3Body} has a natural derived twistor structure, we are motivated by Definition \ref{defn: k-shiftedmomentmap} to call the map $\mu$ satisfying (i) and (ii) a \emph{$n$-shifted moment map of twistor-type}. Similarly, in analogy with \eqref{SympRed2}, we refer to the corresponding relative derived stack \eqref{eqn: DerHypRed1} as the \emph{derived hyperk\"ahler reduction}.
\end{remark}

\begin{proof}[Proof of Theorem \ref{MainTheorem4Body}]
 By definition $Z_{\mathrm{HK}-\mathrm{red}}:=(Z/\mathbb{P}_{\mathbb{C}}^1)_{\mathsf{hk}-\mathrm{red}}=[Z/G]\times_{[\mathfrak{g}^*[n]_{\EuScript{O}_{\mathbb{P}^1}(2)}/G]}^h(\mathsf{B}G\times\mathbb{P}_{\mathbb{C}}^1),$ whose first factor $[Z/G]$ carries an induced real structure by Proposition \ref{RealOnQuotients}. Note that the real structure on $Z$ is $G$-equivariant and the canonical map $0:\mathsf{B}G\times\mathbb{P}_{\mathbb{C}}^1\to \big[\mathfrak{g}^*[n]_{\EuScript{O}(2)}/G\big]$ is equivariant for the real structure.
We need to check this map has the structure of a Lagrangian morphism, in the sense of Definition \ref{defn:HyperLagStructure}.
Indeed by assumption, the $\EuScript{O}(2)$-twisted relative moment map denoted
\[
\begin{tikzcd}
    Z\arrow[dr]\arrow[rr,"\mu"] & &  \mathfrak{g}^*[n]\times\mathbb{P}_{\mathbb{C}}^1\arrow[dl]
    \\
    & \mathbb{P}_{\mathbb{C}}^1 &
\end{tikzcd}\]
is $G$-equivariant relative to $\mathbb{P}_{\mathbb{C}}^1.$
For the $G$-action $G\times (\mathfrak{g}^*[n]\times \mathbb{P}^1)\to \mathfrak{g}^*[n]\times \mathbb{P}^1,$ the associated Segal groupoid object has $m$-simplices $G^m\times \mathfrak{g}^*[n]\times \mathbb{P}^1,$ and the fiber-wise action induces an equivalence of $\mathbb{P}^1$-stacks $[\mathfrak{g}^*[n]_{\mathbb{P}^1}/G]\simeq [\mathfrak{g}^*[n]/G]\times\mathbb{P}^1.$ By Proposition \ref{AnCotangentProperties} and Proposition \ref{AnDeRhamAlgs}, we obtain a morphism of fiber sequences of relative cotangent complexes
\begin{equation}
    \label{HKReductionAn}
\begin{tikzcd}
    (\mathbb{L}_{Z_{hk-red}/\mathbb{P}^1})^{\an}\arrow[d]\arrow[r] & \big(\mathbb{L}_{[Z/G]/\mathbb{P}^1}\oplus \mathbb{L}_{\mathsf{B}G\times\mathbb{P}^1/\mathbb{P}^1}\big)^{\an} \arrow[d]\arrow[r] & (\mathbb{L}_{[\mathfrak{g}^*[n]/G]/\mathbb{P}^1})^{\an}\arrow[d]
    \\
    \mathbb{L}^{\an}_{Z_{hk-red}^{\an}/\bfP}\arrow[r] & \mathbb{L}_{[Z^{\an}/G^{\an}]/\bfP}^{\an}\oplus \mathbb{L}^{\an}_{\mathsf{B}G^{\an}\times \bfP/\bfP}\arrow[r] & \mathbb{L}^{\an}_{[\mathfrak{g}^*[n]_{\bfP}/G^{\an}]/\bfP},
\end{tikzcd}\end{equation}
where we have used that analytification is a left-adjoint thus commutes with colimits. In particular, $\mathsf{B}G^{\an}\simeq \mathsf{B}(G^{\an}).$ In \eqref{HKReductionAn} we have omitted the pull-backs of the relative cotangent complexes in the middle column, and note that the equivalence constructed in the proof of Proposition \ref{AnDeRhamAlgs} implies all vertical maps in the diagram are equivalences. Furthermore, by essentially the same argument as in Proposition \ref{PropsAn} (1), we obtain $(\mathsf{T}^*[n+1]\mathsf{B}G)^{\an}\simeq (\mathsf{T}^*)^{\an}[n+1]\mathsf{B}(G^{\an}),$ which holds for shifted $\EuScript{O}(2)$-twisted cotangent bundles \eqref{TwistedCotangent}.
The claim follows by noting that $\mathbb{L}_{\mathsf{B}G}\simeq \mathfrak{g}^*[-1]$ and $\mathbb{L}_{[\mathfrak{g}^*[n]/G]}\simeq\mathfrak{g}^*[-n-1],$ all twisted by $\EuScript{O}_{\mathbb{P}^1}(2),$ and using
Theorem \ref{MainTheorem3Body} which states the homotopy fiber product \eqref{eqn: DerHypRed1} is a Lagrangian intersection and is thus $n$-shifted (pre)twistor.
\end{proof}

\section{Deligne shape and the Hodge stack}
For any derived (pre)stack $X$ with a nil-isomorphism $X\rightarrow Y$ where $Y\in \mathrm{PSt}_{\mathrm{laft-def}}$, following \cite[§ 5.1.3]{GR17b}, there is a formal moduli problem under $X$. We are interested in the formal moduli problem $X\to X_{\mathrm{DR}}$, and the associated deformation to the normal bundle construction \cite[Sect.~9.2.4]{GR17b}, and the associated derived stacks of perfect complexes, and coherent complexes on it.
\subsection{Filtered, graded prestacks}
Put $\mathsf{Perf}^{\mathrm{filt}}:=\mathsf{Perf}\big([\mathbb{A}_{\mathbb{C}}^1/\mathbb{G}_m]\big)$ and $\mathsf{Perf}^{\mathrm{gr}}:=\mathsf{Perf}(\mathsf{B}\mathbb{G}_m).$
Given a derived prestack $Z$, put
\begin{equation}
\label{eqn:FiltGrPst}
\mathsf{Filt}(Z):=\Map([\mathbb{A}^1/\mathbb{G}_m],Z),\hspace{5mm}\mathsf{Grad}(Z):=\Map(\mathsf{B}\mathbb{G}_m,Z).
\end{equation}
The canonical map $\mathsf{B}\mathbb{G}_m\xrightarrow{j}[\mathbb{A}_{\mathbb{C}}^1/\mathbb{G}_m]$ induces $\infty$-functor $j^*:\mathsf{Perf}^{\mathrm{filt}}\to \mathsf{Perf}^{\mathrm{gr}}.$ 
Let $A$ be a connective commutative dg-algebra, then for any filtered complex $E\in \mathsf{Perf}(\mathsf{Spec}A\times [\mathbb{A}^1/\mathbb{G}_m]),$ there is a cartesian diagram,
\[
\begin{tikzcd}
  \mathbb{R}\Gamma\big([\mathbb{A}^1/\mathbb{G}_m],E\big)\arrow[d]\arrow[r] & \mathbb{R}\Gamma(*,1^*E)\arrow[d]
  \\
\mathbb{R}\Gamma\big([*,\mathbb{G}_m],0^*E\big)\arrow[r] & \mathbb{R}\Gamma\big([\mathbb{A}^1/\mathbb{G}_m],E^{\vee}\big)^{\vee}.
\end{tikzcd}
\]
This fact is used in \cite[Prop.~5.17]{KinjoParkSafranov2024} to prove that the diagram (of derived stacks),
\begin{equation}
\label{CoSpan}
[*/\mathbb{G}_m]\xrightarrow{0}[\mathbb{A}^1/\mathbb{G}_m]\xleftarrow{1} *
\end{equation}
defines a $0$-oriented cospan.
Note further that the morphisms $0,1:*\to [\mathbb{A}^1/\mathbb{G}_m],$ and $[*/\mathbb{G}_m]\xrightarrow{0}[\mathbb{A}^1/\mathbb{G}_m]$
induced morphisms of derived stacks, $\mathsf{Grad}(Z)\xleftarrow{\mathrm{gr}}\mathsf{Filt}(Z),$ and $\mathrm{ev}_0,\mathrm{ev}_1:\mathsf{Filt}(Z)\to Z.$ Write $\mathrm{tot}:\mathsf{Grad}(Z)\to Z.$ 

Assume now that we have a derived Artin stack $(Z,\omega)$ with a $n$-shifted symplectic structure. Then the  diagram
\begin{equation}
\label{CartTangentFiltDiagram}
\begin{tikzcd}
\mathbb{T}_{\mathsf{Filt}(Z)}\arrow[d]\arrow[r] & \mathrm{gr}^*\mathbb{T}_{\mathsf{Grad}(Z)}\arrow[d]
\\
\mathrm{ev}_1^*\mathbb{T}_Z\arrow[r] & \mathbb{L}_{\mathsf{Filt}(Z)}[n],
\end{tikzcd}
\end{equation}
is a cartesian diagram in $\mathsf{Perf}\big(\mathsf{Filt}(Z)\big).$ Denote the canonical evaluation map by 
$\mathrm{ev}:[\mathbb{A}^1/\mathbb{G}_m]\times \mathsf{Filt}(Z)\to Z$ and the canonical projection map by $q:[\mathbb{A}^1/\mathbb{G}_m]\times\mathsf{Filt}(Z)\to \mathsf{Filt}(Z).$ Then there are equivalences 
$\mathbb{T}_{\mathsf{Grad}(Z)}\simeq \mathrm{tot}^*(\mathbb{T}_Z)_0,$ where $(-)_0$ indicates the weight $0$-part of of the $\mathbb{G}_m$-action, and $\mathbb{T}_{\mathsf{Filt}(Z)}\simeq q_*\mathrm{ev}^*(\mathbb{T}_Z)$ (see e.g. \cite[Prop.~5.11]{KinjoParkSafranov2024}, \cite[Lem.~1.2.2]{Hal14}).

\subsection{Deformation to the normal bundle}
The idea we follow produces an $\mathbb{A}^1$-family of prestacks $X\times\mathbb{A}^1\rightarrow Y_{scaled}$ from an object $X\rightarrow Y\in \mathrm{FMP}_{X/}$ where $Y_{scaled}$ as constructed in \cite[Ch.~9]{GR17b}. It is left-lax equivariant with respect to $\mathbb{A}^1$ acting by multiplication on itself. Equivariance with respect to $\mathbb{G}_m\subset \mathbb{A}^1$ implies all fibers $Y_{t}$ of $Y_{scaled}$ at $t\neq 0\in \mathbb{A}^1$ are canonically isomorphic to $X.$ The fiber at $0\in \mathbb{A}^1$ identifies with the formal version of the total space of the normal bundle $\mathsf{T}(X/Y)[1]$ of $X$ in $Y$.

Considering the canonical maps
$0:\mathsf{B}\mathbb{G}_m\rightarrow \AG,$ and $1:Spec(k)=pt\rightarrow \AG,$
we can consider for a given object
$X\in \mathrm{PSt}$, its base-change along $0$, denoted by $Gr(X).$ Similarly, base-change along $1$ defines an underlying-stack object $Un(X).$

The formal deformation
to the normal bundle $\widehat{\mathcal{N}}_{X/Y}$ is a prestack over $[\mathbb{A}^1/\mathbb{G}_m]$ with generic fiber $Y$ and special fiber a formal moduli problem under $X$.
\begin{remark}
The latter is given by a colimit of split square-zero extensions by symmetric powers to the normal bundle i.e. $\mathrm{Sym}^n(\mathsf{T}_{X/Y}[1]),$ for each $n.$
\end{remark}
Denote the substacks $[\mathbb{G}_m/\mathbb{G}_m]\simeq *$ by $[1]$ and $[\{0\}/\mathbb{G}_m]\simeq \mathsf{B}\mathbb{G}_m$ by $[0].$

\begin{definition}
\label{Definition: Hodge Prestack}
The \emph{Hodge prestack} $X_{\mathrm{Hod}}$ of $X$ is the deformation to the normal bundle construction of the moduli problem $X\rightarrow X_{\mathrm{DR}}$, described by a prestack
$$\widehat{\mathcal{N}}_{X/X_{\mathrm{DR}}}\in \mathrm{PSt}_{/(\mathbb{A}^1/\mathbb{G}_m)}\simeq \mathrm{FMP}_{X\times \mathbb{A}^1/ / X_{\mathrm{DR}}\times\mathbb{A}^1}:=\big(\mathrm{FMP}_{/X_{\mathrm{DR}}\times \mathbb{A}^1}\big)_{X\times \mathbb{A}^1/}.$$
\end{definition}
In particular, $X_{\mathrm{Hod}}$ is a filtered prestack with generic fiber (fiber over $[1]$) given by $X_{\mathrm{Hod},[1]}\simeq X_{\mathrm{DR}}$ and special fiber (fiber over $[0]$) the formal moduli problem under $X\times[\mathbb{A}^1/\mathbb{G}_m]$ denoted
\begin{equation}
    \label{eqn: Hodge Structure Map}
i:X\times[\mathbb{A}^1/\mathbb{G}_m]\rightarrow X_{\mathrm{Hod}}.
\end{equation}
The special fiber is the {\it Dolbeault stack} $X_{\mathrm{Dol}}$ over $\mathsf{B}\mathbb{G}_m,$ with natural $\mathbb{G}_m$-action.
In other words, it is the Dolbeault degeneration of $X_{\mathrm{DR}}$ to (total space) of $\mathsf{T}[1]X.$

\subsection{Deformation theory}
\label{ssec: DefTheory}
We compute the tangent complexes for de Rham, Dolbeault and Hodge complexes, and perfect complexes. Some general results will be needed.
We make tacit use of Porta--Sala’s notion of a \emph{categorically proper morphism} between derived stacks \cite{PortaSala}. Such a morphism satisfies that perfect complexes pushforward to perfect complexes \cite[Prop 2.3.23]{PortaSala}. 
We will write $\sharp\in \big\{\mathrm{DR},\mathrm{Dol},\mathrm{B},\mathrm{Hod}\big\},$ to mean any of the shape operations. In particular, we also include the trivial shape $(-)_{\emptyset}$ for which $X_{\emptyset}=X.$

All maps $X_{\sharp}\to \mathbb{A}^1,$ are categorically proper \cite[Prop.3.1.1, Prop.4.1.1, Lemma.5.3.2]{PortaSala}. This means that pushforwards along projections $\pi_2$ of the canonical diagram
\begin{equation}
    \label{eqn: RPerf Diagram1}
        \begin{tikzcd}
    X_{\sharp}& \arrow[l,"\pi_1"] X_{\sharp}\times\Map(X_{\sharp},Z)\arrow[d,"\pi"]  \arrow[r,"\mathrm{ev}"] & Z
\\
&\Map(X_{\sharp},Z)&
\end{tikzcd}
\end{equation}
preserves the subcategories $\mathsf{Perf}(-)\subset\mathsf{QCoh}(-)$, of perfect complexes. 

In particular, for $X_{\emptyset}$, and $Z=\Perf$, the diagram \eqref{eqn: RPerf Diagram1} is
\begin{equation}
    \label{eqn: RPerf Diagram}
    \begin{tikzcd}
    X& \arrow[l,"q_X"] X\times \RPerf(X)\arrow[d,"q"] \arrow[r,"\mathrm{ev}"] & \RPerf(X)
    \\
    & \RPerf(X) &
    \end{tikzcd}
\end{equation}
For $\mathsf{Spec}(A)\in \mathsf{dAff}_k$ with $x:\mathsf{Spec}(A)\to \RPerf(X),$
set
\[
\begin{tikzcd}
    X_A\arrow[d]\arrow[r,"\mathrm{id}_X\times x"] & X\times \RPerf(X)\arrow[d,"q"]
    \\
    \mathsf{Spec}_k(A)\arrow[r,"x"] & \RPerf(X),
\end{tikzcd}
\]
where $X_A:=X\times \mathsf{Spec}_k(A)$ and $p_A:X_A\to \mathsf{Spec}_k(A).$
Note that $x^*\circ q_+\simeq p_{A+}(\mathrm{id}_X\times x)^*.$


\subsubsection{Symplectic structures on Higgs relative perfect complexes} 
We now prove there exists symplectic forms on relative Dolbeault mapping stacks.
Let $f:\mathcal{X}\to \mathcal{B}$ be a smooth family of varieties of relative dimension $d$. Let $\mathcal{X}_b:=\mathcal{X}\times_{\mathcal{B}}\{b\},$ for be a generic smooth fiber over $b\in \mathcal{B}.$ 
Denote by $\mathbb{T}_{\mathcal{X}/\mathcal{B}},$ the relative tangent complex, and set $\mathcal{X}_{\mathrm{Dol}/\mathcal{B}}:=(\mathcal{X}/\mathcal{B})_{\mathrm{Dol}},$ the relative Dolbeault stack over $\mathcal{B}$ with natural $\mathbb{G}_m$-action, and put 
\begin{equation}
    \label{eqn: Dolbeault maps}
\mathcal{X}_{\mathrm{Dol}/\mathcal{B}}\xrightarrow{q}\mathcal{X}\xrightarrow{r}\mathcal{B},
\end{equation}
with $p:\mathcal{X}_{\mathrm{Dol}/\mathcal{B}}\to \mathcal{B}.$
\begin{proposition}
\label{Prop:DolbeaultShiftedSympl}
The derived mapping stack $
\Map(\mathcal{X}_{\mathrm{Dol}/\mathcal{B}}, \Perf\times\mathcal{X})$
admits a canonical $(2-d)$-shifted symplectic structure relative to $\mathcal{B}$. More precisely, for any $A \in \mathsf{cdga}_\mathcal{B}^{\le 0}$ and $\mathcal{E}\in \Perf(\mathcal{X}_{\mathrm{Dol}/\mathcal{B}}\times_\mathcal{B} \mathsf{Spec} A)$, the canonical map
\[
-\cap \eta_A: p_{A*}\mathcal{E} \longrightarrow (p_{A*}(\mathcal{E}^\vee))^\vee[2-d]
\]
is a quasi-isomorphism of dg $A$-modules, where $p_A: \mathcal{X}_{\mathrm{Dol}/\mathcal{B}} \times_\mathcal{B} \mathsf{Spec}(A)\to \mathsf{Spec}(A)$ is the projection and $\eta_A$ is the pullback of a global $\mathcal{O}$-orientation $
    \eta:p_*\mathcal{O}_{\mathcal{X}_{\mathrm{Dol}/\mathcal{B}}} \longrightarrow \mathcal{O}_\mathcal{B}[-d],$
    constructed from the relative dualizing complex and Serre duality.
\end{proposition}

\begin{proof}
Let $q:\mathcal{X}_{\mathrm{Dol}/S}\to \mathcal{X}$ be the projection to the underlying family. By definition, $
\mathcal{O}_{\mathcal{X}_{\mathrm{Dol}/S}} \simeq \mathrm{Sym}_{\mathcal{O}_\mathcal{X}}^\bullet(\mathbb{T}_{\mathcal{X}/S}[1]),$ thus applying $q_*\equiv \mathbb{R}q_*$, we obtain a complex $
q_* \mathcal{O}_{\mathcal{X}_{\mathrm{Dol}/S}} \simeq \bigoplus_{i=0}^d \wedge^i \mathbb{T}_{\mathcal{X}/S}[-i],$
with each $\wedge^i \mathbb{T}_{\mathcal{X}/S}$ is perfect over $\mathcal{O}_\mathcal{X}$. Hence $q_* \mathcal{O}_{\mathcal{X}_{\mathrm{Dol}/S}}$ is perfect over $S$. This shows that $\mathcal{X}_{\mathrm{Dol}/S}$ is \emph{$\mathcal{O}$-compact relative to $S$}.

Consider $p:\mathcal{X}_{\mathrm{Dol}/S}\to S$. We need an isomorphism
\[
\eta:p_* \mathcal{O}_{\mathcal{X}_{\mathrm{Dol}/S}} \longrightarrow (Rp_* \mathcal{O}_{\mathcal{X}_{\mathrm{Dol}/S}})^\vee[-d].
\]
Since $Rp_* \mathcal{O}_{\mathcal{X}_{\mathrm{Dol}/S}} \simeq \bigoplus_{i=0}^d Rf_* \wedge^i \mathbb{T}_{\mathcal{X}/S}[-i],$ by Grothendieck–-Serre duality, we have for each $i$:
\[
(Rf_* \wedge^i \mathbb{T}_{\mathcal{X}/S}[-i])^\vee[-d] \simeq Rf_* (\wedge^{d-i} \Omega^1_{\mathcal{X}/S}[d-i]).
\]

In particular, the top wedge gives $H^0(Rp_* \mathcal{O}_{\mathcal{X}_{\mathrm{Dol}/S}}) \simeq \mathcal{O}_S$, and we choose a trivialization
\[
\eta:p_*\mathcal{O}_{\mathcal{X}_{\mathrm{Dol}/S}} \xrightarrow{\sim} \mathcal{O}_S[-d],
\]
defining the relative $\mathcal{O}$-orientation.

Let $A \in \mathsf{cdga}_S^{\le 0}$ and $\mathcal{E}\in \Perf(\mathcal{X}_{\mathrm{Dol}/S}\times_S \mathsf{Spec}(A))$. Locally on $\mathcal{X}$, the Dolbeault complex of $(\mathcal{E},\phi)$ is the complex
\[
\mathsf{Dol}(\mathcal{E}_A,\phi_A) := 
\big[\mathcal{E}_A \xrightarrow{\phi_A} \mathcal{E}_A \otimes \Omega^1_{\mathcal{X}_A/A} \xrightarrow{\phi_A} \dots \xrightarrow{\phi_A} \mathcal{E}_A \otimes \Omega^d_{\mathcal{X}_A/A}\big],
\]
concentrated in cohomological degrees $0,\dots,d$, where $\phi_A$ is the relative Higgs field 
\(\phi_A: \mathcal{E}_A \to \mathcal{E}_A \otimes \Omega^1_{\mathcal{X}_A/A}\) satisfying $\phi_A\wedge \phi_A = 0$.  Equivalently, $\mathcal{E}_A$ is a complex on $\mathcal{X}_A$ equipped with a module structure over $\mathrm{Sym}^\bullet \mathbb{T}_{\mathcal{X}_A/A}[-1]$, and $\mathsf{Dol}(\mathcal{E}_A,\phi_A)$ is the pushforward to $S$ along $r_A: \mathcal{X}_A \to \mathsf{Spec} A$:
\[
p_{A*} \mathcal{E}_A \simeq r_{A*} \mathsf{Dol}(\mathcal{E}_A,\phi_A),
\]
a complex in degrees $0$ to $d$.
Define the dual Dolbeault complex
\[
\mathsf{Dol}(\mathcal{E}_A,\phi_A)^\vee := 
\big[\mathcal{E}_A^\vee \xrightarrow{-\phi_A^\vee} \mathcal{E}_A^\vee \otimes \Omega^1_{\mathcal{X}_A/A} \to \dots \to \mathcal{E}_A^\vee \otimes \Omega^d_{\mathcal{X}_A/A}\big],
\]
concentrated in cohomological degrees $-d,\dots,0$. Then
\[
(p_{A*}(\mathcal{E}_A^\vee))^\vee[2-d] \simeq r_{A*}(\mathsf{Dol}(\mathcal{E}_A,\phi_A)^\vee)^\vee[2-d].
\]
By Grothendieck–Serre duality for the map $r_A:\mathcal{X}_A\to \mathsf{Spec} A$, the pairing
\[
\mathsf{Dol}(\mathcal{E}_A,\phi_A) \otimes \mathsf{Dol}(\mathcal{E}_A,\phi_A)^\vee \longrightarrow \mathcal{O}_{\mathcal{X}_A}[-d] \xrightarrow{\eta_A} A[-d]
\]
induces a quasi-isomorphism of dg-$A$-modules
\[
-\cap \eta_A: p_{A*} \mathcal{E}_A \longrightarrow (p_{A*}(\mathcal{E}_A^\vee))^\vee[2-d].
\]

The map $-\cap \eta_A$ provides the non-degenerate $(2-d)$-shifted symplectic pairing on the tangent complex of the derived mapping stack.  

Hence $\Map(\mathcal{X}_{\mathrm{Dol}/S},\Perf)$ carries a canonical $(2-d)$-shifted symplectic structure relative to $S$.

\end{proof}

We now recall that for any $A$-point, there is a canonical equivalence of $\infty$-categories
\begin{equation}
\label{DolbeaultPerfCATBNR}
\mathsf{Perf}\big(X_{\mathrm{Dol}}\times \mathsf{Spec}_k(A)\big)\simeq \mathsf{Mod}_{\mathrm{Sym}_{\mathcal{O}_{X_A}}(p_X^*\mathsf{T}_X)}\big(\mathsf{Perf}(X_A)\big),
\end{equation}
which is essentially the content of the of the Beauville--Narasimhan--Ramanan correspon
dence for perfect complexes, which we now recall in detail due to its relevance. See \cite[Lem.~6.8]{Simpson1994ModuliII} and more generally, \cite[Cor.~4.42]{DPFSV2026}.
\begin{proposition}
\label{DolDeformationTheory}
Let $X$ be a smooth and proper scheme with cotangent bundle $\pi:\mathsf{T}^*X\to X.$
Then the pushforward functor
$\pi_*^{\mathsf{QCoh}}: \mathsf{QCoh}(\mathsf{T}^*X) \to \mathsf{QCoh}(X),$
restricts to an equivalence
$$\mathsf{Perf}_{\pi-\text{prop}}(\mathsf{T}^*X)\simeq \mathsf{Perf}(X_{\mathrm{Dol}}).$$
Consider the full subcategory 
$\mathsf{Perf}_{X}(\mathsf{T}^*X)\hookrightarrow \mathsf{Perf}(\mathsf{T}^*X)$ of perfect complexes set-theoretically supported on the zero section $X\hookrightarrow \mathsf{T}^*X.$
Then $\pi_*^{\mathsf{QCoh}}$ restrictrs to an equivalence
$$\pi_*^{\mathsf{Perf}}|_{X\subset \mathsf{T}^*X}:\mathsf{Perf}_X(\mathsf{T}^*X) \xrightarrow{\ \sim\ } \mathsf{Perf}(X_{\mathrm{Dol}}^{\mathrm{nil}}).$$

\end{proposition}

\begin{proof}
The functor $\pi_*$ induces a canonical equivalence
$$
\mathsf{QCoh}(\mathsf{T}^*X)
\simeq
\mathsf{Mod}_{\mathrm{Sym}_{\mathcal{O}_X}(\mathsf{T}_X)}(\mathsf{QCoh}(X))
\simeq
\mathsf{QCoh}(X_{\mathrm{Dol}}),$$
where the second equivalence follows from the identification of the Dolbeault stack
$X_{\mathrm{Dol}}$ with modules over the commutative algebra
$\mathrm{Sym}_{\mathcal{O}_X}(\mathsf{T}_X),$ where $\mathsf{T}_X$ is the usual tangent bundle $T_X$ as $X$ is smooth and proper. Then \eqref{DolbeaultPerfCATBNR} arises due to restricting this equivalence to properly supported complexes. If $F \in \mathsf{QCoh}(\mathsf{T}^*X)$ is such that $\pi_*(F) \in \mathsf{Perf}(X_{\mathrm{Dol}})$,
one may show $F$ is properly supported over $X$. Since $X$ is smooth, $F$ has bounded
cohomology, i.e. $\pi_i(F) = 0$ for all but finitely many $i$. It suffices to check
that each cohomology sheaf $\pi_i(F)$ is properly supported over $X$. This follows from the
fact that $\pi_*(F)$ is perfect, hence coherent with proper support, and the equivalence above
preserves support conditions. Indeed, $\pi:\mathsf{T}^*X\to X$ has finite tor-amplitude and hence so does $\pi_*F$. The cohomology spectral sequence implies convergence 
$\mathbb{R}^i\pi_*^{\mathsf{QCoh}}(\pi_jF)\Rightarrow \mathbb{R}^{i+j}\pi_*^{\mathsf{QCoh}}F,$ hence each $\mathbb{R}^i\pi_*^{\mathsf{QCoh}}F$ is coherent and vanishes for sufficiently large $i.$ In other words, $F \in \mathsf{Perf}_{\pi\text{-prop}}(\mathsf{T}^*X)$.
\end{proof}


\subsection{Recollections on the Deligne shape}
We recall in some detail the Deligne-shape \cite[Sect. 6.1]{PortaSala}.
Let $X_{\mathrm{Del},\mathbb{G}_m}^{\bullet}$ be the formal completion defined therein and its geometric realization,
$$X_{\Hod}:=\big|X_{\Hod}^{\bullet}\big|\in \mathsf{dSt}_{/\Theta}.$$
\begin{remark}[Notation]
Pulling back via the atlas $\beta:\mathbb{A}^1\to \Theta,$ which forgets the $\mathbb{G}_m$-action, is denoted in \emph{loc.cit} by $X_{\mathrm{Del}}:=\beta^*(X_{\Hod}).$ We denote it here by $X_{\Hod}\times\mathbb{A}^1.$
    When $X$ is a smooth proper scheme over $\mathbb{C},$ a quasi-coherent complex $E\in \mathsf{QCoh}(X_{\Hod}\times\mathbb{A}^1)$ has, as its fiber $E|_{\lambda}$ over $\lambda\in \mathbb{A}^1,$ a quasi-coherent $\lambda$-connection on $X.$
\end{remark}
We recall and generalize the construction to the relative setting of a smooth morphism $f:\mathcal{X}\to \mathcal{B}$ of smooth projective varieties, of relative dimension $d.$ 

Let $\mathbb{A}_{\mathcal{B}}^1:=\mathbb{A}^1\times\mathcal{B},$ set $\Theta_{\mathcal{B}}:=[\mathbb{A}_{\mathcal{B}}^1/\mathbb{G}_m].$
$$\underline{\mathsf{Del}}^{\bullet}_{\mathcal{B}}:\Delta\to \mathsf{dAff}_{/\mathbb{A}^1_{\mathcal{B}}},$$
be the cosimplicial derived affine scheme over $\mathbb{A}_{\mathcal{B}}^1$, with simplices,
$\underline{\mathsf{Del}}^n:=\mathsf{Spec}(\mathbb{C}[x,y]/(x^n-y^n)),$ with structure map to $\mathbb{A}^1=\mathrm{Spec}(\mathbb{C}[t])$ given by $t\to y.$
The $\mathbb{G}_m$-action on $\mathsf{Del}^n$ is $\mathbb{A}^1$-equviariant, and there is cosimplicial derived stack,
$$\underline{\mathsf{Del}}_{\mathbb{G}_m/\mathcal{B}}^{\bullet}:=[\underline{\mathsf{Del}}^{\bullet}_{\mathcal{B}}/\mathbb{G}_m]:\Delta\to \mathrm{dStk}_{/\Theta_{\mathcal{B}}},$$
where the $\mathbb{G}_m$-action on $\mathcal{B}$ is trivial.
Then,
$$\Map_{/\Theta_{\mathcal{B}}}(\underline{\mathsf{Del}}_{\mathbb{G}_m/\mathcal{B}}^{\bullet},\mathcal{X}\times_{\mathcal{B}} \Theta_{\mathcal{B}}),$$
is simplicial object over $\Theta_{\mathcal{B}}.$ 
Pull-back via $\beta$ gives the $\mathbb{A}^1_{\mathcal{B}}$-cosimplicial object
$\Map_{/\mathbb{A}^1_{\mathcal{B}}}(\mathsf{Del}^{\bullet}_{\mathcal{B}},\mathcal{X}\times_{\mathcal{B}}\mathbb{A}^1_{\mathcal{B}}).$
The natural map $\mathsf{Del}_{\mathbb{G}_m/\mathcal{B}}^{\bullet}\to \Theta_{\mathcal{B}},$ gives a map
$$\delta:\mathcal{X}\times\Theta_{\mathcal{B}}\to \Map_{/\Theta_{\mathcal{B}}}(\mathsf{Del}_{\mathbb{G}_m/\mathcal{B}}^{\bullet},\mathcal{X}\times\Theta_{\mathcal{B}}),$$
and for each $[n]\in \Delta,$ gives a family of morphisms,
\[
\begin{tikzcd}
\mathcal{X}\times\mathbb{A}^1_{\mathcal{B}}\arrow[dr]\arrow[rr] && \Map_{/\mathbb{A}^1_{\mathcal{B}}}(\mathsf{Del}^n,\mathcal{X}\times\mathbb{A}^1_{\mathcal{B}})\arrow[dl]
\\
& \mathbb{A}^1\times\mathcal{B} &
\end{tikzcd}
\]
Noting the morphism $\mathcal{X}\times \mathbb{A}^1\to \mathcal{B}\times\mathbb{A}^1$ is simply $f\times \mathrm{id}_{\mathbb{A}^1},$ then this is nothing but the deformation to the normal cone construction applied to the relative diagonal $\Delta_
{\mathcal{X}/\mathcal{B}}^n:\mathcal{X}\to \mathcal{X}_{\mathcal{B}}^n.$ In particular, for a fiber over $(b,t)\in \mathbb{A}_{\mathcal{B}}^1,$ letting $(\mathcal{X}\times\mathbb{A}^1)_{(b,t)}\simeq \mathcal{X}_b,$ fiberwise this is the relative diagonal $\Delta_{\mathcal{X}_n}^n:\mathcal{X}_b\to \mathcal{X}_b^n,$ for $b\in \mathcal{B}$ and $[n]\in \Delta.$
Then, we obtain a new simplicial object, defined by the homotopy fiber product,
\[
\begin{tikzcd}
    (\mathcal{X}/\mathcal{B})_{\mathrm{Del},\mathbb{G}_m}^{\bullet}\arrow[d]\arrow[r] & \Map_{/\Theta_{\mathcal{B}}}(\mathsf{Del}_{\mathbb{G}_m/\mathcal{B}}^{\bullet},\mathcal{X}\times_{\mathcal{B}} \Theta_{\mathcal{B}})
\arrow[d]
\\
(\mathcal{X}\times\Theta_{\mathcal{B}})_{\mathrm{DR}}\arrow[r,"\delta_{\mathrm{DR}}"] & \Map_{/\Theta_{\mathcal{B}}}(\mathsf{Del}_{\mathbb{G}_m/\mathcal{B}}^{\bullet},\mathcal{X}\times[\mathbb{A}^1_{\mathcal{B}}\times\mathbb{G}_m])_{\mathrm{DR}}.
\end{tikzcd}
\]
Then $(\mathcal{X}/\mathcal{B})_{\mathrm{Del},\mathbb{G}_m}$ is the corresponding formal completion with $(\mathcal{X}/\mathcal{B})_{\mathrm{Del}},$ the pull-back along $\mathbb{A}^1\times\mathcal{B}\to [\mathbb{A}^1/\mathbb{G}_m]\times\mathcal{B}.$
\begin{proposition}
\label{DelDeformationTheory}
Let $f:\mathcal{X}\to \mathcal{B}$ be smooth of relative dimension $d.$ Consider,
$\Map_{/\mathcal{B}\times\mathbb{A}^1}\big((\mathcal{X}/\mathcal{B})_{\mathrm{Del}},\Perf\times \mathbb{A}_{\mathcal{B}}^1\big).$
It is naturally endowed with a relative (to $\mathbb{A}_{\mathcal{B}}^1$) $(2-d)$-shifted symplectic structure.
\end{proposition}
\begin{proof}
 As before, let $p:\mathcal{X}_{\mathrm{Del}/\mathcal{B}}\to \mathcal{B}$ for simplicity and $A\in \mathsf{cdga}_{\mathcal{B}}^{\leq 0},$ and consider
  $(\mathcal{E}_{A},\nabla_A^{\lambda})\in \mathsf{Perf}\big(\mathcal{X}_{\mathrm{Del}/\mathcal{B}}\times_{\mathcal{B}}\mathsf{Spec}(A)\big).$
  The result follows by computing $\mathbb{R}p_{A*}(\mathcal{E}_A)$ and $\big(\mathbb{R}p_{A*}(\mathcal{E}_A^{\vee})\big)[d-2],$ following the same proof as Proposition \ref{Prop:DolbeaultShiftedSympl}.
\end{proof}

\section{Lagrangian correspondences of nonabelian Hodge-type}
\label{sec: LagCorrNAHType}
In this section, we establish a Lagrangian correspondence between the derived moduli stacks underlying the nonabelian Hodge correspondence by interpreting the Hodge stack as a filtration of the de Rham moduli. This yields a natural shifted Poisson structure via Proposition~\ref{LagThicken}. In the curve case, we relate this to Hitchin’s integrable system (Proposition~\ref{NonAbelianCompleteIntegrable}).

Let $(Z,\omega_Z)$ be a derived Artin stack with an $n$-shifted symplectic structure. Then the derived stack of graded objects $\mathsf{Grad}(Z)$ defined by \eqref{eqn:FiltGrPst}, has a canonical $n$-shifted symplectic form $[\omega_{\mathsf{Grad}(Z)}]$. Namely, let $[\mathsf{B}\mathbb{G}_m]:\mathbb{R}\Gamma(\mathsf{B}\mathbb{G}_m,\mathcal{O})\to\mathbb{C}$ be the canonical $0$-orientation \cite[Prop.~5.14]{KinjoParkSafranov2024}. Then the corresponding $2$-form of degree $n$ is induced by the composition,
\[
\begin{tikzcd}
\mathbb{C}[-2-n]
\arrow[r,"{[\omega_{Z}]}"]
\arrow[drr,dotted,swap,"{[\omega_{\mathsf{Grad}(Z)}]}"]
& \mathrm{DR}(Z)
\xrightarrow{\mathrm{eval}^*}
\mathrm{DR}\big(\mathsf{B}\mathbb{G}_m\times \mathsf{Grad}(Z)\big)
\arrow[r,"\kappa"]
& \mathbb{R}\Gamma(\mathsf{B}\mathbb{G}_m,\mathcal{O})\otimes\mathrm{DR}\big(\mathsf{Grad}(Z)\big)
\arrow[d,"{[\mathsf{B}\mathbb{G}_m]}\otimes \mathrm{Id}"]
\\
& &
\mathrm{DR}\big(\mathsf{Grad}(Z)\big)
\end{tikzcd}
\]
where $\mathrm{eval}:\mathsf{B}\mathbb{G}_m\times\mathsf{Grad}(Z)\to Z$ is the canonical evaluation map and $\kappa$ is the K\"unneth morphism.

Since $\mathrm{Spec}(\mathbb{C})$ is $0$-oriented, we prove the following result based on the fact that cospans of derived stacks are sent to Lagrangian correspondences under the derived AKSZ-construction \cite[Prop.~3.4.2]{CalaqueHaugsengScheimbauer2025}.

\begin{proposition}
    \label{MainTheoremBody}
Let $X$ be a smooth proper algebraic variety over $\mathbb{C}$. Let $X^{\an}$ be its analytification. Then 
$$\mathsf{AnFilt}(\RAnPerf(X_{\DR}^{\an})\big): \mathsf{AnGrad}\big(\RAnPerf(X_{\DR}^{\an})\big)\dashrightarrow \RPerf_{\DR}(X)^{\an},$$
is a $2(1-d)$-shifted Lagrangian correspondence of derived analytic stacks with $\mathbf{G}_m$-actions, where $\RAnPerf(X_{\DR}^{\an})$ has the trivial action. Thus, $\mathsf{AnFilt}(\RAnPerf_{\DR}(X^{\an})\big)$ is naturally endowed with $(1-2d)$-shifted Poisson structure.
\end{proposition}
\begin{proof}
We proceed via direct computation of the (relative) tangent complexes using, for instance, the cartesian diagram \eqref{CartTangentFiltDiagram}, as well as Propositions \ref{DolDeformationTheory} and \ref{DelDeformationTheory}.
By our assumptions on $X$, we may proceed in the algebraic setting, and analytify the result by the canonical morphism \eqref{eqn: Map to AnMap}, using Proposition \ref{AnDeRhamAlgs}. 
Since $\mathsf{Filt}\big(\RPerf_{\DR}(X)\big)\to \mathsf{Grad}\big(\RPerf_{\DR}(X)\big)\times \RPerf_{\DR}(X),$ has the structure of a $2(1-d)$-shifted Lagrangian correspondence, let 
$$h:0\sim (\mathsf{Gr},\mathsf{Un})^*(\omega_{\mathsf{Grad}}\times\omega_{\mathrm{DR}}),\hspace{2mm} \text{ in }\mathcal{A}^{2,\mathrm{cl}}\big(\mathsf{Filt}(\RPerf_{\DR}(X), 2(1-d)\big),$$
be the homotopy defining the isotropic structure i.e. it is the homotopy filling the diagram:
\begin{equation}
\begin{tikzcd}
\label{IsoDiagFilt}
\mathbb{T}_{\mathsf{Filt}(\RPerf_{\DR}(X)}\arrow[d]\arrow[ddr,"0"] & 
\\
(\mathsf{Gr},\mathsf{Un})^*\mathbb{T}_{\mathsf{Grad}(\RPerf_{\DR}(X))\times \RPerf_{\DR}(X)} \arrow[d,"(\mathsf{Gr}\times\mathsf{Un})^*\Theta_{\omega}"]& 
\\
(\mathsf{Gr},\mathsf{Un})^*\mathbb{L}_{\mathsf{Grad}(\RPerf_{\DR}(X))\times \RPerf_{\DR}(X)}[n]\arrow[r] & \mathbb{L}_{\mathsf{Filt}(\RPerf_{\DR}(X))}[2(1-d)].
\end{tikzcd}
\end{equation}
Note explicitly the quasi-isomorphism
\begin{equation}
    \label{ThetaLagHod1}
\Theta_h:\mathbb{T}_{\mathsf{Filt}(\RPerf_{\DR}(X))}\rightarrow \mathbb{L}_{\mathsf{Filt}(\RPerf_{\DR}(X))/\mathsf{Grad}(\RPerf_{\DR}(X))\times \RPerf_{\DR}(X)}[1-2d].
\end{equation}

By adjunction 
$$
\Map([\mathbb{A}^1/\mathbb{G}_m], \Map(X_{\mathrm{DR}}, \Perf))
\simeq
\Map\big(X_{\mathrm{DR}} \times [\mathbb{A}^1/\mathbb{G}_m], \Perf\big),$$
it is clear for each base derived affine scheme $T$, there is a symmetric monoidal equivalence
$$\mathsf{Perf}\big([\mathbb{A}^1/\mathbb{G}_m]\times X_{\mathrm{DR}}\times T)\simeq \mathsf{Fil}_{\mathbb{Z}_{\geq }}\big(\mathsf{Perf}(X_{\mathrm{DR}}\times T)\big),$$
where we mean that 
$$\mathsf{Perf}\big([\mathbb{A}_{\mathbb{C}}^1/\mathbb{G}_m]\times X_{\mathrm{DR}}\times T\big)\simeq \mathrm{colim}_{I\subset \mathbb{Z}_{\geq },\text{finite,connected}}\mathsf{Funct}\big(I,\mathsf{Perf}(X_{\mathrm{DR}}\times T)\big),$$
whose objects correspond to diagrams $\cdots E_1\to E_0\to E_{-1}\to \cdots,$ in $\mathsf{Perf}(X_{\mathrm{DR}}\times T)$ with $E_m=0,m\gg0$ and $E_m\simeq E_{m-1}$ an isomorphism for $m\ll0.$ Equivalently, we have a filtered perect complex on $X_{\mathrm{DR}}\times T$, with descending filtration $\supset \mathrm{Fil}^pE^{\bullet}\supset \mathrm{Fil}^{p+1}E^{\bullet}\supset \cdots,$ indexed by $\mathbb{Z}$ for which $\mathrm{Fil}^pE^{\bullet}\equiv 0,$ for large $p$ and $\mathrm{Fil}^pE^{\bullet}\simeq \mathrm{Fil}^{p-1}E^{\bullet}$ an isomorphism for $p$ small, whose total object is $E\simeq\mathrm{colim}_p \mathrm{Fil}^pE,$ given by the fiber over $1$.
which moreover is compatible in the sense that for $f:S\to T$ a morphism of derived affine schemes, denoting all pullbacks by abuse of notation by $f^*,$ we have a commutative diagram
\[
\begin{tikzcd}
    \mathsf{Perf}([\mathbb{A}^1/\mathbb{G}_m]\times X_{\mathrm{DR}}\times T)\arrow[d,"f^*"]\arrow[r,"\sim"] & \mathsf{Fil}_{\mathbb{Z}_{\geq }}\big(\mathsf{Perf}(X_{\mathrm{DR}}\times T)\big)\arrow[d,"f^*"]
    \\
    \mathsf{Perf}([\mathbb{A}^1/\mathbb{G}_m]\times X_{\mathrm{DR}}\times S)\arrow[r,"\sim"]& \mathsf{Fil}_{\mathbb{Z}_{\geq }}\big(\mathsf{Perf}(X_{\mathrm{DR}}\times S)\big).
\end{tikzcd}
\]
Evaluation on $T$-points,
\begin{equation*}
\mathsf{Fil}_{\mathbb{Z}_{\geq}}\big(\mathsf{Perf}(X_{\mathrm{DR}}\times T)\big)
\simeq \mathsf{Perf}([\mathbb{A}^1/\mathbb{G}_m]\times X_{\mathrm{DR}}\times T)^{\simeq}
\simeq \mathsf{Filt}\big(\RPerf
_{\mathrm{DR}}(X)\big)(T).
\end{equation*}
Similarly, $\mathsf{Grad}(\RPerf_{\DR}(X))(T)\simeq \mathsf{Perf}(\mathsf{B}\mathbb{G}_m\times X_{\mathrm{DR}}\times T)^{\simeq}\simeq\mathsf{Perf}^{\mathrm{gr}}(X_{\mathrm{DR}}\times T)^{\simeq}.$
Finally, the existence of the shifted Poisson structure claim follows since
$$\mathsf{Filt}\big(\RPerf_{\DR}(X)\big)\to \mathsf{Grad}\big(\RPerf_{\DR}(X)\big)\times \overline{\RPerf_{\DR}}(X),$$
is Lagrangian with target a $2(1-d)$-shifted symplectic stack, thus is $2(1-d)-1$-shifted Poisson via Proposition \ref{LagThicken}.

  Analytification for stacks is a left Kan extension, thus it preserves colimits. By analytifying the groupoid presentation 
$\mathbb{G}_m\times\mathbb{A}_{\mathbb{C}}^1\rightrightarrows \mathbb{A}_{\mathbb{C}}^1,$ for $[\mathbb{A}_{\mathbb{C}}^1/\mathbb{G}_m],$ we see $[\mathbb{A}_{\mathbb{C}}^1/\mathbb{G}_m]^{\mathrm{an}}\simeq [\mathbf{A}_{\mathbb{C}}^1/\mathbb{C}^{\times}],$ since $(\mathbb{G}_m\times\mathbb{A}_{\mathbb{C}}^1)^{\mathrm{an}}\simeq (\mathbb{G}_m)^{\mathrm{an}}\times (\mathbb{A}_{\mathbb{C}}^1)^{\mathrm{an}}\simeq \mathbb{C}^{\times}\times\mathbf{A}_{\mathbb{C}}^1.$ Thus we consider the analytic stack quotient $[\mathbf{A}_{\mathbb{C}}^1/\mathbb{C}^{\times}]$ for the holomorphic action $\mathbb{C}^{\times}\times\mathbf{A}_{\mathbb{C}}^1\to \mathbf{A}_{\mathbb{C}}^1$ and following \eqref{eqn:FiltGrPst}, we define
\begin{equation}
    \label{eqn:FiltGrAnPst}\mathsf{AnFilt}(-):=\AnMap\big([\mathbf{A}_{\mathbb{C}}^1/\mathbb{C}^{\times}],-\big),\hspace{5mm} \mathsf{AnGrad}(-):=\AnMap\big(\mathbb{G}_m^{\mathrm{an}},-\big).
    \end{equation}
Then, by \cite[Prop.~5.2, Thm.~5.5]{HolsteinPorta2025}, 
\begin{eqnarray}
\label{FiltAn}
\mathsf{Filt}\big(\RPerf(X_{\mathrm{DR}})\big)^{\mathrm{an}}&=& \Map([\mathbb{A}_{\mathbb{C}}^1/\mathbb{G}_m],\RPerf(X_{\mathrm{DR}})\big)^{\mathrm{an}} \nonumber
\\
&\simeq& \AnMap\big([\mathbf{A}_{\mathbb{C}}^1/\mathbb{C}^{\times}],\RAnPerf(X_{\mathrm{DR}}^{\mathrm{an}})\big) \nonumber
\\
&=& \mathsf{AnFilt}\big(\RAnPerf(X_{\mathrm{DR}}^{\mathrm{an}})\big).
\end{eqnarray}
Note we also used \cite[Lem.~5.24, Lem.~5.31]{HolsteinPorta2025}. Similarly, we obtain an equivalence 
\begin{equation}
    \label{GradAn}
\mathsf{Grad}\big(\RPerf(X_{\mathrm{DR}})\big)^{\mathrm{an}}\simeq \mathsf{AnGrad}\big(\RAnPerf(X_{\mathrm{DR}}^{\mathrm{an}})\big),
\end{equation}
of derived analytic stacks. Since the derived stacks of interest are all (locally) geometric, then by relative GAGA \cite[Prop.~5.2, Thm.~5.5]{HolsteinPorta2025}, which give equivalences \eqref{eqn: RPerfAnCommutes}, we have that
since $\mathcal{A}^{p,\mathrm{cl}}(Z,n)\simeq \Map(Z,\mathcal{A}^{2,\mathrm{cl}}(-,n)),$ we have an equivalence of complexes 
$\mathcal{A}^{2,\mathrm{cl}}(Z,n)^{\mathrm{an}}\xrightarrow{\simeq} \mathcal{A}_{\mathbf{C}}^{2,\mathrm{cl}}(Z^{\mathrm{an}},n),$
and the Lagrangian structure also preserved by analytification since
\begin{eqnarray*}
(\mathbb{T}_{\mathsf{Filt}\big(\RPerf(X_{\mathrm{DR}})\big)})^{\mathrm{an}}&\xrightarrow{\Theta_h^{\mathrm{an}}}&(\mathbb{L}_{\mathsf{gr}\times\mathsf{un}})^{\mathrm{an}}[1-2d]
\\
&\simeq& \mathbb{L}_{(\mathsf{gr}\times\mathsf{un})^{\mathrm{an}}}[1-2d]
\\
&\simeq& \mathbb{L}_{\mathsf{Filt}\big(\RPerf(X_{\mathrm{DR}})\big)^{\mathrm{an}}/(\mathsf{Grad}(\RPerf(X_{\mathrm{DR}})\times\RPerf(X_{\mathrm{DR}})^{\mathrm{an}}}[1-2d].
\end{eqnarray*}
Then, from \eqref{FiltAn} and \eqref{GradAn} by composing we obtain
$$\Theta_{h^{\mathrm{an}}}:\mathbb{T}_{\mathsf{AnFilt}\big(\RAnPerf(X_{\mathrm{DR}}^{\mathrm{an}})\big)}\xrightarrow{\sim}\mathbb{L}_{\mathsf{AnFilt}\big(\RAnPerf(X_{\mathrm{DR}}^{\mathrm{an}})\big)/\mathsf{AnGrad}(\RAnPerf(X_{\mathrm{DR}}^{\mathrm{an}}))\times \RAnPerf(X_{\mathrm{DR}}^{\mathrm{an}})}[1-2d].$$

\end{proof}
We now consider the Hodge stack of perfect complexes $\RPerf_{\Hod}(X)\to \Theta.$ By taking invariants for the natural $\mathbb{G}_m$-action we identify 
$\Map_{\Theta}(X_{\Hod},\Perf_{\mathbb{C}}\times\Theta)^{\mathbb{G}_m}$ with $\Map(X_{\Dol},\Perf_{\mathbb{C}})^{\mathbb{G}_m}.$
The Lagrangian correspondence follows from a twisted version of the following general observation.
Let $Z$ be an $n$-shifted symplectic derived Artin stack. Then $Z\times\mathbb{A}^1\to \mathbb{A}^1$ is relative $n$-shifted symplectic. Since $Z\to Z\times Z$ is shifted Lagrangian, then the morphism of $\mathbb{A}^1$-stacks $(Z\times\mathbb{A}^1)\to (Z\times\mathbb{A}^1)\times(Z\times\mathbb{A}^1)$ is relative shifted Lagrangian.
\begin{theorem}
\label{MainTheoremBodyBody}
Let $X$ be a smooth proper scheme over $\mathbb{C}$ of dimension $d.$ Then $X_{\Hod}\to \Theta,$ carries a canonical $\mathcal{O}_{\Theta}(1)$-twisted $2d$-orientation and by pullback along $\mathbb{A}^1\to \Theta,$ there is a relative shifted Lagrangian correspondence
$$\RPerf_{\Hod}(X)\times\mathbb{A}^1\dashrightarrow \big(\RPerf_{\Hod}(X)\times\mathbb{A}^1\big)\times_{\mathbb{A}^1}\big(\RPerf_{\Hod}(X)\times\mathbb{A}^1\big),$$
compatible with analytifcation.
\end{theorem}
\begin{proof}
By \cite[Prop.~5.1]{calaque2024shifted}, the morphism $\mathsf{B}\mathbb{G}_m\to \Theta$ is $\mathcal{O}_{\Theta}(1)[-1]$-oriented over $\Theta$. By construction of $r:X_{\Hod}\to \Theta,$ consider 
$\Gamma_{/\Theta}(X_{\Hod},\mathcal{O})\to \mathcal{O}_{\Theta}[2d],$ which gives an $\mathcal{O}_{\Theta}$-twisted relative orientation of degree $2d$ interpolating the untwisted $2d$-orientations $[X_{\DR}]$ and $[X_{\Dol}]$ \cite{PTVV13}. Then, consider the pull-back deformation space via
\[
\begin{tikzcd}
X_{\Del}\arrow[d]\arrow[r] & X_{\Hod}\arrow[d]
    \\
    \mathbb{A}^1\arrow[r] & \Theta.
\end{tikzcd}
\]
Using \cite[Prop.~5.10, Cor.~5.13]{calaque2024shifted},
by considering $\RPerf_{\mathrm{Hod}}(X):=\Map_{/\Theta}(X_{\Hod},\Perf\times\Theta)\to\Theta,$ and the pull-back along the canonical map $\mathbb{A}^1\to \Theta,$ we get $\RPerf_{\Hod}(X)\times\mathbb{A}^1\to \mathbb{A}^1,$ and consider the diagram 
    \[
    \begin{tikzcd}
     \RPerf_{\Hod}(X)\times\mathbb{A}^1 \arrow[dr] \arrow[rr] & & \arrow[dl] (\RPerf_{\Hod}(X)\times\mathbb{A}^1)\times (\RPerf_{\Hod}(X)\times\mathbb{A}^1)
     \\ 
     &\mathbb{A}^1 &
    \end{tikzcd}
    \]
where the top map is given by $(\mathrm{ev}_0,\mathrm{ev}_1).$ Each factor of the target is $2(1-d)$-shifted symplectic relative to $\mathbb{A}^1,$ thus so is their product, for which the above map is endowed with a relative shifted Lagrangian structure. Compatibility with analytification is clear from \eqref{eqn: Map to AnMap} and the main result of \cite{HolsteinPorta2025}.
\end{proof}
While Theorem \ref{MainTheoremBody} is independent of any choice of stability condition, we mention for completeness the corresponding open substacks and their classical truncations. In the particular case of a complex curve $C$, we also describe the Jordan--H\"older maps, acting by semisimplification from the Artin stack of flat (resp. $\lambda$-connections) to its corresponding good moduli space, in the sense of \cite{alper2013good}. These Artin stacks will be the classical truncations of corresponding derived Artin stacks described we now discuss.
\subsection{Stability considerations}
\label{ssec: Stability considerations}
Let $H$ be an ample divisor and fix $P=P(m)\in\mathbb{Q}[m].$ By \cite[Def.~1.2.4, Prop.~2.3.1]{HuybrechtsLehn2010}, we may study the $H$-semistable coherent sheaves on a projective scheme, making use of openness of $H$-semistability in families.
We can define derived open substacks of the moduli presented so far, with fixed numerical behaviour, following e.g \cite[Cons.~2.5]{PortaSala2023}. There exists an open substack
$$^{\mathrm{cl}}\RCoh^P(X)\hookrightarrow \hspace{.2mm}^{\mathrm{cl}}(\RCoh(X)),$$
parameterizing flat families of coherent sheaves $F$ on $X$ with $P$ fixed. Namely, for $n\gg 0,$ 
$$\dim H^0\big(X,F\otimes\mathcal{O}_X(mH)\big)=P(m).$$
Then, as an open substack, by \cite[Prop.~2.1]{STV}, there is a canonical derived enhancement $\RCoh^P(X)$. The same holds true for $\mathbb{R}\underline{\mathrm{Vect}}^P(X).$
Moreover, for any non-zero $P(m)\in \mathbb{Q}[m]$ of a fixed degree $r$, let $\overline{P}(m):=P(m)/a_r,$ with $a_r$ the leading coeffficient in $P(m).$ Given a monic polynomial $q\in\mathbb{Q}[m],$ we have
$$\RCoh^q(X):=\coprod_{\overline{P}=q}\RCoh^P(X),$$
and similarly for $\mathbb{R}\underline{\mathrm{Vect}}^q(X).$ Assuming $\deg =\dim X,$
Gieseker $H$-semistability is an open property, thus there are open substacks
$^{\mathrm{cl}}\RCoh^{\mathrm{ss},q}(X)\hookrightarrow \hspace{.2mm} ^{\mathrm{cl}}\RCoh^{\mathrm{ss},q}(X),$
parameterizing families of $H$-semistable coherent sheaves on $X$ with fixed reducted polynomial $q$. Again by \cite[Prop~2.1]{STV}, $\RCoh^{\mathrm{ss},q}(X)$ is the canonical derived enhancement.
Similarly for $\mathbb{R}\underline{\mathrm{Vect}}^{\mathrm{ss},q}(X).$

\subsubsection{Variants of Dolbeault Complexes}

Let $X$ be a smooth projective complex scheme. For any monic numerical polynomial $q(m)\in \mathbb{Q}[m],$ consider cartesian diagrams\footnote{Since from \cite[Lem.~ 5.3.1]{PortaSala}, the canonical map $X\to X_{\mathrm{Dol}}$ is a flat effective epimorphism.} in derived stacks

\begin{equation}
    \label{qCohDol}
\begin{tikzcd}
    \RCoh_{\mathrm{Dol}}^q(X)\arrow[d]\arrow[r] & \RCoh^q(X)\arrow[d]
    \\
   \RPerf(X_{\mathrm{Dol}})\arrow[r] & \RPerf(X),
\end{tikzcd}\hspace{5mm} \begin{tikzcd}
    \mathbb{R}\underline{\mathrm{Vect}}_{\mathrm{Dol}}^q(X)\arrow[d]\arrow[r] & \mathbb{R}\underline{\mathrm{Vect}}^q(X)\arrow[d] 
    \\
     \RPerf(X_{\mathrm{Dol}})\arrow[r] & \RPerf(X).
\end{tikzcd}
\end{equation}
The homotopy fiber products \eqref{qCohDol} are both geometric derived stacks, locally of finite presentation.
The higher-dimensional analogue of semistability of Higgs bundles \cite{Simpson1994ModuliI,Simpson1994ModuliII} originally defined over a curve, 
is an instance of the
Gieseker stability condition for modules over a sheaf of rings of differential operators,
when such a sheaf is induced by $\Omega_X^1$ with zero symbol. This
semistability condition is an open property for flat families. Thus,
there exists an open substack
$^{\mathrm{cl}}\RCoh_{\mathrm{Dol}}^{\mathrm{ss},q}(X)\hookrightarrow \hspace{.2mm}^{\mathrm{cl}}\RCoh_{\mathrm{Dol}}^q(X),$
parameterizing families of semistable Higgs sheaves on $X$, with fixed reduced polynomial $\overline{P}(m)=q(m).$ We denote by
$$\RCoh_{\mathrm{Dol}}^{\mathrm{ss},q}(X),$$
its canonical derived enhancement. Similarly, we have geometric derived stack which is locally of finite presentation $\RCoh_{\mathrm{DR}}(X):=\RCoh(X_{\mathrm{DR}})$, as a full substack of $\Map(X_{\mathrm{DR}},\Perf).$ 
Since $X$ is smooth, the canonical map $q_X:X\to X_{\mathrm{DR}}$ is a flat effective epimorphism \cite[Prop.~4.1.1]{PortaSala}, hence we may identify $\RCoh_{\mathrm{DR}}(X)$ as in \eqref{qCohDol}, with the homotopy pull-back,
\begin{equation}
\label{RCohDR}
\begin{tikzcd}
\RCoh(X_{\mathrm{DR}})\arrow[d]\arrow[r] & \RCoh(X)\arrow[d]
    \\
\RPerf(X_{\mathrm{DR}})\arrow[r] & \RPerf(X).
\end{tikzcd}
\end{equation}
The following two results are then direct corollaries of Theorem \ref{MainTheoremBody}.
\begin{corollary}
\label{SSCohCor}
Let $q(m)\in\mathbb{Q}[m],$ be a monic numerical polynomial. Restriction to $\RCoh^{\mathrm{ss},q}$, induces a shifted Lagrangian correspondence $$\RCoh_{\mathrm{Dol}}^{\mathrm{ss},q}(X)\leftarrow \RCoh_{\mathrm{Hod}}^{\mathrm{ss},q}(X)\rightarrow \RCoh_{\mathrm{DR}}^{\mathrm{ss},q}(X),$$
where we consider the (relative) derived stack $\RCoh_{\mathrm{Hod}}(X):(\mathsf{dAff}_{/\mathbb{A}^1})^{\mathrm{op}}\to \mathsf{Spc}$ mapping $(S\to \mathbb{A}^1)$ to $\mathsf{Coh}_{S}\big(S\times_{\mathbb{A}^1}X_{\mathrm{Hod}}\big)^{\simeq}.$
\end{corollary}
\begin{proof}
    Since semistability is an open property, $\RCoh^{\mathrm{ss},q}(X)\hookrightarrow \RCoh^q(X)$ is a Zariski open immersion. Note that $\RCoh_{\mathrm{Hod}}(X)$ is a geometric derived stack, locally of finite presentation, since the map  $X\times\mathbb{A}^1\to X_{\mathrm{Hod}},$ induced from \eqref{eqn: Hodge Structure Map} is a flat effective $\mathbb{A}^1$-map. Then, we have the following explicit identification as a homotopy fiber product,
    \[
    \begin{tikzcd}
        \RCoh_{\mathrm{Hod}}(X)\arrow[d]\arrow[r] & \RCoh(X)\times\mathbb{A}^1\arrow[d]
        \\
        \RPerf_{\mathrm{Hod}}(X)\arrow[r] & \RPerf(X)\times\mathbb{A}^1.
    \end{tikzcd}
    \]
    Hence, $\RCoh_{\mathrm{Hod}}(X)\to \mathbb{A}^1$ when viewed as a derived stack over $\mathbb{A}^1$ (by pull-back), with fibers
    $\RCoh_{\mathrm{Hod}}(X)\times_{\mathbb{A}^1}\{0\}\simeq \RCoh(X_{\mathrm{Dol}}),$ and $\RCoh_{\mathrm{Hod}}(X)\times_{\mathbb{A}^1}\{1\}\simeq \RCoh(X_{\mathrm{DR}}).$
    To see the symplectic structures and Lagrangian morphism are preserved, we use the fact that the canonical map $\RCoh(X)\to \RPerf(X)$ is formally \'etale \cite[Cor.~2.19]{PortaSala2023}, thus if $\RPerf(X)$ has a cotangent complex, so does $\RCoh(X)$ \cite[Cor.~2.20]{PortaSala2023} and therefore inherits the shifted-symplectic structure.
\end{proof}
When $X$ is a smooth and proper complex scheme there is a geometric derived stack 
$$\mathbb{R}\mathrm{Higgs}(X):=\mathsf{T}^*[0]\RCoh(X):=\mathsf{Spec}_{\RCoh(X)}\big(\mathsf{Sym}_{\mathcal{O}_{\RCoh(X)}}^{\bullet}(\mathbb{T}_{\RCoh(X)})\big),$$
and a natural morphism 
$\RCoh_{\mathrm{Dol}}(X)\to \mathbb{R}\mathrm{Higgs}(X).$
It is an equivalence when $X$ is a smooth projective curve, however, in higher dimensions it is no longer an equivalence, since in $\dim_{\mathbb{C}}\geq 2,$ the symmetric algebra and the tensor algebra on $\mathbb{T}_{\RCoh(X)}$ differ. In the curve case, there are natural maps from truncations to underlying good moduli schemes, which we now describe.
\subsubsection{Jordan--H\"older maps}
\label{sssec: JHMaps}
When $C$ is a smooth projective curve, 
$\mathcal{M}_{\mathrm{DR}}(C)$ is the de Rham moduli stack parameterizing flat connections on $C$, with good moduli space $M_{\mathrm{DR}}(C)$ parameterizing semisimple flat connections on $C$. Similar good moduli spaces exists for the Dolbeault stacks $\mathcal{M}_{\mathrm{Dol}}(C)\to M_{\mathrm{Dol}}(C)$. The natural maps between them are given by the Jordan--H\"older maps
\begin{equation}
\label{JHmaps}
JH_{\mathrm{DR}}:\mathcal{M}_{\mathrm{DR}}(C)\to M_{\mathrm{DR}}(C),\hspace{3mm} JH_{\mathrm{Dol}}:\mathcal{M}_{\mathrm{Dol}}(C)\to M_{\mathrm{Dol}}(C).
\end{equation}
Similarly, for the Hodge moduli stack, we have
$JH_{\mathrm{Hod}}:\mathcal{M}_{\mathrm{Hod}}(C)\to M_{\mathrm{Hod}}(C).$ There are obvious $\mathbb{C}^{*}$-actions on both $\mathcal{M}_{\mathrm{DR}}(C)$ and $M_{\mathrm{Dol}}(C)$, covering the natural action on $\mathbb{A}^1.$ The map $JH_{\mathrm{Hod}}$ is $\mathbb{C}^*$-equivariant.
Note the following defining diagrams, whose squares are all cartesian,
\[
\begin{tikzcd}
    \mathcal{M}_{\mathrm{Dol}}(C)\arrow[d]\arrow[r,"JH_{\mathrm{Dol}}"] & M_{\mathrm{Dol}}(C)\arrow[d]\arrow[r] & \{0\}\arrow[d]
    \\
    \mathcal{M}_{\mathrm{Hod}}(C)\arrow[r,"JH_{\mathrm{Hod}}"] & M_{\mathrm{Hod}}(C)\arrow[r] & \mathbb{A}^1,
\end{tikzcd}\hspace{5mm}
\begin{tikzcd}
    \mathcal{M}_{\mathrm{DR}}(C)\arrow[d]\arrow[r,"JH_{\mathrm{DR}}"] & M_{\mathrm{DR}}(C)\arrow[d]\arrow[r] & \{1\}\arrow[d]
    \\
    \mathcal{M}_{\mathrm{Hod}}(C)\arrow[r,"JH_{\mathrm{Hod}}"] & M_{\mathrm{Hod}}(C)\arrow[r] & \mathbb{A}^1.
\end{tikzcd}
\]
Moreover, for every $\lambda\in\mathbb{A}^1,$ there are fiber diagrams, whose horizontal arrows are inclusion of the fiber over $\lambda$:
\[
\begin{tikzcd}
\mathcal{M}_{\lambda}^{\mathrm{Hod}}(C)\arrow[d,"JH_{\lambda,\mathrm{Hod}}"]\arrow[r] & \mathcal{M}_{\mathrm{Hod}}(C)\arrow[d,"JH_{\mathrm{Hod}}"]
    \\
    M_{\lambda}^{\mathrm{Hod}}(C)\arrow[d]\arrow[r]& M_{\mathrm{Hod}}(C)\arrow[d]
    \\
    \{\lambda\}\arrow[r] & \mathbb{A}^1.
\end{tikzcd}
\]
Taking the fiber over $\lambda\in \mathbb{A}^1,$ the Hodge moduli stack admits a good moduli space $\mathcal{M}_{\lambda}^{\mathrm{Hod}}(C)\to M_{\lambda}^{\mathrm{Hod}}(C).$

We conclude by relating these classical spaces with their derived enhancements. 

We let $\RCoh_{\mathrm{DR}}(C)\simeq \RPerf_{\mathrm{DR}}(C)\times_{\RPerf(C)}\RCoh(C)$ be the geometric derived stack, locally of finite presentation, obtained as a full substack of $\Map(C_{\mathrm{DR}},\mathbb{R}\Perf).$ Similarly, for $\mathbb{R}\mathrm{Vect}_{\mathrm{DR}}(C),$ as well as $\mathbb{R}\mathrm{Vect}_{\mathrm{Hod}}(C).$ This derived Artin stack parameterizes mixed graded complexes in degree zero with $\lambda$-connection. There is an equivalence of (higher) Artin stacks $\mathcal{M}_{\mathrm{Hod}}(C)\simeq\hspace{.2mm}^{\mathrm{cl}}\mathbb{R}\mathrm{Vect}_{\mathrm{Hod}}(C).$ Taking the fiber over $\lambda$, we obtain an equivalence of higher Artin stacks of $\lambda$-complexes,
$$\mathcal{M}_{\lambda}^{\mathrm{Hod}}(C)\simeq\hspace{1mm} ^{\mathrm{cl}}\mathbb{R}\mathrm{Vect}_{\lambda,\mathrm{Hod}}(C).$$

We now recall a well-known fact that the Hodge--Deligne moduli space carries a natural Poisson structure. We give a nonabelian description in Proposition \ref{NonAbelianCompleteIntegrable}.
\subsubsection{A complete integrable Poisson variety}
\label{sssec: CompleteIntegrable}
In the particular case when the base is a projective curve $C$, one can describe $\lambda$-connections as pairs $(E,D)$, where $E$ is a vector bundle and $
D \in H^0\big(C, \mathrm{At}(E) \otimes K_C \big)$
is a section of the Atiyah algebroid of $E$ tensored with the canonical bundle $K_C$ of $C$. The symbol map defines a morphism $D\mapsto \lambda$, given by the composite,
$$
\sigma \otimes \mathrm{Id} \colon H^0\big(\mathrm{At}(E) \otimes K_C\big) \longrightarrow H^0(T_C \otimes K_C) \simeq H^0(\mathcal{O}_C) = \mathbb{C}.$$
Variants for fixed determinant also arise. Let $M_{SL_m}$ be the smooth quasi-projective variety of stable rank
$m$ bundles with trivial determinant on $C$. Let $\pi:M^{\mathrm{st}}\times C\to M^{\mathrm{st}}$ be the projection and assume there is a universal bundle $\mathcal{E}$.
Consider the determinant of cohomology
line bundle 
$\mathcal{L}_{\mathrm{det}}:=\det(\pi_*\mathcal{E})\otimes\det \mathbb{R}^1\pi_*(\mathcal{E})^{\vee}.$
Restricting to the locus of the moduli spaces where the underlying bundle is stable, twe obtain morphisms,
$
M_{\mathrm{Dol}}^s \simeq T^* M_{\mathrm{SL}_m} \longrightarrow M_{\mathrm{SL}_m},
M_{\mathrm{Hod}} \longrightarrow M_{\mathrm{SL}_m},$
 interpreted as vector bundles over $M_{\mathrm{SL}_mr}$. Together with the symbol map sending a $\lambda$-connection to $\lambda$, one obtains a short exact sequence
\begin{equation} \label{eq:HodgeSeq}
0 \longrightarrow M_{\mathrm{Dol}}^s \simeq T^* M_{\mathrm{SL}_r} 
\longrightarrow M_{\mathrm{Hod}} \simeq \pi^*\big(\mathrm{At}_0(E) \otimes K_C \big) \longrightarrow \mathcal{O}_{M_{\mathrm{SL}_r}} \longrightarrow 0.
\end{equation}
Over the locus of stable underlying vector bundles, the Hodge--Deligne space $M_{\mathrm{Hod}}$ for fixed determinants is dual to the total space of the Atiyah algebroid $\mathrm{At}_0(\mathcal{L}_{\det})$. Moreover, 
$
M_{\mathrm{Hod}} \simeq \mathrm{At}_0(L_{\det})^* \longrightarrow M_{\mathrm{SL}_r},$ is an isomorphism preserving the natural Poisson structures on total spaces. 
In the next subsection we provide an interpretation of this (relative) Poisson structure at the (nonabelian) level of moduli spaces as the classical shadow of the shifted Poisson structure in Theorem \ref{MainTheoremBody}. 

\subsection{Semi-stable curves}
Consider the moduli stack of relative flat $\lambda$-connections on a semi-stable family of curves $\mathcal{X}\to \mathcal{B}.$ 

\begin{definition}
\label{defn: Family of ss curves}
Let $f:\mathcal{X}\to \mathcal{B}$ be a flat representable morphism of classical Artin stacks of relative dimension $1$. It is a \emph{semi-stable family of curves} if every geometric fiber of $\mathcal{X}\to \mathcal{B}$ is isomorphic to a semi-stable curve i.e. satisfies:
\begin{itemize}
    \item Every geometric fiber is a reduced, connected, projective, nodal curve;
    \item Every component $E$ of a geometric fiber which isomorphic to $\mathbb{P}^1$ must intersect the
 union of all other components at at least two smooth points.
\end{itemize}
\end{definition}
Consider the moduli of rank $n$-vector bundles $B\mathrm{GL}_n$ viewed as a sub-stack of $\Perf.$
A \emph{relative $\lambda$-connection} over a family of semi-stable curves $\mathcal{X}\to\mathcal{B}$ is a locally free $\mathcal{O}_{\mathcal{X}}$-module $\mathcal{E}$ with a relative $\lambda$-connection $\nabla_{\mathcal{X}/\mathcal{B}}^{\lambda}:\mathcal{E}\to \mathcal{E}\otimes \omega_{\mathcal{X}/\mathcal{B}}.$ 
\begin{corollary}
\label{DelCohomIsom}
If $(\mathcal{E},\nabla_{\mathcal{X}/\mathcal{B}}^{\lambda})$ is a quasi-coherent relative $\lambda$-connection on $\mathcal{X}/\mathcal{B}$, and if $E$ is the corresponding quasi-coherent complex on $\mathcal{X}_{\mathrm{Del}/\mathcal{B}},$ there is a canonical isomorphism of quasi-coherent
 complexes on $\mathcal{B}$,
 $$H(\mathcal{X}_{\mathrm{Del}/\mathcal{B}}/\mathcal{B},E)\simeq H_{\mathrm{Del}}(\mathcal{X}/\mathcal{B},(\mathcal{E},\nabla_{\mathcal{X}/\mathcal{B}}^{\lambda}),$$
 where the right-hand side is $R(\mathcal{X}\to \mathcal{B})_*\big[\mathcal{E}\xrightarrow{\nabla^{\lambda}}\mathcal{E}\otimes \omega_{\mathcal{X}/\mathcal{B}}\big],$ in degrees $[0,1].$
\end{corollary}
We immediately have that
\begin{proposition}
\label{RMDel0shifted}
    The moduli stack $\Map_{/\mathcal{B}\times\mathbb{A}^1}(\mathcal{X}_{\mathrm{Del}/\mathcal{B}},\mathrm{\mathsf{B}GL}_n(\mathcal{O}_{\mathcal{B}})\times\mathbb{A}_{\mathcal{B}}^1),$
     has a natural relative $0$-shifted symplectic structure.
\end{proposition}
\begin{proof}
    Let $u_A:\mathcal{X}_{\mathrm{Dol}/A}\to \mathrm{\mathsf{B}GL}_n\times \mathsf{Spec}(A),$ be a point corresponding to a rank $n$-vector bundle $\mathcal{E}_A$ on $\mathcal{X}_{\mathrm{Dol},A}.$ Thus, it corresponds to a $\lambda$-connection 
    $$\big[E_A\xrightarrow{\nabla_A^{\lambda}}E_A\otimes \omega_{\mathcal{X}_A,A}\big],$$
    on $\mathcal{X}_A.$ There is a natural pairing 
$$u_A^*\mathbb{T}_{\mathrm{\mathsf{B}GL}_n}\otimes u_A^*\mathbb{T}_{\mathrm{\mathsf{B}GL}_n}\simeq \mathcal{E}nd(\mathcal{E}_A)[1]\otimes \mathcal{E}nd(\mathcal{E}_A)[1]\xrightarrow{\mathrm{Tr}}\mathcal{O}_{\mathcal{X}_{\mathrm{Dol},A}}[2],$$
    via the trace-pairing and pushforward along the map $q_A$, gives 
    $$q_{A,*}\big(\mathcal{E}nd(E_A[1])\oplus \mathcal{E}nd(E_A[1]))\to (\mathcal{O}_{\mathcal{X}_A}\oplus \omega_{\mathcal{X}_A,A}[-1])[2].$$
    Note that the complex
    $q_{A,*}\big(\mathcal{E}nd(\mathcal{E}_A[1])\big)\simeq \big[\mathcal{E}nd(E_A)\xrightarrow{\mathrm{ad}_{\nabla^{\lambda}}}\mathcal{E}nd(E_A)\otimes\omega_{\mathcal{X}_A,A}\big],$
    has differential $\mathrm{ad}_{\nabla^{\lambda}}:=[-,\nabla^{\lambda}],$ and the result follows from inspection.
\end{proof}

\begin{corollary}
\label{SScurveLAGNAH}
    Let $f:\mathcal{X}\to\mathcal{B}$ be a family of semi-stable curves of genus $g$ and let $\lambda\in \mathbb{A}^1$ be
a deformation parameter for the (relative) Deligne shape $\mathcal{X}_{\mathrm{Del}/\mathcal{B}}\to \mathcal{B}\times \mathbb{A}^1.$
Then, $R\mathcal{M}^{\mathrm{Del}}(\mathcal{X}/\mathcal{B};n)\to \mathcal{B}\times\mathbb{A}^1,$ carries a natural $0$-shifted relative symplectic structure and
\begin{enumerate}
    \item $(\lambda=0):$ the relative Dolbeault stack $R\mathcal{M}^{\mathrm{Dol}}(\mathcal{X}/\mathcal{B};n)$ is $0$-shifted symplectic,
\item $(\lambda \neq 0):$ the relative de Rham stack $R\mathcal{M}^{\mathrm{DR}}(\mathcal{X}/\mathcal{B};n)$ is $0$-shifted symplectic.
\end{enumerate}
\end{corollary}
\begin{proof}
  This follows directly from the main Theorem \ref{MainTheorem}. Indeed, from Corollary \ref{DelCohomIsom} which describes the tangent spaces to the relative Deligne-moduli stack, which is $0$-shifted by Proposition \ref{RMDel0shifted}, the result follows from the general fact that a Lagrangian morphsim with $n$-shifted symplectic target is $(n-1)$-shifted. 
\end{proof}

A nonabelain interpretation of Subsect.~\ref{sssec: CompleteIntegrable} is available as follows. Let
$
q_{\mathrm{Hod}} :C \times_S M_{\mathrm{Hod}} \longrightarrow M_{\mathrm{Hod}}$
be the natural projection and suppose $(\mathcal{E}, \gamma)$ is a universal $\lambda$-connection on $C \times_S M_{\mathrm{Hod}} \to M_{\mathrm{Hod}}$. Let $\EuScript{T}_{M/S}$ denote the tangent complex of the classical moduli stack $M$ over $S$, and let $T_{M/S} := H^0(\EuScript{T}_{M/S})$ be its zero-th cohomology sheaf. We then have the following lemma. 
\begin{proposition}
\label{NonAbelianCompleteIntegrable}
  Let $C$ be a smooth projective curve and 
$\mathcal{M}_{\mathrm{Hod}}(C)\to \mathbb{A}^1$ the Hodge moduli stack of vector bundles with $\lambda$-connection. 
For each $\lambda \in \mathbb{A}^1$, the fiber 
\[
\mathcal{M}_{\lambda}^{\mathrm{Hod}}(C) \simeq {}^{\mathrm{cl}}\mathbb{R}\mathrm{Vect}_{\lambda,\mathrm{Hod}}(C)
\]
admits a natural Poisson structure on its smooth locus, induced via the trace pairing
$$
    \mathrm{tr} : \mathrm{End}^0(\mathcal{E}_{\mathrm{Hod}}) \otimes \mathrm{End}^0(\mathcal{E}_{\mathrm{Hod}}) \longrightarrow \mathcal{O}_{C\times_S M_{\mathrm{Hod}}}.$$
    
\end{proposition}
\begin{proof}
    There is a canonical isomorphism $
T_{M_{\mathrm{DR}}(C/S)/S} \simeq \mathbb{R}^1 q_* \mathrm{End}^0(\mathcal{E})_{\mathrm{DR}},$
where $\mathcal{E}$ is the universal flat bundle. Moreover,
$$
H^0(\EuScript{T}_{M_{\mathrm{DR}}(C/S)}) \simeq \pi^* T_S \times_{\mathbb{R}^1 q_* T_{C \times_S M_{\mathrm{DR}}/M_{\mathrm{DR}}}} \mathbb{R}^1 q_*(\mathrm{At}^0(\mathcal{E})_1^\bullet),$$
via the Atiyah complex, where the subscripts indicate taking the fiber $\lambda =1.$
Similarly,
$$
T_{M_{\mathrm{Hod}}(C/S)/\mathbb{A}^1} \simeq \mathbb{R}^1 q^{\mathrm{Hod}}_* \mathrm{End}^0(\mathcal{E})_{\lambda\text{-}\mathrm{DR}} \simeq \pi_{\mathrm{Hod}}^* T_S \times_{\mathbb{R}^1 q^{\mathrm{Hod}}_* T_{X \times M_{\mathrm{Hod}}/M_{\mathrm{Hod}}}} \mathbb{R}^1 q^{\mathrm{Hod}}_*(\mathrm{At}^0(\mathcal{E}_{\mathrm{Hod}})_{\lambda}^\bullet).
$$
Then, via relative Serre duality along the fibers of $q^{\mathrm{Hod}}$, there is a canonical map
$$   \mathbb{L}_{M_{\mathrm{Hod}}(C/S)/\mathbb{A}^1} \longrightarrow T_{M_{\mathrm{Hod}}(C/S)/\mathbb{A}^1},$$
    endowing $M_{\mathrm{Hod}}(C/S)$ with a relative Poisson structure over $\mathbb{A}^1$. This is verified by noting this map is equivalently
\[\begin{tikzcd} T_{M_{\mathrm{Hod}}/\mathbb{A}^1} \ar[r,"\sim"] \ar[d] & \pi_{\mathrm{Hod}}^* T_S \times_{\mathbb{R}^1 q_* T_{C\times M/M}} \mathbb{R}^1 q_* (\mathrm{At}^0(\mathcal{E}_\lambda)^\bullet) \ar[d] \\ \mathbb{L}_{M_{\mathrm{Hod}}/\mathbb{A}^1}[1] \ar[r,"\sim"] & \pi_{\mathrm{Hod}}^* T_S^\vee \times_{\mathbb{R}^1 q_* T_{C\times M/M}^\vee} \mathbb{R}^1 q_* (\mathrm{At}^0(\mathcal{E}_\lambda)^\bullet \otimes K_{C/S}). \end{tikzcd}\]
\end{proof}

\subsection{The surface case}

Definition \ref{defn: Family of ss curves} adapts to the case of a smooth family of smooth projective surfaces over $\mathbb{C}.$ Consequently, adjusting the argument of Corollary \ref{DelCohomIsom}
to the relative dimension $2$ context, we have that for each perfect complex $E$, the corresponding Deligne--Dolbeault and de Rham complexes are concentrated in degrees $[0,2],$ and in particular, we have a quasi-isomorphism
$$R(\mathcal{X}_{\mathrm{Hod}/\mathcal{B}}\to \mathcal{B})_*(E)\simeq R(\mathcal{X}\to \mathcal{B})_*\big[\mathcal{E}\xrightarrow{\nabla^{\lambda}}\mathcal{E}\otimes \Omega_{\mathcal{X}/\mathcal{B}}\to\mathcal{E}\otimes \Omega_{\mathcal{X}/\mathcal{B}}^2\big].$$
This directly implies the following analogue of Proposition \ref{RMDel0shifted}.

\begin{corollary}
\label{RelDim2NAHLag}
    Let $f:\mathcal{X}\to \mathcal{B}$ be a smooth family of projective surfaces. Then the derived moduli stack $\mathcal{M}_{\Hod}(\mathcal{X}/\mathcal{B};n)$ is naturally $(-2)$-shifted symplectic, and is naturally part of a $(-2)$-shifted relative Lagrangian correspondence.
\end{corollary}
\begin{proof}
The first claim follows directly from the proof of Corollary \ref{DelCohomIsom} and of Proposition \ref{RMDel0shifted} via the fiberwise integration map
$\int_{\mathcal{X}_A/\mathcal{B}}: R(\mathcal{X}\to\mathcal{B})_*\Omega^2_{\mathcal{X}/\mathcal{B}}[-2] \to \mathcal{O}_{\mathcal{B}}[-2].$
The second claim is a consequence of the proof of Corollary \ref{SScurveLAGNAH}.
\end{proof}

\subsection{Nonabelian Hodge theory in the $\mathbb{C}$-analytic setting}
\label{ssec: Analytification}

We now consider a complex analytic analogue of the above results. In particular, we study the enhancement of the RH-isomorphism \eqref{RHIntro} in the derived analytic setting \cite{porta2017derived}. 

All of $X_{\mathrm{DR}},X_{\mathrm{Dol}}$ and $X_{\mathrm{Betti}}$ exist in the analytic setting, and if $X$ is a smooth proper connected $\mathbb{C}$-scheme, forming these shape operations commutes with analytification i.e. $(X_{\mathrm{DR}})^{\mathrm{an}}\simeq (X^{\mathrm{an}})_{\mathrm{DR}},(X_{\mathrm{Dol}})^{\mathrm{an}}\simeq (X^{\mathrm{an}})_{\mathrm{Dol}},$ \cite[Lem.~5.24, Lem.~5.31]{HolsteinPorta2025}. Similarly for $(-)_{\mathrm{Betti}}.$

Since analytification functor commutes with colimits, note the canonical $(\mathsf{B}G)^{\mathrm{an}} \simeq \mathsf{B}(G^{\mathrm{an}}),$ and that there is a canonical $t$-exact conservative functor
$$\mathsf{QCoh}(\mathsf{B}G)\to \mathsf{Mod}(\mathcal{O}_{\mathsf{B}G^{\mathrm{an}}}). $$
Moreover, both canonical $t$-exact functors
$$\mathsf{Perf}(\mathsf{B}G)\to \mathsf{Perf}(\mathsf{B}G^{\mathrm{an}}),\hspace{2mm}\mathsf{Coh}(\mathsf{B}G)\to\mathsf{Coh}(\mathsf{B}G^{\mathrm{an}}),$$
are equivalences of stable $\infty$-categories where both sides are equipped with complete $t$-structures.
\begin{remark}
We restrict to the study of 
$\RAnCoh(X^{\mathrm{an}})$ where $X$ is an algebraic variety, although this derived stack makes sense for more general derived analytic stacks $X$ which are $\mathcal{O}$-compact and oriented.
\end{remark}

If $X$ is a proper complex scheme, by \cite[Prop.~5.2, Thm.~5.5]{HolsteinPorta2025}, there is an equivalence,
\begin{equation}
    \label{eqn: RPerfAnCommutes}
\RPerf(X)^{\mathrm{an}}\simeq \RAnPerf(X^{\mathrm{an}}).
\end{equation}

It extends to $\RAnCoh(X^{\mathrm{an}})$ as follows. There is a canonical map $\RPerf(X)^{\mathrm{an}}\to \RAnPerf(X^{\mathrm{an}})$, obtained by adjunction from the canonical morphism 
$$\mu_{X,\mathrm{perf}}:\RPerf(X)\to \RAnPerf(X^{\mathrm{an}})\circ (-)^{\mathrm{an}}.$$

For $S\in \mathsf{dAff}^{\mathrm{afp}}$, this map is induced by applying the underlying $\infty$-groupoid functor $(-)^{\simeq}:\mathsf{Cat}_{\infty}\to \mathsf{Spc},$ to the analytification functor
\begin{equation}
    \label{relAnalytic}
\mathsf{Perf}(Z\times S)\to \mathsf{Perf}(Z^{\mathrm{an}}\times S^{\mathrm{an}}).
\end{equation}

If $f:B\to S$ is a morphism of derived complex stacks, lfp, with $B$ geometric and $S$ affine, then $\mathcal{F}\in \mathsf{APerf}(B)$ has tor-amplitude within $[a,b]$ relative to $S$ if and only if $\mathcal{F}^{\mathrm{an}}\in \mathsf{APerf}(B^{\mathrm{an}})$ has tor amplitude within $[a,b]$ relative to $S^{\mathrm{an}},$ thus analytification preserves subcategories of families of coherent sheaves relative to $S$ and $S^{\mathrm{an}}.$ 

Indeed, for $j_U:U\hookrightarrow S^{\mathrm{an}}$ an open derived Stein subspace, let
$A_U:=\Gamma(U;\mathcal{O}_{S^{\mathrm{an}}}^{\mathrm{alg}}|_{U}),$ where $S=\mathsf{Spec}A.$ Then, there is a canonical map $a_U:\mathsf{Spec}A_U\to S$ induced by the morphism $A\to A_U.$ We have natural pullback squares, 
\[
\begin{tikzcd}
X_U\arrow[d]\arrow[r] & X\arrow[d]
\\
\mathsf{Spec}A_U\arrow[r] & S,
\end{tikzcd}\hspace{2mm} \begin{tikzcd}
X_U\arrow[d]\arrow[r] & X\arrow[d]
\\
U\arrow[r] & S,
\end{tikzcd}
\]
and thus a relative analytification functor
$$(-)_{U}^{\mathrm{an}}:\mathsf{APerf}(X_U)\to \mathsf{APerf}(X_U^{\mathrm{an}}).$$

Consequently, there is a morphism 
$\RCoh(X)\to \RAnCoh(X^{\mathrm{an}})\circ(-)^{\mathrm{an}},$ compatible with $\mu_{X,\mathrm{perf}},$ induces
\begin{equation}
    \label{muXCoh}
\mu_X:\RCoh(X)^{\mathrm{an}}\to \RAnCoh(X^{\mathrm{an}}),
\end{equation}
by adjunction. Since $X$ is a complex projective scheme, \eqref{muXCoh} is an equivalence \cite{PortaSala2023}.

Indeed, this can be seen by reducing to the case of derived Stein spaces $U\in \mathsf{dStn}_{\mathbb{C}}.$ Namely, for every derived Stein space $U$ and every
compact derived Stein subspace $K$ of $U$, the natural morphism
$$``\underset{K\subset V\subset U}{\mathrm{hocolim}}" \mathsf{Coh}\big(\mathsf{Spec}(A_V)\times X/\mathsf{Spec}(A_V)\big)\to``\underset{K\subset V\subset U}{\mathrm{hocolim}}" \mathsf{Coh}(V\times X^{\mathrm{an}}/V),$$
is an equivalence in $\mathsf{Ind}\big(\mathsf{Cat}_{\infty}^{\mathrm{st}}\big)$. The colimit is taken over the family of open Stein neighborhoods $V$ of $K$ inside $U.$

\subsubsection{Shifted-symplectic derived Riemann--Hilbert}
In this subsection, we study the derived analog of the RH-map \eqref{RHIntro} and upgrade the derived RH-equivalence \cite[Thm.~6.11]{porta2017derived} to an equivalence between \emph{shifted-symplectic} derived analytic stacks.

As before, we keep that $X$ is a smooth proper connected complex scheme. First, there is a natural transformation
\begin{equation}
    \label{EtaRH}
\eta_{RH}:X_{\mathrm{DR}}^{\mathrm{an}}\to X_{\mathrm{Betti}}^{\mathrm{an}}.
\end{equation}
As mentioned in \cite[Rem.~3.7]{porta2017derived} this map is part of the following diagram,
\begin{equation}
    \label{RHPFDiagram}
  \begin{tikzcd}
      & \mathrm{Maps}_{\mathsf{Spc}}\big(\Pi_{\infty}\big((-)^{\mathrm{an}}\big),\Pi_{\infty}X\big)^{\mathrm{an}} & 
      \\
X_{\mathrm{DR}}^{\mathrm{an}}\arrow[dr,"q"]\arrow[ur] \arrow[rr,"\eqref{EtaRH}"] & & \arrow[ul,"\sim"] \arrow[dl,"p"] X_{\mathrm{Betti}}^{\mathrm{an}}
\\
& * &
  \end{tikzcd}  
\end{equation}
whose lower triangle naturally commutes. 
\begin{remark}
\label{AnalyticRequirementRemark}
  This diagram \eqref{RHPFDiagram} also exists for $X\in\mathsf{dAff}_{\mathbb{C}},$ but the right-most map can not be inverted, as it requires the analytic topology.
  \end{remark}
The transformation \eqref{EtaRH} induces for every $Z\in\mathsf{dAnSt},$ a morphism
\begin{equation}
\label{DerRH}
\eta_{RH}^*:\AnMap(X_{\B}^{\mathrm{an}},Z)\to \AnMap(X_{\DR}^{\mathrm{an}},Z).
\end{equation}
When $Z$ is $\AnPerf,$ so that $\RAnPerf(X_{\mathrm{DR}}^{\mathrm{an}})=\AnMap(X_{\mathrm{DR}}^{\mathrm{an}},\AnPerf),$ by \cite[Thm.~6.11]{porta2017derived} the Riemann--Hilbert morphism \eqref{DerRH} is an equivalence of derived analytic stacks.

Since $X$ is compact connected oriented manifold of dimension $d$, the Betti stack $X_{\mathrm{Betti}}$ is $\mathcal{O}$-compact. By Poincar\'e duality, it is $d'$-oriented as a topological manifold, and $\Map(X_B,\RPerf)$ is locally of finite presentation with natural $(2-d')$-shifted symplectic structure \cite{PTVV13}. Here, $d'=2\dim_{\mathbb{C}}(X),$ and thus the underlying $2$-forms and their cohomological degrees agree. By relative GAGA, via the analytic de Rham theorem the fundamental classes coincide, hence so do the PTVV-orientation data.  

In other words, applying the analytification functor to the correspondence in Theorem \ref{MainTheoremBody}, then the morphism \eqref{eqn: Map to AnMap} in this case is an equivalence \cite{HolsteinPorta2025} and we obtain the following result.
\begin{theorem}
\label{AnalyticNAH}
   Let $X$ be a smooth proper connected complex algebraic variety of dimension $d$. Consider the $2d$-oriented stacks $X_{\DR},X_{\B}$. Then the Riemann--Hilbert transformation preserves the orientation data. In other words,
   the following statements hold:
   \begin{itemize}
   \item[(i)] The derived RH morphism \eqref{DerRH} is compatible with the $(2-2d)$-shifted symplectic structures via AKSZ i.e.
$$\eta_{\mathrm{RH}}^*:\AnMap(X_{\B}^{\mathrm{an}},\AnPerf)\xrightarrow{\simeq}\AnMap(X_{\DR}^{\mathrm{an}},\AnPerf),$$
    is an equivalence of $2(1-d)$-shifted sympelctic derived analytic stacks. 
       \item[(ii)] More generally, let $Z\in\mathsf{dSt}^{\mathrm{afp}}_{\mathbb{C}}$ be a geometric derived stack locally almost of finite presentation over $\mathbb{C},$ with $n$-shifted symplectic structure.
Assume that $Z$ is Tannakian and that $\Q(Z)\simeq\mathsf{Ind}\big(\mathsf{Perf}(Z)\big).$
Assume that $X$ is a smooth $\mathbb{C}$-analytic space of dimension $d$. Then the Riemann--Hilbert transformation $\eta_{\mathrm{RH}}:X_{\DR}\to X_{\B}$ induces an equivalence
$$\eta_{\mathrm{RH}}^*:\AnMap(X_{\B},Z^{\an})\rightarrow \AnMap(X_{\DR},Z^{\an}),$$
of $(n-2d)$-shifted symplectic derived analytic stacks.
    \end{itemize}
    \end{theorem}
\begin{proof}
We need only show orientations are preserved, which is an argument independent of the target. We thus phrase the proof for (ii). 
Consider the diagram \eqref{RHPFDiagram}. It is enough to prove that the natural orientations $p_*(\mathcal{O}_{X_{\mathrm{B}}})\to \mathbb{C}[2-d']$ and $q_*(\mathcal{O}_{X_{\mathrm{DR}}})\to \mathbb{C}[2-2d],$ are preserved by the pull-back $\eta_{\mathrm{RH}}^*,$ for $\mathcal{F}\in\mathsf{APerf}(X_{\mathrm{B}}^{\mathrm{an}}).$ In this case, compatibility of lower-$+$ pushforward, is just a rephrasing of the de Rham theorem i.e. de Rham and Betti cohomology agree. 
By Beck-Chevalley, there is a natural map 
$q_+\eta_{\mathrm{RH}}^*(\mathcal{F})\leftarrow p_+(\mathcal{F}),$
and the natural composition 
$$\mathcal{F}\otimes p^*p_*(\mathcal{F}^{\vee})\to \mathcal{F}\otimes \mathcal{F}^{\vee}\to \mathcal{O}_{X_{\mathrm{B}}^{\mathrm{an}}},$$
induces a natural morphism $\nu_{\mathcal{F}}:\mathcal{F}\to p^*p_+(\mathcal{F}),$ functorial in $\mathcal{F}\in\mathsf{APerf}.$ Similarly for $q_+$. The result follows by considering the unit map 
$$p_*\to p_*\eta_{\mathrm{RH}*}\eta_{\mathrm{RH}}^*\simeq q_+\eta_{\mathrm{RH}}^*,$$
which is an equivalence by \eqref{RHPFDiagram}.
For completeness, we also confirm the agreement of the symplectic structures via the AKSZ construction (although more or less clear).
To this end, consider the canonical diagram induced by the RH-transform \eqref{EtaRH}, and \eqref{DerRH}:
\[
\begin{tikzcd} X_{\mathrm{DR}}^{\mathrm{an}}\times \AnMap(X_{\mathrm{DR}}^{\mathrm{an}},Z) \arrow[r,"\mathrm{ev}_{\mathrm{DR}}"] \arrow[d,"\eta_{\mathrm{RH}}\times \eta_{\mathrm{RH}}^{-1}"] & Z \\ X_{\mathrm{Betti}}^{\mathrm{an}}\times \AnMap(X_{\mathrm{Betti}}^{\mathrm{an}},Z) \arrow[ur,"\mathrm{ev}_{\mathrm{B}}"'] \end{tikzcd}
\]
Consider also the diagram of projections
\[
\begin{tikzcd} X_{\mathrm{DR}}^{\mathrm{an}}\times \AnMap(X_{\mathrm{DR}}^{\mathrm{an}},Z) \arrow[r] \arrow[d,"p^{\mathrm{DR}}"] & X_{\mathrm{Betti}}^{\mathrm{an}}\times \AnMap(X_{\mathrm{Betti}}^{\mathrm{an}},Z) \arrow[d,"p^{\mathrm{B}}"] \\ \AnMap(X_{\mathrm{DR}}^{\mathrm{an}},Z) \arrow[r,"\eqref{DerRH}"] & \AnMap(X_{\mathrm{Betti}}^{\mathrm{an}},Z) \end{tikzcd}
\]
Applying $\mathrm{DR}^{\mathrm{an}},$ we first obtain the diagram

\begin{equation}
    \label{SympRHDiag1}
\begin{tikzcd}[column sep=huge] \mathrm{DR}(Z) \arrow[r,"ev_{dR}^*"] \arrow[d,equal] & \mathrm{DR}\big(X_{\mathrm{DR}}^{\mathrm{an}}\times \AnMap(X_{\mathrm{DR}}^{\mathrm{an}},Z)\big) \arrow[d,"(\eta_{RH}\times id)^*"] \\ \mathrm{DR}(Z) \arrow[r,"ev_B^*"'] & \mathrm{DR}\big(X_{\mathrm{B}}^{\mathrm{an}}\times    \AnMap(X_{\mathrm{B}}^{\mathrm{an}},Z)\big) \end{tikzcd}
\end{equation}
and secondly a diagram via the choices of orientations,
\begin{equation}
    \label{SympRHDiag2}
\begin{tikzcd}[column sep=huge] \mathrm{DR}(X_{\mathrm{DR}}^{\mathrm{an}}\times \AnMap(X_{\mathrm{DR}}^{\mathrm{an}},Z)) \arrow[r,"p^{\mathrm{DR}}_+"] \arrow[d,"(\eta_{RH}\times id)^*"] & \mathrm{DR}(\AnMap(X_{\mathrm{DR}}^{\mathrm{an}},Z))[-2d] \arrow[d,"\eta_{RH}^*"] \\ \mathrm{DR}(X_{\mathrm{B}}^{\mathrm{an}}\times \AnMap(X_{\mathrm{B}}^{\mathrm{an}},Z)) \arrow[r,"p^{\mathrm{B}}_+"] & \mathrm{DR}(\AnMap(X_{\mathrm{B}}^{\mathrm{an}},Z))[-2d]. \end{tikzcd}
\end{equation}
Pasting together diagram \eqref{SympRHDiag1} with  \eqref{SympRHDiag2} via composition (c.f. composite maps \eqref{eqn: DR-composites}), we obtain a large diagram from which the Beck--Chevalley condition gives that
$p(_{\mathrm{DR}})_+\circ (\eta_{\mathrm{RH}}\circ \mathrm{Id})^*\simeq \eta_{\mathrm{RH}}^*\circ (p_{\mathrm{B}})_+,$
which is the content of the de Rham theorem. Then, explicitly:
$$\eta_{\mathrm{RH}}^*(\omega_{\mathrm{Betti}})\simeq\eta_{\mathrm{RH}}^*\big(\int_{[X_{\mathrm{B}}]}\mathrm{ev}_{\mathrm{B}}^*\omega_Z\big)\simeq\int_{[X_{\mathrm{DR}}]}(\eta_{\mathrm{RH}}\times \mathrm{Id})^*\mathrm{ev}_{\mathrm{B}}^*\omega_Z\simeq \int_{[X_{\mathrm{DR}}]}\circ \mathrm{ev}_{\mathrm{DR}}^*\omega_Z=\omega_{\mathrm{DR}}.$$
At the level of points, the argument is carried out via a relative GAGA argument; it suffices to evaluate $\eta_{\mathrm{RH}}^*$ on $S\in \mathsf{dStein}_{\mathbb{C}}$, which gives a morphism
$$\mathrm{Maps}_{\mathsf{dAnSt}_{\mathbb{C}}}\big(S\times X_{\B},Z^{\an})\rightarrow \mathrm{Maps}_{\mathsf{dAnSt}_{\mathbb{C}}}(S\times X_{\DR},Z^{\an}).$$
One must show that the analytic mapping stacks on both sides are geometric and this morphism is an equivalence which preserves the orientations on $X_{\B},X_{\DR}.$ 
These statements are clear from \cite{HolsteinPorta2025}, but we mention details for completeness. Geometricity is immediate, while  the equivalence follows if we can show that the induced diagram,
\[
\begin{tikzcd}[column sep=6em]
\mathrm{Maps}_{\mathsf{dAnSt}_{\mathbb{C}}}\big(S\times X_{\B},Z^{\an})\arrow[r,"\eta_{\mathrm{RH}}^*(S)"] \arrow[d,"\hat{P}"] & \mathrm{Maps}_{\mathsf{dAnSt}_{\mathbb{C}}}\big(S\times X_{\DR},Z^{\an})\arrow[d,"\hat{P}"]
    \\
    \mathsf{Funct}_{L,\mathbb{C}}^{\otimes}\big(\mathsf{Perf}(Z),\mathsf{Perf}(S\times X_{\B})\big)\arrow[r,"G"]& \mathsf{Funct}_{L,\mathbb{C}}^{\otimes}\big(\mathsf{Perf}(Z),\mathsf{Perf}(S\times X_{\DR})\big)
\end{tikzcd}
\]
is commutative.
Since $Z$ is Tannakian, the vertical morphisms are fully faithful by application of \cite[Thm.~6.4]{HolsteinPorta2025}. The functor $G$ is an equivalence by \cite[Thm.~6.11]{porta2017derived}.
Commutativity follows from application of \cite[Prop.~6.12]{HolsteinPorta2025}. In order to apply that result, one proceeds as follows. Replace $X_{\B}$ by a colimit of derived Stein spaces, and apply the same argument given in \cite[Lem.~5.14]{PortaYuHigherGaga2016}, which finds a finite double hypercovering of $X$ by disjoint unions of contractible open Stein subspaces satisfying a nesting property \cite[Rmk.~5.15]{PortaYuHigherGaga2016}, for which the above results apply and which we glue to give the desired statement.
\end{proof}
\begin{remark}
In practice, if $\mathcal{F}$ is a derived local system i.e. a perfect complex of local systems on $X$, then $\eta_{\mathrm{RH}}^*(\mathcal{F})\simeq\mathcal{F}\otimes_{\underline{\mathbb{C}}_X}\mathcal{O}_X$ is the induced perfect complex with flat connection $\nabla$. Then, 
$q_*\eta_{\mathrm{RH}}^*(\mathcal{F})\simeq H_{\mathrm{DR}}^*(\mathcal{F}\otimes_{\underline{\mathbb{C}}_X}\mathcal{O}_{X^{\mathrm{an}}},\nabla\big)\simeq H^*(\mathcal{F}),$ precisely by the de Rham theorem (cf. remark \ref{UniversalDolbeault}).
\end{remark}

\section{Deligne--Hitchin--Simpson twistor stack}
\label{sec: Derived NAH and KW}
In this section,  
for simplicity, we focus on the absolute case, so that $\mathcal{B}$ is $Spec(\mathbb{C}).$ 
 \subsection{Deligne homotopy pushouts}
 Let $X$ be a smooth proper variety over $\mathbb{C}$ and $X^{\mathrm{conj}}$ the complex (Galois) conjugate variety as in Subsection.~\ref{PreTwistorSection}, defined by taking complex conjugates of the coefficients of equations defining $X.$
 This defines a real analytic homemorphism,
 $\gamma:X_{\mathrm{B}}\xrightarrow{\simeq}X^{\mathrm{conj}}_{\mathrm{B}},$
 where $X_{\mathrm{B}}$ is the topological space underlying the complex analytic manifold $X^{\mathrm{an}}.$ There is an induced equivalence,
 \begin{equation}
     \label{eqn: Gamma_B}
 \gamma_{\mathrm{B}}^*:\Map(X_{\mathrm{B}},\RPerf)\xrightarrow{\simeq}\Map(X_{\mathrm{B}}^{\mathrm{conj}},\RPerf).
 \end{equation}
 By \eqref{eqn: Map to AnMap} it is compatible with analytification. We will combine this equivalence with the derived Riemann-Hilbert equivalence \eqref{DerRH}, to compute the derived twistor space of perfect complexes on a smooth proper $\mathbb{C}$-analytic space $X$.
\begin{remark}[Notation]
The stack quotient $\EuScript{P}=[\mathbb{P}^1/\mathbb{G}_m]$ has two closed points isomorphic to $\mathsf{B}\mathbb{G}_m$ and one open point with trivial stabilizers. Indeed, the fixed points of the $\mathbb{G}_m$-action on $\mathbb{P}^1$ are $0=[1:0]$ and $\infty=[0:1].$ In other words, $\EuScript{P}$ is equivalent to the pushout $\EuScript{A}_0\sqcup_{\mathrm{Spec}(\mathbb{C})}\EuScript{A}_{\infty},$ where $\EuScript{A}_i$ are the copies at $0,\infty.$ Moreover, $\mathbb{P}^1\to \EuScript{P}$ is a $\mathbb{G}_m$-torsor over $\mathbb{P}^1\backslash \{0,\infty\}.$
\end{remark}
Recall $(\mathbb{A}^1\backslash \{0\})^{\mathrm{an}}\simeq\mathbb{C}^*\simeq \mathbf{G}_m,$ and our conventions regarding $\mathbf{P}^1$ are that it is always endowed with its $\mathrm{Gal}(\mathbb{C}/\mathbb{R})$-action (cf. Rem.~\ref{GalActionRemark}).
\begin{theorem}
\label{DHSTwistor1}
Let $X$ be a smooth $\mathbb{C}$-analytic space of dimension $d$. Then, the following hold:
\begin{itemize}
    \item[(i)] The homotopy pushout in derived analytic stacks
    \begin{equation}
        \label{DeligneHoPushout}
    \begin{tikzcd}
        \RAnPerf_{\B}(X)\times \mathbf{G}_m\arrow[d,"\EuScript{J}"] \arrow[r,"\EuScript{J}_{\mathrm{conj}}"] & \RAnPerf_{\mathrm{Hod}}(X^{\mathrm{an,conj}})\arrow[d]
\\
\RAnPerf_{\mathrm{Hod}}(X^{\mathrm{an,conj}})\arrow[r] & \RAnPerf_{\mathrm{Del}}(X),
    \end{tikzcd}
\end{equation}
    is a derived analytic stack.
    
    \item[(ii)] There is a natural $\mathbb{G}_m$-equivariant map to $\mathbf{P}^1,$ descending to the stack theoretic analytic quotient $\EuScript{P}^{\mathrm{an}}$ of $\mathbf{P}^1$ by $\mathbf{G}_m$, denoted 
$\eta_{\mathrm{Del}}:\RAnPerf_{\mathrm{Del}}(X)\to \EuScript{P}^{\mathrm{an}},$ supplying the Deligne-stack with an automorphism covering the antipode. In particular, there are natural $\mathbb{G}_m$-equivariant restriction maps $\eta_{Del}^1:\RAnPerf_{\mathrm{Hod}}(X)\to \mathbf{A}_{\mathbb{C}}^1,$ and $\eta_{Del}^2:\RAnPerf_{\mathrm{Hod}}(X^{\mathrm{an,conj}})\to \mathbf{A}_{\mathbb{C}}^1,$ recording the values of the $\lambda$-connections, which are equal under the antipodal map $\sigma_{\mathbf{P}^1}:\mathbf{P}^1\to \mathbf{P}^1$. 

\item[(iii)] Let $t_0\big(\RAnPerf_{\mathrm{Hod}}(X)\big)$ be the classical truncation and denote by
$\mathcal{M}_{\mathrm{Hod}}^{\mathrm{ss}}(X)\subset t_0\big(\RAnPerf_{\mathrm{Hod}}(X)\big)$
the open substack of semistable objects. Similarly, for $\mathcal{M}_{\mathrm{Hod}}^{\mathrm{ss}}(X^{\mathrm{conj}}).$ Then the pushout of ordinary classical analytic stacks
$\mathcal{M}_{\mathrm{Hod}}(X)^{\mathrm{ss}}\sqcup \mathcal{M}_{\mathrm{Hod}}(X^{\mathrm{conj}})^{\mathrm{ss}},$
 is a good moduli space.
\end{itemize}
\end{theorem}
\begin{proof}
\textit{(i)}:
We begin by constructing closed-immersions of analytic stacks, along which we form a homotopy-pushout. 
To this end, combining \eqref{eqn: Gamma_B} with the derived Riemann-Hilbert map \eqref{DerRH}, we have  
\begin{equation}
    \label{GluingDeRhamConjugates}
\begin{tikzcd}
\Map(X_{\mathrm{B}},\Perf)^{\mathrm{an}}\arrow[d,"\eqref{eqn: Gamma_B}"]\arrow[r,"\eqref{eqn: Map to AnMap}"]& \AnMap(X_{\mathrm{B}}^{\mathrm{an}},\AnPerf)
\arrow[d,"\gamma_{\mathrm{B}}^*\circ (-)^{\mathrm{an}}"]\arrow[r,"\eta_{\mathrm{RH}}^*"]  & \AnMap(X_{\mathrm{DR}}^{\mathrm{an}},\AnPerf)\arrow[d,"\gamma_{\mathrm{DR}}^*"]
\\
\Map(X_{\mathrm{B}}^{\mathrm{conj}},\Perf)^{\mathrm{an}} \arrow[r,"\eqref{eqn: Map to AnMap}"] & \AnMap(X_{\mathrm{B}}^{\mathrm{an},\mathrm{conj}},\AnPerf)\arrow[r,"\overline{\eta}_{\mathrm{RH}}^*"] & \AnMap(X^{\mathrm{conj}}_{\mathrm{DR}},\AnPerf).
\end{tikzcd}
\end{equation}
This gives an equivalence of derived analytic stacks
$$\gamma_{\mathrm{DR}}:\RPerf(X_{\mathrm{DR}})^{\mathrm{an}}\simeq\RAnPerf(X_{\mathrm{DR}}^{\mathrm{an}})\xrightarrow{\simeq}\RPerf(X^{\mathrm{conj}})^{\mathbf{an}}\simeq \RAnPerf\big((X_{\mathrm{DR}}^{\mathrm{conj}})^{\mathrm{an}}\big).$$
There is a natural trivialization map away from the origin in $\mathbb{A}^1,$ denoted
$$\tau:\big(\Map_{/\Theta}\big(X_{\mathrm{Hod}},\RPerf\times\Theta)\times_{\mathbb{A}^1}\mathbf{G}_m\big)^{\mathrm{an}}\simeq \RAnPerf(X_{\mathrm{DR}}^{\mathrm{an}})\times (\mathbb{A}^1\backslash\{0\})^{\mathrm{an}},$$
which sends $(\mathcal{E}^{\bullet},\nabla_{\lambda})$ to $(\mathcal{E}^{\bullet},\lambda^{-1}\cdot \nabla^{\lambda},\lambda).$ There is an obvious natural map to $\mathbb{A}^1,$ given by $(\mathcal{E}^{\bullet},\lambda^{-1}\cdot \nabla^{\lambda},\lambda)\mapsto \lambda,$ whose inverse is given by rescaling $\lambda$ with the $\mathbb{G}_m$-action $(\nabla,\lambda)\mapsto \lambda\cdot \nabla,$ to be denoted $j:=j_{X_{\mathrm{Hod}}}.$ 
A similar morphism holds for the conjugate variety $X^{\mathrm{conj}},$
\begin{equation}
    \label{ClImmerHodge}
j_{\mathrm{conj}}:\RAnPerf(X_{\mathrm{DR}}^{\mathrm{an}})\times\mathbf{G}_m\to \RAnPerf_{\mathrm{Hod}}(X^{\mathrm{conj},\mathrm{an}}).
\end{equation}

Recall that $\RAnPerf(X_{\mathrm{Hod}}^{\mathrm{an}})\to \EuScript{A}^{\mathrm{an}}.$
The morphism $j$ and its conjugate morphism \eqref{ClImmerHodge} induce closed immersions $\EuScript{J}$ and $\EuScript{J}_{\mathrm{conj}}$ of derived analytic stacks, given by

\begin{equation}
     \label{ClImmersion1}
\begin{tikzcd}[column sep=6em]
    \RAnPerf(X_{\mathrm{B}}^{\mathrm{conj}})\times \mathbf{G}_m
\arrow[r,"j_{\mathrm{conj}}^{-1}\circ (\eta_{RH,\mathrm{conj}}^*\times \operatorname{id})"]
    \arrow[rr, bend left=13, "\EuScript{J}_{\mathrm{conj}}"]
    &
    \RAnPerf(X_{\mathrm{Hod}}^{\mathrm{conj}})\times_{\mathbb{A}^1}\mathbf{G}_m
    \arrow[r]
    &
    \RAnPerf(X_{\mathrm{Hod}}^{\mathrm{conj}})
\end{tikzcd}
\end{equation}
and by the composition
\begin{equation}
    \label{ClImmersion2}
    \begin{tikzcd}[column sep=6em]
    \RAnPerf(X_{\mathrm{B}})\times \mathbf{G}_m
    \arrow[r,"j_{X_{\mathrm{Hod}}}^{-1}\circ (\eta_{RH}^*\times \operatorname{id})"]
    \arrow[rr, bend right=10, "\EuScript{J}"]
    &
    \RAnPerf(X_{\mathrm{Hod}})\times_{\mathbb{A}^1}\mathbf{G}_m
    \arrow[r]
    &
    \RAnPerf(X_{\mathrm{Hod}}).
\end{tikzcd}
\end{equation}
Via the derived Riemann--Hilbert equivalence \eqref{DerRH}, we may use \eqref{eqn: Gamma_B} to glue the sources of the morphisms \eqref{ClImmersion1},\eqref{ClImmersion2}. Equivalently, one may use the induced conjugate de Rham diagram via \eqref{GluingDeRhamConjugates}.
We compute the homotopy-pushout in $\mathsf{dAnSt}_{\mathbb{C}/\mathbf{P}^1},$ along the closed-immersions
\begin{equation}
    \label{DelPushOutProof}
\begin{tikzcd}        \RAnPerf(X_{\mathrm{B}}^{\mathrm{conj}})\times (\mathbb{A}^1\backslash\{0\})^{\mathrm{an}}\overset{\eta_{RH}^{*,\mathrm{conj}}}{\simeq}\RAnPerf(X_{\mathrm{DR}}^{\mathrm{conj}})\times (\mathbb{A}^1\backslash\{0\})^{\mathrm{an}}\arrow[d,"\eqref{ClImmersion2}"]\arrow[r,"\eqref{ClImmersion1}"] & \RAnPerf_{\mathrm{Hod}}(\overline{X}^{\mathrm{an}})\arrow[d]
        \\
        \RAnPerf_{\mathrm{Hod}}(X)\arrow[r] & \RAnPerf_{\mathrm{Del}}(X).
    \end{tikzcd}
\end{equation}
We must prove that $\RPerf_{\mathrm{Del}}(X)$ is a derived analytic stack. This is achieved by analyzing the analytic Postnikov tower decomposition, which allows to construct the pushout of derived
analytic spaces along closed immersions. 

Since the pushout is formed in the category of derived analytic stacks, an open cover of the Deligne stack is induced by open covers of $\RPerf_{\mathrm{Hod}}(X)$ and $\RPerf_{\mathrm{Hod}}(X^{\mathrm{conj}})$ compatible over $\RAnPerf_{\mathrm{B}}(X)\times (\mathbb{A}^1\backslash\{0\})^{\mathrm{an}}.$ 
Hence we may assume each space is a derived affinoid or derived Stein space. By resentability criterion \cite[Thm. 7.1]{PortaYu2020}, the derived mapping stacks are covered by derived affinoid (or Stein) spaces, so are
locally of the form 
$\mathbb{R}\mathrm{Sp}(A)$ for a derived analytic algebra $A$. As $X$ is compact, by the proper mapping theorem for analytic stacks, the intersection is a closed derived affinoid subspace.
By local reductions, we first will prove that the pushout \eqref{DelPushOutProof} exists in the category of topologically ringed 
$\infty$-topoi \cite[Prop 6.2]{PortaYu2020}. Its underlying $\infty$-topos is the pushout of the underlying $\infty$-topoi. 
\begin{remark}[Notation]
    Given an affinoid algebra $A$, set $X=\mathsf{Sp}(A).$ Then, the \'etale topos is denoted $\EuScript{X}_{\mathsf{Sp}(A)}.$ This notation will be extended to our derived analytic stacks.
\end{remark}
Thus, we denote by $\EuScript{X}_{\mathrm{Betti}}:=\EuScript{X}_{\RAnPerf_{\B}(X)},$ the $\infty$-topose for the Betti stack of perfect complexes. Similarly, for notational simplicity, put $\EuScript{X}_{\mathrm{Hod}},\EuScript{X}_{\mathrm{Hod}^{\mathrm{conj}}}$ and $\EuScript{X}_{\mathrm{Del}},$ for the corresponding $\infty$-toposes. 
By definition of closed immersion (cf. Definition \ref{ClImmersions}), there are morphisms of topoi
$$j_{\mathrm{H},*}^{\mathrm{B}}:\EuScript{X}_{\mathrm{Betti}}\leftrightarrows \EuScript{X}_{\mathrm{Hod}^{\mathrm{conj}}}:j_{\mathrm{H},\mathrm{conj}}^{\mathrm{B},-1},\hspace{3mm}\text{ and }\hspace{1mm}
j_{H,*}^B:\EuScript{X}_{\mathrm{Betti}}\leftrightarrows \EuScript{X}_{\mathrm{Hod}}:j_{H}^{B,-1}.$$
By \cite[Lem.~6.5]{PortaYu2020}, we have that for each $n\geq 0$,
\begin{equation}
\label{CommuteTruncate}
j_{H,*}^B\circ \tau_{\leq n}^{\EuScript{X}_{\mathrm{Betti}}}\simeq \tau_{\leq n}^{\EuScript{X}_{\mathrm{Hod}^{\mathrm{conj}}}}:j_{H,*}^B,\hspace{5mm}\text{ and }j_{H,*}^B\circ \tau_{\leq n}^{\EuScript{X}_{\mathrm{Betti}}}\simeq \tau_{\leq n}^{\EuScript{X}_{\mathrm{Hod}}}:j_{H,*}^B.
\end{equation}
From \cite[Thm.~5.1]{Lurie2011}, the underlying $\infty$-topoi $\EuScript{X}_{\mathrm{Del}}$ of $\RAnPerf_{\mathrm{Del}}(X)$ is given by the pushout
\[
\begin{tikzcd}
    \EuScript{X}_{\mathrm{Betti}}\arrow[d,"j_{H,*}^B"]\arrow[r,"j_{H^{\mathrm{conj}},*}^B"]& \EuScript{X}_{\mathrm{Hod}^{\mathrm{conj}}}\arrow[d,"p_*"]
    \\
    \EuScript{X}_{\mathrm{Hod}}\arrow[r,"q_*"]& \EuScript{X}_{\mathrm{Del}}.
\end{tikzcd}
\]
By Proposition \ref{PullbackClosedImmersionLemma}, there is a corresponding pull-back 
\begin{equation}
    \label{StrSheafDiagram}
\begin{tikzcd}
\mathcal{O}_{\RAnPerf_{\Del}(X)}\arrow[d]\arrow[r] & \mathbb{R}p_*\mathcal{O}_{\mathbb{R}Hod}(X^{\mathrm{conj}})\arrow[d]
    \\
    \mathbb{R}q_*\mathcal{O}_{\RAnPerf_{\mathrm{Hod}}(X)}\arrow[r] & \mathbb{R}h_*\mathcal{O}_{\RAnPerf_{\mathrm{B}}(X)\times\mathbf{G}_m},
\end{tikzcd}
\end{equation}
where the pushout-maps $p:\RAnPerf_{\mathrm{Hod}}(X)\to \RAnPerf_{\mathrm{Del}}(X)$ and $q:\RAnPerf_{\mathrm{Hod}}(X^{\mathrm{conj}})\to \RAnPerf_{\mathrm{Del}}(X)$ satisfy 
$h=p\circ \iota=q\circ \iota^{conj}:\RAnPerf_{\mathrm{B}}(X)\times\mathbf{G}_m\to \RAnPerf_{\mathrm{Del}}(X).$
In particular, the structure sheaf $\mathcal{O}_{\RAnPerf_{\Del}(X)}$ is equivalent to 
\begin{equation}
    \label{eqn: DeligneDGSheaf}
\mathbb{R}q_*\mathcal{O}_{\RAnPerf_{\mathrm{\mathrm{Hod}}}(X^{\mathrm{an,conj}})}\otimes_{\mathbb{R}q_*\mathbb{R}j_{H^{\mathrm{conj}},*}^{\mathrm{B}}\mathcal{O}_{\RAnPerf_{\mathrm{B}}(X)\times\mathbb{G}_m}}\mathbb{R}p_*\mathcal{O}_{\RAnPerf_{\mathrm{Hod}}(X)}.\end{equation}
This homotopy pullback of sheaves induces a corresponding long exact sequence of homotopy groups by \cite[Thm.~5.1]{Lurie2011}.
Corresponding to \eqref{StrSheafDiagram},
it reads
\begin{align}
\label{LESDelign}
&\pi_1\big(\mathbb{R}p_*\mathcal{O}_{\RAnPerf_{\mathrm{Hod}}(X^{\mathrm{conj}})}^{\mathrm{alg}})\oplus \pi_1\big(\mathbb{R}q_*\mathcal{O}_{\RAnPerf_{\mathrm{Hod}}(X)}^{\mathrm{alg}}\big)\to  \pi_1\big(\mathbb{R}h_*\mathcal{O}_{\RAnPerf_{\mathrm{B}}(X)\times\mathbf{G}_m})
\nonumber
\\
&\to\pi_0\mathcal{O}_{\RAnPerf_{\mathrm{Del}}(X)}^{\mathrm{alg}}\rightarrow \to \pi_0\big(\mathbb{R}p_*\mathcal{O}_{\RAnPerf_{\mathrm{Hod}}(X^{\mathrm{conj}}))}^{\mathrm{alg}})\oplus\pi_0\big(\mathbb{R}q_*\mathcal{O}_{\RAnPerf_{\mathrm{Hod}}(X)}\big)\nonumber
\\
&\to \pi_0(\mathbb{R}h_*\mathcal{O}_{\RAnPerf_{\mathrm{B}}(X)\times\mathbf{G}_m}^{alg})\to 0.
\end{align}
Before we turn to this, note that since $t_0$ commute with colimits and preserves
closed immersions, we obtain a classical pushout square in classical analytic stacks,
\[
\begin{tikzcd}
   t_0(\mathbb{R}\AnPerf_{\B}(X))\times\mathbb{A}^1\backslash\{0\})\arrow[d]\arrow[r] & t_0(\RAnPerf_{\mathrm{Hod}}(X^{\mathrm{conj}}))\arrow[d]
   \\
   t_0(\RAnPerf_{\mathrm{Hod}}(X))\arrow[r] & t_0\big(\RAnPerf_{\Del}(X)\big))=:\underline{\mathrm{Del}}(X).
\end{tikzcd}
\]
Thus, $\underline{\mathrm{Del}}(X)$ exists as a $\mathbb{C}$-analytic space \cite[Prop. 6.2]{PortaYu2020}. Moreover by \eqref{eqn: DeligneDGSheaf}, its structure sheaf satisfies,
$$\mathcal{O}_{\underline{\mathrm{Del}}(X)}\simeq \mathbb{R}p_*\pi_0\mathcal{O}_{\RAnPerf_{\mathrm{Hod}}(X^{\mathrm{an,conj}})}\times_{\mathbb{R}h_*\pi_0\mathcal{O}_{\RAnPerf_{\mathrm{B}}(X)\times\mathbf{G}_m}}\mathbb{R}q_*\pi_0\mathcal{O}_{\RAnPerf_{\mathrm{Hod}}(X)}.$$
Since push-forwards of $\infty$-topoi along closed immersions commutes with truncations \cite[Prop.~7.3.2.5]{Lur09}, relations 
\eqref{CommuteTruncate}
directly imply that

$$\mathbb{R}p_*\pi_0\mathcal{O}_{\RAnPerf_{\mathrm{Hod}}(X^{\mathrm{an,conj}})}\simeq \pi_0\big(\mathbb{R}p_*\mathcal{O}_{\RAnPerf_{\mathrm{Hod}}(X^{\mathrm{an,conj}})}\big),$$
with similar commutation of $\pi_0$ with the other pushforwards $\mathbb{R}q_*$ and $\mathbb{R}h_*.$
Long-exact sequence \eqref{LESDelign},
gives a sequence 
$$0\to \mathcal{J}\to \pi_0\mathcal{O}_{\RAnPerf_{\mathrm{Del}}(X)}\to \mathcal{O}_{\underline{\mathrm{Del}}(X)}\to 0,$$
with cokernel
$$\mathcal{J}=\mathrm{coker}\big(\pi_1\mathbb{R}p_*\mathcal{O}_{\RAnPerf_{\mathrm{Hod}}(X^{\mathrm{conj}}))}\oplus \pi_1(\mathbb{R}q_*\mathcal{O}_{\RAnPerf_{\mathrm{Hod}}(X)}\to \pi_1\mathbb{R}h_*\mathcal{O}_{\RAnPerf_{\mathrm{B}}(X)\times\mathbf{G}_m}).$$
By similar reasoning, we have $\pi_1(\mathbb{R}p_*\mathcal{O})\simeq\mathbb{R}p_*(\pi_1\mathcal{O})$ as well as $\pi_1\mathbb{R}q_*\mathcal{O}\simeq\mathbb{R}q_*\pi_1\mathcal{O}$ and similarly for $h.$ In other words $\mathcal{O}_{\RAnPerf_{\mathrm{Hod}}(X^{\mathrm{conj}}))}$ and $\mathcal{O}_{\RAnPerf_{\mathrm{Del}}(X)}$ are structure sheaves of derived affinoid spaces. Moreover, $\pi_i$ of them are coherent modules over the classical analytic ring $\pi_0\mathcal{O}$ and since $\mathcal{J}$ is a coherent sheaf of $\mathcal{O}_{\underline{\mathrm{Del}}(X)}$-modules, $\mathcal{J}$ is nilpotent ideal since $\pi_0\mathcal{O}_{DelX'}$ and $\mathcal{O}_{\underline{\mathrm{Del}}(X)}$ have the same support. 
Therefore, $\pi_0\mathcal{O}_{\RAnPerf_{\mathrm{Del}}(X)}$ is an analytic square-zero extensio of $\mathcal{O}_{\underline{\mathrm{Del}}(X)}$ by the coherent sheaf $\mathcal{J}$. The sheaves $\pi_0\mathcal{O}_{\RAnPerf_{\mathrm{Del}}(X)}$ are structure sheaves of an analytic space (in the sense of Definition \ref{AnalyticSpace}), whose reduced space is $\underline{\mathrm{Del}}(X).$
Therefore, 
$$\big(\RAnPerf_{\mathrm{Del}}(X),\pi_0\mathcal{O}_{\RAnPerf_{\mathrm{Del}}(X)}\big),$$
is an analytic ring i.e. an element of $\EuScript{T}^{an}$-structure topos (in the sense of \ref{AnalyticTopos}).
Finally, to see $\pi_i(\mathcal{O}_{\RAnPerf_{\mathrm{Del}}(X)}$ is a coherent $\pi_0\mathcal{O}_{\RAnPerf_{\mathrm{Del}}(X)}$-module, for each $i,$ we may use the above long-exact sequence,
$$\cdots \pi_i\mathbb{R}p_*\mathcal{O}_{\RAnPerf_{\mathrm{Hod}}(X^{\mathrm{conj}}))}\oplus \pi_i\mathbb{R}q_*\mathcal{O}_{\RAnPerf_{\mathrm{Hod}}(X)}\to \pi_i\mathbb{R}h_*\mathcal{O}_{\RAnPerf_{\mathrm{B}}(X)\times\mathbf{G}_m}\to \pi_{i-1}\mathcal{O}_{\RAnPerf_{\mathrm{Del}}(X)}\to \dots.$$
Since $\pi_0\mathcal{O}_{\RAnPerf_{\mathrm{Del}}(X)}\to \pi_0\mathbb{R}p_*\mathcal{O}_{\RAnPerf_{\mathrm{Hod}}(X^{\mathrm{conj}}))}$ are epimorphisms, by surjectivity of the pushout maps, by induction on each $i$, and properness of the mapping stacks each $\pi_i\mathcal{O}_{\RAnPerf_{\mathrm{Del}}(X)}$ is a coherent $\pi_0\mathcal{O}_{\RAnPerf_{\mathrm{Del}}(X)}$-module, for ever $i$. Thus, we have prove by local affinoid reduction that $\RAnPerf_{\Del}(X)$ is derived analytic stack.
\medskip

\textit{Proof of (ii)}:
We now prove compatibility with real-structures and $\mathbb{G}_m$-actions (cf., Definition \ref{RealStructure}). To this end, restricting to the components of the pushout, we obtain maps
$$\eta_{\mathrm{Hod}}:=\eta_{\mathrm{Del}}|_{\RAnPerf_{\mathrm{Hod}}(X)}:\RAnPerf_{\mathrm{Hod}}(X)\to \mathbf{A}^1_{\mathbb{C}},$$
and similarly for $\eta_{\mathrm{Hod}}^{\mathrm{conj}}:\RAnPerf_{\mathrm{Hod}}(X^{\mathrm{conj}})\to \mathbf{A}^1_{\mathbb{C}}.$
Both structure maps are $\mathbf{G}_m$-equivariant and given $\beta\in\mathbf{G}_m,$ the conjugate element is $\beta^{\mathrm{conj}}=1/\beta.$ Such conjugate actions supply automorphisms,
$$\sigma_{\mathrm{Hod}}:\RAnPerf_{\mathrm{Hod}}(X)\to \RAnPerf_{\mathrm{Hol}}(X),\hspace{2mm}\sigma_{\mathrm{Hod}}^{\mathrm{conj}}:\RAnPerf_{\mathrm{Hod}}(X^{\mathrm{conj}})\to \RAnPerf_{\mathrm{Hod}}(X^{\mathrm{conj}}),$$
which are equivariant with respect to the natural actions on $\mathbf{A}_{\mathbf{C}}^1.$
In other words, there are commutative diagrams covering the natural involutions on $\mathbf{A}^1_{\mathbb{C}}$:
\begin{equation}
    \label{HodgeCommutingDiagrams}
\begin{tikzcd}
    \RAnPerf_{\mathrm{Hod}}(X)\arrow[d,"\eta_{\mathrm{Hod}}"] \arrow[r,"\sigma_{\mathrm{Hod}}"] & \RAnPerf_{\mathrm{Hod}}(X)\arrow[d,"\eta_{\mathrm{Hod}}"]
    \\
    \mathbf{A}_{\mathbb{C}}^1\arrow[r,"\sigma"]& \mathbf{A}_{\mathbb{C}}^1,
\end{tikzcd}\hspace{5mm}\begin{tikzcd}
    \RAnPerf_{\mathrm{Hod}}(X^{\mathrm{conj}})\arrow[d,"\eta_{\mathrm{Hod}}^{\mathrm{conj}}"] \arrow[r,"\sigma_{\mathrm{Hod}}^{\mathrm{conj}}"] & \RAnPerf_{\mathrm{Hod}}(X^{\mathrm{conj}})\arrow[d,"\eta_{\mathrm{Hod}}^{\mathrm{conj}}"]
    \\
    \mathbf{A}_{\mathbb{C}}^1\arrow[r,"\sigma"]& \mathbf{A}_{\mathbb{C}}^1.
\end{tikzcd}
\end{equation}
Consequently, by identifying $\mathbf{P}_{\mathbb{C}}^1$ with gluing two copies of $\mathbf{A}_{\mathbb{C}}^1$ with their conjugate variables, by the construction of the homotopy pushout \eqref{DeligneHoPushout},
the diagrams \eqref{HodgeCommutingDiagrams} glue to give an automorphism of $\RAnPerf_{\mathrm{Del}}(X)$, via the commutative diagram
\begin{equation}
    \label{DerDeligneCommuting}
\begin{tikzcd}
   \RAnPerf_{\mathrm{Del}}(X)\arrow[d,"\eta_{\mathrm{Del}}"] \arrow[r,"\sigma_{\mathrm{Del}}"] & \RAnPerf_{\mathrm{Del}}(X)\arrow[d,"\eta_{\mathrm{Del}}"]
    \\
    \mathbf{P}_{\mathbb{C}}^1\arrow[r,"\sigma"]& \mathbf{P}_{\mathbb{C}}^1,
    \end{tikzcd}
\end{equation}
covering the antipodal map on $\mathbf{P}^1,$ as required. Compose with the stack theoretic quotient $g:\mathbf{P}^1\to [\mathbf{P}^1/\mathbf{G}_m]$, gives the result, as the latter is precisely $\EuScript{P}^{\mathrm{an}}.$
\medskip

\textit{Proof of (iii)}:
Since truncation is compatible with taking colimits and with the derived Riemann-Hilbert correspondence, the derived pushout square \eqref{DeligneHoPushout} truncates to an ordinary pushout of schemes; there is an equivalence of classical analytic stacks,
\begin{eqnarray*}
t_0\big(\RAnPerf_{\mathrm{Del}}(X)\big)&\simeq&t_0\big(\RAnPerf_{\mathrm{Hod}}(X)\sqcup^{\mathbb{L}}_{\RAnPerf(X_{\mathrm{Betti}})}\RAnPerf_{\mathrm{Hod}}(X^{\mathrm{conj}})\big)
\\
&\simeq& t_0\big(\RAnPerf_{\mathrm{Hod}}(X)\sqcup_{t_0\RAnPerf_{\mathrm{Betti}}(X)}t_0(\RAnPerf_{\mathrm{Hod}}(X^{\mathrm{conj}}).
\end{eqnarray*}
By applying \cite[Prop~2.1]{STV}, we may again introduce the sub stack of Gieseker semistable objects of the truncations
$$\mathrm{M}_{\mathrm{Hod}}^{\mathrm{ss}}(X):=t_0\big(\RAnPerf_{\mathrm{Hod}}(X)\big)^{\mathrm{ss}}\subset t_0\big(\RAnPerf_{\mathrm{Hod}}(X)\big),$$
and similarly for 
$$\mathrm{M}_{\mathrm{Hod}}^{\mathrm{ss}}(X^{conj}):=t_0\big(\RAnPerf_{\mathrm{Hod}}(X^{\mathrm{conj}})^{ss}\subset t_0\big(\RAnPerf_{\mathrm{Hod}}(X^{\mathrm{conj}}).$$
They define good moduli spaces for the $1$-stacks $\mathcal{M}_{\Hod}(X)$ and $\mathcal{M}_{\Hod}(X^{\mathrm{conj}})$ as in 

It follows from \cite[Prop.~7.9]{alper2013good}, that we may repeat the argument in \cite[Prop.~3.3]{FrancoHanson2024} that states their pushout, which glue to give
$$t_0\big(\RAnPerf_{\Del}(X))^{\mathrm{ss}}\to \mathrm{Tw}\big(M_{\Dol}(X)\big),$$
is a good moduli space for the classical twistor space of $M_{\Dol}(X).$
\end{proof}

In Theorem \ref{DHSTwistor1}, we have established the derived Deligne-stack is a derived analytic space. To establish the existence of a relative shifted symplectic structure, we will first show it admits a (relative) cotangent complex. For this, note any locally geometric derived algebraic stack admits a global cotangent complex (e.g, \cite[Cor.~3.17]{toenvaquie2007}, whereby we may use Proposition \ref{AnDeRhamAlgs} and thus \cite[Thm.~7.1]{PortaYu2020}.
\begin{proposition}
\label{DHSTwistor2}
The homotopy pushout in derived analytic stacks $\RPerf_{\mathrm{Hod}}(X)^{\mathrm{an}}\sqcup_{\RAnPerf(X_{\mathrm{DR}}^{\mathrm{an}})}\RPerf_{\mathrm{Hod}}(X^{\mathrm{conj}})^{\mathrm{an}},$
is a locally geometric derived analytic stack over $\EuScript{P}^{\mathrm{an}}.$ 
\end{proposition}
\begin{proof}
    Let $\mathbb{R}\mathbf{M}_{\mathrm{Hod}}(X):=\big[\RAnPerf_{\mathrm{Hod}}(X^{\mathrm{an}})/\mathbb{G}_m^{\mathrm{an}}]\to \EuScript{A}^{\mathrm{an}},$ and consider 
    $j_X:\RAnPerf_{\mathrm{DR}}(X^{\mathrm{an}})\hookrightarrow \mathbb{R}\mathbf{M}_{\mathrm{Hod}}(X),$ as well as
    $$\RAnPerf_{\mathrm{DR}}(X)\xrightarrow{\gamma_{\mathrm{DR}}^*}\RAnPerf_{\mathrm{DR}}(X^{\mathrm{conj}})\xrightarrow{j_{X^{\mathrm{conj}}}}\mathbb{R}\mathbf{M}_{\mathrm{Hod}}(X^{\mathrm{conj}}).$$
Consider the stratification by tor-amplitude. We thus impost some conditions, define $\RPerf_{\mathrm{Hod}}^{\mathrm{ss}}(X)$ as the derived moduli functor whose $f_T:T\to \mathbb{A}^1,$ points are:
$$\RPerf_{\mathrm{Hod}}^{\mathrm{ss}}(X)(T)\simeq \big\{E\in \mathcal{O}_{X\times T}-\text{perfect }: \forall t\in T, \lambda\circ f_T(t)\neq0 E_t:=E|_{X\times\{t\}} ,H^i(E_t)\in\mathcal{M}_{\mathrm{Hod}}^{\mathrm{ss}}(X)\big\}.$$
This is an open condition \cite[Thm.~7.1]{Sim09}, therefore defines a locally geometric open substack,
$$\RPerf_{\mathrm{Hod}}^{\mathrm{ss}}(X)\hookrightarrow \RPerf_{\mathrm{Hod}}(X).$$
Since $\RPerf(X)\simeq \coprod_{p,q}\RPerf^{[p,q]}(X),$
one has a similar decomposition for $\RPerf_{\mathrm{Hod}}(X)$, and thus there is a covering by open substacks with fixed tor-amplitudes,
\[
\begin{tikzcd}
\RPerf_{\mathrm{Hod}}^{\mathrm{ss},[p,q]}(X)\arrow[d]\arrow[r] & \RPerf^{[p,q]}_{\mathrm{Hod}}(X)\arrow[d]
    \\
\RPerf_{\mathrm{Hod}}^{\mathrm{ss}}(X)\arrow[r] & \RPerf_{\mathrm{Hod}}(X).
\end{tikzcd}
\]
By relative GAGA, there are equivalences,
$\RPerf_{\mathrm{Hod}}^{\mathrm{ss},[p,q]}(X)^{\mathrm{an}}\simeq \RAnPerf_{\mathrm{Hod}}^{\mathrm{ss},[p,q]}(X^{\mathrm{an}}),$
and similarly the conjugate of $X.$ 
We claim the functor of analytification makes the following diagram commute
\[
\begin{tikzcd}
\RPerf_{\mathrm{Hod}}(X)\arrow[d,"(-)^{\mathrm{an}}"]\arrow[r,"\simeq"]& \prod_{p,q}\RPerf_{\mathrm{Hod}}(X)^{[p,q]}\arrow[d,"(-)^{\mathrm{an}}"]
    \\
    \RAnPerf_{\mathrm{Hod}}(X^{\mathrm{an}})\arrow[r,"\simeq"] & \prod_{p, q}\RAnPerf_{\mathrm{Hod}}(X^{\mathrm{an}})^{[p,q]}.
\end{tikzcd}
\]
Indeed, since analytification is t-exact and conservative, let $S$ be a test scheme and consider the induced functor of relative analytification \eqref{relAnalytic}, as above. Assume that $\mathcal{F}^{\mathrm{an}}$ is perfect of tor-amplitude $[p,q]$, relative to $S^{\mathrm{an}}.$ Then for any $\mathcal{G}\in \mathsf{Perf}^{\heartsuit}(S),$ we must show 
$\pi_i(\mathcal{F}\otimes f^*\mathcal{G})=0,$ for every $i\notin [p,q],$ where $f:X\to S$ is the induced map. By $t$-exactness and conservative properties of analytification, this is equivalent to checking that 
$\pi_i\big((\mathcal{F}\otimes f^*\mathcal{G})^{\mathrm{an}})\big)=0.$ However, since 
$$(\mathcal{F}\otimes f^*\mathcal{G})^{\mathrm{an}}\simeq \mathcal{F}^{\mathrm{an}}\otimes f^{\mathrm{an},*}(\mathcal{G}^{\mathrm{an}}),$$
the claim follows since $\mathcal{G}^{\mathrm{an}}\in \mathsf{Perf}^{\heartsuit}(S^{\mathrm{an}}).$ 

Since $\mathsf{dAnSt}_{\mathbb{C}}$ has all colimits, 
we may consider the homotopy-pushout diagram gluing the conjugate analytic Hodge stacks 
via the derived Riemann--Hilbert equivalence  
$\RAnPerf(X_{\mathrm{B}}^{\mathrm{an}})^{[p,q]}\simeq \RAnPerf(X_{\mathrm{DR}}^{\mathrm{an}})^{[p,q]}$, and we obtain 
$\RAnPerf_{\mathrm{DH}}^{\mathrm{ss}}(X^{\mathrm{an}})^{[p,q]}\to \mathbb{P}_{\mathbb{C}}^1.$
\end{proof}

\subsubsection{}
Let $X$ be a smooth proper scheme over $\mathbb{C}.$ A more general version of Theorem \ref{DHSTwistor1} can be obtained, which could be of interest for future applications. Namely, let $Z$ be a geometric derived stack locally finitely presented over $\mathbb{C}$, which is Tannakian such that $\Q(Z)\simeq \mathrm{Ind}(\Per(Z)).$ 
Assume the relative derived stack $\Map_{/\Theta}(X_{\Hod},Z)\to \Theta,$ is locally geometric, and call it the \emph{$Z$-valued (derived) $\lambda$-connections on $X$.} 
From \eqref{eqn: Map to AnMap}, there is a canonical map 
$$\big(\Map_{\Theta}(X_{\Hod},Z\times\Theta)\times_{\Theta}\mathbb{A}_{\mathbb{C}}^1\big)^{\an}\to \AnMap_{/\Theta^{\an}}(X_{\Hod}^{\an},Z^{\an}\times\Theta^{\an})\times_{\Theta^{\an}}\mathbf{A}^1,$$
which is an equivalence of relative derived analytic stacks by \cite{HolsteinPorta2025}.
By Theorem \ref{AnalyticNAH} (ii), we may glue along the generalized RH-correspondence to obtain:
\begin{equation}
    \label{GluingDeRhamConjugates2}
\begin{tikzcd}
\Map(X_{\mathrm{B}},Z)^{\mathrm{an}}\arrow[d,"\sim"]\arrow[r,"\eqref{eqn: Map to AnMap}"]& \AnMap(X_{\mathrm{B}}^{\mathrm{an}},Z^{\an})
\arrow[d,"\gamma_{\mathrm{B}}^*\circ (-)^{\mathrm{an}}"]\arrow[r,"\eta_{\mathrm{RH}}^*"]  & \AnMap(X_{\mathrm{DR}}^{\mathrm{an}},Z^{\an})\arrow[d,"\gamma_{\mathrm{DR}}^*"]
\\
\Map(X_{\mathrm{B}}^{\mathrm{conj}},Z)^{\mathrm{an}} \arrow[r,"\eqref{eqn: Map to AnMap}"] & \AnMap(X_{\mathrm{B}}^{\mathrm{an},\mathrm{conj}},Z^{\an})\arrow[r,"\overline{\eta}_{\mathrm{RH}}^*"] & \AnMap(X^{\mathrm{conj}}_{\mathrm{DR}},Z^{\an}).
\end{tikzcd}
\end{equation}
This gives an equivalence of derived analytic stacks
$$\Map(X_{\mathrm{DR}},Z)^{\mathrm{an}}\simeq\AnMap(X_{\mathrm{DR}}^{\mathrm{an}},Z^{\an})\simeq \AnMap\big((X_{\mathrm{DR}}^{\mathrm{conj}})^{\mathrm{an}},Z^{\an}\big).$$
The trivialization map away from the origin in $\mathbb{A}^1,$ is 
$$\tau:\big(\Map_{/\Theta}\big(X_{\mathrm{Hod}},Z\times\Theta)\times_{\mathbb{A}^1}\mathbf{G}_m\big)^{\mathrm{an}}\simeq \AnMap(X_{\mathrm{DR}}^{\mathrm{an}},Z^{\an}\times (\mathbb{A}^1\backslash\{0\})^{\mathrm{an}},$$
and there are analogs of the closed-immersions \eqref{ClImmersion1} and \eqref{ClImmersion2}. The corresponding homotopy pushout \eqref{DeligneHoPushout} gives a $\mathbf{G}_m$-equivariant map 
$\AnMap_{\Del}(X^{\an},Z^{\an})\to \bfP,$
with a relative $(n-2d)$-shifted symplectic structure.

As a final remark, we mention the possibility to further generalize the construction to obtain a family of derived analytic twistor spaces parametrized by complex structures by adapting \cite[Sect.~2]{Hu2024}.
\begin{remark}
   Replacing the fiber at $\infty$ by the Hodge stack of analytic perfect complexes with $\lambda$-connection over any $X'$ which lies in the Teichm\"uller space $\mathrm{Teich}(X)$ of complex structures on the underlying $C^{\infty}$-manifold of $X.$ Having fixed $X$, the analytic structure of such a generalized Deligne--
Hitchin twistor stack with $X'$ in $\mathrm{Teich}(X)$ which does not affect the $C^{\infty}$-structure and yields a family of relative derived analytic stacks parametrized by $\mathrm{Teich}(X).$ 
\end{remark}
\subsubsection{}

There are variants with fixed determinant, leading to rigidification of $\lambda$-perfect complexes. They should play an important role in stability questions, so we mention their construction for completeness.
\begin{remark}
Fixing $\mathcal{L}\in \mathsf{Pic},$ via the derived determinant construction of \cite{STV}, there are homotopy fiber products; note the first defines the second:
\[
\begin{tikzcd} \RPerf_{\mathrm{Hod}}^{\mathrm{si},>0}(X)_{\mathcal{L}} \ar[d] \ar[r] & \RPerf^{\mathrm{si},>0}(X)_{\mathcal{L}}\times\mathbb{A}^1\ar[d] \\ \RPerf_{\mathrm{Hod}}(X) \ar[r] & \RPerf(X)\times\mathbb{A}^1, \end{tikzcd}\hspace{5mm}
\begin{tikzcd} \RCoh_{\mathrm{Hod}}^{\mathrm{si},>0}(X)_{\mathcal{L}} \ar[d] \ar[r] & \RCoh_{\mathrm{Hod}}(X)\ar[d] \\ \RPerf_{\mathrm{Hod}}^{\mathrm{si},>0}(X)_{\mathcal{L}} \ar[r] & \RPerf_{\mathrm{Hod}}(X). \end{tikzcd}\]
\end{remark}

\begin{proposition}
\label{DHSTwistor3}
Consider Theorem \ref{DHSTwistor1}. Pullback along  $g:\mathbf{P}^1\to\EuScript{P}^{\mathrm{an}}$ gives a relative derived analytic stack to be  denoted by
\begin{equation}
    \label{TwistorPB}
\begin{tikzcd}
    \mathrm{Tw}^{\mathbb{R}}(\Perf)\arrow[d,"p"]\arrow[r,"q"] & \RAnPerf_{\mathrm{Del}}(X)\arrow[d,"\eta_{\mathrm{Del}}"]
    \\
\mathbf{P}^{1}\arrow[r,"g"] & \EuScript{P}^{\mathrm{an}},
\end{tikzcd}
\end{equation}
with shifted symplectic fibers compatible with the real structures given by diagram \eqref{DerDeligneCommuting}.

In other words, the derived relative analytic Deligne stack $\RAnPerf_{\Del}(X)\to \mathbf{P}^1$ has the structure of a $2(1-d)$-shifted analytic pretwistor structure of hyperkahler type. 
\end{proposition}
\begin{proof}
   Recall $\EuScript{P}$ is equivalent to the pushout $\EuScript{A}_0\sqcup_{\mathrm{Spec}(\mathbb{C})}\EuScript{A}_{\infty},$ where $\EuScript{A}_i$ are the copies at $0,\infty.$ Moreover, $\mathbb{P}^1\to \EuScript{P}$ is a $\mathbb{G}_m$-torsor over $\mathbb{P}^1\backslash \{0,\infty\},$ whose analytification is $g:\bfP\to \EuScript{P}^{\an}.$ In Theorem \ref{DHSTwistor1} (ii), we have verified the compatibility of real involutive structures via the commutative diagram \eqref{DerDeligneCommuting}. It remains to prove there exists a relative shifted-symplectic form, twisted by $\EuScript{O}_{\mathbf{P}^1}(2).$ 
To this end, since \eqref{TwistorPB} is a cartesian diagram in $\mathsf{dAnSt}_{/\EuScript{P}^{\mathrm{an}}},$ there is an induced pushout square in $\mathcal{O}_{\mathrm{Tw}^{\mathbb{R}}(\Perf)}-\mathsf{Mod}$ for analytic relative cotangent complexes, given by Proposition \ref{AnCotangentProperties}.
The computation is clear by relating $\mathbb{L}_{\RAnPerf_{\mathrm{Del}}(X)}^{\mathrm{an}}$ to the analytic cotangent complexes of the derived Hodge stack and of the derived Betti stack using \eqref{DeligneHoPushout}. 
Indeed, since the morphisms \eqref{ClImmersion1} and \eqref{ClImmersion2} are closed-immersions, by \cite[Cor.~5.33, Thm.~5.29]{PortaYu2020}, we have equivalences
$$\mathbb{L}_{\RAnPerf_{\mathrm{B}}(X)^{\mathrm{alg}}/\RAnPerf_{\mathrm{Hod}}(X^{\mathrm{conj}})^{\mathrm{alg}}}\xrightarrow{\simeq} \mathbb{L}_{\RAnPerf_{\mathrm{B}}(X)/\RAnPerf_{\mathrm{Hod}}(X^{\mathrm{conj}})}^{\mathrm{an}},$$
where $\RAnPerf_{\mathrm{B}}(X)^{\mathrm{alg}}$ are the $\EuScript{T}_{et}(\mathbb{C})$-structure topoi $(\EuScript{X}_{\mathrm{Betti}},\mathcal{O}^{\mathrm{alg}}).$

Moreover, since 
$\mathbb{L}_{\RAnPerf_{\mathrm{B}}(X)/\RAnPerf_{\mathrm{Hod}}(X^{\mathrm{conj}})}^{\mathrm{an}}$ is the analytic cotangent complex of the morphism 
$$\EuScript{J}_{\mathrm{conj}}^{-1}\mathcal{O}_{\RAnPerf_{\mathrm{Hod}}(X^{\mathrm{conj}})}\to \mathcal{O}_{\RAnPerf_{\mathrm{B}}(X)},$$
in $\mathsf{AnRing}_{\mathbb{C}}(\mathsf{Shv}(X_{\mathrm{top}}),$ since $\EuScript{J}_{\mathrm{conj}}$ is a closed-immersion, it is an effective epimorphism. Similarly for $\EuScript{J},$ thus as a homotopy push-out \eqref{DelPushOutProof} along closed immersions, we apply Theorem \ref{MainTheoremBodyBody} along each copy $\EuScript{A}_i$ of $\mathbf{A}^1,i=0,\infty.$

\end{proof}

Note that by Theorem \ref{DHSTwistor1}, together with Lemmas \ref{DHSTwistor2} and \ref{DHSTwistor3}, we may compute the fibers of the homotopy--pushout over $\EuScript{P}^{\mathrm{an}},$ as:
$$\begin{cases}
\AnPerf_{\mathrm{Dol}}(X),\hspace{7.5mm}\text{ fiber at } 0;
\\
\AnPerf_{\mathrm{DR}}(X),\hspace{7.5mm}\text{ fiber at } 1;
\\
\AnPerf_{\mathrm{Dol}}(X^{\mathrm{conj}}),\hspace{2mm}\text{ fiber at } \infty.
\end{cases}
$$
Moreover, there are equivalences
$$\RAnPerf_{\mathrm{Del}}(X^{\mathrm{an}})\times (\mathbf{P}^1\backslash\infty)\simeq \AnPerf_{\mathrm{Hod}}(X^{\mathrm{an}}),\hspace{3mm}\RAnPerf_{\mathrm{Del}}(X^{\mathrm{an}})\times (\mathbf{P}^1\backslash 0)\simeq \AnPerf_{\mathrm{Hod}}(X^{\mathrm{an,conj}}).$$
\begin{remark}
    The computation of the fibers, and the general theory of twistor spaces associated to a hyperk\"ahler manifold \cite{HKLR87}, indicate that the pull-back given by Proposition \ref{DHSTwistor3}, is really the derived twistor space of the derived analytic Dolbeault stack; we could write 
$\mathbf{Tw}^{\mathbb{R}}\big(\RAnPerf_{\mathrm{Dol}}(X^{\mathrm{an}})\big)\simeq \RAnPerf_{\mathrm{Del}}(X^{\mathrm{an}}),$ but refrain from using this notation as it could suggest a derived NAH correspondence at the global level has been established.
\end{remark}
\subsection{$\lambda\leftrightarrow t$ and Kapustin--Witten equations}
We conclude by discussing the derived analytic Deligne-stack of perfect complexes and $G$-bundles over a surface $S$, to a derived enhancement of the Kapustin--Witten moduli space as studied in \cite{LRT2022}.

Note that the constructions in the previous subsections hold for smooth projective manifolds of any dimension. For gauge-theoretic applications we consider the setting of Subsect.\ref{ssec: KW} and by specifying to the
surface case, we obtain a $(-2)$-shifted derived pre-twistor family of hyperk\"ahler type, whose fibers are described as follows.

\begin{proposition}
Consider $\mathsf{B}G$ for $G$ a connected complex reductive group or the analytification of a reductive algebraic group. There is a good moduli space for the classical twistor construction with $(-2)$-shifted symplectic fibers over $\lambda=0,1$. 
\end{proposition}
\begin{proof}
Considering $\EuScript{M}_{\mathrm{Hod}}(S,n):=\EuScript{M}_{\mathrm{Hod}}(S,\mathsf{B}GL_n),$ as before, by geometric invariant theory, these are good categorical quotients and hence 
the maps
$\mathrm{M}_{\mathrm{Hod}}^{ss}(S)\to \EuScript{M}_{\mathrm{Hod}}(S,n),$ and $\mathrm{M}_{\mathrm{Hod}}^{ss}(S^{\mathrm{conj}})\to \EuScript{M}_{\mathrm{Hod}}(S^{\mathrm{conj}},n)$ are good moduli spaces. Consequently, their colimit
$$t_0\RAnPerf_{\mathrm{Del}}(S)^{\mathrm{ss}}=:\mathrm{M}_{\mathrm{Del}}^{\mathrm{ss}}(S)\simeq \mathrm{M}_{\mathrm{Hod}}^{\mathrm{ss}}(S)\sqcup_{\RAnPerf_{\mathrm{B}}(S)\times\mathbf{G}_m} \mathrm{M}_{\mathrm{Hod}}^{\mathrm{ss}}(S^{\mathrm{conj}}),$$ is a good moduli space for $\mathrm{Tw}(\EuScript{M}_{\mathrm{Dol}}(S)).$
\end{proof}

\bibliographystyle{amsalpha}
\bibliography{Bibliography}
\end{document}